\documentclass[12pt,twoside]{amsart}

\usepackage{amssymb}
\usepackage{graphicx}
\usepackage[pdfauthor={David Cook II, Uwe Nagel},
            pdftitle={Enumerations of lozenge tilings, lattice paths, and perfect matchings and the weak Lefschetz property},
            pdfsubject={Commutative algebra and enumerative combinatorics},
            pdfkeywords={Monomial ideals, weak Lefschetz property, stable syzygy bundles, determinants, permanents,
                lozenge tilings, non-intersecting lattice paths, perfect matchings},
            pdfproducer={LaTeX with hyperref},
            pdfcreator={latex->dvips->ps2pdf},
            pdfpagemode=UseOutlines,
            bookmarksopen=false,
            letterpaper,
            bookmarksnumbered=true,
            pdfborder={0 0 0},
            colorlinks=false]{hyperref}
\usepackage[margin=1in]{geometry}

\makeindex
\newcommand{\emdx}[1]{\emph{#1}\index{#1}}

\newcommand{\dntri}{\bigtriangledown}
\newcommand{\uptri}{\triangle}
\newcommand{\HF}{\mathcal{H}}
\newcommand{\SO}{\mathcal{O}}

\newcommand{\PP}{\mathbb{P}}
\newcommand{\ZZ}{\mathbb{Z}}
\newcommand{\fa}{\mathfrak{a}}
\newcommand{\mo}{\mathfrak{o}}
\newcommand{\PS}{\mathfrak{S}}
\DeclareMathOperator{\CEKZ}{CEKZ}
\DeclareMathOperator{\charf}{char}

\DeclareMathOperator{\Mac}{Mac}
\DeclareMathOperator{\per}{perm}
\DeclareMathOperator{\rank}{rank}
\DeclareMathOperator{\reg}{reg}
\DeclareMathOperator{\sgn}{sgn}
\DeclareMathOperator{\msgn}{msgn}
\DeclareMathOperator{\lpsgn}{lpsgn}
\DeclareMathOperator{\soc}{soc}
\DeclareMathOperator{\syz}{syz}
\DeclareMathOperator{\Hom}{Hom}

\newcommand{\st}{\; \mid \;}
\def\urltilda{\kern -.15em\lower .7ex\hbox{\~{}}\kern .04em}
\newcommand{\flfr}[2]{\left\lfloor\frac{#1}{#2}\right\rfloor}
\newcommand{\clfr}[2]{\left\lceil\frac{#1}{#2}\right\rceil}

\numberwithin{figure}{section}
\numberwithin{equation}{section}

\newtheorem{theorem}{Theorem}[section]
\newtheorem{lemma}[theorem]{Lemma}
\newtheorem{proposition}[theorem]{Proposition}
\newtheorem{corollary}[theorem]{Corollary}
\newtheorem{conjecture}[theorem]{Conjecture}

\newtheorem{assumption}[theorem]{Assumption}
\theoremstyle{definition}
\newtheorem{definition}[theorem]{Definition}
\newtheorem{remark}[theorem]{Remark}
\newtheorem{example}[theorem]{Example}
\newtheorem{question}[theorem]{Question}
\newtheorem*{acknowledgement}{Acknowledgement}

\begin{document}

\title[Enumerations and the weak Lefschetz property]{Enumerations of lozenge tilings, lattice paths, and perfect matchings and the weak Lefschetz property}
\author[D.\ Cook II]{David Cook II${}^{\star}$}
\address{Department of Mathematics, University of Notre Dame, Notre Dame, IN 46556, USA}
\email{\href{mailto:dcook8@nd.edu}{dcook8@nd.edu}}
\author[U.\ Nagel]{Uwe Nagel}
\address{Department of Mathematics, University of Kentucky, 715 Patterson Office Tower, Lexington, KY 40506-0027, USA}
\email{\href{mailto:uwe.nagel@uky.edu}{uwe.nagel@uky.edu}}
\thanks{
    Part of the work for this paper was done while the authors were partially supported by the National Security Agency
    under Grant Number H98230-09-1-0032.
    The second author was also partially supported by the National Security Agency under Grant Number H98230-12-1-0247
    and by the Simons Foundation under grant \#208869.\\
    \indent ${}^{\star}$ Corresponding author.}

\begin{abstract}
    MacMahon enumerated the plane partitions in an $a \times b \times c$ box. These are in bijection  to lozenge tilings of a hexagon, to certain perfect matchings,
    and to families of non-intersecting lattice paths. In this work we consider more general regions, called triangular regions, and establish signed versions of
    the latter three bijections. Indeed, we use perfect matchings and families of non-intersecting lattice paths to define two signs of a lozenge tiling. A
    combinatorial argument involving a new method, called resolution of a puncture, then shows that the signs are in fact equivalent. This provides in
    particular two different determinantal enumerations of these families. These results are then applied to study the weak Lefschetz property of Artinian
    quotients by monomial ideals of a three-dimensional polynomial ring. We establish sufficient conditions guaranteeing the weak Lefschetz property as well
    as the semistability of the syzygy bundle of the ideal, classify the type two algebras with the weak Lefschetz property, and study monomial almost
    complete intersections in depth. Furthermore, we develop a general method that often associates to an algebra that fails the weak Lefschetz property a
    toric surface that satisfies a Laplace equation.  We also present examples of toric varieties that satisfy arbitrarily many Laplace equations. Our
    combinatorial methods allow us to address the dependence on the characteristic of the base field for many of our results.
\end{abstract}

\vspace*{-1\baselineskip}
\maketitle

\vspace*{-2\baselineskip}
\setcounter{tocdepth}{1}
\tableofcontents

\section{Introduction} \label{sec:intro}

A plane partition is a rectangular array of nonnegative integers such that the entries in each row and each column
are weakly decreasing. It is in an $a \times b \times c$ box if the array has $a$ rows, $b$ columns, and all entries
are at most $c$.  MacMahon \cite{MacMahon} showed that the number of plane partitions in an $a \times b \times c$ box is
\[
    \frac{\HF(a) \HF(b) \HF(c) \HF(a+b+c)}{\HF(a+b) \HF(a+c) \HF(b+c)},
\]
where $a$, $b$, and $c$ are nonnegative integers and $\HF(n) := \prod_{i=0}^{n-1}i!$ is the \emph{hyperfactorial} of
$n$. It is well-known that this result can be interpreted as counting the number of lozenge tilings of a hexagon with
side lengths $a, b, c$. Think of a plane partition as a stack of unit cubes, where the number of stacked cubes in
position $(i, j)$ is given by the corresponding entry in the array, as illustrated in Figure~\ref{fig:pp-tile-intro}.

\begin{figure}[!ht]
    \begin{minipage}[b]{0.48\linewidth}
        \[
            \vspace{1em}
            \begin{array}{cccccc}
               3 & 3 & 2 & 2 & 2 & 1 \\
               3 & 2 & 2 & 1 & 0 & 0
            \end{array}
        \]
    \end{minipage}
    \begin{minipage}[b]{0.48\linewidth}
        \centering
        \includegraphics[scale=1]{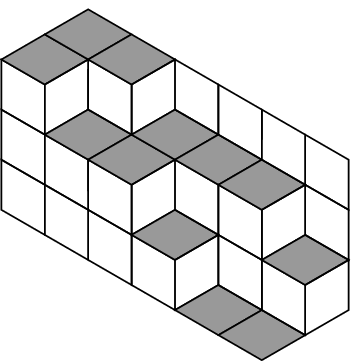}
    \end{minipage}
    \caption{A $2 \times 6 \times 3$ plane partition and the corresponding stack of cubes.  The grey
        lozenges are the tops of the boxes.}
    \label{fig:pp-tile-intro}
\end{figure}

The projection of the stack to the plane with normal vector $(1,1,1)$ gives a lozenge tiling of a hexagon with side
lengths $a, b, c$ (see Figure~\ref{fig:hex-tile-intro}). Here the hexagon is considered as a union of equilateral
triangles of side length one, and a lozenge is obtained by gluing together two such triangles along a shared edge.

\begin{figure}[!ht]
    \begin{minipage}[b]{0.48\linewidth}
        \centering
        \includegraphics[scale=0.5]{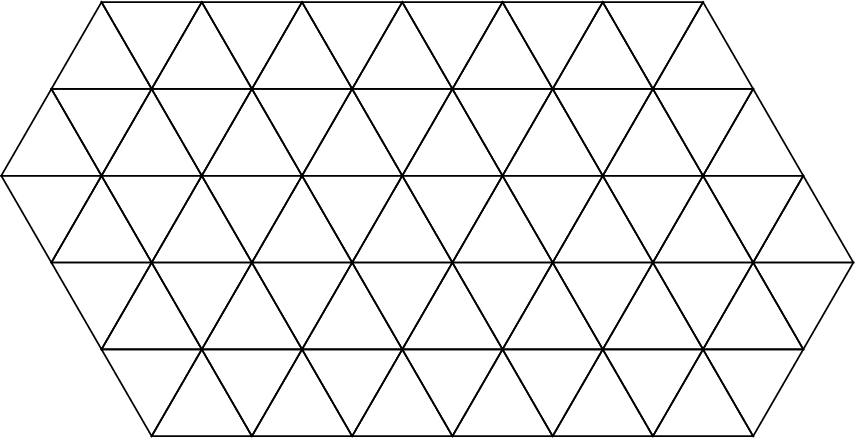}
    \end{minipage}
    \begin{minipage}[b]{0.48\linewidth}
        \centering
        \includegraphics[scale=1]{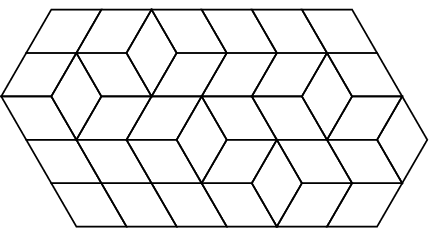}
    \end{minipage}
    \caption{A $2 \times 6 \times 3$ hexagon and the lozenge tiling associated to the plane partition
        in Figure~\ref{fig:pp-tile-intro}.}
    \label{fig:hex-tile-intro}
\end{figure}

In this work we view the above hexagon as a subregion of a triangular region $\mathcal{T}_d$, which is an equilateral
triangle of side length $d$ subdivided by equilateral triangles of side length one.  See Figure~\ref{fig:triregion-R-intro}
for an illustration.

\begin{figure}[!ht]
    \begin{minipage}[b]{0.32\linewidth}
        \centering
        \includegraphics[scale=0.7]{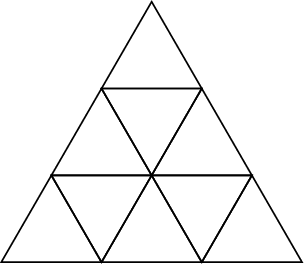}\\
    \end{minipage}
    \begin{minipage}[b]{0.32\linewidth}
        \centering
        \includegraphics[scale=0.7]{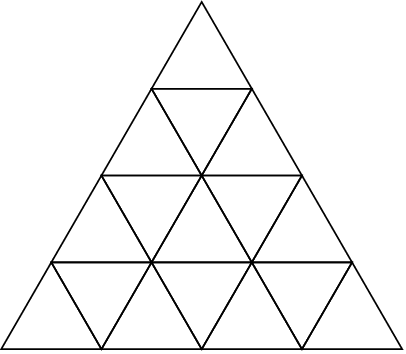}\\
    \end{minipage}
    \caption{The triangular regions $\mathcal{T}_3$ and $\mathcal{T}_4$.}
    \label{fig:triregion-R-intro}
\end{figure}

The hexagon with side lengths $a, b, c$ is obtained by removing triangles of side lengths $a, b$, and $c$ at the vertices
of $\mathcal{T}_d$, where $d = a + b + c$. We refer to the removed upward-pointing triangles as \emph{punctures}.
More generally, we consider subregions $T \subset \mathcal{T}$ that arise from $\mathcal{T}$ by removing upward-pointing
triangles, each of them being a union of unit triangles. The punctures, that is, the removed upward-pointing triangles may
overlap (see Figure \ref{fig:triregion-intro}). We call the resulting subregions of $\mathcal{T}$ \emph{triangular subregions}.
Such a region is said to be \emph{balanced} if it contains as many upward-pointing unit triangles as down-pointing pointing
unit triangles. For example, hexagonal subregions are balanced. Lozenge tilings of triangular subregions have been studied
in several areas. For example, they are used in statistical mechanics for modeling bonds in dimers (see, e.g., \cite{Ke})
or in statistical mechanics when studying phase transitions (see, e.g., \cite{Ci-2005}).

\begin{figure}[!ht]
    \begin{minipage}[b]{0.48\linewidth}
        \centering
        \includegraphics[scale=1]{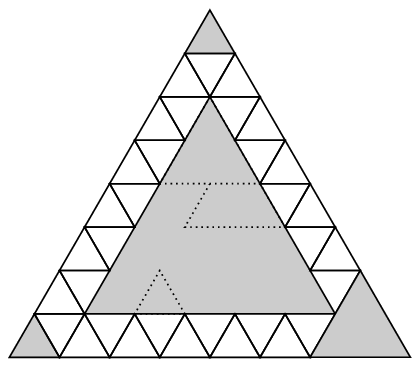}
    \end{minipage}
    \begin{minipage}[b]{0.48\linewidth}
        \centering
        \includegraphics[scale=1]{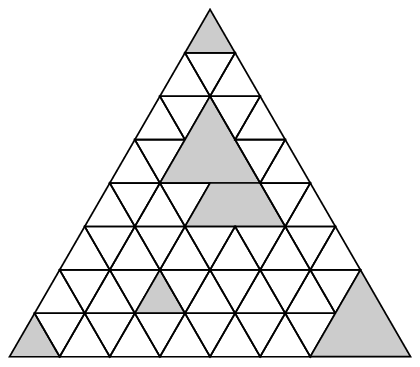}
    \end{minipage}
    \caption{Two triangular subregions.}
    \label{fig:triregion-intro}
\end{figure}

If a triangular subregion $T$ is a hexagon with side lengths $a, b, c$, then the plane partitions in an $a \times b \times c$
box are not only in bijection to lozenge tilings of $T$, but also to perfect matchings determined by $T$ as well as to families
of non-intersecting lattice paths in $T$ (see, e.g., \cite{Pr}). Moreover, all these objects are enumerated
by a determinant of an integer matrix.  For more general balanced triangular subregions,  the latter three bijections remain true,
whereas the bijection to plane partitions is lost.

\begin{figure}[!ht]
   \begin{minipage}[b]{0.42\linewidth}
        \centering
        \includegraphics[scale=1]{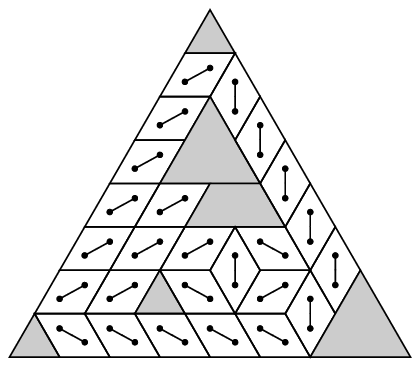}\\
        \emph{A perfect matching.}
    \end{minipage}
    \begin{minipage}[b]{0.48\linewidth}
        \centering
        \includegraphics[scale=1]{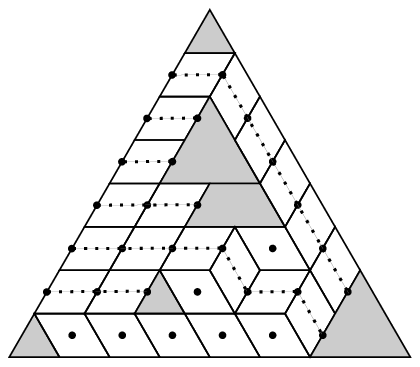}\\
        \emph{A family of non-intersecting lattice paths.}
    \end{minipage}
    \caption{Bijections to lozenge tilings.}
    \label{fig:bijections}
\end{figure}

Here we establish a signed version of these bijections. Introducing suitable signs, one of our main results says that,
for each balanced triangular subregion $T$, there is a bijection between the signed perfect matchings and the signed
families of non-intersecting lattice paths. This is achieved via the links to lozenge tilings.

Indeed, using the theory pioneered by Gessel and Viennot \cite{GV-85}, Lindstr\"om \cite{Li},  Stembridge \cite{Stembridge}, and Krattenthaler \cite{Kr-95}, the sets of signed families of non-intersecting lattice paths in $T$
can be enumerated by the determinant of a matrix $N(T)$ whose entries are binomial coefficients, once a suitable sign is
assigned to each such family. We define this sign as the \emph{lattice path sign} of the corresponding lozenge tiling of the region $T$.

The perfect matchings determined by $T$ can be enumerated by the permanent of a zero-one matrix $Z(T)$ that is the
bi-adjacency matrix of a bipartite graph. This suggests to introduce the sign of a perfect matching such that the
signed perfect matchings are enumerated by the determinant of $Z(T)$. We call this sign the \emph{perfect matching sign}
of the lozenge tiling that corresponds to the perfect matching. Typically, the matrix $N(T)$ is much smaller than the
matrix $Z(T)$. However, the entries of $N(T)$ can be much bigger than one. Nevertheless, a delicate combinatorial
argument shows that the perfect matching sign and the lattice path sign are equivalent, and thus (see Theorem \ref{thm:detZN})
\[
    |\det Z(T)| = |\det N(T)|.
\]
The proof also reveals instances where the absolute value of $\det Z(T)$ is equal to the permanent of $Z(T)$
(see Proposition~\ref{prop:same-sign}). This includes hexagonal regions, for which the result is well-known.

The above results allow us to obtain explicit enumerations in many new instances. They also suggest several intriguing conjectures.
\smallskip

Another starting point and motivation for our investigations has been the problem of deciding the presence of
the Lefschetz properties. A standard graded Artinian algebra $A$ over a field $K$ is said to have the \emph{weak Lefschetz property}
if there is a linear form $\ell \in A$ such that the multiplication map $\times \ell : [A]_i \rightarrow [A]_{i+1}$ has maximal rank
for all $i$ (i.e., it is  injective or surjective).  The algebra $A$ has the  \emph{strong Lefschetz property} if
$\times \ell^d : [A]_i \rightarrow [A]_{i+d}$ has maximal rank for all $i$ and $d$. The names are motivated by the conclusion of
the Hard Lefschetz Theorem on the cohomology ring of a compact K\"ahler manifold. Many algebras are expected to have the Lefschetz
properties. However, establishing this fact is often very challenging.

The Lefschetz properties play a crucial role in the proof of the so-called $g$-Theorem. It characterises the face vectors of simplicial
polytopes, confirming a conjecture of McMullen. The sufficiency of McMullen's condition was shown by Billera and Lee \cite{BL2} by
constructing suitable polytopes. Stanley \cite{St-faces}  established the necessity of the conditions by using the Hard Lefschetz
Theorem to show that the Stanley-Reisner ring of a simplicial polytope modulo a general linear system of parameters has the strong
Lefschetz property. It has been a longstanding conjecture whether McMullen's conditions also characterise the face vectors of all
triangulations of a sphere. This conjecture would follow if one can show that the Stanley-Reisner ring of such a triangulation
modulo a general linear system of parameters has the weak Lefschetz property.  The algebraic $g$-Conjecture posits that this algebra
even has the strong Lefschetz property. If true, this would imply strong restrictions on the face vectors of  all orientable
$K$-homology manifolds  (see \cite{NS1} and \cite{NS2}). Although there has been a flurry of papers studying the Lefschetz properties
in the last decade (see, e,g,  \cite{BMMNZ2, BK, BK-p, CGJL, GIV, HSS, HMMNWW, KRV, KV, LZ, MMO, MMN-2012}), we currently seem far
from being able to decide the above conjectures. Indeed, the need for new methods has led us to consider lozenge tilings, perfect
matchings, and families of non-intersecting lattice paths. We use this approach to establish new results about the presence or the
absence of the weak Lefschetz property of quotients of a polynomial ring $R = K[x, y, z]$ by a monomial ideal $I$ that contains
powers of each of the variables $x$, $y$, and $z$.

In the case where the ideal $I$ has only three generators, the powers of the variables, the algebra $R/I$ has the  Lefschetz
properties if the base field has characteristic zero (see \cite{Stanley-1980, ikeda, Wa, BTK}). In this case the algebra $R/I$
has Cohen-Macaulay type one. We extend this result in several directions.

First, one of the main results in \cite{BMMNZ} says that the monomial algebras $R/I$ of type two that are also level have the
weak Lefschetz property if $K$ has characteristic zero. Examples show that this may fail if one drops the level assumption or
if $K$ has positive characteristic. However, the intricate proof in the level case in \cite{BMMNZ} did not give any insight when
such failures occur. We resolve this by completely classifying all type two algebras that have the weak Lefschetz property if
the characteristic is zero or large enough (Theorem~\ref{thm:type-two} and Proposition~\ref{prop:char-0-to-p}).

Second, we consider the case where the ideal $I$ is an almost complete intersection, that is, $I$ is minimally generated by four
monomials. We decide the presence of the weak Lefschetz property in a broad range of cases, adding, for example, new evidence to a
conjecture in \cite{MMN-2011}. In particular, we show that the weak Lefschetz property may fail in at most one degree, that is, the
multiplication by a general linear form $[R/I]_{j-1} \to [R/I]_j$ has maximal rank for all but at most one integer $j$
(see Theorem~\ref{thm:amaci-wlp}).

Furthermore, we establish the weak Lefschetz property for various other infinite classes of algebras $R/I$, where the ideal $I$
can have arbitrarily many generators.

If an algebra that is expected to have the weak Lefschetz property actually fails to have it, this is often of interest too.
A projective variety is said to satisfy a Laplace equation of order $s$ if its $s$-th osculating space at a general (smooth) point
has smaller dimension than expected. Togliatti \cite{To} started investigating such varieties and obtained the first classification
results. Very recently, Mezzetti, Mir\'o Roig, and Ottaviani \cite{MMO} showed that the existence of Laplace equations is
closely related to the failure of the weak Lefschetz property. Using this, we prove that \emph{every}  Artinian monomial
ideal $I \subset R$ such that $R/I$ fails injectivity in degree $d-1$ as predicted by the weak Lefschetz property, that is,
the multiplication map $\times \ell: [R/I]_{d-1} \to [R/I]_d$ is not injective and $0 < \dim_K [R/I]_{d-1}  \le \dim_K [R/I]_{d}$,
gives rise to a toric surface satisfying a Laplace equation of order $d-1$ (see Theorem~\ref{thm:laplace-eq}).  Furthermore, we use
our approach via lozenge tilings to construct toric surfaces that satisfy any desired number of independent Laplace equations of
order $d-1$ whenever $d$ is sufficiently large (Corollary~\ref{cor:many-laplance-eq}).
\smallskip

The key to relating results on lozenge tilings to the study of the Lefschetz properties is to label the unit triangles in a
triangular region by monomials. This allows us to translate properties of a monomial ideal $I$ into properties of its
associated triangular subregions $T_d (I) \subset {\mathcal T}_d$. This is described in Section~\ref{sec:dictionary}.

In Section~\ref{sec:tiling} we establish sufficient and necessary conditions for a balanced triangular subregion to be
tileable (see, e.g., Theorem~\ref{thm:tileable}). Our arguments also give an algorithm for constructing a tiling of a
triangular subregion if any such tiling exists.

It turns out that the tileability of a triangular subregion $T_d (I)$ is related to the semistability of the syzygy bundle of
the ideal $I$. This is established in Section~\ref{sec:syz} (see Theorem~\ref{thm:tileable-semistable}).

Key results of our approach are developed in Section~\ref{sec:signed}. First, in Subsection~\ref{sub:pm}  we recall that
every non-empty subregion $T$ of ${\mathcal T}_d$ corresponds to a bipartite graph. We use this bijection to define the
bi-adjacency matrix $Z(T)$ and to introduce the perfect matching sign of a lozenge tiling. Second, we consider families of
non-intersecting lattice path in $T$ and introduce the lattice path matrix $N(T)$ as well as the lattice path sign of a
lozenge tiling in Subsection~\ref{sub:nilp}. In order to compare the perfect matching and the lattice path sign of lozenge
tilings we introduce a new combinatorial construction that we call \emph{resolution of a puncture} (see Subsection~\ref{sub:signs}).
Roughly speaking, it replaces a triangular subregion with a fixed lozenge tiling by a larger triangular subregion with a
compatible lozenge tiling and one puncture less. Carefully analyzing the change of sign under resolutions of punctures and
using induction on the number of punctures of a given region, we establish that, for each balanced triangular subregion,
the two defined signs of a lozenge tiling are in fact equivalent, and thus, $|\det N(T)| = |\det Z(T)|$. This results
allows us to move freely between signed perfect matchings and families of non-intersecting lattice paths.

In Section~\ref{sec:det}  we use this interplay and MacMahon's enumeration of plane partitions to establish various
explicit enumerations. We also give sufficient conditions that guarantee that all lozenge tilings of a triangular
subregion have the same sign (see Proposition~\ref{prop:same-sign}). In this case, the permanent of $Z(T)$, which gives
the total number of perfect matchings determined by $T$, is equal to $|\det Z(T)|$.

The special case of a mirror symmetric region is considered in Section~\ref{sec:mirror}. Using a result by Ciucu \cite{Ci-2005},
we provide some explicit enumerations of signed perfect matchings (see Theorems~\ref{thm:ciucu-11} and \ref{thm:ciucu-corrected}).
We also offer a conjecture (Conjecture~\ref{con:zero-mirror}) on the regularity of the bi-adjacency matrix of a mirror symmetric
region and provide evidence for it.

In the remainder of this work we apply the results on lozenge tilings to study the Lefschetz properties. In Section~\ref{sec:wlp}
we first present some general tools for establishing the weak Lefschetz property. Then we show that, for an Artinian monomial
ideal $I \subset R$, the rank of the multiplication $ \times \ell : [R/I]_{d-2} \rightarrow [R/I]_{d-1}$ by a general linear
form $\ell$ is governed by the rank of the bi-adjacency matrix $Z(T)$ (see Proposition~\ref{prop:interp-Z}) and the rank of
the lattice path matrix $N(T)$ (see Proposition~\ref{prop:interp-N}) of the region $T = T_d (I)$, respectively. Since these
are integer matrices, it follows that in the case where these matrices have maximal rank,  the prime numbers dividing  all
their maximal minors are the positive characteristics of the base field $K$ for which $R/I$ fails to have the weak Lefschetz
property. We first draw consequences to monomial complete intersections and then establish some sufficient conditions on a
monomial ideal $I$ such that $R/I$ has the weak Lefschetz property and the syzygy bundle of $I$ is semistable in characteristic
zero (see Theorem~\ref{thm:wlp-to-semistab}).

Algebras of type two are investigated in Section~\ref{sec:type-two}. Theorem~\ref{thm:type-two} gives the mentioned
classification of such algebras with the weak Lefschetz property in characteristic zero, extending the earlier result
for level algebras in \cite{BMMNZ}.

In Section~\ref{sec:amaci} we consider an  Artinian monomial ideal $I \subset R$ with four minimal generators. Our
results on the weak Lefschetz property of $R/I$ are summarised in Theorem~\ref{thm:amaci-wlp}. In particular, they
provide further evidence for a conjecture in \cite{MMN-2011}, which concerns the case where $R/I$ is also level. Furthermore,
we determine the generic splitting type of the syzygy bundle of $I$ in all cases but one (see Propositions~\ref{pro:st-nss}
and \ref{prop:splitt-type-semist}). In the remaining case we show that determining the generic splitting type is equivalent
to deciding whether $R/I$ has the weak Lefschetz property (see Theorem~\ref{thm:equiv}).

The results on varieties satisfying Laplace equations are established in Section~\ref{sec:failure}. They are based on
Proposition~\ref{cor:unit-reduc}. It says that in order to decide whether the bi-adjacency matrix $Z(T)$ of a triangular
region has maximal rank, it is enough to decide the same problem for a modification $\hat{T}$ of $T$ whose punctures all
have side length one. In Subsection~\ref{sub:large-p} we also give examples of balanced triangular subregions $T$ such that
its bi-adjacency matrix $Z(T)$ is regular and $\det Z(T)$ has remarkably large prime divisors. In fact, assuming a
number-theoretic conjecture by Bouniakowsky, we exhibit triangular subregions $T_d \subset {\mathcal T}_d$ such that
$\det Z(T_d) \neq 0$ has prime divisors of the order $d^2$.

We conclude by discussing some open problems that are motivated by this work in Section~\ref{sec:open-problems}.

\section{Ideals and triangular regions: a dictionary} \label{sec:dictionary}

In this section, we introduce a correspondence between monomial ideals and triangular regions. We define a few helpful
terms for triangular regions which allow us to interpret properties of monomial ideals as properties of triangular regions.

~\subsection{Monomial ideals in three variables}~\par

Let $R = K[x,y,z]$ be the standard graded polynomial ring over the field $K$, i.e., $\deg{x} = \deg{y} = \deg{z} = 1$.
Unless specified otherwise, $K$ is always an arbitrary field.

Let $I$ be a monomial ideal of $R$. As $R/I$ is standard graded, we can decompose it into finite vector spaces called
the \emdx{homogeneous components (of $R/I$) of degree $d$}, denoted $[R/I]_d$. For $d \in \ZZ$, the monomials of $R$ of degree $d$
that are \emph{not} in $I$ form a $K$-basis of $[R/I]_d$.

~\subsection{The triangular region in degree \texorpdfstring{$d$}{d}}\label{sub:trideg}~\par

Let $d \geq 1$ be an integer. Consider an equilateral triangle of side length $d$ that is composed of $\binom{d}{2}$
downward-pointing ($\dntri$) and $\binom{d+1}{2}$ upward-pointing ($\uptri$) equilateral unit triangles. We label the
downward- and upward-pointing unit triangles by the monomials in $[R]_{d-2}$ and $[R]_{d-1}$, respectively, as
follows: place $x^{d-1}$ at the top, $y^{d-1}$ at the bottom-left, and $z^{d-1}$ at the bottom-right, and continue
labeling such that, for each pair of an upward- and a downward-pointing triangle that share an edge, the label of the
upward-pointing triangle is obtained from the label of the downward-pointing triangle by multiplying with a variable.
The resulting labeled triangular region is the \emph{triangular region (of $R$) in degree $d$}%
\index{triangular region!of $R$ in degree $d$}
\index{0@\textbf{Symbol list}!Td@$\mathcal{T}_d$}
and is denoted $\mathcal{T}_d$. See Figure~\ref{fig:triregion-R} for an illustration.

\begin{figure}[!ht]
    \begin{minipage}[b]{0.32\linewidth}
        \centering
        \includegraphics[scale=1]{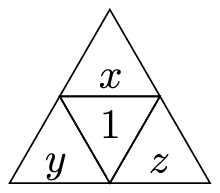}\\
        \emph{(i) $\mathcal{T}_2$}
    \end{minipage}
    \begin{minipage}[b]{0.32\linewidth}
        \centering
        \includegraphics[scale=1]{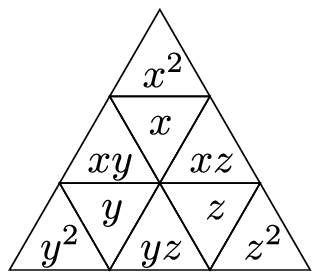}\\
        \emph{(ii) $\mathcal{T}_3$}
    \end{minipage}
    \begin{minipage}[b]{0.32\linewidth}
        \centering
        \includegraphics[scale=1]{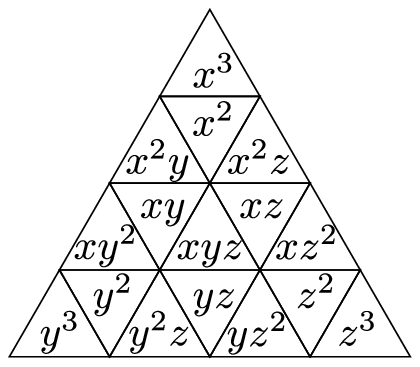}\\
        \emph{(iii) $\mathcal{T}_4$}
    \end{minipage}
    \caption{Some triangular regions $\mathcal{T}_d$.}
    \label{fig:triregion-R}
\end{figure}

Throughout this manuscript we order the monomials of $R$ with the \emdx{graded reverse-lexicographic order}, that is,
$x^a y^b z^c > x^p y^q z^r$ if either $a+b+c > p+q+r$ or $a+b+c = p+q+r$ and the \emph{last} non-zero entry in
$(a-p, b-q, c-r)$ is \emph{negative}. For example, in degree $3$,
\[
    x^3 > x^2y > xy^2 > y^3 > x^2z > xyz > y^2z > xz^2 > yz^2 > z^3.
\]
Thus in $\mathcal{T}_4$, see Figure~\ref{fig:triregion-R}(iii), the upward-pointing triangles are ordered starting at
the top and moving down-left in lines parallel to the upper-left edge.

We generalise this construction to quotients by monomial ideals. Let $I$ be a monomial ideal of $R$. The
\emph{triangular region (of $R/I$) in degree $d$},%
\index{triangular region!of $R/I$ in degree $d$}
\index{0@\textbf{Symbol list}!TdI@$T_d(I)$}
denoted by $T_d(I)$, is the part of $\mathcal{T}_d$ that is obtained
after removing the triangles labeled by monomials in $I$. Note that the labels of the downward- and
upward-pointing triangles in $T_d(I)$ form $K$-bases of $[R/I]_{d-2}$ and $[R/I]_{d-1}$, respectively. It is sometimes
more convenient to illustrate such regions with the removed triangles darkly shaded instead of being removed; both
illustration methods will be used throughout this manuscript. See Figure~\ref{fig:triregion-RmodI} for an example.

\begin{figure}[!ht]
    \includegraphics[scale=1]{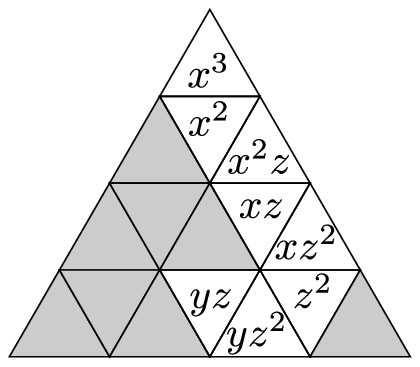}
    \caption{The triangular region $T_4(xy, y^2, z^3)$.}
    \label{fig:triregion-RmodI}
\end{figure}

Notice that the regions missing from $\mathcal{T}_d$ in $T_d(I)$ can be viewed as a union of (possibly overlapping)
upward-pointing triangles of various side lengths that include the upward- and downward-pointing triangles inside them.
Each of these upward-pointing triangles corresponds to a minimal generator of $I$ that has, necessarily, degree at most
$d-1$. We can alternatively construct $T_d(I)$ from $\mathcal{T}_d$ by removing, for each minimal generator $x^a y^b
z^c$ of $I$ of degree at most $d-1$, the \emph{puncture associated to $x^a y^b z^c$}%
\index{puncture!associated to a monomial}
which is an upward-pointing
equilateral triangle of side length $d-(a+b+c)$ located $a$ triangles from the bottom, $b$ triangles from the
upper-right edge, and $c$ triangles from the upper-left edge. See Figure~\ref{fig:triregion-punctures} for an example.
We call $d-(a+b+c)$ the \emph{side length of the puncture associated to $x^a y^b z^c$},%
\index{puncture!side length}
regardless of possible overlaps with other punctures in $T_d (I)$.

\begin{figure}[!ht]
    \begin{minipage}[b]{0.48\linewidth}
        \centering
        \includegraphics[scale=1]{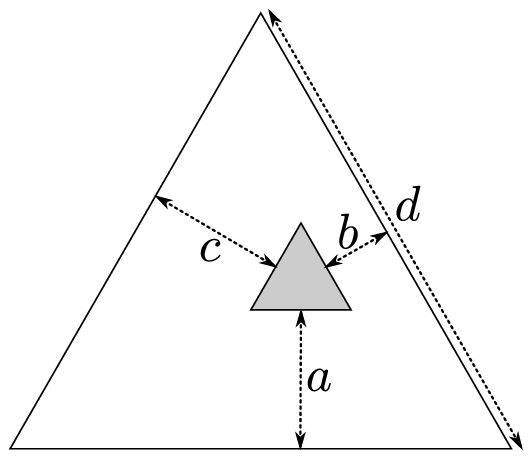}\\
        \emph{(i) $T_{d}(x^a y^b z^c)$}
    \end{minipage}
    \begin{minipage}[b]{0.48\linewidth}
        \centering
        \includegraphics[scale=1]{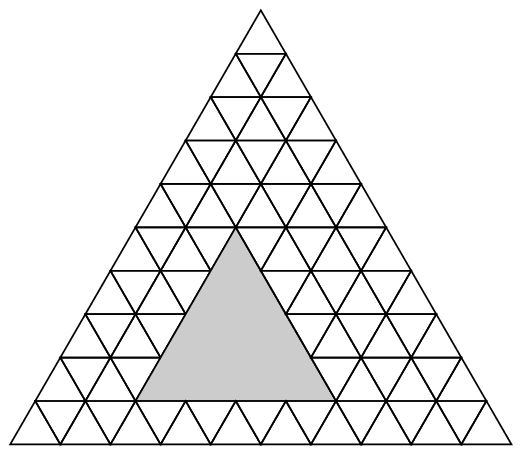}\\
        \emph{(ii) $T_{10}(xy^3z^2)$}
    \end{minipage}
    \caption{$T_d(I)$ as constructed by removing punctures.}
    \label{fig:triregion-punctures}
\end{figure}

We say that two punctures \emph{overlap}%
\index{puncture!overlapping}
if they share at least an edge. Two punctures are said to be \emph{touching}%
\index{puncture!touching}
if they share precisely a vertex.

~\subsection{The Hilbert function and \texorpdfstring{$T_d(I)$}{T\_d(I)}}~\par

Let $I$ be a monomial ideal of $R$. Recall that each component $[R/I]_d$ is a finite dimensional vector space. The
\emdx{Hilbert function} of $R/I$ is the function $h_{R/I}: \ZZ \to \ZZ$, where $h(d) := h_{R/I}(d) := \dim_K [R/I]_{d}$.%
\index{0@\textbf{Symbol list}!hRI@$h_{R/I}$}
By construction, $T := T_d(I)$ has $h(d-2)$ downward-pointing triangles and $h(d-1)$ upward-pointing triangles. Notice
also that $h(d)$ is the number of vertices in $T_d (I)$. Later it will become important to distinguish whether $h(d-2)$
and $h(d-1)$ are equal. We say $T$ is \emph{balanced}%
\index{triangular region!balanced}
if $h(d-2) = h(d-1)$, and otherwise we say $T$ is \emph{unbalanced}. Moreover, for $T$ unbalanced, if $h(d-2) < h(d-1)$,
then we say $T$ is \emph{$\dntri$-heavy}, and otherwise we say $T$ is \emph{$\uptri$-heavy}.
\index{triangular region!$\dntri$- or $\uptri$-heavy}

~\subsection{Socle elements}\label{sub:socle}~\par

Let $I$ be a monomial ideal of $R$. The quotient ring $R/I$ or simply $I$ is called \emph{Artinian} if $R/I$ is a finite
$K$-vector space. In the language of triangular regions, this translates as $R/I$ is Artinian if and only if $T_d(I)$
has a puncture in each corner of $\mathcal{T}_d$ for some $d$.

The \emdx{socle} of $R/I$ is the annihilator of $\mathfrak{m} = (x,y,z)$, the homogeneous maximal ideal of $R$, that is,%
\index{0@\textbf{Symbol list}!socRI@$\soc{R/I}$}
$\soc{R/I} = \{ f \in R/I \st fx = fy = fz = 0 \}$. As $I$ is a monomial ideal, $\soc{R/I}$ can be generated by
monomials. The monomials $m \in \soc{R/I}$ of degree $m-2$ are precisely those that are the center of ``triads'' in $T_d(I)$.%
\index{puncture!triad} \index{triad|see{puncture}}
See Figure~\ref{fig:socle-element} for an illustration of such a triad.

\begin{figure}[!ht]
    \includegraphics[scale=1.75]{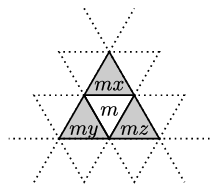}
    \caption{A triad around the socle element $m$.}
    \label{fig:socle-element}
\end{figure}

It is often important to determine the minimal degree, or bounds thereon, of socle elements of $R/I$. If the punctures
of $T_d(I)$ corresponding to the minimal generators of $I$ do not overlap, then the minimal degree of a socle element is
$d-2$ provided $T_d(I)$ contains a triad, otherwise the minimal degree of a socle element of $R/I$ is at least $d-1$. On
the other hand, if $T_d(I)$ has overlapping punctures, then the degrees of socle generators cannot be immediately
estimated.

If $R/I$ is Artinian, then the least degree $j$ such that $[R/I]_j \neq 0$ is called the \emph{socle degree}%
\index{socle!degree}
or Castelnuovo-Mumford regularity of $R/I$. The \emph{type}%
\index{socle!type}
of $R/I$ is the $K$-dimension of $\soc{R/I}$. Notice that $[R/I]_e \subset [\soc{R/I}]_e$ if $e$ is the socle
degree of $R/I$. Further, $R/I$ is said to be \emdx{level} if its socle is concentrated in one degree, i.e., in
its socle degree.

~\subsection{Greatest common divisors}~\par\label{subsec:gcd}

Let $I$ be a monomial ideal of $R$ minimally generated by the monomials $f_1, \ldots, f_n$. Without loss of generality,
assume $f_1, \ldots, f_m$ have degrees bounded above by $d-1$. Set $g = \gcd\{f_1, \ldots, f_m\}$. In
$\mathcal{T}_d$, the puncture associated to $g$ is exactly the smallest upward-pointing triangle that contains the
punctures associated to $f_1, \ldots, f_m$. See Figure~\ref{fig:triregion-gcd} for an example.

\begin{figure}[!ht]
    \begin{minipage}[b]{0.48\linewidth}
        \centering
        \includegraphics[scale=1]{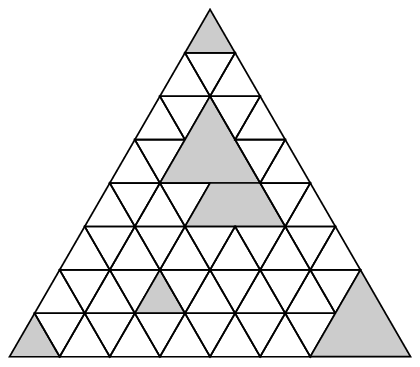}\\
        \emph{(i) $T_{8}(x^7, y^7, z^6, \underbrace{x y^4 z^2, x^3 y z^2, x^4 y z})$}
    \end{minipage}
    \begin{minipage}[b]{0.48\linewidth}
        \centering
        \includegraphics[scale=1]{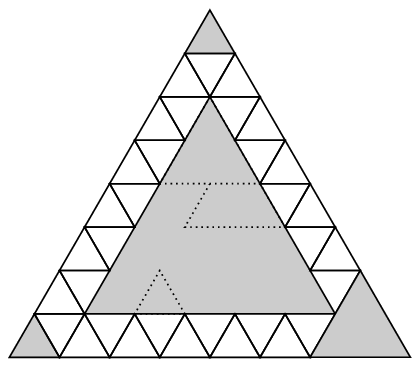}\\
        \emph{(ii) $T_{8}(x^7, y^7, z^6, xyz)$}
    \end{minipage}
    \caption{The greatest common divisor is associated with the minimal containing puncture.}
    \label{fig:triregion-gcd}
\end{figure}

The monomial ideal $J = (I, g)$ is minimally generated by $g$ and $f_{m+1}, \ldots, f_n$. Its triangular region $T_d(J)$
is obtained from $T_d(I)$ by replacing the punctures associated to $f_1, \ldots, f_m$ by their smallest containing
puncture in $T_d(I)$. This replacing operation can be reversed. A given puncture can be broken into smaller punctures
whose smallest containing puncture is the primal puncture. This corresponds to replacing a minimal generator in a
monomial ideal by several multiples whose greatest common divisor is the generator.

Observe that different monomial ideals can determine the same monomial region of $\mathcal{T}_d$. Consider, for example,
$I_1 = (x^5, y^5, z^5, xyz^2, xy^2z, x^2yz)$ and $I_2 = (x^5, y^5, z^5, xyz)$. Then $T_6 (I_1) = T_6 (I_2)$. However,
given a triangular region $T = T_d (I)$, there is a unique largest ideal $J$ that is generated by monomials whose
degrees are bounded above by $d-1$ and that satisfies $T = T_d (J)$. We call $J(T) := J$ the \emph{monomial ideal of the
triangular region $T$}.%
\index{triangular region!monomial ideal of}
\index{0@\textbf{Symbol list}!JT@$J(T)$}
In the example, $I_2 = J(T_6 (I_1))$.

Recall that each monomial of degree less than $d$ determines a puncture in $\mathcal{T}_d$. Thus, the punctures of a
monomial ideal $I \subset R$ in $\mathcal{T}_d$ correspond to the minimal generators of $I$ of degree less than $d$.
However, the punctures of the triangular region $T = T_d(I)$ correspond to the minimal generators of $J(T)$. In the
above example, $I_1$ determines six punctures in $T = T_6(I_1)$, but the region $T$ has four punctures.

\section{Tilings with lozenges} \label{sec:tiling}

In this section, we consider the question of tileability of a triangular region. Here we use monomial ideals merely as a
bookkeeping tool in order to describe the considered regions. If possible we want to tile such a region by lozenges. A
\emdx{lozenge} is a union of two unit equilateral triangles glued together along a shared edge, i.e., a rhombus with
unit side lengths and angles of $60^{\circ}$ and $120^{\circ}$. Lozenges are also called calissons and diamonds in the
literature.

Fix a positive integer $d$ and consider the triangular region $\mathcal{T}_d$ as a union of unit triangles. Thus a \emph{subregion}%
\index{triangular region!subregion}
$T \subset \mathcal{T}_d$ is a subset of such triangles. We retain their labels. As above, we say that
a subregion $T$ is \emph{$\dntri$-heavy}, \emph{$\uptri$-heavy}, or \emph{balanced} if there are more downward pointing
than upward pointing triangles or less, or if their numbers are the same, respectively. A subregion is \emph{tileable}%
\index{triangular region!tileable}
if either it is empty or there exists a tiling of the region by lozenges such that every triangle is part of exactly one
lozenge. See Figure~\ref{fig:triregion-tiling}.
\begin{figure}[!ht]
    \includegraphics[scale=1]{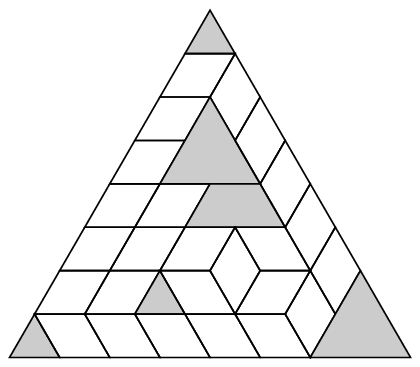}
    \caption{One of $13$ tilings of $T_{8}(x^7, y^7, z^6, x y^4 z^2, x^3 y z^2, x^4 y z)$ (see Figure~\ref{fig:triregion-gcd}(i)).}
    \label{fig:triregion-tiling}
\end{figure}
Since a lozenge in $\mathcal{T}_d$ is the union of a downward-pointing and an upward-pointing triangle, and every
triangle is part of exactly one lozenge, a tileable subregion is necessarily balanced.

Let $T \subset \mathcal{T}_d$ be any subregion. Given a monomial $x^a y^b z^c$ with degree less than $d$, the
\emph{monomial subregion}%
\index{triangular region!subregion associated to a monomial}
of $T$ associated to $x^a y^b z^c$ is the part of $T$ contained in the triangle $a$ units from
the bottom edge, $b$ units from the upper-right edge, and $c$ units from the upper-left edge. In other words, this
monomial subregion consists of the triangles that are in $T$ and the puncture associated to the monomial $x^a y^b z^c$.
See Figure~\ref{fig:triregion-subregion} for an example.

\begin{figure}[!ht]
    \includegraphics[scale=1]{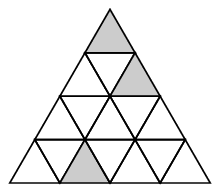}
    \caption{The monomial subregion of $T_{8}(x^7, y^7, z^6, x y^4 z^2, x^3 y z^2, x^4 y z)$
        (see Figure~\ref{fig:triregion-gcd}(i)) associated to $x y^2 z$.}
    \label{fig:triregion-subregion}
\end{figure}

Replacing a tileable monomial subregion by a puncture of the same size does not alter tileability.

\begin{lemma} \label{lem:replace-tileable}
    Let $T \subset \mathcal{T}_d$ be any subregion.  If the monomial subregion $U$ of $T$ associated to $x^a y^b z^c$ is tileable,
    then $T$ is tileable if and only if $T \setminus U$ is tileable.

    Moreover, each tiling of $T$ is obtained by combining a tiling of $T \setminus U$ and a tiling of $U$.
\end{lemma}
\begin{proof}
    Suppose $T$ is tileable, and let $\tau$ be a tiling of $T$.  If a tile in $\tau$ contains a downward-pointing triangle of $U$, then the
    upward-pointing triangle of this tile also is in $U$. Hence, if any lozenge in $\tau$ contains exactly one triangle of $U$, then
    it must be an upward-pointing triangle. Since $U$ is balanced, this would leave $U$ with a downward-pointing triangle that is
    not part of any tile, a contradiction. It follows that $\tau$ induces a tiling of $U$, and thus $T \setminus U$ is tileable.

    Conversely, if $T \setminus U$ is tileable, then a tiling of $T \setminus U$ and a tiling of $U$ combine to a tiling of $T$.
\end{proof}

Let $U \subset \mathcal{T}_d$ be a monomial subregion, and let $T, T' \subset \mathcal{T}_d$ be any subregions such that
$T \setminus U = T' \setminus U$. If $T \cap U$ and $T' \cap U$ are both tileable, then $T$ is tileable if and only if
$T'$ is, by Lemma \ref{lem:replace-tileable}. In other words, replacing a tileable monomial subregion of a triangular
region by a tileable monomial subregion of the same size does not affect tileability.

Using the above observation to reduce to the simplest case, we find a tileability criterion of triangular regions
associated to monomial ideals.

\begin{theorem} \label{thm:tileable}
    Let $T = T_d(I)$ be a balanced triangular region, where $I \subset R$ is any monomial ideal.  Then $T$ is tileable if and only if
    $T$ has no $\dntri$-heavy monomial subregions.
\end{theorem}
\begin{proof}
    Suppose $T$ contains a $\dntri$-heavy monomial subregion $U$.  That is, $U$ has more downward-pointing triangles than upward-pointing
    triangles.  Since the only triangles of $T \setminus U$ that share an edge with  $U$ are downward-pointing triangles, it is impossible to cover every
    downward-pointing triangle of $U$ with a lozenge.  Thus, $T$ is non-tileable.

    Conversely,  suppose $T$ has no $\dntri$-heavy monomial subregions.  In order to show that $T$ is tileable, we may also assume
    that $T$ has no non-trivial tileable monomial subregions by Lemma~\ref{lem:replace-tileable}.

    Consider any pair of touching or overlapping punctures in $\mathcal{T}_d$. The smallest monomial subregion $U$ containing both punctures
    is tileable. (In fact, such a monomial region is uniquely tileable by lozenges.)
    If further triangles stemming from other punctures of $T$ have been removed from $U$, then the resulting region
    $T \cap U$ becomes $\dntri$-heavy or empty. Thus, our assumptions imply that $T$ has no overlapping and no touching
    punctures.

    Now we proceed by induction on $d$. If $d \leq 2$, then $T$ is empty or consists of one lozenge.  Thus, it is tileable.
    Let $d \geq 3$, and let $U$ be the monomial subregion of $T$ associated to $x$, i.e., $U$ consists of  the upper $d-1$
    rows of $T$.  Let $L$ be the bottom row of $T$.  If $L$ does not contain part of a puncture of $T$, then $L$ is
    $\uptri$-heavy forcing $U$ to be a $\dntri$-heavy monomial subregion, contradicting an assumption on $T$.  Hence, $L$
    must contain part of at least one puncture of $T$.  See Figure~\ref{fig:thm-tileable}(i).

    \begin{figure}[!ht]
        \begin{minipage}[b]{0.48\linewidth}
            \centering
            \includegraphics[scale=1]{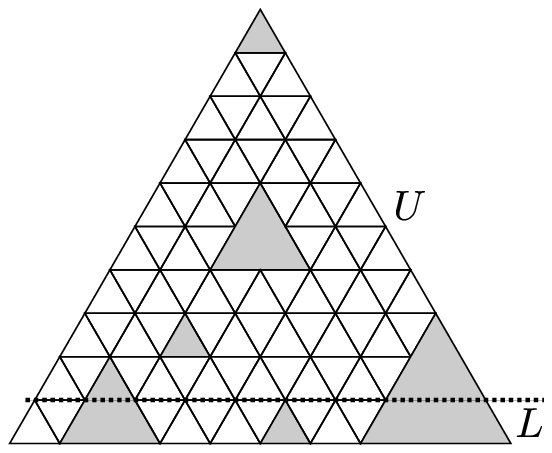}\\
            \emph{(i) The region $T$ split in to $U$ and $L$.}
        \end{minipage}
        \begin{minipage}[b]{0.48\linewidth}
            \centering
            \includegraphics[scale=1]{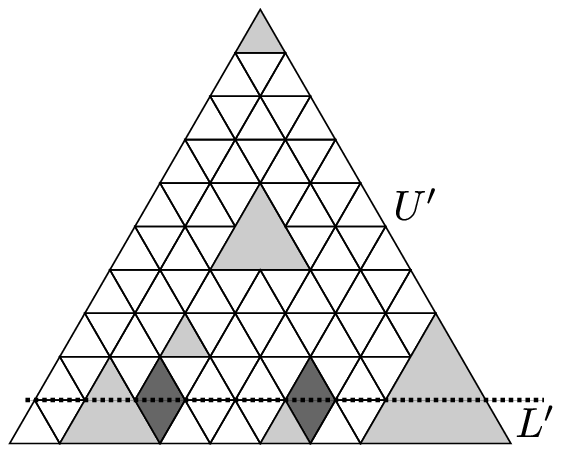}\\
            \emph{(ii) Creating $U'$ and $L'$.}
        \end{minipage}
        \caption{Illustrations for the proof of Theorem~\ref{thm:tileable}.}
        \label{fig:thm-tileable}
    \end{figure}

    Place an up-down lozenge in $T$ just to the right of each puncture along the bottom row \emph{except} the farthest right puncture.
    Notice that putting in all these tiles is  possible since punctures are non-overlapping and non-touching.  Let $U' \subset U$
    and $L' \subset L$ be the subregions that are obtained by removing the relevant upward-pointing and  downward-pointing triangles
    of the added lozenges from $U$ and $L$, respectively.  See Figure~\ref{fig:thm-tileable}(ii). Notice, $L'$ is uniquely tileable.

    As $T$ and $L'$ are balanced,  so is $U'$.   Assume $U'$ contains a monomial subregion $V'$ that is $\dntri$-heavy.
    Then $V' \neq U'$, and hence $V'$ fits into a triangle of side length $d-2$. Furthermore, the assumption on $T$ implies
    that $V'$ is not a monomial subregion of $U$. In particular, $V'$ must be located at the bottom of $U'$. Let $\tilde{V}$
    be the smallest monomial subregion of $U$ that contains $V'$. It is obtained from $V'$ by adding suitable upward-pointing
    triangles that are parts of the added lozenges.  Expand $\tilde{V}$ down one row to a monomial subregion $V$ of $T$.
    Thus, $V$ fits into a triangle of side length $d-1$ and is not $\dntri$-heavy. If $V$ is balanced, then, by induction,
    $V$ is tileable.  However, we assumed $T$ contains no such non-trivial regions.  Hence, $V$ is $\uptri$-heavy.
    Observe now that the region $V \cap L'$ is either balanced or has exactly one more upward-pointing triangle than
    downward-pointing triangles.  Since $V'$ is obtained from $V$ by removing $V \cap L$ and some of the added lozenges,
    it follows that $V'$ cannot be $\dntri$-heavy, a contradiction.

    Therefore, we have shown that each monomial subregion of $U'$ is not $\dntri$-heavy. By induction on $d$, we conclude
    that $U'$ is tileable. Using the lozenges already placed, along with the tiling of $L'$, we obtain a tiling of $T$.
\end{proof}

\begin{remark} \label{rem:complexity}
    The preceding proof yields a recursive construction of a canonical tiling of the triangular region.  In fact, the tiling can
    be seen as minimal, in the sense of Subsection~\ref{sub:nilp}.  Moreover, the theorem yields an exponential (in the number of
    punctures) algorithm to determine the tileability of a region.

    Thurston~\cite{Th} gave a linear (in the number of triangles) algorithm to determine the tileability of a \emph{simply-connected region},%
    \index{triangular region!simply-connected}
    i.e., a region with a polygonal boundary.  Thurston's algorithm also yields a minimal canonical tiling.
\end{remark}

Let $I$ be a monomial ideal of $R$ whose punctures in $\mathcal{T}_d$ (corresponding to the minimal generators of $I$ having degree less than $d$)
have side lengths that sum to $m$. Then we define the \emdx{over-puncturing coefficient} of $I$ in degree $d$ to be $\mo_d (I) = m - d$.
If $\mo_d (I) < 0$, $\mo_d (I) = 0$, or $\mo_d (I) > 0$, then we call $I$ \emph{under-punctured}, \emph{perfectly-punctured}, or
\emph{over-punctured} in degree $d$, respectively.
\index{over-puncturing coefficient!over-punctured}\index{over-puncturing coefficient!under-punctured}\index{over-puncturing coefficient!perfectly-punctured}
\index{0@\textbf{Symbol list}!oI@{$\mo_d (I)$, $\mo_d (T)$}}

Let now $T = T_d(I)$ be a triangular region with punctures whose side lengths sum to $m$.  Then we define similarly the \emph{over-puncturing coefficient}%
of $T$ to be $\mo_T = m - d$.  If $\mo_T < 0$, $\mo_T = 0$, or $\mo_T > 0$, then we call $T$ \emph{under-punctured}, \emph{perfectly-punctured},
or \emph{over-punctured}, respectively. Note that $\mo_T = \mo_d (J(T)) \leq \mo_d (I)$, and equality is true if and only if the ideals $I$ and
$J(T)$ are the same in all degrees less than $d$.

Perfectly-punctured regions admit a numerical tileability criterion.

\begin{corollary} \label{cor:pp-tileable}
    Let $T = T_d(I)$ be a triangular region.  Then any two of the following conditions imply the third:
    \begin{enumerate}
        \item $T$ is perfectly-punctured;
        \item $T$ has no over-punctured monomial subregions; and
        \item $T$ is tileable.
    \end{enumerate}
\end{corollary}
\begin{proof}
    If $T$ is unbalanced, then $T$ is not tileable.  Moreover, if $T$ is perfectly-punctured, then at least two punctures must
    overlap, thus creating an over-punctured monomial subregion.   Hence, we may assume $T$ is balanced for the remainder of the argument.

    Suppose $T$ is tileable. Then $T$ has no $\dntri$-heavy monomial subregions, by Theorem \ref{thm:tileable}.  Thus every monomial
    subregion of $T$ is not over-punctured if and only if no punctures of $T$ overlap, i.e., $T$ is perfectly-punctured.

    If $T$ is non-tileable, then $T$ has a $\dntri$-heavy monomial subregion.   Since every $\dntri$-heavy monomial subregion is
    also over-punctured, it follows that $T$ has an over-punctured monomial subregion.
\end{proof}

\section{Stability of syzygy bundles} \label{sec:syz}

Let $I$ be an Artinian ideal of $S = K[x_1, \ldots, x_n]$ that is minimally generated by forms $f_1, \ldots, f_m$.
The \emdx{syzygy module} of $I$ is the graded module $\syz{I}$ that fits into the exact sequence
\[
    0 \rightarrow \syz{I} \rightarrow \bigoplus_{i=1}^{m}S(-\deg f_i) \rightarrow I \rightarrow 0.
\]
Its sheafification $\widetilde\syz{I}$ is a vector bundle on $\PP^{n-1}$, called the \emdx{syzygy bundle} of $I$. It has rank $m-1$.
\index{0@\textbf{Symbol list}!syzI@$\syz{I}$}
\index{0@\textbf{Symbol list}!syzItilde@$\widetilde\syz{I}$}

Semistability is an important property of a vector bundle.  Let $E$ be a vector bundle on projective space.  The \emph{slope}%
\index{vector bundle!slope}
of $E$ is defined as $\mu(E) := \frac{c_1(E)}{rk(E)}$.  Furthermore, $E$ is said to be \emph{semistable}%
\index{vector bundle!semistable}
if the inequality $\mu(F) \leq \mu(E)$ holds for every coherent subsheaf $F \subset E$.  If the inequality is always strict, then
$E$ is said to be \emph{stable}.%
\index{vector bundle!stable}

Brenner established a beautiful characterisation of the semistability of syzygy bundles to monomial ideals.
Since we only consider monomial ideals in this work, the following may be taken as the definition of (semi)stability herein.

\begin{theorem}{\cite[Proposition~2.2 \& Corollary~6.4]{Br}} \label{thm:stable-syz}
    Let $I$ be an Artinian ideal in $K[x_1, \ldots, x_n]$ that is minimally generated by monomials $g_1, \ldots, g_m$, where $K$ is a field of
    characteristic zero. Then $I$ has a semistable syzygy bundle if and only if, for every proper subset $J$ of $\{1, \ldots, m\}$ with at
    least two elements, the inequality
    \[
        \frac{d_J - \displaystyle\sum_{j \in J} \deg g_j}{|J|-1} \leq \frac{-\displaystyle\sum_{i=1}^{m} \deg g_i}{m-1}
    \]
    holds, where $d_J$ is the degree of the greatest common divisor of the $g_j$ with $j \in J$.  Further, $I$ has a stable
    syzygy bundle if and only if the above inequality is always strict.
\end{theorem}

Notice that the right-hand side in the above inequalities is the slope of the syzygy bundle of $I$.

We use the above criterion to rephrase (semi)stability in the case of a monomial ideal in $K[x,y,z]$ in terms of the
over-puncturing coefficients of ideals.  To do this, we first reinterpret the slope. Throughout this section we continue
to assume that $K$ is a field of characteristic zero.

\begin{lemma} \label{lem:T-slope}
    Let $I$ be an Artinian ideal in $R = K[x,y,z]$ that is minimally generated by monomials $g_1, \ldots, g_m$ whose degrees are bounded
    above by $d$.  Then
    \[
        \mu(\widetilde\syz{I}) = -d + \frac{\mo_d (I)}{m-1}.
    \]
\end{lemma}
\begin{proof}
    Recall that the side length of the puncture associated to $g_i$ in $T = T_d(I)$ is $d - \deg{g_i}$.
    Thus, $\mo_d (I)= \sum_{i=1}^{m} (d - \deg g_i) - d$. Hence we obtain
    \begin{equation*}
        \begin{split}
            \mu(\widetilde\syz{I}) &= \frac{-\sum_{i=1}^{m} \deg g_i}{m-1} \\
                                    &= \frac{[\sum_{i=1}^{m} (d - \deg g_i) - d] - d(m-1)}{m-1} \\
                                    &= -d + \frac{\sum_{i=1}^{m} (d - \deg g_i) - d}{m-1} \\
                                    &= -d + \frac{\mo_d (I)}{m-1}.
        \end{split}
    \end{equation*}
\end{proof}

Observe that $\mo_{d+1} (I) = \mo_d (I) + m-1$.

Now we reinterpret Theorem~\ref{thm:stable-syz} by using over-puncturing coefficients.

\begin{corollary} \label{cor:T-stable}
    Let $I$ be an Artinian ideal in $R = K[x,y,z]$ that is minimally generated by monomials $g_1, \ldots, g_m$ of degree at most $d$.
    For every proper subset $J$ of $\{1, \ldots, m\}$ with at least two elements, let $I_J$ be the monomial ideal that is
    generated by $\{ g_j / g \st j \in J\}$, where $g = \gcd\{g_j \st j \in J \}$.

    Then $I$ has a semistable syzygy bundle if and only if, for every proper subset $J$ of $\{1, \ldots, m\}$ with at
    least two elements, the inequality
    \[
        \frac{\mo_{d-\deg{g}} (I_J)}{|J| - 1} \leq \frac{\mo_{d} (I)}{m-1}
    \]
    holds.  Furthermore, $I$ has a stable syzygy bundle if and only if the above inequality is always strict.
\end{corollary}
\begin{proof}
    Let $J$ be a  proper subset of $\{1, \ldots, m\}$ with $n \geq 2$ elements. Then a computation similar to the one in Lemma~\ref{lem:T-slope} provides
    \begin{equation*}
        \begin{split}
            \frac{\deg{g} - \displaystyle\sum_{j \in J} \deg g_j}{|J|-1} &= \frac{-(n-1)\deg{g} - \displaystyle\sum_{j \in J}(\deg{g_j} - \deg{g})}{|J|-1} \\
                &= -\deg{g} + \frac{-\displaystyle\sum_{j \in J} \deg(g_j/g)}{|J|-1} \\
                &= -\deg{g} + (\deg{g} - d) + \frac{\mo_{T_{d-\deg{g}}(I_J)}}{|J| - 1} \\
                &= -d + \frac{\mo_{d-\deg{g}} (I_J)}{|J| - 1}.
        \end{split}
    \end{equation*}

    Taking into account also Lemma~\ref{lem:T-slope},  Theorem~\ref{thm:stable-syz} shows that we need to compare
    $-d + \frac{\mo_{d-\deg{g}} (I_J)}{|J| - 1}$ and $-d +  \frac{\mo_d (I)}{m-1}$.
\end{proof}

In order to better interpret the last result we slightly extend the concept of a triangular region $T_d (I)$ in the
remainder of this section. Label the vertices in $\mathcal{T}_d$ by monomials of degree $d$ such that the label of
each unit triangle is the greatest common divisor of its vertex labels.  Then a minimal monomial generator of $I$ with
degree $d$ corresponds to a vertex of $\mathcal{T}_d$ that is removed in $T_d (I)$. We consider this removed vertex
as a puncture of side length zero. Observe that this is in line with our general definition of the side length of a
puncture (see Subsection \ref{sub:trideg}).    Now Corollary~\ref{cor:T-stable} can be rephrased as saying that
the syzygy bundle of $I$ is semistable if and only if the ``average'' over-puncturing per puncture for any
non-trivial collection of punctures of $I$ (restricted to their smallest containing triangle) is at most the ``average''
over-puncturing per puncture for the entire ideal $I$.

\begin{example} \label{ex:stability}
    We illustrate this point of view by giving quick proofs of some known results.
    \begin{enumerate}
        \item (\cite[Theorem~0.2]{Ba}) \emph{For each $d \geq 1$, the syzygy bundle of $(x, y, z)^d$ is stable.}
            \begin{proof}
                Consider $T_d = T_d ((x, y, z)^{d-1})$, that is, $T_d$ is obtained from $\mathcal{T}_d$ by removing all $\binom{d+1}{2}$ upward-pointing triangles.  Then
                \[
                    \frac{\mo_{T_d}}{\binom{d+1}{2}-1} = \frac{\binom{d+1}{2} - d}{\binom{d+1}{2}-1} = \frac{d}{d+2}.
                \]

                Now we consider the average over-puncturing of any non-trivial collection of punctures in $T$
                in a triangle of side length $e$, where $2 \leq e \leq d$:

                \begin{enumerate}
                    \item If $e < d$, then the average over-puncturing is maximised when all the punctures in the triangle are present, i.e.,
                        when the associated monomial subregion is $T_e$.  Clearly, $\frac{e}{e+2} < \frac{d}{d+2}$ if  $2\leq e < d$.
                    \item If $e = d$, then it is maximised when all but one puncture is present.  The over-puncturing is thus
                        \[
                            \frac{\binom{d+1}{2} - d-1}{\binom{d+1}{2}-2} < \frac{\binom{d+1}{2} - d}{\binom{d+1}{2}-1}.\qedhere
                        \]
                \end{enumerate}
            \end{proof}
    \item (\cite[Corollary~7.1]{Br} \& \cite[Lemma 2.1]{HMNW}) \emph{Let $I = (x^c, y^{c-a}, z^{c-b})$ be a monomial complete intersection where,
        without loss of generality, $0 \leq a \leq b < c$. Then $\widetilde\syz{I}$ is semistable (or stable) if and only if the punctures
        in $T_c(I)$ do not overlap (or touch).}
        \begin{proof}
            Notice that $T_c(I)$ has a puncture of side length zero at its top, corresponding to $x^c$. Thus,
            the average over-puncturing of $T_c(I)$ is $\frac{1}{2}(a+b-c)$ and the average over-puncturing of the
            three non-trivial collections of punctures is $a-c$, $b-c$, and $a+b-c$ for $T_c(x^c, y^{c-a})$, $T_c(x^c, z^{c-b})$,
            and $T_c(y^{c-a}, z^{c-b})$, respectively.  The latter is maximised at $a+b-c$.

            \begin{figure}[!ht]
                \includegraphics[scale=1]{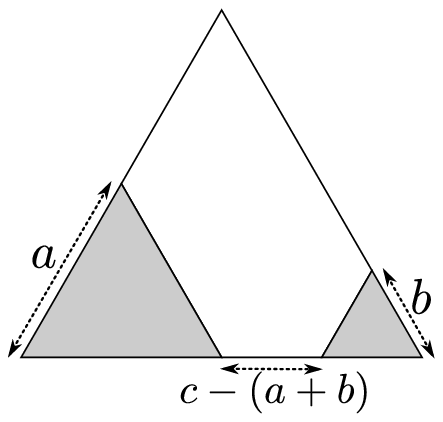}
                \caption{The region $T_c(x^c, y^{c-a}, z^{c-b})$, where $0 \leq a \leq b < c$.}
                \label{fig:ci-abc-c}
            \end{figure}

            Using Corollary~\ref{cor:T-stable}, we see that $I$ has a semistable (stable) syzygy bundle if and only if
            $\frac{1}{2}(a+b-c) \geq a+b-c$ (strictly), i.e., $c \geq a+b$ (strictly).  Interpreting this in  $T_c(I)$
            (see Figure~\ref{fig:ci-abc-c}) yields the desired conclusion.
        \end{proof}
    \end{enumerate}
\end{example}

Using Corollary~\ref{cor:pp-tileable}, we see that semistability is strongly related to tileability of a region.

\begin{theorem} \label{thm:tileable-semistable}
    Let $I$ be an Artinian ideal in $R = K[x,y,z]$ generated by monomials whose degrees are bounded above by $d$,
    and let $T = T_d(I)$.  If $T$ is non-empty, then any two of the following conditions imply the third:
    \begin{enumerate}
        \item $I$ is perfectly-punctured;
        \item $T$ is tileable; and
        \item $\widetilde\syz{I}$ is semistable.
    \end{enumerate}
\end{theorem}
\begin{proof}
    Assume $T$ is perfectly-punctured, that is, $\mo_d (T) = 0$.  By Corollary~\ref{cor:pp-tileable}, $T$ is tileable
    if and only if $T$ has no over-punctured monomial subregions.  The latter condition is equivalent to every monomial subregion of
    $T$ having a non-positive over-puncturing coefficient.  By Corollary~\ref{cor:T-stable}, this is equivalent
    to $\widetilde\syz{I}$ being semistable.

    Now assume $I$ is not perfectly-punctured, but $T$ is tileable. We have to show that $\widetilde\syz{I}$ is not semistable.

    Observe that $\mathcal{T}_d$ has exactly $d$ more upward-pointing triangles than downward-pointing triangles. It follows
    that every balanced monomial subregion of $\mathcal{T}_d$ cannot be under-punctured. Since $T$ is balanced, but not perfectly
    punctured we conclude that $T$ is over-punctured. Using again that $T$ is balanced, $T$ must have overlapping punctures.
    Consider two such overlapping punctures of $T$. Then the smallest monomial subregion $U$ containing these two punctures does
    not overlap with any other puncture of $T$ with positive side length  (because it is not $\dntri$-heavy by Theorem~\ref{thm:tileable}) and is uniquely tileable.
    Hence $T' = T \setminus U$ is tileable (see Lemma~\ref{lem:replace-tileable}) and $0 \leq \mo_{T'} < \mo_d (I)$.  If $T'$ is still
    over-punctured, then we repeat the above replacement procedure until we get a perfectly-punctured monomial subregion of $T$.
    Abusing notation slightly, denote this region by $T'$.  Let $J$ be the largest monomial ideal containing $I$ and with generators
    whose degrees are bounded above by $d$   such that $T' = T_d (J)$. Observe  that $\mo_d (J) = \mo_{T'} = 0$.

    Notice that a single replacement step above amounts to replacing the triangular region to an ideal $I'$ by the region
    to the ideal $(I', f)$, where $f$ is the greatest common divisor of a family of minimal generators of $I'$ having
    degree less than $d$.

    Assume now that $T'$ is empty.  By the above considerations, this means that $I$ has a family of minimal generators, say,
    $g_1, \ldots, g_t$ of degrees $d-a_1, \ldots, d-a_t < d$ that are relatively prime and whose corresponding punctures form
    two overlapping punctures of $T$.  Thus, $a_1 + \cdots + a_t > d$ (see Example~\ref{ex:stability}(ii)).  Furthermore,
    all other minimal generators of $I$, of which there must be at least one as $I$ is Artinian, must have degree $d$
    since $T$ is balanced.  Hence the average over-puncturing of $I$ is
    \[
        \frac{\mo_d (I)}{m-1} = \frac{a_1 + \cdots + a_t - d}{m-1} \leq \frac{a_1 + \cdots + a_t - d}{t},
    \]
    where $m \geq t+1$ is the number of minimal generators of $I$.  However, the average over-puncturing corresponding
    to the ideal $I'$ generated by $g_1, \ldots, g_t$ is
    \[
      \frac{  \mo_d (I')}{t - 1} = \frac{a_1 + \cdots + a_t - d}{t-1} > \frac{a_1 + \cdots + a_t - d}{t}.
    \]
    Hence, Corollary~\ref{cor:T-stable} shows that $\widetilde\syz{I}$ is not semistable.

    It remains to consider the case where $T'$ is not empty, i.e., $J$ is a proper ideal of $R$.  Let $g_1, \ldots, g_m$ and
    $f_1, \ldots, f_n$ be the minimal monomial generators of $I$ and $J$, respectively. Partition the generating set of $I$
    into $F_j = \{ g_i \st g_i \mbox{~divides~} f_j\}$. Notice $f_j = \gcd\{F_j\}$. In particular, $n > 1$ as $I$ is an
    Artinian ideal.

    Set $\mo_j = \sum_{g \in F_j} (d - \deg{g}) - (d - \deg{f_j})$. Observe $\mo_j \geq 0$ as the region associated
    to the ideal generated by $F_j$ is tileable, hence not under-punctured.  Moreover,
    \begin{equation*}
        \begin{split}
            \mo_d (J)  = \sum_{j = 1}^{n}(d - \deg{f_j}) - d & = \sum_{j=1}^{n} \left( \sum_{g \in F_j} (d - \deg{g}) - \mo_j \right) - d \\
                                                            & = \sum_{j=1}^{n} \sum_{g \in F_j} (d - \deg{g}) - d - \sum_{j=1}^{n} \mo_j.
        \end{split}
    \end{equation*}
    As $\mo_d (J)  = 0$, we conclude that $\mo_d (I) = \sum_{j=1}^{n} \mo_j$ and, in particular, $\mo_d (I) \geq \mo_j$ for each $j$.

    Assume $m \cdot \mo_j < \# F_j \cdot \mo_d (I)$ for all $j$.  Then $m \sum_{j=1}^{n}\mo_j < \mo_d (I) \sum_{j=1}^{n} \# F_j$, but this
    implies $m \cdot \mo_d (I) < m \cdot \mo_d (I)$, which is absurd.  Hence, there is some $k$ such that  $m \cdot \mo_k \geq \# F_k \cdot \mo_d (I)$.
    Since $\mo_d (I)\geq \mo_k$ it follows that $\frac{\mo_k}{\# F_k-1} > \frac{\mo_T}{m-1}$.  Indeed, this is immediate if
    $\mo_d (I) > \mo_k$. If $\mo_d (I) = \mo_k$, then it is also true because $\# F_k  < m$. Now Corollary~\ref{cor:T-stable} provides
    that $\widetilde\syz{I}$ is not semistable.
\end{proof}

We get the following criterion when focusing solely on the triangular region. Recall that $J(T)$ denotes the monomial ideal of
a triangular region $T$ (see Subsection~\ref{subsec:gcd}).

\begin{corollary}  \label{cor:semistability-by-region}
    Let $I$ be an Artinian ideal in $R = K[x,y,z]$ generated by monomials whose degrees are bounded above by $d$,
    and let $T = T_d(I)$.  Assume $T$ is non-empty and tileable.
    \begin{enumerate}
        \item If $I \neq I + J(T)$, then $\widetilde\syz{I}$ is not semistable.
        \item $\widetilde\syz(I + J(T))$ is semistable if and only if  $T$ is perfectly-punctured.
    \end{enumerate}
\end{corollary}
\begin{proof}
    Since $T$ is balanced, we get $0 \leq \mo_T = \mo_d (J(T)) = \mo_d (I + J(T))$. Hence Theorem~\ref{thm:tileable-semistable} provides our assertions.
    Note for claim (i) that $I \ne I + J(T)$ implies $\mo_d (I+J(T)) < \mo_d (I)$.
\end{proof}


For stability, we obtain the following result.

\begin{proposition} \label{pro:pp-stable}
    Let $I$ be an Artinian ideal in $R = K[x,y,z]$ generated by monomials whose degrees are bounded above by $d$.
    If $T = T_d(I)$ is non-empty, tileable, and perfectly-punctured, then $\widetilde\syz(I + J(T))$ is stable
    if and only if every proper monomial subregion of $T$ is under-punctured.
\end{proposition}
\begin{proof}
    We may assume $I = I + J(T)$. As $T$ is perfectly-punctured, we have that $\mo_d (I) = \mo_T = 0$.  Using Corollary~\ref{cor:T-stable}, we see that
    $\widetilde\syz{I}$ is stable if and only if $\mo_{T_{d-\deg{g}}(I_J)} < 0$, where $g = \gcd\{J\}$, for all
    proper subsets $J$ of the set of minimal generators of $I$.  This is equivalent to every proper monomial subregion of $T$ being under-punctured.
\end{proof}

By the preceding theorem and proposition, we have an understanding of semistability and stability for
perfectly-punctured triangular regions.  However, when a region is over-punctured and non-tileable  more
information is needed to decide semistability.

\begin{example} \label{exa:stability}
    There are monomial ideals with stable syzygy bundles whose corresponding triangular regions are
    over-punctured and  non-tileable.  See Example~\ref{ex:stability}(i) and Figure~\ref{fig:nss-examples}(i) for a specific
    example.

    \begin{figure}[!ht]
        \begin{minipage}[b]{0.30\linewidth}
            \centering
            \includegraphics[scale=1]{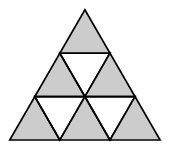}\\
            \emph{(i) $T_3(x^2, y^2, z^2, xy, xz, yz)$}
        \end{minipage}
        \begin{minipage}[b]{0.28\linewidth}
            \centering
            \includegraphics[scale=1]{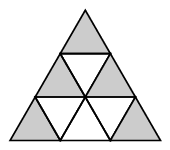}\\
            \emph{(ii) $T_3(x^2, y^2, z^2, xy, xz)$}
        \end{minipage}
        \begin{minipage}[b]{0.40\linewidth}
            \centering
            \includegraphics[scale=1]{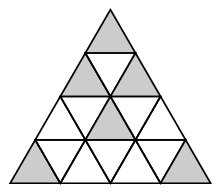}\\
            \emph{(iii) $T_4(x^3, y^3, z^3, xyz, x^2y, x^2z)$}
        \end{minipage}
        \caption{Over-punctured, non-tileable regions and various levels of stability.}
        \label{fig:nss-examples}
    \end{figure}

    Moreover, the ideal $(x^2, y^2, z^2, xy, xz)$ has a semistable, but non-stable  syzygy bundle (the monomial subregion associated
    to $x$ breaks stability), and the ideal $(x^3, y^3, z^3, xyz, x^2y, x^2z)$ has a non-semistable
   syzygy bundle (the monomial subregion associated to $x^2$ breaks semistability).  Both of their triangular regions,
    see Figures~\ref{fig:nss-examples}(ii) and (iii), respectively, are over-punctured and non-tileable.
\end{example}

\section{Signed lozenge tilings and enumerations} \label{sec:signed}

In Section~\ref{sec:tiling} we considered whether a triangular region $T_d(I)$ is tileable by lozenges.  Now we want to
\emph{enumerate} the tilings of tileable regions $T_d(I)$. In fact, we introduce two ways for assigning a sign to a lozenge tiling and then compare the resulting enumerations.

In order to derive the (unsigned) enumeration, we consider the enumeration of perfect matchings of an associated bipartite
graph.  If we consider the bi-adjacency matrix, a zero-one matrix, of the bipartite graph, then the permanent of the matrix yields
the desired enumeration.  However, the determinant of this matrix yields a (possibly different) integer, which may be negative.
We consider this a \emph{signed} enumeration of the perfect matchings of the graph, and hence of lozenge tilings.

We derive a different \emph{signed} enumeration of the lozenge tilings by considering the enumeration of families of non-intersecting
lattice paths on an associated finite sub-lattice of $\ZZ^2$.  Using the Lindstr\"om-Gessel-Viennot Theorem (\cite{Li}, \cite{GV};
see Theorem~\ref{thm:lgv}), we generate a binomial matrix for the finite sub-lattice with a determinant that gives a signed enumeration
of families of non-intersecting lattice paths, hence of lozenge tilings.  The two signed enumerations appear to be different, but we show
that they are indeed the same, up to sign.

~\subsection{Perfect matchings}\label{sub:pm}~\par

A subregion $T (G) \subset \mathcal{T}_d$ can be associated to a bipartite planar graph $G$ that is an induced subgraph
of the honeycomb graph. Lozenge tilings of $T(G)$ can be then associated to perfect matchings on $G$. The connection was
used by Kuperberg in~\cite{Kup}, the earliest citation known to the authors, to study symmetries on plane partitions.
Note that $T(G)$ is often called the \emph{dual graph} of $G$ in the literature (e.g., \cite{Ci-1997}, \cite{Ci-2005},
and \cite{Ei}). Here we begin with a subregion $T$ and then construct a graph $G$.

Let $T \subset \mathcal{T}_d$ be any subregion. As above, we consider $T$ as a union of unit triangles. We associate to
$T$ a bipartite graph. First, place a vertex at the center of each triangle. Let $B$ be the set of centers of the
downward-pointing triangles, and let $W$ be the set of centers of the upward-pointing triangles. Consider both sets
ordered by the reverse-lexicographic ordering applied to the monomial labels of the corresponding triangles (see
Subsection~\ref{sub:trideg}). The \emph{bipartite graph associated to $T$}%
\index{triangular region!bipartite graph associated to}
is the bipartite graph $G(T)$ on the vertex set $B \cup W$ that has an edge between vertices $B_i \in B$ and $W_j \in W$
if the corresponding upward- and downward-pointing triangle share are edge. In other words, edges of $G(T)$ connect vertices
of adjacent triangles. See Figure~\ref{fig:build-pm}(i).

\begin{figure}[!ht]
    \begin{minipage}[b]{0.32\linewidth}
        \centering
        \includegraphics[scale=1]{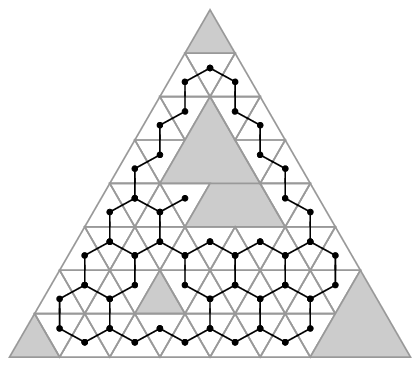}\\
        \emph{(i) The graph $G(T)$.}
    \end{minipage}
    \begin{minipage}[b]{0.32\linewidth}
        \centering
        \includegraphics[scale=1]{figs/build-pm-2}\\
        \emph{(ii) Selected covered edges.}
    \end{minipage}
    \begin{minipage}[b]{0.32\linewidth}
        \centering
        \includegraphics[scale=1]{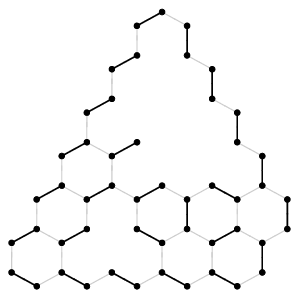}\\
        \emph{(iii) The perfect matching.}
    \end{minipage}
    \caption{Given the tiling $\tau$ in Figure~\ref{fig:triregion-tiling} of $T$, we construct the perfect
        matching $\pi$ of the bipartite graph $G(T)$ associated to $\tau$.}
    \label{fig:build-pm}
\end{figure}

Using the above ordering of the vertices, we define the \emph{bi-adjacency matrix}%
\index{triangular region!bi-adjacency matrix of}
\index{0@\textbf{Symbol list}!ZT@$Z(T)$}
of $T$ as the bi-adjacency matrix $Z(T) := Z(G(T))$ of the graph $G(T)$. It is the zero-one matrix $Z(T)$ of
size $\# B \times \# W$ with entries $Z(T)_{(i,j)}$ defined by
\begin{equation*}
    Z(T)_{(i,j)} =
    \begin{cases}
        1 & \text{if $(B_i, W_j)$ is an edge of $G(T)$ } \\
        0 & \text{otherwise.}
    \end{cases}
\end{equation*}

\begin{remark} \label{rem:Z-non-square}
    Note that that $Z(T)$ is a square matrix if and only of the region $T$ is balanced.  Observe also that the construction
    of $G(T)$ and $Z(T)$ do not require any restrictions on $T$.  In particular, $T$ need not be
    balanced and so $Z(T)$ need not be square.  This generality is needed in Section~\ref{sub:wlp-pm}.
\end{remark}

A \emph{perfect matching of a graph $G$}%
\index{perfect matching}
is a set of pairwise non-adjacent edges of $G$ such that each vertex is matched. There is well-known bijection between lozenge
tilings of a balanced subregion $T$ and perfect matchings of $G(T)$. A lozenge tiling $\tau$ is transformed in to a perfect
matching $\pi$ by overlaying the triangular region $T$ on the bipartite graph $G(T)$ and selecting the edges of the graph that
the lozenges of $\tau$ cover. See Figures~\ref{fig:build-pm}(ii) and~(iii) for the overlayed image and the perfect matching by
itself, respectively.

\begin{remark}
    The graph $G(T)$ is a ``honeycomb graph,'' a type of graph that has been studied, especially for its perfect matchings.
    \begin{enumerate}
        \item In particular, honeycomb graphs are investigated for their connections to physics.  Honeycomb graphs model the bonds
            in dimers (polymers with only two structural units), and perfect matchings correspond to so-called \emdx{dimer coverings}.
            Kenyon~\cite{Ke} gave a modern recount of explorations on dimer models, including random dimer coverings and their
            limiting shapes.  The recent memoir~\cite{Ci-2005} of Ciucu, which has many connections to this paper (see Section~\ref{sec:mirror}),
            describes further results in this direction.
        \item Kasteleyn~\cite{Ka} provided, in 1967, a general method for computing the number of perfect matchings of a planar graph by means
            of a determinant.  In the following proposition, we compute the number of perfect matchings on $G(T)$ by means of a permanent.
    \end{enumerate}
\end{remark}

Recall that the \emdx{permanent} of an $n \times n$ matrix $M = (M_{(i, j)})$ is given by
\[
    \per{M} := \sum_{\sigma \in \PS_n} \prod_{i=1}^{n} M_{(i, \sigma(i))}.
\]
\index{0@\textbf{Symbol list}!perM@$\per{M}$}

\begin{proposition} \label{pro:per-enum}
    Let $T \subset \mathcal{T}_d$ be a non-empty balanced subregion.  Then the lozenge tilings of $T$ and the perfect matchings of $G(T)$ are both
    enumerated by $\per{Z(T)}$.
\end{proposition}
\begin{proof}
    As $T$ is balanced, $Z(T)$ is a square zero-one matrix. Each non-zero summand of $\per{Z(T)}$ corresponds to a
    perfect matching, as it corresponds to a bijection between the two colour classes $B$ and $W$ of $G(T)$ (determined
    by the downward- and upward-pointing triangles of $T$).  Hence, $\per{Z(T)}$ enumerates the perfect matchings of
    $G(T)$, and thus the tilings of $T$.
\end{proof}

Recall that the \emdx{determinant} of an $n \times n$ matrix $M$ is given by
\[
    \det{M} := \sum_{\sigma \in \PS_n} \prod_{i=1}^{n} \sgn{\sigma} M_{(i, \sigma(i))}.
\]
\index{0@\textbf{Symbol list}!detM@$\det{M}$}
Each non-zero summand of the determinant of $M$ is given a sign based on the signature (or sign) of the permutation associated to it. We take the
convention that the permanent and determinant of a $0 \times 0$ matrix its one.

By the proof of Proposition~\ref{pro:per-enum}, each lozenge tiling $\tau$ corresponds to a perfect matching $\pi$ of $G(T)$, that is, a bijection
$\pi: B \to W$. Considering $\pi$ as a permutation on $\#\uptri(T) = \#\dntri (T)$ letters, it is natural to assign a sign to each lozenge tiling
using the signature of the permutation $\pi$.

\begin{definition} \label{def:pm-sign}
    Let $T \subset \mathcal{T}_d$ be a non-empty balanced subregion. Then we define the \emph{perfect matching sign}%
    \index{perfect matching!sign}
    \index{0@\textbf{Symbol list}!msgn@$\msgn{\tau}$}
    of a lozenge tiling $\tau$ of $T$ as $\msgn{\tau} := \sgn{\pi}$, where $\pi \in \PS_{\#\uptri(T)}$ is the perfect
    matching determined by $\tau$.
\end{definition}

It follows that the determinant of $Z(T)$ gives an enumeration of the \emph{perfect matching signed lozenge tilings} of $T$.

\begin{theorem} \label{thm:pm-matrix}
     Let $T \subset \mathcal{T}_d$ be a non-empty balanced subregion. Then the perfect matching signed lozenge tilings of $T$
    are enumerated by $\det{Z(T)}$, that is,
    \[
        \sum_{\tau \text{tiling of } T}  \msgn{\tau} = \det Z(T).
    \]
\end{theorem}~

\begin{example} \label{exa:Z-matrix}
    Consider the triangular region $T = T_6(x^3, y^4, z^5)$, as seen in the first picture of Figure~\ref{fig:three-rotations}.
    Then $Z(T)$ is the $11 \times 11$ matrix
    \[
        Z(T) =
        \left[
            \begin{array}{ccccccccccc}
                1&1&0&0&0&0&0&0&0&0&0\\
                0&1&1&0&0&0&0&0&0&0&0\\
                0&0&1&1&0&0&0&0&0&0&0\\
                1&0&0&0&1&0&0&0&0&0&0\\
                0&1&0&0&1&1&0&0&0&0&0\\
                0&0&1&0&0&1&1&0&0&0&0\\
                0&0&0&1&0&0&1&1&0&0&0\\
                0&0&0&0&1&0&0&0&1&0&0\\
                0&0&0&0&0&1&0&0&1&1&0\\
                0&0&0&0&0&0&1&0&0&1&1\\
                0&0&0&0&0&0&0&1&0&0&1
            \end{array}
        \right].
    \]
    We note that $\per Z(T) = \det{Z(T)} = 10$. Thus, $T$ has exactly $10$ lozenge tilings, all of which have the same sign.
    We derive a theoretical explanation for this fact in the following two subsections.
\end{example}~

\subsection{Families of non-intersecting lattice paths}\label{sub:nilp}~\par

We follow~\cite[Section~5]{CEKZ} (similarly,~\cite[Section~2]{Fi}) in order to associate to a subregion $T \subset \mathcal{T}_d$ a
finite set  $L(T)$ that can be identified with a subset of  the lattice $\ZZ^2$.  Abusing notation, we refer to $L(T)$ as a
sub-lattice of $\ZZ^2$. We then translate lozenge tilings of $T$ into families of non-intersecting lattice paths on $L(T)$.%
\index{triangular region!lattice associated to}
\index{0@\textbf{Symbol list}!LT@$L(T)$}

We first construct $L(T)$ from $T$.  Place a vertex at the midpoint of the edge of each triangle of $T$  that is parallel to the upper-left
boundary of the triangle $\mathcal{T}_d$.  These vertices form $L(T)$. We will consider paths in $L(T)$. There we think of rightward motion parallel to the
bottom edge of $\mathcal{T}_d$ as ``horizontal'' and downward motion parallel to the upper-right edge of $\mathcal{T}_d$ as ``vertical'' motion.  If one simply orthogonalises
$L(T)$ with respect to the described ``horizontal'' and ``vertical'' motions, then we can consider $L(T)$ as a finite sub-lattice of $\ZZ^2$.
As we can translate $L(T)$ in $\ZZ^2$ and not change its properties, we may assume that the vertex associated to the
lower-left triangle of $\mathcal{T}_d$ is the origin.  Notice that each vertex of $L(T)$  is on the upper-left edge of an upward-pointing triangle of $\mathcal{T}_d$ (even if
this triangle is not present in $T$). We use the monomial label of this upward-pointing triangle to specify a vertex of $L(T)$. Under this identification the mentioned
orthogonalisation of $L(T)$ moves the vertex associated to the monomial $x^a y^b z^{d-1-(a+b)}$ in $L(T)$  to the point $(d-1-b, a)$ in $\ZZ^2$.

We next single out special vertices of $L(T)$. We label the vertices of $L(T)$ that are only on upward-pointing triangles in $T$, from smallest to largest
in the reverse-lexicographic order, as $A_1, \ldots, A_m$.  Similarly, we label the vertices of $L(T)$ that are only on downward-pointing triangles in $T$,
again from smallest to largest in the reverse-lexicographic order, as $E_1, \ldots, E_n$.  See Figure~\ref{fig:build-nilp}(i).  We note that there are an equal
number of vertices  $A_1, \ldots, A_m$ and $E_1, \ldots, E_n$ if and only if the region $T$ is balanced.  This follows from the fact the these vertices are
precisely the vertices of $L(T)$ that are in exactly one unit triangle of $T$.

A \emdx{lattice path} in a lattice $L \subset \ZZ^2$ is a finite sequence of vertices of $L$  so that all single steps move either to the right or down.
Given any vertices $A, E \in \ZZ^2$, the number of lattice paths in $\ZZ^2$ from $A$ to $E$ is a binomial coefficient.  In fact, if $A$ and $E$ have
coordinates $(u,v), (x,y) \in \ZZ^2$ as above, there are $\binom{x-u+v-y}{x-u}$ lattice paths from $A$ to $E$ as each path has $x-u + v-y$ steps
and $x-u \geq 0$ of these must be horizontal steps.

Using the above identification of $L(T)$ as a sub-lattice of $\ZZ^2$, a \emph{lattice path} in $L(T)$ is a finite sequence of vertices of $L(T)$
so that all single steps move either to the East or to the Southeast. The \emph{lattice path matrix}%
\index{triangular region!lattice path matrix of}
of $T$ is the $m \times n$ matrix  $N(T)$ with entries $N(T)_{(i,j)}$ defined by
\[
    N(T)_{(i,j)} = \# \text{lattice paths in $\ZZ^2$ from $A_i$ to $E_j$}.
\]
Thus,  the entries of $N(T)$ are binomial coefficients.
\index{0@\textbf{Symbol list}!NT@$N(T)$}

Next we consider several lattice paths simultaneously. A \emph{family of non-intersecting lattice paths}%
\index{lattice path!family of non-intersecting}
is a finite collection of lattice paths such that no two lattice paths have any points in common.  We call a family of non-intersecting
lattice paths \emph{minimal}%
\index{lattice path!minimal family of non-intersecting}
if every path takes vertical steps before it takes horizontal steps, whenever possible.  That is, every time a horizontal step is followed by a vertical step, then replacing
these with a vertical step followed by a horizontal step would cause paths in the family to intersect.

Assume now that the subregion $T$ is balanced, so $m = n$. Let $\Lambda$ be a family of $m$ non-intersecting lattice paths in $L(T)$ from
$A_1, \ldots, A_m$ to $E_1, \ldots, E_m$. Then $\Lambda$ determines a permutation $\lambda \in \PS_m$ such that the path in $\Lambda$ that
begins at $A_i$ ends at $E_{\lambda(i)}$.

Now we are ready to apply a beautiful theorem relating enumerations of signed families of non-intersecting lattice paths and
determinants.  In particular, we use a theorem first given by Lindstr\"om in~\cite[Lemma~1]{Li} and stated independently
in~\cite[Theorem~1]{GV} by Gessel and Viennot.  Stanley gives a very nice exposition of the topic in~\cite[Section~2.7]{Stanley-2011}.

\begin{theorem}{\cite[Lemma~1]{Li} \& \cite[Theorem~1]{GV}} \label{thm:lgv}
    Assume $T \subset \mathcal{T}_d$ is a non-empty balanced subregion with identified lattice points
    $A_1, \ldots, A_m, E_1, \ldots, E_m \in L(T)$ as above. Then
    \[
        \det{N(T)} = \sum_{\lambda \in \PS_m} \sgn(\lambda) P^+_\lambda(A\rightarrow E),
    \]
    where, for each permutation $\lambda \in \PS_m$, $P^+_\lambda(A \rightarrow E)$ is the number of families of non-intersecting
    lattice paths with paths in $L(T)$ going from $A_i$ to $E_{\lambda(i)}$.
\end{theorem}

We now use a well-know bijection between lozenge tilings of $T$ and families of non-intersecting lattice paths from $A_1, \ldots, A_m$ to $E_1, \ldots, E_m$;
see, e.g., the survey~\cite{Pr}.  Let $\tau$ be a lozenge tiling of $T$.  Using the lozenges of $\tau$ as a guide,
we connect each pair of vertices of $L(T)$ that occur on a single lozenge.  This generates the family of non-intersecting lattice
paths $\Lambda$ of $L(T)$ corresponding to $\tau$.  See Figures~\ref{fig:build-nilp}(ii) and~(iii) for the overlayed image and the family of non-intersecting
lattice paths by itself, respectively.

\begin{figure}[!ht]
    \begin{minipage}[b]{0.32\linewidth}
        \centering
        \includegraphics[scale=1]{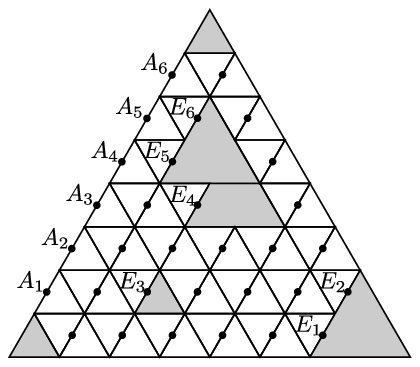}\\
        \emph{(i) The sub-lattice $L(T)$.}
    \end{minipage}
    \begin{minipage}[b]{0.32\linewidth}
        \centering
        \includegraphics[scale=1]{figs/build-nilp-2}\\
        \emph{(ii) The overlayed image.}
    \end{minipage}
    \begin{minipage}[b]{0.32\linewidth}
        \centering
        \includegraphics[scale=1]{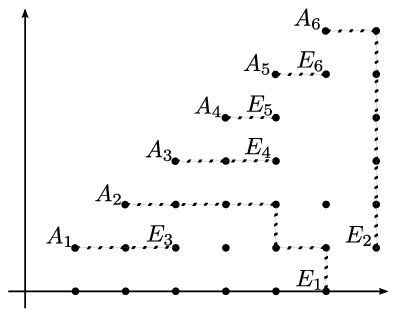}\\
        \emph{(iii) The family $\Lambda$.}
    \end{minipage}
    \caption{The family of non-intersecting lattice paths $\Lambda$ associated to the tiling $\tau$ in Figure~\ref{fig:triregion-tiling}.}
    \label{fig:build-nilp}
\end{figure}

This bijection provides another way for assigning a sign to a lozenge tiling, this time using the signature of the permutation $\lambda$.

\begin{definition} \label{def:nilp-sign}
    \index{lattice path!sign}
    \index{0@\textbf{Symbol list}!lpsgn@$\lpsgn{\tau}$}
    Let $T \subset \mathcal{T}_d$ be a non-empty balanced subregion as above, and let $\tau$ be a lozenge tiling of $T$.
    Then we define the \emph{lattice path sign} of $\tau$ as $\lpsgn{\tau} := \sgn{\lambda}$, where $\lambda \in \PS_m$ is
    the permutation such that, for each $i$, the lattice path determined by $\tau$ that starts at $A_i$ ends at $E_{\lambda (i)}$.
\end{definition}

It follows that the determinant of $N(T)$ gives an enumeration of the \emph{lattice path signed lozenge tilings of $T$}.

\begin{theorem} \label{thm:nilp-matrix}
    Let $T \subset \mathcal{T}_d$ be a non-empty balanced subregion.  Then the lattice path signed lozenge tilings of $T$
    are enumerated by $\det{N(T)}$, that is,
    \[
        \sum_{\tau \text{tiling of } T} \lpsgn{\tau} = \det{N(T)}.
    \]
\end{theorem}

\begin{remark} \label{rem:rotations}
    Notice that we can use the above construction to assign, for each subregion $T$,  three (non-trivially) different lattice path matrices.
    The matrix $N(T)$ from Theorem~\ref{thm:nilp-matrix} is one of these matrices, and the other two are the $N(\cdot)$ matrices of the
    $120^{\circ}$ and $240^{\circ}$ rotations of $T$.  See Figure~\ref{fig:three-rotations} for an example.

    \begin{figure}[!ht]
        \includegraphics[scale=1]{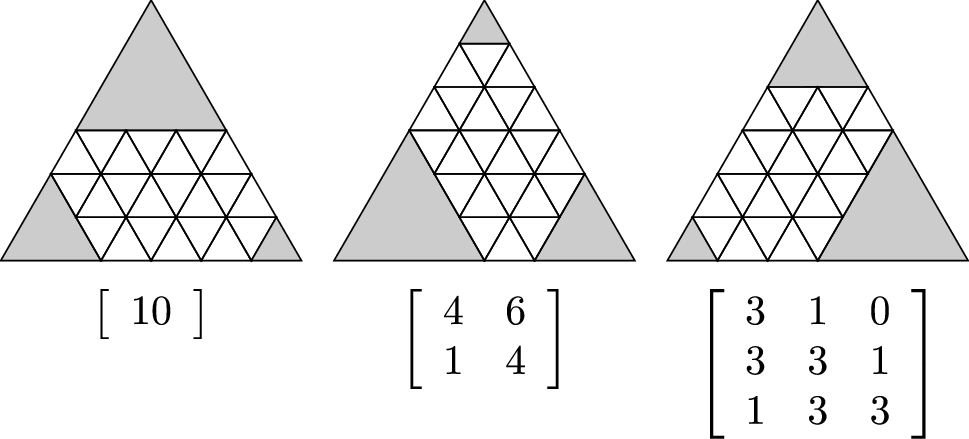}
        \caption{The triangular region $T_6(x^3, y^4, z^5)$ and its rotations, along with the $N(\cdot)$.}
        \label{fig:three-rotations}
    \end{figure}
\end{remark}

Finally, we note that in~\cite{Pr}, Propp gave a history of the connections between lozenge tilings (of non-punctured hexagons),
perfect matchings, plane partitions, and non-intersecting lattice paths.

\subsection{Interlude of signs}\label{sub:signs}~\par

We now have two different signs, the perfect matching sign and the lattice path sign, associated to each lozenge tiling of a
balanced region $T$.  In the case where $T$ is a triangular region, we demonstrate in this subsection that the
signs are equivalent, up to a scaling factor dependent only on $T = T_d(I)$. In particular, the main result of this
section (Theorem~\ref{thm:detZN}) states that $|\det{Z(T)}| = |\det{N(T)}|$. In order to prove this theorem, we
first make a few definitions. Throughout this subsection $T = T_d(I)$ is a tileable triangular region as introduced in
Section~\ref{sec:dictionary}. In particular, $T$ is balanced, and each puncture of $T$ has positive side length.

\subsubsection{Resolution of punctures}\label{subsub:rez}~\par

The first is a tool to remove a puncture from a triangular region, relative to some tiling, in a controlled fashion.

First, suppose that $T \subset \mathcal{T}_d$ has at least one puncture, call it $\mathcal{P}$, that is not overlapped
by any other puncture of $T$. Let $\tau$ be some lozenge tiling of $T$, and denote by $k$ the side length of
$\mathcal{P}$. Informally, we will replace $T$ by a triangular region in $\mathcal{T}_{d + 2k}$, where the place of
the puncture $\mathcal{P}$ of $T$ is taken by a tiled regular hexagon of side length $k$ and three corridors to the outer
vertices of $\mathcal{T}_{d + 2k}$ that are all part of the new region.

\begin{figure}[!ht]
    \begin{minipage}[b]{0.48\linewidth}
        \centering
        \includegraphics[scale=1]{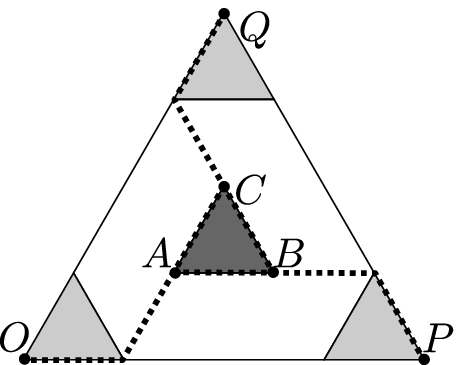}\\
        \emph{(i)  The splitting chains.}
    \end{minipage}
    \begin{minipage}[b]{0.48\linewidth}
        \centering
        \includegraphics[scale=1]{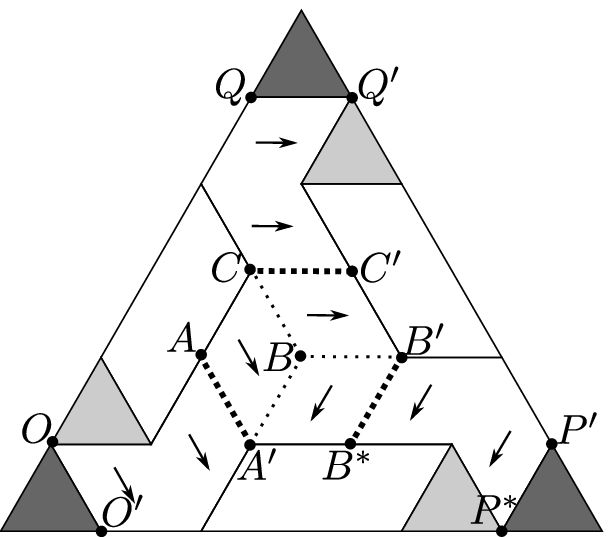}\\
        \emph{(ii) The resolution $T'$.}
    \end{minipage}
    \caption{The abstract resolution of a puncture.}
    \label{fig:resolve-abstract}
\end{figure}

As above, we label the vertices of $\mathcal{T}_d$ such that the label of each unit triangle is the greatest common
divisor of its vertex labels. For ease of reference, we denote the lower-left, lower-right, and top vertex of the
puncture $\mathcal{P}$ by $A, B$, and $C$, respectively. Similarly, we denote the lower-left, lower-right, and top vertex
of $\mathcal{T}_d$ by $O, P$, and $Q$, respectively. Now we select three chains of unit edges such that each edge
is either in $T$ or on the boundary of a puncture of $T$. We start by choosing chains connecting $A$ to $O$, $B$
to $P$, and $C$ to $Q$, respectively, subject to the following conditions:
\begin{itemize}
    \item The chains do not cross, that is, do not share any vertices.
    \item There are no redundant edges, that is, omitting any unit edge destroys the connection between the desired end points of a chain.
    \item There are no moves to the East or Northeast on the lower-left chain $OA$.
    \item There are no moves to the West or Northwest on the lower-right chain $PB$.
    \item There are no moves to the Southeast or Southwest on the top chain $CQ$.
\end{itemize}
For these directions we envision a particle that starts at a vertex of the puncture and moves on a chain to the corresponding corner vertex of $\mathcal{T}_d$.

Now we connect the chains $OA$ and $CQ$ to a chain of unit edges $OACQ$ by using the Northeast edge of $\mathcal{P}$. Similarly we
connect the chains $OA$ and $BP$ to a chain $OABP$ by using the horizontal edge of $\mathcal{P}$, and we connect
$PB$ and $CQ$ to the chain $PBCQ$ by using the Northwest side of $\mathcal{P}$. These three chains subdivide
$\mathcal{T}_d$ into four regions. Part of the boundary of three of these regions is an edge of $\mathcal{T}_d$.
The fourth region, the central one, is the area of the puncture $\mathcal{P}$. See Figure~\ref{fig:resolve-abstract}(i)
for an illustration.

Now consider $T \subset \mathcal{T}_d$ as embedded into $\mathcal{T}_{d+ 2k}$ such that the original region
$\mathcal{T}_d$ is identified with the triangular region $T_{d+2k} (x^k y^k)$. Retain the names $A, B, C, O, P$, and $Q$
for the specified vertices of $T$ as above. We create new chains of unit edges in $\mathcal{T}_{d+ 2k}$.

First, multiply each vertex in the chain $PBCQ$ by $\frac{z^k}{y^k}$ and connect the resulting vertices to a chain
$P'B'C'Q'$ that is parallel to the chain $PBCQ$. Here $P', B', C'$, and $Q'$ are the images of $P, B, C$, and $Q$ under
the multiplication by $\frac{z^k}{y^k}$. Informally, the chain $P'B'C'Q'$ is obtained by moving the chain $PBCQ$ just
$k$ units to the East.

Second, multiply each vertex in the chain $OA$ by $\frac{z^k}{x^k}$ and connect the resulting vertices to a chain $O'A'$
that is parallel to the chain $OA$. Here $A'$ and $O'$ are the points corresponding to $A$ and $O$. Informally the chain
$O'A'$ is obtained by moving the chain $OA$ just $k$ units to the Southeast.

Third, multiply each vertex in the chain $P'B'$ by $\frac{y^k}{x^k}$ and connect the resulting vertices to a chain
$P^*B^*$ that is parallel to the chain $P'B'$, where $P^*$ and $B^*$ are the images of $P'$ and $B'$, respectively.
Thus, $P^*B^*$ is $k$ units to the Southwest of the chain $P'B'$. Connecting $A'$ and $B^*$ by horizontal edges, we
obtain a chain $O'A'B^*P^*$ that has the same shape as the chain $OABP$. See Figure~\ref{fig:resolve-abstract}(ii) for
an illustration.

We are ready to describe the desired triangular region $T' \subset \mathcal{T}_{d+2k}$ along with a tiling. Place
lozenges and punctures in the region bounded by the chain $OACQ$ and the Northeast boundary of $\mathcal{T}_{d+2k}$ as
in the corresponding region of $T$. Similarly place lozenges and punctures in the region bounded by the chain $P'B'C'Q'$
and the Northwest boundary of $\mathcal{T}_{d+2k}$ as in the corresponding region of $T$ that is bounded by $PBCQ$.
Next, place lozenges and punctures in the region bounded by the chain $O'A'B^*P^*$ and the horizontal boundary of
$\mathcal{T}_{d+2k}$ as in the exterior region of $T$ that is bounded by $OABP$. Observe that corresponding vertices of
the parallel chains $BCQ$ and $B'C'Q'$ can be connected by horizontal edges. The region between two such edges that are
one unit apart is uniquely tileable. This gives a lozenge tiling for the region between the two chains. Similarly, the
corresponding vertices of the parallel chains $OAC$ and $O'A'C'$ can be connected by Southeast edges. Respecting these
edges gives a unique lozenge tiling for the region between the chains $OAC$ and $O'A'C'$. In a similar fashion, the
corresponding vertices of the parallel chains $P'B'$ and $P^*B^*$ can be connected by Southwest edges, which we use as a
guide for a lozenge tiling of the region between the two chains. Finally, the rhombus with vertices $A', B^*, B'$, and
$B$ admits a unique lozenge tiling. Let $\tau'$ the union of all the lozenges we placed in $\mathcal{T}_{d+2k}$, and
denote by $T'$ the triangular region that is tiled by $\tau'$. Thus, $T' \subset \mathcal{T}_{d+2k}$ has a puncture of
side length $k$ at each corner of $\mathcal{T}_{d+2k}$. See Figure~\ref{fig:resolve-simple} for an illustration of this.
We call the region $T'$ with its tiling $\tau'$ a \emph{resolution of the puncture $\mathcal{P}$ in $T$ relative to $\tau$}
or, simply, a \emph{resolution of $\mathcal{P}$}.%
\index{puncture!resolution of}

Observe that the tiles in $\tau'$ that were not carried over from the tiling $\tau$ are in the region that is the union
of the regular hexagon with vertices $A, A', B^*, B', C'$ and $C$ and the regions between the parallel chains $OA$ and
$O'A'$, $CQ$ and $C'Q'$ as well as $P'B'$ and $P^*B^*$. We refer to the latter three regions as the \emph{corridors} of
the resolution. Furthermore, we call the chosen chains $OA$, $PB$, and $CQ$ the \emph{splitting chains}%
\index{puncture!resolution of!splitting chains of}
of the resolution. The resolution blows up each splitting chain to a corridor of width $k$.

\begin{figure}[!ht]
    \begin{minipage}[b]{0.48\linewidth}
        \centering
        \includegraphics[scale=1]{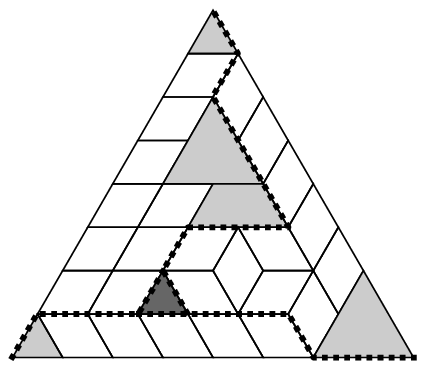}\\
        \emph{(i)  The selected lozenge and puncture edges.}
    \end{minipage}
    \begin{minipage}[b]{0.48\linewidth}
        \centering
        \includegraphics[scale=1]{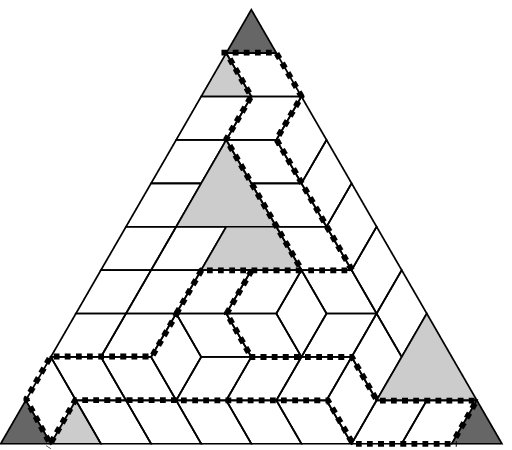}\\
        \emph{(ii) The resolution $T'$ with tiling $\tau'$.}
    \end{minipage}
    \caption{A resolution of the puncture associated to $x y^4 z^2$, given the tiling $\tau$ in Figure~\ref{fig:triregion-tiling} of $T$.}
    \label{fig:resolve-simple}
\end{figure}

Second, suppose the puncture $\mathcal{P}$ in $T$ is overlapped by another puncture of $T$. Then we cannot resolve
$\mathcal{P}$ using the above technique directly as it would result in a non-triangular region. Thus, we adapt it. Since
$T$ is balanced, $\mathcal{P}$ is overlapped by exactly one puncture of $T$ (see Theorem~\ref{thm:tileable}). Let $U$ be
the smallest monomial subregion of $T$ that contains both punctures. We call $U$ the \emph{minimal covering region}%
\index{puncture!minimal covering region}
of the two punctures. It is is uniquely tileable, and we resolve the puncture $U$ of $T \setminus U$. Notice that the
lozenges inside $U$ are lost during resolution. However, since $U$ is uniquely tileable, they are recoverable from the
two punctures of $T$ in $U$. See Figure~\ref{fig:resolve-family} for an illustration.

\begin{figure}[!ht]
    \begin{minipage}[b]{0.48\linewidth}
        \centering
        \includegraphics[scale=1]{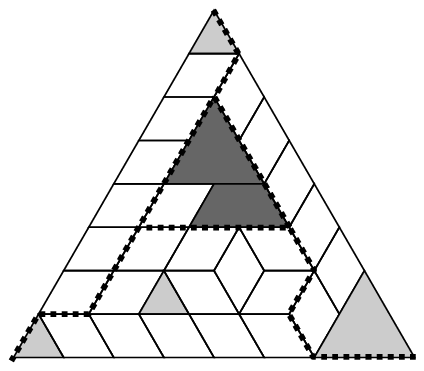}\\
        \emph{(i)  The selected lozenge and puncture edges.}
    \end{minipage}
    \begin{minipage}[b]{0.48\linewidth}
        \centering
        \includegraphics[scale=0.75]{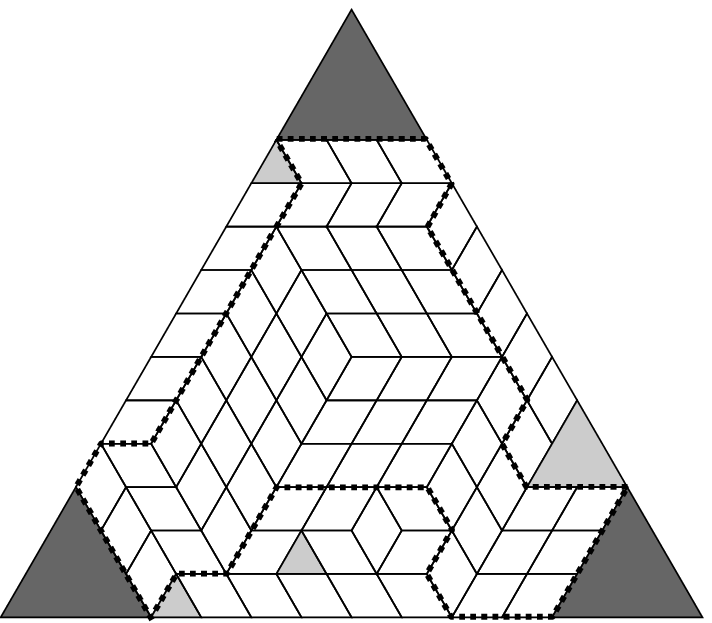}\\
        \emph{(ii) The resolution $T'$ with tiling $\tau'$.}
    \end{minipage}
    \caption{Resolving overlapping punctures, given the tiling in Figure~\ref{fig:triregion-tiling}.}
    \label{fig:resolve-family}
\end{figure}

\subsubsection{Cycles of lozenges}\label{subsub:cyc}~\par

Let $\tau$ be some tiling of $T$. An \emph{$n$-cycle (of lozenges)}%
\index{lozenge!cycle of}
$\sigma$ in $\tau$ is an ordered collection of distinct lozenges $\ell_1, \ldots, \ell_n$ of $\tau$ such that the
downward-pointing triangle of $\ell_i$ is adjacent to the upward-pointing triangle of $\ell_{i+1}$ for $1 \leq i < n$
and the downward-pointing triangle of $\ell_n$ is adjacent to the upward-pointing triangle of $\ell_1$. The smallest
cycle of lozenges is a three-cycle; see Figure~\ref{fig:three-cycle}.

\begin{figure}[!ht]
    \includegraphics[scale=1]{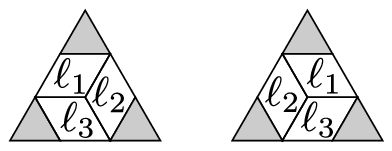}
    \caption{$T_3(x^2, y^2, z^2)$ has two tilings, both are three-cycles of lozenges.}
    \label{fig:three-cycle}
\end{figure}

Let $\sigma = \{\ell_1, \ldots, \ell_n\}$ be an $n$-cycle of lozenges in the tiling $\tau$ of $T$. If we replace the
lozenges in $\sigma$ be the $n$ lozenges created by adjoining the downward-pointing triangle of $\ell_i$ with the
upward-pointing triangle of $\ell_{i+1}$ for $1 \leq i < n$ and the downward-pointing triangle of $\ell_n$ with the
upward-pointing triangle of $\ell_1$, then we get a new tiling $\tau'$ of $T$. We call this new tiling the \emph{twist of $\sigma$}%
\index{lozenge!cycle of!twist of}
in $\tau$. The two three-cycles in Figure~\ref{fig:three-cycle} are twists of each other. See
Figure~\ref{fig:cycle-twist} for another example of twisting a cycle. A puncture is \emph{inside}%
\index{puncture!inside a cycle}
the cycle $\sigma$ if the lozenges of the cycle fully surround the puncture. In Figure~\ref{fig:cycle-twist}(i), the
puncture associated to $x y^4 z^2$ is inside the cycle $\sigma$ and all other punctures of $T$ are not inside the cycle $\sigma$.

\begin{figure}[!ht]
    \begin{minipage}[b]{0.48\linewidth}
        \centering
        \includegraphics[scale=1]{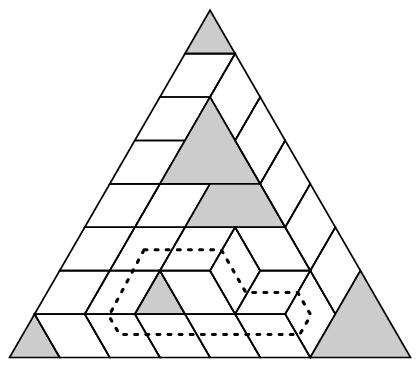}\\
        \emph{(i) A $10$-cycle $\sigma$.}
    \end{minipage}
    \begin{minipage}[b]{0.48\linewidth}
        \centering
        \includegraphics[scale=1]{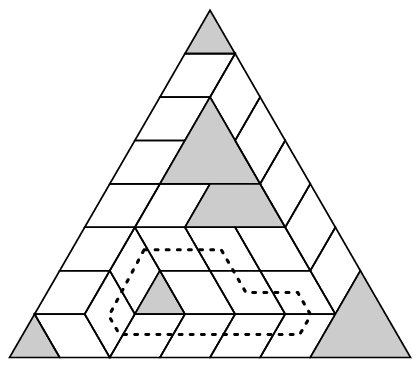}\\
        \emph{(ii) The twist of $\sigma$ in $\tau$.}
    \end{minipage}
    \caption{A $10$-cycle $\sigma$ in the tiling $\tau$ (see Figure~\ref{fig:triregion-tiling}) and its twist.}
    \label{fig:cycle-twist}
\end{figure}

Recall that the perfect matching sign of a tiling $\tau$ is denoted by $\msgn{\tau}$ (see Definition~\ref{def:pm-sign}).

\begin{lemma} \label{lem:twist-sign}
    Let $\tau$ be a lozenge tiling of a triangular region $T = T_d(I)$, and let $\sigma$ be an $n$-cycle of lozenges in $\tau$.
    Then the twist $\tau'$ of $\sigma$ in $\tau$ satisfies $\msgn{\tau'} = (-1)^{n-1}\msgn{\tau}$.
\end{lemma}
\begin{proof}
    Let $\pi$ and $\pi'$ be the perfect matching permutations associated to $\tau$ and $\tau'$, respectively.  Without loss of generality,
    assume $\ell_i$ corresponds to the upward- and downward-pointing triangles labeled $i$.  As $\tau'$ is a twist of $\tau$ by $\sigma$,
    then $\pi'(i) = i+1$ for $1 \leq i < n$ and $\pi'(n) = 1$.  That is, $\pi' = (1, 2, \ldots, n) \cdot \pi$, as permutations.  Hence,
    $\msgn{\tau'} = (-1)^{n-1}\msgn{\tau}$.
\end{proof}

\subsubsection{Resolutions, cycles of lozenges, and signs}\label{subsub:rez-cyc}~\par

Resolving a puncture modifies the length of a cycle of lozenges in a predictable fashion. We first need a definition. It
uses the starting and end points of lattice paths $A_1,\ldots,A_m$ and $E_1,\ldots,E_m$, as introduced at the beginning
of Subsection~\ref{sub:nilp}.

The \emph{$E$-count}%
\index{lozenge!cycle of!$E$-count}
of a cycle is the number of lattice path end points $E_j$ ``inside'' the cycle. Alternatively, this
can be seen as the sum of the side lengths of the non-overlapping punctures plus the sum of the side lengths of the
minimal covering regions of pairs of overlapping punctures. For example, the cycles shown in
Figure~\ref{fig:three-cycle} have $E$-counts of zero, the cycles shown in Figure~\ref{fig:cycle-twist} have $E$-counts
of $1$, and the (unmarked) cycle going around the outer edge of the tiling shown in Figure~\ref{fig:cycle-twist}(i) has an
$E$-count of $1 + 3 = 4$.

Now we describe the change of a cycle surrounding a puncture when this puncture is resolved.

\begin{lemma} \label{lem:cycle-res}
    Let $\tau$ be a lozenge tiling of $T = T_d(I)$, and let $\sigma$ be an $n$-cycle of lozenges in $\tau$.
    Suppose $T$ has a puncture $P$ (or a minimal covering region of a pair of overlapping punctures) with $E$-count $k$.
    Let $T'$ be a resolution of $P$ relative to $\tau$.  Then the resolution takes $\sigma$ to an $(n+kl)$-cycle of
    lozenges $\sigma'$ in the resolution, where $l$ is the number of times the splitting chains of the resolution
    cross the cycle $\sigma$ in $\tau$.  Moreover, $l$ is odd if and only if $P$ is inside $\sigma.$
\end{lemma}
\begin{proof}
    Fix a resolution $T' \subset \mathcal{T}_{d+2k}$ of $P$ with tiling $\tau'$ as induced by $\tau$.

    First, note that if $P$ is a minimal covering region of a pair of overlapping punctures, then any cycle of lozenges must
    avoid the lozenges present in $P$ as all such lozenges are forcibly chosen, i.e., immutable.  Thus, all lozenges of
    $\sigma$ are present in $\tau'$.

    The resolution takes the cycle $\sigma$ to a cycle $\sigma'$ by adding $k$ new lozenges for each unit edge of a lozenge in $\sigma$
    that belongs to a splitting chain.  More precisely, such an edge is expanded to $k+1$ parallel edges.  Any two consecutive
    edges form the opposite sides of a lozenge (see Figure~\ref{fig:cycle-insert}).  Thus, each time a splitting chain of the resolution
    crosses the cycle $\sigma$ we insert $k$ new lozenges.  As $l$ is the number of times the splitting chains of the resolution
    cross the cycle $\sigma$ in $\tau$, the resolution adds exactly $k l$ new lozenges to the extant lozenges of $\sigma$.
    Thus, $\sigma'$ is an $(n + k l)$-cycle of lozenges in $\tau'$.

    \begin{figure}[!ht]
        \includegraphics[scale=1]{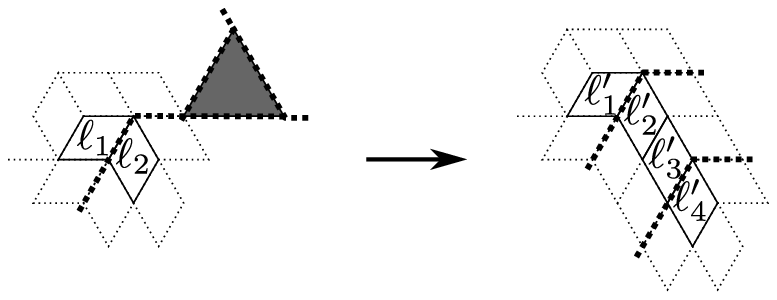}
        \caption{Expansion of a lozenge cycle at a crossing of a splitting chain.}
        \label{fig:cycle-insert}
    \end{figure}

    Since the splitting chains are going from $P$ to the exterior triangle of $\mathcal{T}_d$, the splitting chains terminate
    outside the cycle.  Hence if the splitting chain crosses into the cycle, it must cross back out.  If $P$ is outside
    $\sigma$, then the splitting chains start outside $\sigma$, and so $l$ must be even.  On the other hand,
    if $P$ is inside $\sigma$, then the splitting chains start inside of $\sigma$, and so $l = 3 + 2j$, where $j$
    is the number of times the splitting chains cross \emph{into} the cycle.
\end{proof}

Let $\tau_1$ and $\tau_2$ be tilings of $T$, and let $\pi_1$ and $\pi_2$ be their respective perfect matching permutations.  Suppose
$\pi_2 = \rho \pi_1$, for some permutation $\rho$.  Write $\rho$ as a product of disjoint cycles whose length is at least two.
(Note that these cycles will be of length at least three, as discussed above.)
Then each factor corresponds to cycle of lozenges of $\tau_1$. If all these cycles are twisted we get $\tau_2$. We call these
lozenge cycles the \emph{difference cycles} of $\tau_1$ and $\tau_2$.%
\index{lozenge!tiling!difference cycles of}

Using the idea of difference cycles, we characterise when two tilings have the same perfect matching sign.

\begin{corollary} \label{cor:msgn-cycle}
    Let $\tau$ be a lozenge tiling of $T = T_d(I)$, and let $\sigma$ be an $n$-cycle of lozenges in $\tau$.  Then
    the following statements hold.
    \begin{itemize}
        \item The $E$-count of $\sigma$ is even if and only if $n$ is odd.
        \item Two lozenge tilings of $T$ have the same perfect matching sign if and only if the sum of the
            $E$-counts of the difference cycles is even.
    \end{itemize}
\end{corollary}
\begin{proof}
    Suppose $T$ has $a$ punctures and pairs of overlapping punctures, $P_1, \ldots, P_a$, inside $\sigma$ that are
    \emph{not} in a corner, i.e., not associated to $x^k$, $y^k$, or $z^k$, for some $k$.  Let $j_i$ be the
    $E$-count of $P_i$.  Similarly, suppose $T$ has $b$ punctures and pairs of overlapping punctures, $Q_1, \ldots, Q_b$,
    outside $\sigma$ that are \emph{not} in a corner, i.e., not associated to $x^k$, $y^k$, or $z^k$, for some $k$.
    Let $k_i$ be the $E$-count of $Q_i$.

    If we resolve all of the punctures $P_1, \ldots, P_a, Q_1, \ldots, Q_b$, then $\sigma$ is taken to a  cycle
    $\sigma'$.  By Lemma~\ref{lem:cycle-res}, $\sigma'$ has length
    \[
        n' := n + (j_1 l_1 + \cdots + j_a l_a) + (k_1 m_1 + \cdots + k_b m_b),
    \]
    where the integers $l_1, \ldots, l_a$ are odd and the integers $m_1, \ldots, m_b$ are even.

    Moreover, the $a + b$ times resolved region $T'$ is, after merging touching punctures, the region of
    some complete intersection, i.e., of the form $T_e(x^a, y^b, z^c)$, for appropriate values of
    $a, b, c,$ and $e$.  By \cite[Theorem~1.2]{CGJL}, every tiling of $T'$ is thus obtained from any
    other tiling of $T'$ through a sequence of three-cycle twists, as in Figure~\ref{fig:three-cycle}.
    By Lemma~\ref{lem:twist-sign}, such twists do not change the perfect matching sign of the tiling,
    hence $n'$ is an odd integer.

    Since $n'$ is odd, $n' - (k_1 m_1 + \cdots + k_b m_b) = n + (j_1 l_1 + \cdots + j_a l_a)$ is also odd.
    Thus, $n$ is odd if and only if $j_1 l_1 + \cdots + j_a l_a$ is even.  Since the integers $l_1, \ldots, l_a$ are odd,
    we see that $j_1 l_1 + \cdots + j_a l_a$ is even if and only if an even number of the $l_i$ are odd, i.e., the sum
    $l_1 + \cdots + l_a$ is even.  Notice that this sum is the $E$-count of $\sigma$. Thus, claim (i) follows.

    Suppose two tilings $\tau_1$ and $\tau_2$ of $T$ have difference cycles $\sigma_1, \ldots, \sigma_p$.  Then
    by Lemma~\ref{lem:twist-sign}, $\msgn{\tau_2} = \sgn{\sigma_1} \cdots \sgn{\sigma_p} \msgn{\tau_1}$.  By claim (i), $\sigma_i$
    is a cycle of odd length if and only if the $E$-count of $\sigma_i$ is even.  Thus, $\sgn{\sigma_1} \cdots \sgn{\sigma_p} = 1$ if and only if
    an even number of the $\sigma_i$ have an odd $E$-count.  An even number of the $\sigma_i$ have an odd $E$-count if and only
    if the sum of the $E$-counts of $\sigma_1, \ldots, \sigma_p$ is even.  Hence, claim (ii) follows.
\end{proof}

\begin{remark}
    In \cite[Theorem~1.2]{CGJL}, Chen, Guo, Jin, and Liu show that three-cycles in the perfect matching permutation
    correspond to adding or removing boxes from the stacks the tiling represents.  See Figure~\ref{fig:pp-tile} where
    a lozenge tiling of a complete intersection region is shown in correspondence with a plane partition, i.e.,
    boxes stacked in a room.
\end{remark}

The lattice path permutation also changes predictably when twisting a cycle of lozenges.  To see this we single out certain punctures. We recursively define a puncture of $T \subset \mathcal{T}_d$ to be
a \emph{non-floating} puncture if it  touches the boundary of $ \mathcal{T}_d$ or if it overlaps or touches a non-floating puncture of $T$. Otherwise we call a puncture
a \emph{floating} puncture.%
\index{puncture!floating}

We also distinguish between \emph{preferred} and \emph{acceptable directions} on the splitting chains used for resolving
a puncture. Here we use again the perspective of a particle that starts at a vertex of the puncture and moves on a chain
to the corresponding corner vertex of $\mathcal{T}_d$. Our convention is:

\begin{itemize}
    \item On the lower-left chain the preferred direction are Southwest and West, the acceptable directions are Northwest and Southeast.
    \item On the lower-right chain the preferred directions are Southeast and East, the acceptable directions are Northeast and Southwest.
    \item On the top chain the preferred directions are Northeast and Northwest, the acceptable directions are East and West.
\end{itemize}

\begin{lemma} \label{lem:lpsgn-cycle}
    Let $\tau$ be a lozenge tiling of $T = T_d(I)$,  and let $\sigma$ be a cycle of lozenges in $\tau$.  Then the lattice
    path signs of $\tau$ and the twist of $\sigma$ in $\tau$ are the same if and only if the $E$-count of $\sigma$ is even.
\end{lemma}
\begin{proof}
    Suppose $T$ has $n$  floating punctures.  We proceed by induction on $n$ in five steps.

    \emph{Step $1$: The base case.}

    If $n = 0$, then every tiling of $T$ induces the same bijection $\{A_1,\ldots,A_m\} \to \{E_1,\ldots,E_m\}$. Thus, all tilings have the same lattice path sign.  Since $T$ has no
    floating punctures, $\sigma$  has an $E$-count of zero.  Hence, the claim is true if $n=0$.

    \emph{Step $2$: The set-up.}

    Suppose now that $n > 0$, and choose $P$ among the floating punctures and the minimal covering regions of two
    overlapping floating punctures of $T$ as the one that covers the upward-pointing unit triangle of $\mathcal{T}_d$
    with the smallest monomial label. Let $s > 0$ be the side length of $P$, and let $k$ be the $E$-count of $\sigma$.
    Furthermore, let $\upsilon$ be the lozenge tiling of $T$ obtained as twist of $\sigma$ in $\tau$. Both, $\tau$ and
    $\upsilon$, induce bijections $\{A_1,\ldots,A_m\} \to \{E_1,\ldots,E_m\}$, and we denote by $\lambda \in \PS_m$ and
    $\mu \in \PS_m$ the corresponding lattice path permutations, respectively. We have to show $\lpsgn \tau = (-1)^k \lpsgn \upsilon$, that is,
    \[
        \sgn \lambda = (-1)^k \sgn \mu.
    \]

    \emph{Step $3$: Resolutions.}

    We resolve $P$ relative to the tilings $\tau$ and $\upsilon$, respectively. For the resolution of $P$ relative to
    $\tau$, choose the splitting chains so that each unit edge has a preferred direction, except possibly the unit edges on
    the boundary of a puncture of $T$;  this is always possible. By our choice of $P$, no other floating punctures are to the
    lower-right of $P$. It follows that no edge on the lower-right chain crosses a lattice path, except possibly at the end
    of the lattice path.

    For the resolution of $P$ relative to $\upsilon$, use the splitting chains described in the previous paragraph, except
    for the edges that cross the lozenge cycle $\sigma$. They have to be adjusted since these unit edges disappear when
    twisting $\sigma$. We replace each such unit edge by a unit edge in an acceptable direction followed by a unit edge in a
    preferred direction so that the result has the same starting and end point as the unit edge they replace. Note that this
    is always possible and that this determines the replacement uniquely. The new chains meet the requirements on splitting
    chains.

    Using these splitting chains we resolve the puncture $P$ relative to $\tau$ and $\upsilon$, respectively. The result is
    a triangular region $T' \subset \mathcal{T}_{d+2s}$ with induced tilings $\tau'$ and $\upsilon'$, respectively. Denote
    by $\sigma'$ the extension of the cycle $\sigma$ in $T'$ (see Lemma~\ref{lem:cycle-res}). Since $\tau$ and $\upsilon$
    differ exactly on the cycle $\sigma$ and the splitting chains were chosen to be the same except on $\sigma$, it follows
    that twisting $\sigma'$ in $\tau'$ results in the tiling $\upsilon'$ of $T'$.

      \begin{figure}[!ht]
        \includegraphics[scale=1]{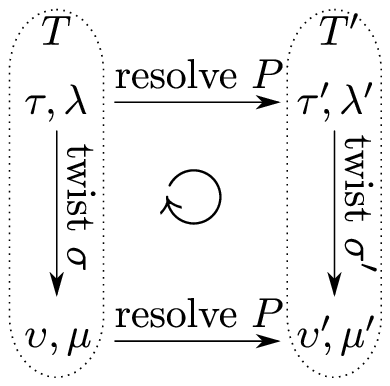}
        \caption{The commutative diagram used in the proof of Lemma~\ref{lem:lpsgn-cycle}.}
        \label{fig:res-comm-diag}
    \end{figure}

    \emph{Step $4$: Lattice path permutations.}

    Now we compare the signs of $\lambda, \mu \in \PS_m$ with the signs of  $\lambda'$ and $\mu'$, the lattice path permutations
    induced by the tilings $\tau'$ and $\upsilon'$ of $T'$, respectively.

    First, we compare the starting and end points of lattice paths in $T$ and $T'$. Resolution of the puncture identifies
    each starting and end point in $T$ with one such point in $T'$. We refer to these points as the \emph{old} starting and
    end points in $T'$. Note that the end points on the puncture $P$ correspond to the end points on the puncture in the
    Southeast corner of $T'$. The starting points in $T$ that are on one of the splitting chains used for resolving $P$
    relative to $\tau$ and $\upsilon$ are the same. Assume there are $t$ such points. After resolution, each point gives
    rise to a new starting and end point in $T'$. Both are connected by a lattice path that is the same in both resolutions
    of $P$. Hence, in order to compare the signs of the permutations $\lambda'$ and $\mu'$ on $m+t$ letters, it is enough to
    compare the lattice paths between the old starting and end points in both resolutions.

    Retain for these points the original labels used in $T$. Using this labeling, the lattice paths induce permutations
    $\tilde{\lambda}$ and $\tilde{\mu}$ on $m$ letters. Again, this is the same process in both resolutions. It follows that
    \begin{equation}\label{eq:compare-res-signs}
        \sgn (\tilde{\lambda}) \cdot \lpsgn (\tau') = \sgn (\tilde{\mu}) \cdot \lpsgn (\upsilon').
    \end{equation}

    Assume now that $P$ is a puncture. Then the end points on $P$ are indexed by $s$ consecutive integers. Since we retain
    the labels, the same indices label the end points on the puncture in the Southeast corner of $T'$. The end points on $P$
    correspond to the points in $T'$ whose labels are obtained by multiplying by $x^s y^s$. Consider now the case, where all
    edges in the lower-right splitting chain in $T$ are in preferred directions. Then the lattice paths induced by $\tau'$
    connect each point in $T'$ that corresponds to an end point on $P$ to the end point in the Southeast corner of $T'$ with
    the same index. Thus, $\sgn (\lambda) = \sgn (\tilde{\lambda})$. Next, assume that there is exactly one edge in
    acceptable direction on the lower-right splitting chain of $T$. If this direction is Northeast, then the $s$ lattice
    paths passing through the points in $T'$ corresponding to the end points on $P$ are moved one unit to the North. If the
    acceptable direction was Southwest, then the edge in this direction leads to a shift of these paths by one unit to the
    South. In either case, this shift means that the paths in $T$ and $T'$ connect to end points that differ by $s$
    transpositions, so $\sgn (\tilde{\lambda})= (-1)^{s} \sgn (\lambda)$. More generally, if $j$ is the number of unit edges
    on the lower-right splitting chain of $T$ that are in acceptable directions, then
    \[
        \sgn (\tilde{\lambda}) = (-1)^{js} \sgn (\lambda).
    \]

    Next, denote by $c$ the number of unit edges on the lower-right splitting chain that have to be adjusted when twisting
    $\sigma$. Since each of these edges is replaced by an edge in a preferred and one edge in an acceptable direction, after
    twisting the lower-right splitting chain in $T$ has exactly $j+c$ unit edges in acceptable directions. It follows as
    above that
    \[
        \sgn (\tilde{\mu}) = (-1)^{(j+c)s} \sgn (\mu).
    \]
    Since a unit edge on the splitting chain has to be adjusted when twisting if and only if it is shared by two
    consecutive lozenges in the cycle $\sigma$, the number $c$ is even if and only if the puncture $P$ is outside
    $\sigma$.

   Moreover, as the puncture $P$ has been resolved in $T'$, we conclude by induction that $\tau'$ and $\upsilon'$ have
   the same lattice path sign if and only if the $E$-count of $\sigma'$ is even. Thus, we get
    \begin{equation}
        \lpsgn (\upsilon') =
        \begin{cases}
            (-1)^{k-s} \lpsgn (\tau') & \text{if $P$ is inside $\sigma$}, \\
            (-1)^{k} \lpsgn (\tau') & \text{if $P$ is outside $\sigma$}.
        \end{cases}
    \end{equation}

    \emph{Step $5$: Bringing it all together.}

    We consider the two cases separately:
    \begin{enumerate}
        \item Suppose $P$ is inside $\sigma$.  Then $c$ is odd. Hence,  the above considerations imply
            \begin{equation*}
                \begin{split}
                    \sgn (\lambda) & =  (-1)^{js}\sgn (\tilde{\lambda}) =  (-1)^{js + k-s}\sgn (\tilde{\mu}) \\
                                   & =  (-1)^{js + k-s + (j+c) s} \sgn ({\mu}) \\
                                   & = (-1)^k  \sgn ({\mu}),
                \end{split}
            \end{equation*}
            as desired.

        \item Suppose $P$ is outside of $\sigma$.  Then $c$ is even, and we conclude
            \begin{equation*}
                \begin{split}
                    \sgn (\lambda) & =  (-1)^{js}\sgn (\tilde{\lambda}) =  (-1)^{js + k}\sgn (\tilde{\mu}) \\
                                   & =  (-1)^{js + k-s + (j+c) s} \sgn ({\mu}) \\
                                   & = (-1)^k  \sgn ({\mu}).
                 \end{split}
            \end{equation*}
    \end{enumerate}

    Finally, it remains to consider the case where $P$ is the minimal covering region of two overlapping punctures of
    $T$. Let $\hat{T}$ be the triangular region that differs from $T$ only by having $P$ as a puncture, and let
    $\hat{\tau}$ and $\hat{\upsilon}$ be the tilings of $\hat{T}$ induced by $\tau$ and $\upsilon$, respectively. Since
    we order the end points of lattice paths using monomial labels, it is possible that the indices of the end points on
    the Northeast boundary of $P$ in $\tilde{T}$ differ from those of the points on the Northeast boundary of the
    overlapping punctures in $T$. However, the lattice paths induced by $\tau$ and $\upsilon$ connecting the points on
    the Northeast boundary of $P$ to the points on the Northeast boundary of the overlapping punctures are the same.
    Hence the lattice paths sign of $\tau$ and $\hat{\tau}$ differ in the same ways as the signs of $\upsilon$ and
    $\hat{\upsilon}$. Since we have shown our assertion for $\hat{\tau}$ and $\hat{\upsilon}$, it also follows for
    $\tau$ and $\upsilon$.
\end{proof}

Using difference cycles, we now characterise when two tilings of a region have the same lattice path sign.

\begin{corollary} \label{cor:lpsgn-cycle}
    Let $T = T_d(I)$ be a non-empty, balanced triangular region. Then two tilings of $T$ have the same lattice path sign
    if and only if the sum of the $E$-counts (which may count some end points $E_j$ multiple times) of the difference cycles
    is even.
\end{corollary}
\begin{proof}
    Suppose two tilings $\tau_1$ and $\tau_2$ of $T$ have difference cycles $\sigma_1, \ldots, \sigma_p$.
    By Lemma~\ref{lem:lpsgn-cycle}, $\lpsgn{\tau_1} = \lpsgn{\tau_2}$ if and only if an even number of the $\sigma_i$ have an odd $E$-count.
    The latter is equivalent to the sum of the $E$-counts of $\sigma_1, \ldots, \sigma_p$ being even.
\end{proof}

Our above results imply that the two signs that we assigned to a given lozenge tiling, the perfect matching sign (see
Definition~\ref{def:pm-sign}) and the lattice path sign (see Definition~\ref{def:nilp-sign}), are the same up to a
scaling factor depending only on $T$. The main result of this section follows now easily.

\begin{theorem} \label{thm:detZN}
    Let $T = T_d(I)$ be a balanced triangular region.  The following statements hold.
    \begin{enumerate}
        \item Let $\tau$ and $\tau'$ be two lozenge tilings of $T$. Then their perfect matching signs are
            the same if and only ifvtheir lattice path signs are the same, that is,
            \[
                \msgn (\tau) \cdot \lpsgn (\tau) = \msgn (\tau') \cdot \lpsgn (\tau').
            \]
        \item In particular, we have that
            \[
                |\det{Z(T)}| = |\det{N(T)}|.
            \]
    \end{enumerate}
\end{theorem}
\begin{proof}
    Consider two lozenge tilings of $T$. According to Corollaries~\ref{cor:msgn-cycle} and~\ref{cor:lpsgn-cycle}, they have
    the same perfect matching and the same lattice path signs if and only if the sum of the $E$-counts of the difference
    cycles is even. Hence using Theorems~\ref{thm:pm-matrix} and~\ref{thm:nilp-matrix}, it follows that $|\det{Z(T)}| =
    |\det{N(T)}|$.
\end{proof}

This result allows us to move freely between the points of view using lozenge tilings, perfect matchings, and families
of non-intersecting lattice paths, as needed. In particular, it implies that rotating a triangular region by
$120^{\circ}$ or $240^{\circ}$ does not change the enumerations. Thus, for example, the three matrices described in
Remark~\ref{rem:rotations} as well as the matrix given in Example~\ref{exa:Z-matrix} all have the same determinant, up
to sign.

\section{Determinants} \label{sec:det}

By Theorems~\ref{thm:pm-matrix} and~\ref{thm:nilp-matrix}, the enumerations of signed lozenge tilings of a balanced
triangular region, where the sign is determined by perfect matchings (see Definitions~\ref{def:pm-sign}) or lattice
paths (see Definition~\ref{def:nilp-sign}), are both given by determinants of integer matrices. Furthermore, the
absolute values of the two determinants are the same by Theorem~\ref{thm:detZN}. In this section we determine these
determinants in various cases. If the determinant is non-vanishing, then we are also interested in its prime divisors,
and, failing that, an upper bound on the prime divisors of the enumeration. This is important for applications later on.

~\subsection{Building enumerations by replacement}\label{sub:det-replace}~\par

Recall that, by Lemma~\ref{lem:replace-tileable}, removing a tileable region does not affect unsigned tileability. Using
the structure of the bi-adjacency matrix $Z(T)$, we analyse how removing a balanced region affects signed enumerations.

\begin{proposition} \label{pro:rep-enum}
    Let $T \subset \mathcal{T}_d$   be a balanced  subregion, and let $U$ be a balanced monomial subregion of $T$.   The
    following statements hold.
    \begin{enumerate}
        \item $\per{Z(T)} = \per{Z(T \setminus U)} \cdot \per{Z(U)}$; and
        \item $|\det{Z(T)}| = |\det{Z(T \setminus U)} \cdot \det{Z(U)}|$.
    \end{enumerate}
\end{proposition}
\begin{proof}
    Recall that the rows of the matrices $Z(\cdot)$ are indexed by the downward-pointing triangles, and the columns of
    the matrices $Z(\cdot)$ are indexed by the upward-pointing triangles, using the reverse-lexicographic order of their
    monomial labels. Reorder the downward-pointing (respectively, upward-pointing) triangles of $T$ so that the triangles of
    $T \setminus U$ come first and the triangles of $U$ come second, where we preserve the internal order of the triangles
    of $T \setminus U$ and $U$. Using this new ordering, we reorder the rows and columns of $Z(T)$. The result is a block
    matrix
    \[
        \left( \begin{array}{cc} Z(T \setminus U) & X \\ Y & Z(U) \end{array} \right).
    \]
    Since the downward-pointing triangles of $U$ are not adjacent  to any upward-pointing triangle of $T \setminus U$, the
    matrix $Y$ is a zero matrix. Thus,  the claims follow by using block matrix formul\ae\ for permanents and determinants.
\end{proof}

In particular, if we remove a monomial region with a unique lozenge tiling, then we do not modify the enumerations of
lozenge tilings in that region. This is true in greater generality.

\begin{proposition}\label{prop:remove-unique-tileable}
    Let $T \subset \mathcal{T}_d$ be a balanced subregion, and let $U$ be any subregion of $T$ such that each lozenge
    tiling of $T$ induces a tiling of $U$ and all the induced tilings of $U$ agree. Then:
    \begin{enumerate}
        \item $Z(T)$ has maximal rank if and only if $Z(T \setminus U)$ has maximal rank.
        \item $\per{Z(T)} = \per{Z(T \setminus U)}$ and $|\det{Z(T)}| = |\det{Z(T \setminus U)}|$.
    \end{enumerate}
\end{proposition}
\begin{proof}
    Part (ii) follows from Theorem \ref{thm:pm-matrix} and Proposition \ref{pro:per-enum}, and it implies part (i).
\end{proof}

We point out the following special case.

\begin{corollary}\label{cor:replace-two-punctures}
    Let $T = T_d (I)$ be a balanced triangular region with two punctures $P_1$ and $P_2$ that overlap or touch each
    other. Let $P$ be the minimal covering region of $P_1$ and $P_2$. The following statements hold.
    \begin{enumerate}
        \item $\per{Z(T)} = \per{Z(T \setminus P)}$; and
        \item $|\det{Z(T)}| = |\det{Z(T \setminus P)} |$.
    \end{enumerate}
\end{corollary}
\begin{proof}
    The monomial region $U := P \setminus (P_1 \cup P_2)$ is uniquely tileable. Hence the claims follows from
    Proposition~\ref{prop:remove-unique-tileable} because $T \setminus U = T \setminus P$.
\end{proof}

We give an example of such a replacement.

\begin{example} \label{exa:split-puncture}
    Let $T = T_d(I)$ be a balanced triangular region. Suppose the ideal $I$ has minimal generators $x^{a + \alpha} y^b
    z^c$ and $x^a y^{b + \beta} z^{c+\gamma}$. The punctures associated to these generators overlap or touch if and only if $a + \alpha
    + b + \beta + c + \gamma \leq d$. In this case, the minimal overlapping region $U$ of the two punctures is associated to
    the greatest common divisor $x^a y^b z^c$. Assume that $U$ is not overlapped by any other puncture of $T$. Then $U$ is
    uniquely tileable. Hence the regions $T$ and $T' = T \setminus U = T_d (I, x^a y^b z^c)$ have the same enumerations.
    Note that the ideal $(I, x^a y^b z^c)$ has fewer minimal generators than $I$. See Corollary~\ref{cor:ci-split} and
    Figure~\ref{fig:ci-split} for an illustration of a special case of splitting a puncture.
\end{example}~

The above procedure allows us in some cases to pass from a triangular region to a triangular region with fewer punctures.
Enumerations are typically more amenable to explicit evaluations if we have few punctures, as we will see in the next
subsection.

~\subsection{Mahonian determinants}~\par

MacMahon computed the number of plane partitions (finite two-dimensional arrays that weakly decrease in all columns and rows)
in an $a \times b \times c$ box as (see, e.g., \cite[Page 261]{Pr})
\[
    \Mac(a,b,c) := \frac{\HF(a) \HF(b) \HF(c) \HF(a+b+c)}{\HF(a+b) \HF(a+c) \HF(b+c)},
\]
\index{MacMahon's enumeration of plane partitions}
\index{0@\textbf{Symbol list}!Mac@$\Mac(a,b,c)$}
\index{0@\textbf{Symbol list}!HFn@$\HF(n)$}
where $a$, $b$, and $c$ are nonnegative integers and $\HF(n) := \prod_{i=0}^{n-1}i!$ is the \emph{hyperfactorial} of
$n$. David and Tomei proved in~\cite{DT} that plane partitions in an $a \times b \times c$ box are in bijection with
lozenge tilings in a hexagon with side lengths $(a,b,c)$, that is, a hexagon whose opposite sides are parallel and have
lengths $a, b$, and $c$, respectively. However, Propp states on~\cite[Page 258]{Pr} that Klarner was likely the first to
have observed this. See Figure~\ref{fig:pp-tile} for an illustration of the connection.

\begin{figure}[!ht]
    \begin{minipage}[b]{0.48\linewidth}
        \[
            \vspace{1em}
            \begin{array}{cccccc}
               3 & 3 & 2 & 2 & 2 & 1 \\
               3 & 2 & 2 & 1 & 0 & 0
            \end{array}
        \]
    \end{minipage}
    \begin{minipage}[b]{0.48\linewidth}
        \centering
        \includegraphics[scale=1]{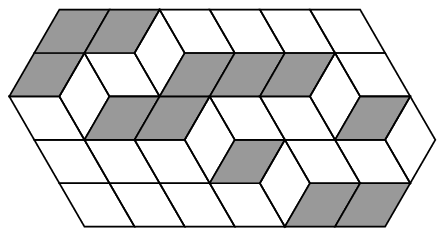}
    \end{minipage}
    \caption{An example of a $2 \times 6 \times 3$ plane partition and the associated lozenge tiling of a hexagon.
        The grey lozenges are the tops of the boxes.}
    \label{fig:pp-tile}
\end{figure}

We use the above formula in many explicit determinantal evaluations. As we are also interested in the prime divisors of
the various non-trivial enumerations we consider, we note that $\Mac(a,b,c) > 0$ and the prime divisors of $\Mac(a,b,c)$
are at most $a+b+c-1$ if $a,b$, and $c$ are positive. This bound is sharp if $a+b+c-1$ is a prime number. If one of $a,
b$, or $c$ is zero, then $\Mac(a,b,c) = 1$.

As a first example, we enumerate the (signed) lozenge tilings of a hexagon, i.e., the triangular region of a complete
intersection. (See Remark~\ref{rem:ci-history} for a brief history of results related to the following one.)

\begin{proposition} \label{pro:ci-enum}
    Let $a$, $b$, and $c$ be positive integers such that $a \leq b+c$,\; $b \leq a+c$, and $c \leq a+b$. Suppose that
    $d = \frac{1}{2}(a+b+c)$ is an integer. Then $T = T_d(x^a, y^b, z^c)$ is a hexagon with side lengths $(d-a, d-b, d-c)$ and
    \[
        |\det{Z(T)}| = \per{Z(T)} = \Mac(d-a, d-b, d-c).
    \]

    Moreover, the prime divisors of the enumeration are  bounded above by $d-1$.
\end{proposition}
\begin{proof}
    As $a \leq b + c$, we have $d = \frac{1}{2}(a+b+c) \geq \frac{1}{2}(a + a) = a$. Similarly, $d \geq b$ and $d \geq
    c$. Thus $T$ has three punctures of length $d-a$, $d-b$, and $d-c$ in the three corners. Moreover, $d -(d-a + d-b) =
    d-c$ is the distance between the punctures of length $d-a$ and $d-b$, and similarly for the other two puncture
    pairings.

    Thus, the unit triangles of $T$ form a hexagon with side lengths $(d-a,d-b,d-c)$. By MacMahon's formula we have
    $\per{Z(T)} = \Mac(d-a,d-b,d-c)$. Moreover, each lozenge tiling of $T$ induces the identity permutation as its
    lattice path permutation. Thus, all tilings of $T$ have the same perfect matching sign by Theorem~\ref{thm:detZN}.
    Hence, $|\det{N(T)}| = |\det{Z(T)}| = \Mac(d-a,d-b,d-c)$. The prime divisors of this integer are bounded above by
    $(d-a) + (d-b) + (d-c) - 1 = d-1$.
\end{proof}

Combining Example~\ref{exa:split-puncture} and the preceding proposition, we get the following result for a certain
region with four punctures.

\begin{corollary} \label{cor:ci-split}
    Let $T = T_d(x^{a+\alpha}, y^b, z^c, x^a y^{\beta} z^{\gamma})$, where $a$, $b$, $c$, and $d$ are as in
    Proposition~\ref{pro:ci-enum}, $\alpha$ is a positive integer, and $\beta$ and $\gamma$ are nonnegative integers,
    not both zero. Suppose additionally that $\alpha + \beta + \gamma \leq \frac{1}{2} (b+c-a)$. Then
    \[
        |\det{Z(T)}| = \per{Z(T)} = \Mac(d-a, d-b, d-c),
    \]
    and the prime divisors of the enumeration are bounded above by $d-1$.
\end{corollary}
\begin{proof}
    The assumption $\alpha + \beta + \gamma \leq \frac{1}{2} (b+c-a)$ is equivalent to $a + \alpha + \beta + \gamma \le
    d$. Thus, Example~\ref{exa:split-puncture} shows that $T$ has the same enumerations as $T_d(x^a, y^b, z^c)$, and we
    conclude using Proposition~\ref{pro:ci-enum}. See Figure~\ref{fig:ci-split} for an illustration of the triangular
    regions.
    \begin{figure}[!ht]
        \includegraphics[scale=1]{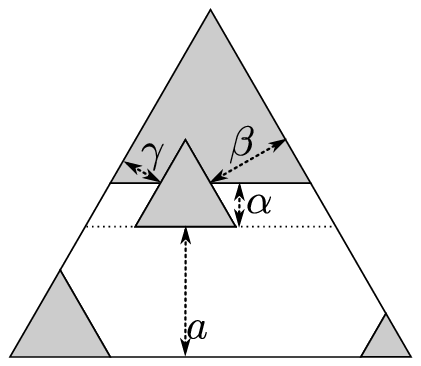}
        \caption{Covering the punctures associated to $x^{a+\alpha}$ and $x^a y^{\beta} z^{\gamma}$ by $x^a$.}
        \label{fig:ci-split}
    \end{figure}
\end{proof}

Moreover, combining Propositions~\ref{pro:rep-enum} and~\ref{pro:ci-enum} we get the enumeration for a slightly more
complicated triangular region. (We will use this observation in Section~\ref{sec:type-two}.) Clearly, the process of
removing a hexagon from a puncture can be repeated.

\begin{corollary} \label{cor:ci-nest}
    Let $T = T_d(x^{a+\alpha}, y^b, z^c, x^a y^{\beta}, x^a z^{\gamma})$, where the quadruples $(a,b,c,d)$ and $(\alpha,
    \beta, \gamma, d-a)$ are both as in Proposition~\ref{pro:ci-enum}. In particular, $a + \alpha + \beta + \gamma = b +
    c$ and $d = \frac{1}{2}(a+b+c)$. Then
    \[
        |\det{Z(T)}| = \per{Z(T)} = \Mac(d-a, d-b, d-c) \Mac(d-a-\alpha, d-a-\beta, d-a-\gamma),
    \]
    and the prime divisors of the enumeration are bounded above by $d-1$.
\end{corollary}
\begin{proof}
    The region $T$ is obtained from $T_d(x^a, y^b, z^c)$ by replacing the puncture associated to $x^a$ by
    $T_{d-a}(x^{\alpha}, y^{\beta}, z^{\gamma})$. See Figure~\ref{fig:ci-ci}.
    \begin{figure}[!ht]
        \includegraphics[scale=1]{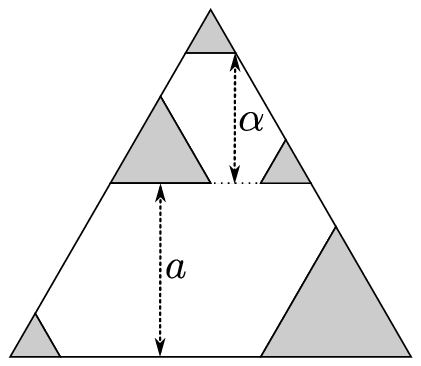}
        \caption{A simple hexagon with a puncture replaced by a simple hexagon.}
        \label{fig:ci-ci}
    \end{figure}
\end{proof}

A tileable triangular region with punctures in two corners and a third non-corner puncture also has a Mahonian
determinant since many of the tiles are fixed.

\begin{proposition} \label{pro:C-is-zero}
    Let $T = T_d(x^a, y^b, x^{\alpha} y^{\beta} z^{2d-(a+b+\alpha+\beta)})$, where $\alpha + b, a + \beta \leq d \leq a + b$.  Then
    \[
        |\det{Z(T)}| = \per{Z(T)} = \Mac(a+b-d, d-(\alpha + b), d-(a + \beta)),
    \]
    and the prime divisors of the enumeration are sharply bounded above by $d-(\alpha + \beta) - 1$.
\end{proposition}
\begin{proof}
    First we note that as $\alpha + b, a + \beta \leq d \leq a + b$, none of the punctures overlap.

    Now we consider the families of non-intersecting lattice paths in the lattice $L(T)$. Their end points $E_j$ are
    \emph{all} along the Northeast boundary of the puncture associated to the monomial $x^{\alpha} y^{\beta}
    z^{2d-(a+b+\alpha+\beta)}$. The lattice paths are thus confined to the region bounded by the starting and end points
    of the paths. This region is a hexagon of side lengths $(a+b-d, d-(\alpha + b), d-(a + \beta))$. See
    Figure~\ref{fig:two-corners} for an illustration.
    \begin{figure}[!ht]
        \includegraphics[scale=1]{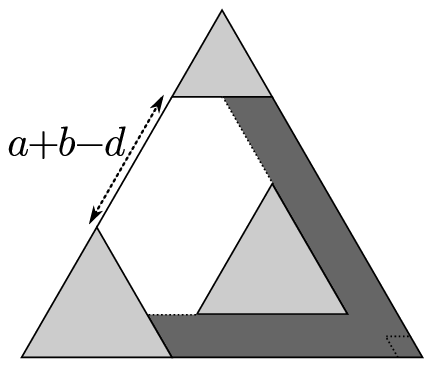}
        \caption{The portion of $T$ shaded dark grey has fixed tiles of the same orientation, leaving a monomial subregion of $T$ that is a  hexagon.}
        \label{fig:two-corners}
    \end{figure}
\end{proof}

Before computing our final Mahonian-type determinant, we need a more general determinant calculation, which may be of
independent interest.

\begin{lemma} \label{lem:split-binom-det}
    Let $M$ be an $n$-by-$n$ matrix  with entries
    \[
        (M)_{i,j} = \left\{
            \begin{array}{ll}
                \displaystyle \binom{p}{q + j - i}     & \mbox{if } 1     \leq j \leq m, \\[1.1em]
                \displaystyle \binom{p}{q + r + j - i} & \mbox{if } m + 1 \leq j \leq n, \\
            \end{array}
        \right.
    \]
    where $p,q,r,$ and $m$ are nonnegative integers and $1 \leq m \leq n$.  Then
    \[
        \det{M} = \Mac(m,q,r) \Mac(n-m, p-q-r, r) \frac{\HF(q+r)\HF(p-q)\HF(n+r)\HF(n+p)}{\HF(n+p-q)\HF(n+q+r)\HF(p)\HF(r)}.
    \]
\end{lemma}
\begin{proof}
    We begin by  using~\cite[Equation~(12.5)]{CEKZ} to evaluate $\det{M}$ to be
    \[
        \prod_{1\leq i < j \leq n} (L_j - L_i) \prod_{i=1}^n \frac{(p+i-1)!}{(n+p-L_i)!(L_i-1)!},
    \]
    where
    \[
        L_j =
        \begin{cases}
            q + j & \text{if } 1 \leq j \leq m, \\
            q + r + j & \text{if } m + 1 \leq j \leq n.
        \end{cases}
    \]
    If we split the products in the previously displayed equation relative to the split in $L_j$, then we get the
    following equations:
    \begin{equation*}
        \begin{split}
            \prod_{1\leq i < j \leq n} (L_j - L_i)
                = & \left(\prod_{1\leq i < j \leq m} (j - i)\right) \left(\prod_{m< i < j \leq n} (j - i)\right) \left(\prod_{1\leq i \leq m < j \leq n} (r+j-i)\right) \\[0.3em]
                = & \left(\HF(m)\right) \left(\HF(n-m)\right) \left(\frac{\HF(n+r) \HF(r)}{\HF(n+r-m) \HF(m+r)}\right)
        \end{split}
    \end{equation*}
    and
    \begin{equation*}
        \begin{split}
            \prod_{i=1}^n \frac{(p+i-1)!}{(n+p-L_i)!(L_i-1)!}
                = & \left( \prod_{i=1}^{n}(p+i-1)! \right) \left( \prod_{i=1}^m \frac{1}{(n+p-q-i)!(q+i-1)!}\right) \\[0.3em]
                  & \left( \prod_{i=m+1}^{n} \frac{1}{(n+p-q-r-i)!(q+r+i-1)!}\right) \\[0.3em]
                = & \left( \frac{\HF(n+p)}{\HF(p)} \right) \left( \frac{\HF(n+p-m-q)\HF(q)}{\HF(n+p-q)\HF(m+q)}  \right) \\[0.3em]
                  & \left( \frac{\HF(p-q-r)\HF(m+q+r)}{\HF(n+p-m-q-r)\HF(n+q+r)}  \right).
        \end{split}
    \end{equation*}
    Bringing these equations together we get that $\det{M}$ is
    {\footnotesize \[
        \frac{\HF(m)\HF(q)\HF(r)\HF(m+q+r)}{\HF(m+r)\HF(m+q)}
        \frac{\HF(n-m)\HF(p-q-r)\HF(n+p-m-q)}{\HF(n+r-m)\HF(n+p-m-q-r)}
        \frac{\HF(n+r)\HF(n+p)}{\HF(p)\HF(n+p-q)\HF(n+q+r)},
    \] }
    which, after minor manipulation, yields the claimed result.
\end{proof}

\begin{remark} \label{rem:split-binom-det}
    The preceding lemma generalises \cite[Lemma~2.2]{LZ}, which handles the case $r = 1$.  Furthermore, if $r = 0$,
    then $\det{M} = \Mac(n, p-q, q)$, as expected (see the running example, $\det \binom{a+b}{a-i+j}$, in~\cite{Kr}).
\end{remark}

A tileable, simply-connected triangular region with four non-floating punctures has a Mahonian-type determinant. This
particular region is of interest in Section~\ref{sec:type-two}. While in the previous evaluations we directly considered
a bi-adjacency matrix, this time we work primarily with a lattice path matrix and then use Theorem~\ref{thm:detZN}.

\begin{proposition} \label{pro:two-mahonian}
    Let $T = T_d(x^a, y^b, z^c, x^{\alpha} y^{\beta})$, where $d = \frac{1}{3}(a+b+c+\alpha+\beta)$ is an integer,
    $0 < \alpha < a$, $0 < \beta < b$, and $\max\{a,b,c,\alpha+\beta\} \leq d \leq \min\{a + \beta, \alpha + b, a+c, b+c\}$.
    Then $|\det{Z(T)}| = \per{Z(T)}$ is
    \begin{equation*}
        \begin{split}
                   & \Mac(a+\beta - d,d-a,d-(\alpha+\beta)) \Mac(\alpha+b-d,d-b,d-(\alpha+\beta)) \\[0.3em]
            \times & \frac{\HF(d-a+d-(\alpha+\beta))\HF(d-b+d-(\alpha+\beta))\HF(d-c+d-(\alpha+\beta))\HF(d)}{\HF(a)\HF(b)\HF(c)\HF(d-(\alpha+\beta))}.
        \end{split}
    \end{equation*}
    Moreover, the prime divisors of the enumeration are  bounded above by $d - 1$.
\end{proposition}
\begin{proof}
    Note that $\max\{a,b,c,\alpha+\beta\} \leq d$ implies that all four punctures have nonnegative side length. Further,
    the condition $d \leq \min\{a + \beta, \alpha + b, a+c, b+c\}$ guarantees that none of the punctures overlap.

    We now compute the lattice path matrix $N(T)$ as introduced in Subsection \ref{sub:nilp}. Recall that a point in the
    lattice $L(T)$ with label $x^u y^v z^{d-1-(u+v)}$ is identified with the point $(d-1-v, u) \in \ZZ^2$. Thus, the
    starting points of the lattice paths are
    \begin{equation*}
        A_i =
        \begin{cases}
            (d-b+i-1, d-b+i-1) & \text{if } 1 \leq i \leq \alpha + b - d, \\
            (2d-(\alpha + \beta+b) + i-1, 2d-(\alpha + \beta+b) + i-1) & \text{if } \alpha + b - d <  i \leq d-c.
        \end{cases}
    \end{equation*}
    For the end points of the lattice paths, we get
    \begin{equation*}
        E_j  = (c-1+j, j-1), \quad \text{where } 1 \leq j \leq d-c.
    \end{equation*}
    Thus, the entries of the lattice path matrix $N(T)$ are
    \[
        (N(T))_{i, j} =
        \left\{
            \begin{array}{ll}
                \displaystyle \binom{c}{d-b+i-j}                    & \mbox{if } 1                \leq i \leq \alpha+b-d, \\[0.8em]
                \displaystyle \binom{c}{2d-(\alpha+\beta+b)+i-j}    & \mbox{if } \alpha+b-d  <  i \leq d-c. \\
            \end{array}
        \right.
    \]
    Transposing $N(T)$, we get a matrix of the form in Lemma~\ref{lem:split-binom-det}, where $m = \alpha + b -d$, $n =
    d-c$, $p = c$, $q = d-b$, and $r = d-(\alpha+\beta)$. Thus, we have the desired determinant evaluation.

    Moreover, the only lattice path permutation that admits non-intersecting lattice paths is the identity permutation,
    so all families of non-intersecting lattice paths have the same sign. Hence $|\det{Z(T)}| = \per{Z(T)}$ by Theorem
    \ref{thm:detZN}.

    Finally, as $d-\alpha$ and $d-\beta$ are smaller than $d$, the prime divisors of $|\det{N(T)}| = |\det{Z(T)}|$ are
    bounded above by $d-1$.
\end{proof}

\begin{remark} \label{rem:two-mahonian}
    The evaluation of the determinant in the preceding proposition includes two Mahonian terms and a third non-Mahonian
    term. It should be noted that both hexagons associated to the Mahonian terms actually show up in the punctured hexagon.
    \begin{figure}[!ht]
        \includegraphics[scale=1]{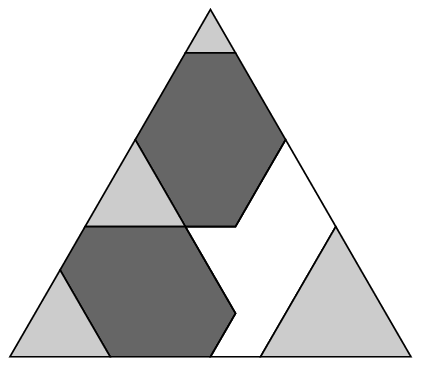}
        \caption{The darkly shaded hexagons correspond to the two Mahonian terms.}
        \label{fig:two-mahonian}
    \end{figure}
    See Figure~\ref{fig:two-mahonian} where the darkly shaded hexagons correspond to the Mahonian terms. It is not clear
    (to us) where the third term comes from, though it may be of interest that if one subtracts $d-(\alpha+\beta)$ from
    each hyperfactorial parameter, before the evaluation, then what remains is $\Mac(d-a,d-b,d-c)$.
\end{remark}~

\subsection{A single sign}~\par

We exhibit further triangular regions such that all lozenge tilings have the same sign by partially extending
Proposition~\ref{pro:ci-enum}. Indeed, this is guaranteed to happen if all floating punctures (see the definition preceding
Lemma~\ref{lem:lpsgn-cycle}) have an even side length.

\begin{proposition} \label{pro:same-sign}
    Let $T$ be a tileable triangular region, and suppose all floating punctures of $T$ have an even side length. Then
    every lozenge tiling of $T$ has the same perfect matching sign as well as the same lattice path sign, and so
    $\per{Z(T)} = |\det{Z(T)}|$.

    In particular, simply-connected regions that are tileable have this property.
\end{proposition}
\begin{proof}
    The equality of the perfect matching signs follows from Corollary~\ref{cor:msgn-cycle}, and the equality of the lattice
    path signs from Corollary~\ref{cor:lpsgn-cycle}. Now Theorem~\ref{thm:pm-matrix} implies $\per{Z(T)} = |\det{Z(T)}|$.

    The second part is immediate as simply-connected regions have no floating punctures.
\end{proof}

\begin{remark}
    The above proposition vastly extends \cite[Theorem~1.2]{CGJL}, where hexagons (as in Proposition~\ref{pro:ci-enum})
    are considered, using a different approach. This special case was also established independently in
    \cite[Section~3.4]{Ke}, with essentially the same proof as~\cite{CGJL}.

    Proposition~\ref{pro:same-sign} can also be derived from Kasteleyn's theorem on enumerating perfect
    matchings~\cite{Ka}. To see this, notice that in the case, where all floating punctures have even side lengths, all
    ``faces'' of the bipartite graph $G(T)$ have size congruent to $2 \pmod{4}$.
\end{remark}

We now extend Proposition~\ref{pro:same-sign}. To this end we define the \emph{shadow} of a puncture to be the region of
$T$ that is both below the puncture and to the right of the line extending from the upper-right edge of the puncture.
See Figure~\ref{fig:puncture-shadow}.

\begin{figure}[!ht]
    \includegraphics[scale=1]{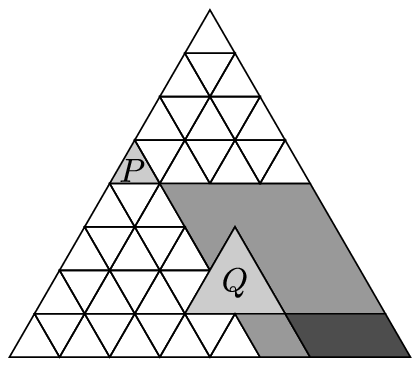}
    \caption{The puncture $P$ has the puncture $Q$ in its shadow (light grey), but $Q$ does not have a puncture in its shadow (dark grey).}
    \label{fig:puncture-shadow}
\end{figure}

\begin{proposition}\label{prop:same-sign}
    Let $T = T_d(I)$ be a balanced triangular region. If all floating punctures (and minimal covering regions of
    overlapping punctures) with other punctures in their shadows have even side length, then any two lozenge tilings of
    $T$ have the same perfect matching and the same lattice path sign. Thus, $\per{Z(T)} = |\det{Z(T)}|$.

    In particular, $Z(T)$ has maximal rank over a field of characteristic zero if and only if $T$ is tileable.
\end{proposition}
\begin{proof}
    Let $P$ be a floating puncture or a minimal covering region with no punctures in its shadow. Then the shadow of $P$ is
    uniquely tileable, and thus the lozenges in the shadow are fixed in each lozenge tiling of $T$. Hence, no cycle of
    lozenges in any tiling of $T$ can contain $P$. Using Corollary~\ref{cor:msgn-cycle} and
    Corollary~\ref{cor:lpsgn-cycle}, we see that $P$ does not affect the sign of the tilings of $T$.

    Now our assumptions imply that all floating punctures (or minimal covering regions of overlapping punctures) of
    $T$ that can be contained in a difference cycle of two lozenge tilings of $T$ have even side length. Thus, we
    conclude $\per{Z(T)} = |\det{Z(T)}|$ as in the proof of Proposition~\ref{pro:same-sign}.

    The last assertion follows by Proposition \ref{pro:per-enum}.
\end{proof}~

\subsection{An axes-central puncture}\label{sub:axes-central}~\par

Last, we consider the case studied by Ciucu, Eisenk\"olbl, Krattenthaler, and Zare in~\cite{CEKZ}, that is, the case of
a hexagon with a central puncture that is equidistant from the hexagon sides along axes through the puncture. We call
such a puncture an \emph{axes-central puncture}.%
\index{puncture!axes-central}

We now describe the ideals whose triangular regions have an axes-central puncture. Let $A$, $B$, $C$, and $M$ be
nonnegative integers with at most one of $A$, $B$, and $C$ being zero. We must consider two cases, depending on parity.

First, suppose $A$, $B$, and $C$ all share the same parity.  We form the triangular region
\[
    T_{A+B+C+M}(x^{B+C+M}, y^{A+C+M}, z^{A+B+M}, x^{\frac{1}{2}(B+C)} y^{\frac{1}{2}(A+C)} z^{\frac{1}{2}(A+B)}).
\]
By construction, the region has a puncture of side length $A$, $B$, and $C$ in the top, bottom-left, and bottom-right
corners, respectively. Further, there is a puncture of side length $M$ that is axes-central. If we let $\alpha$,
$\beta$, and $\gamma$ be the exponents of the mixed term, then we get Figure~\ref{fig:axes-central}(i).

Now suppose $A$ and $B$ differ in parity from $C$. In this case, the axes-central puncture would have to be located a
non-integer distance from the edges of the triangle. To fix this, we shift the puncture up and right one-half unit to
create a valid triangular region. In particular, we form the triangular region
\[
    T_{A+B+C+M}(x^{B+C+M}, y^{A+C+M}, z^{A+B+M}, x^{\frac{1}{2}(B+C+1)} y^{\frac{1}{2}(A+C-1)} z^{\frac{1}{2}(A+B)}).
\]
As in the previous case, we get the desired punctures. Moreover, if we let $\alpha$, $\beta$, and $\gamma$ be the
exponents of the mixed term, then we get Figure~\ref{fig:axes-central}(ii).

\begin{figure}[!ht]
    \begin{minipage}[b]{0.48\linewidth}
        \centering
        \includegraphics[scale=1.0]{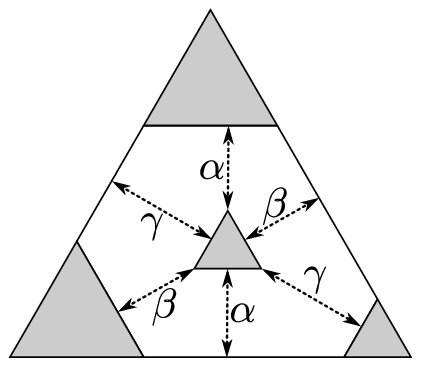}\\
        \emph{(i) The parity of $C$ agrees with $A$ and $B$.}
    \end{minipage}
    \begin{minipage}[b]{0.48\linewidth}
        \centering
        \includegraphics[scale=1.0]{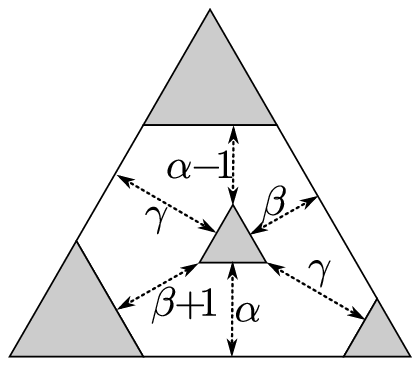}\\
        \emph{(ii) The parity of $C$ differs from $A$ and $B$.}
    \end{minipage}
    \caption{The two prototypical figures with axes-central punctures.}
    \label{fig:axes-central}
\end{figure}

Let $H$ be the punctured hexagon with an axes-central puncture associated to the nonnegative integers $A$, $B$, $C$,
and $M$. The total number of lozenge tilings of $H$ as well as a certain $(-1)$-enumeration of them have been calculated
in four theorems (two for each type and parity condition). We recall the four theorems here, although we forgo the exact
statements of the enumerations; the explicit enumerations can be found in~\cite{CEKZ}.

\begin{theorem}{\cite[Theorems 1, 2, 4, and 5]{CEKZ}} \label{thm:CEKZ-1245}
    Let $A, B, C,$ and $M$ be nonnegative integers, and let $H$ be the associated hexagon with an axes-central puncture.  Then
    the following statements hold:
    \index{0@\textbf{Symbol list}!CEKZ@$\CEKZ_i(A,B,C,M)$}
    \begin{enumerate}
        \item[(1)] The number of lozenge tilings of $H$ is $\CEKZ_1(A,B,C,M)$, if $A, B,$ and $C$ share a common parity.
        \item[(2)] The number of lozenge tilings of $H$ is $\CEKZ_2(A,B,C,M)$, if $A, B,$ and $C$ do not share a common parity.
        \item[(4)] The $(-1)$-enumeration of lozenge tilings of $H$ is
        \begin{enumerate}
            \item[(i)] $0$, if $A, B,$ and $C$ are all odd;
            \item[(ii)] $\CEKZ_4(A,B,C,M)$, if $A, B,$ and $C$ are all even.
        \end{enumerate}
        \item[(5)] The $(-1)$-enumeration of lozenge tilings of $H$ is $\CEKZ_5(A,B,C,M)$, if $A, B,$ and $C$ do not share a common parity.
    \end{enumerate}

    Moreover, the four functions $\CEKZ_i$ are polynomials in $M$ which factor completely into linear terms.  Further, each can be expressed
    as a quotient of products of hyperfactorials and, in each case, the largest hyperfactorial term is $\HF(A+B+C+M)$.
\end{theorem}

Using Proposition~\ref{pro:per-enum} and Theorem~\ref{thm:pm-matrix} together with the fact that the sign used in the
$(-1)$-enumeration in \cite{CEKZ} is equivalent to our perfect matching sign, we find the permanent and the determinant
of the bi-adjacency matrix of $H$.

\begin{corollary} \label{cor:axes-central-Z}
    Let $A, B, C,$ and $M$ be nonnegative integers, and let $H$ be the associated hexagon with an axes-central puncture.  Then
    \[
        \per{Z(H)} =
        \left\{ \begin{array}{ll}
            \CEKZ_1(A,B,C,M) & \mbox{if $A,B,$ and $C$ share a common parity;} \\[0.3em]
            \CEKZ_2(A,B,C,M) & \mbox{otherwise.}
        \end{array} \right.
    \]
    Further, if $M$ is even, then
    \[
        |\det{Z(H)}| =
        \left\{ \begin{array}{ll}
            \CEKZ_1(A,B,C,M) & \mbox{if $A,B,$ and $C$ share a common parity;} \\[0.3em]
            \CEKZ_2(A,B,C,M) & \mbox{otherwise.}
        \end{array} \right.
    \]
    And if $M$ is odd, then
    \[
        |\det{Z(H)}| =
        \left\{ \begin{array}{ll}
            0                  & \mbox{if $A,B,$ and $C$ are all odd;} \\[0.3em]
            |\CEKZ_4(A,B,C,M)| & \mbox{if $A,B,$ and $C$ are all even;} \\[0.3em]
            |\CEKZ_5(A,B,C,M)| & \mbox{otherwise.}
        \end{array} \right.
    \]

    Moreover, the prime divisors of the enumerations are  bounded above by $A+B+C+M - 1$.
\end{corollary}~

\section{Mirror symmetric triangular regions}~\label{sec:mirror}

A \emph{mirror symmetric}%
\index{triangular region!mirror symmetric}
region is a triangular region $T = T_d(I)$ that is invariant under reflection about the
vertical line that goes through the top center vertex of the containing triangular region $\mathcal{T}_d$. Furthermore,
we call a puncture an \emph{axial puncture}%
\index{puncture!axial}
if its top vertex is on the axis of symmetry, i.e., it is itself symmetric.

In this section we consider mirror symmetric regions under some strong restrictions. We collect the assumptions used for
the entirety of the section here.
\begin{assumption} \label{assump:mirror}
    Let $T$ be a triangular region that satisfies the following conditions:
    \begin{enumerate}
        \item $T$ is balanced and mirror symmetric.
        \item With the exception of a pair of punctures in the bottom corners, all punctures of $T$ are axial punctures.
        \item The top-most axial puncture of $T$ is in the top corner of ${\mathcal T}_d$.
        \item No punctures of $T$ touch or overlap.
    \end{enumerate}
\end{assumption}

Recall that such a region $T \subset {\mathcal T}_d$ is balanced if and only if the sum of the side lengths of its
punctures equals $d$ because the punctures of $T$ do not overlap. Note that the assumptions (iii) and (iv) above are
harmless. Indeed, if one of them is not satisfied, then placing the forced lozenges leads to a mirror symmetric
triangular region satisfying (iii) and (iv) (see Corollary~\ref{cor:replace-two-punctures} and
Figure~\ref{fig:two-corners}).

\begin{remark}\label{rem:mirror-reg-tileable}
    Each mirror symmetric region satisfying Assumption \ref{assump:mirror} is tileable by lozenges. This follows, for
    example, from Theorem \ref{thm:tileable}.
\end{remark}

We need some notation to specify a triangular region $T$ as above. We denote the side length of the base punctures
(those in the bottom corners) by $b$. We label the $m$ axial punctures $1, 2, \ldots, m$, starting from the top. The
vertical position and the side length of the $i^{\rm th}$ axial puncture are denoted by $h_i$ and $d_i$, respectively.
As the punctures do not touch or overlap, the numbers $b$ and $(h_1, d_1), \ldots, (h_m, d_m)$ uniquely define $T$. See
Figure~\ref{fig:mirror-symmetric} for an example.

\begin{figure}[!ht]
    \includegraphics[scale=1]{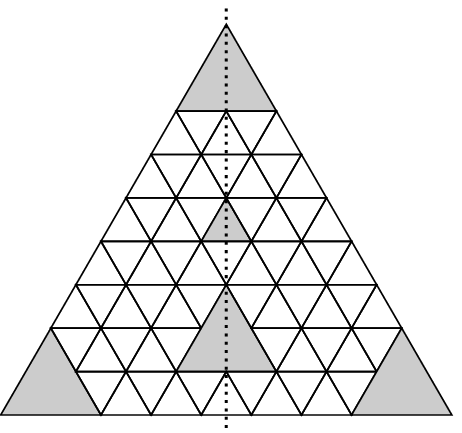}
    \caption{A mirror symmetric region with corner punctures of length $b = 2$ and axial punctures with parameters $(7, 2)$, $(4, 1)$, and $(1, 2)$.}
    \label{fig:mirror-symmetric}
\end{figure}

It is worth recording the conditions the parameters defining $T$ have to satisfy. These parameters also allow us to
describe the associated ideal.

\begin{remark} \label{rem:mirror-prop}
    Let $T \subset {\mathcal T}_d$ be a region with parameters $b$ and $(h_1, d_1), \ldots, (h_m, d_m)$. Then:
    \begin{enumerate}
        \item Since $T$ is balanced, we have $d := 2b + d_1 + \cdots + d_m$.
        \item The $1^{\rm st}$ axial puncture has height $h_1 = d - d_1$ by Assumption~\ref{assump:mirror}(iii).
        \item Assumption~\ref{assump:mirror}(iv) forces the inequalities $h_i - h_{i+1} > d_{i+1}$ for  all $i = 1,\ldots,m$.
        \item For each $i \in \{1,\ldots,m \}$, the integer $h_i$ is even if and only if $d - d_i$ is even because $d-d_i - h_i$ is even by symmetry.
        \item The region $T$ is $T_d (I)$, where
            \[
               I =  (x^{h_1}, y^{d-b}, z^{d-b},
                     x^{h_2} (y z)^{\frac{1}{2}(d - d_2 - h_2)},
                     \ldots,
                     x^{h_m} (y z)^{\frac{1}{2}(d - d_m - h_m)}).
            \]
    \end{enumerate}
\end{remark}~

\subsection{Regularity of the bi-adjacency matrix}~\label{sub:vanishing-det}

By assumption, the mirror symmetric region $T$ is balanced. Thus, its bi-adjacency matrix $Z(T)$ is a square matrix. Our
first main result of this section gives a condition, which implies that $Z(T)$ is not regular.

\begin{theorem} \label{thm:mirror-odd23}
    Let $T = T_d(I)$ be a mirror symmetric region that satisfies Assumption~\ref{assump:mirror}. If the number of its
    axial punctures (including the top-most puncture) with odd side length is either $2$ or $3$ modulo $4$, then
    $\det{Z(T)} = 0$.
\end{theorem}
\begin{proof}
    Instead of directly considering the bi-adjacency matrix $Z(T)$ we study the lattice path matrix $N(T)$. To this end we
    first turn the region $T$ by $120^{\circ}$. Then $T = T_d (J)$, where
    \[
        J = (x^{d-b}, y^{d-b}, z^{h_1},
             (x y)^{\frac{1}{2}(d - d_2 - h_2)} z^{h_2} ,
             \ldots,
             (x y)^{\frac{1}{2}(d - d_m - h_m)} z^{h_m}).
    \]
    Note that the axis of symmetry of $T$ passes through its lower right corner. Thus the lattice path matrix $N(T)$ has
    $d - 2b = d_1 + \cdots + d_m$ rows and columns.

    Consider a family $\Lambda$ of non-intersecting lattice paths on $L(T)$. Let $\lambda$ be the associated
    permutation, that is, for each $i \in \{1,\ldots, d_1 + \cdots + d_m \}$, the path starting at the vertex $A_i$ ends
    at the vertex $E_{\lambda (i)}$.

    Recall that  a permutation of the form
    \[
        \sigma =
        \begin{pmatrix}
            1 & 2 & \ldots & n-1 & n \\
            n & n-1 & \ldots & 2 & 1
        \end{pmatrix}
    \]
    has signature
    \[
        \sgn \sigma = (-1)^{\flfr{n}{2}}.
    \]
    We refer to $\sigma$ as a reflection of length $n$.

    Consider now the family of lattice paths $\hat{\Lambda}$ that is obtained from $\Lambda$ by reflecting it across the
    axis of symmetry of $T$. Let $\hat{\lambda}$ be its associated permutation. (Note it is possible that $\Lambda =
    \hat{\Lambda}$.) This reflection permutes the $d_1 + \cdots + d_m$ starting points of the lattice paths and, for
    each $i = 1,\ldots,m$, the $d_i$ end points on the $i^{\rm th}$ axial puncture. It follows that
    \[
        \hat{\lambda} = \alpha \cdot \lambda \cdot \beta,
    \]
    where $\alpha$ is a reflection of length $d_1 + \cdots + d_m$ and $\beta$ is the product of $m$ reflections having
    lengths $d_1,\ldots,d_m$. Hence, the signature of $\hat{\lambda}$ is
    \[
        \sgn \hat{\lambda} = \sgn \lambda \cdot (-1)^{\flfr{d_1 + \cdots + d_m}{2} + \flfr{d_1}{2} + \cdots + \flfr{d_m}{2}}.
    \]

    Let $q$ be the number of odd side lengths $d_1,\ldots,d_m$. Then the last formula implies
    \[
        \sgn \hat{\lambda} = \sgn \lambda \cdot (-1)^{\flfr{q}{2}}.
    \]
    Observe that $\flfr{q}{2}$ is odd if and only if $q$ is either $2$ or $3$ modulo $4$.

    Finally, suppose that the integer $q$ is either $2$ or $3$ modulo $4$. Then $\sgn(\lambda) = -\sgn(\hat{\lambda})$.
    Therefore, we have found a bijection from the families of non-intersecting lattice paths of $L(T)$ with positive
    sign to the families with negative sign. Hence, Theorem~\ref{thm:nilp-matrix} gives $\det{N(T)} = 0$. Using
    Theorem~\ref{thm:detZN}, we obtain $\det{Z(T)} = 0$.
\end{proof}

In particular, the above result applies to the following ideals.

\begin{example}\label{exa:mirror-symm-singular}
    Consider the ideal
    \[
        I = (x^a, y^c, z^c,  x^{\alpha} y^{\gamma} z^{\gamma}),
    \]
    where $a, c, \alpha, \gamma$ are integers satisfying
    \begin{equation*}
        \begin{split}
            0 < \alpha < a,\; 0 < \gamma < c,\; 2 (c- 2 \gamma) < 2a - \alpha, \; \text{and } \\
            \max \{a, c, \alpha + 2 \gamma \} < \frac{1}{3} (a + \alpha + 2c + 2 \gamma) \in \ZZ.
        \end{split}
    \end{equation*}
    Then Theorem \ref{thm:mirror-odd23} gives $\det Z(T) = 0$, where $T = T_d (I)$ and
    $d = \frac{1}{3} (a + \alpha + 2c + 2 \gamma)$.
\end{example}

Based on an extensive computer search using \emph{Macaulay2}~\cite{M2}, we offer the following conjectured
characterisation of the regularity of the bi-adjacency matrix.

\begin{conjecture} \label{con:zero-mirror}
    Let $T = T_d(I)$ be a be a mirror symmetric region that satisfies Assumption~\ref{assump:mirror}. Then $\det{Z(T)}
    \neq 0$ if and only if the number of axial punctures (including the top-most puncture) with odd side length is at
    most one.
\end{conjecture}

We have additional evidence for this conjecture in the case where the side lengths of all axial punctures except
possibly the top axial puncture are even.

\begin{proposition} \label{pro:mirror-even}
    Let $T = T_d(I)$ be a triangular region as in Assumption~\ref{assump:mirror} with parameters $b$ and
    $(h_1, d_1)$, \ldots, $(h_m, d_m)$.  If $d_2, \ldots, d_m$ are even, then $|\det{Z(T)}| = |\per{Z(T)}| > 0$.
\end{proposition}
\begin{proof}
    The assumptions imply that the floating punctures of $T$ all have an even side length. Since $T$ is tileable, the
    conclusion is an immediate consequence of Proposition~\ref{pro:same-sign}.
\end{proof}~

\subsection{Explicit enumerations}~\label{sub:ciucu}

Ciucu~\cite{Ci-2005} gave explicit formul\ae\ for the (unsigned) enumeration of lozenge tilings of mirror symmetric
regions. These formul\ae\ were found using techniques first described in~\cite{Ci-1997}. We recall a few
definitions following~\cite[Part~B, Section~1]{Ci-2005}.

Let $(a)_k := a (a+1) \cdots (a+k-1)$ be the shifted factorial, also known as the rising factorial. For nonnegative
integers $m$ and $n$, let $B_{m,n}(x)$ and $\overline{B}_{m,n}(x)$ be the monic polynomials defined by
\begin{equation} \label{eqn:bmn}
    \begin{split}
        B_{m,n}(x) =& 2^{-mn-m(m-1)/2} (x+n+1)_m (x+n+2)_m \\
                    & \times \prod_{i=1}^{\clfr{n-1}{2}} (x+1+i)_{n+1 - 2i} \; \times \prod_{i=1}^{\clfr{n}{2}} \left(x+\frac{1}{2}+i\right)_{n+2 - 2i} \\
                    & \times \prod_{i=1}^{n} \frac{(x+i)_m}{(x+i+1/2)_m} \times \prod_{i=1}^{m} (2x+n+i+2)_{n+i-1}
    \end{split}
\end{equation}
and
\begin{equation} \label{eqn:b-mn}
    \begin{split}
        \overline{B}_{m,n}(x) =& 2^{-mn-n(n+1)/2} (x+m+1)_n \\
                               & \times \prod_{i=1}^{\clfr{m}{2}} (x+i)_{m+2 - 2i} \; \times \prod_{i=1}^{\clfr{m-1}{2}} \left(x+\frac{1}{2}+i\right)_{m+1 - 2i} \\
                               & \times \prod_{i=1}^{m} \frac{(x+i)_n}{(x+i+1/2)_n} \times \prod_{i=1}^{n} (2x+m+i+2)_{m+i},
    \end{split}
\end{equation}
respectively.

Moreover, for (possibly empty) sequences $\mathbf{p} = (p_1, \ldots, p_m)$ and $\mathbf{q} = (q_1, \ldots, q_n)$ of
strictly increasing positive integers, define rational numbers $c_{\mathbf{p, q}}$ and $\overline{c}_{\mathbf{p, q}}$ by
\begin{equation} \label{eqn:cpq}
    c_{\mathbf{p, q}} =
        2^{\tbinom{n-m}{2}-m}
        \prod_{i=1}^m \frac{1}{(2p_i)!}
        \prod_{i=1}^n \frac{1}{(2q_i-1)!}
        \frac{\prod_{1 \leq i < j \leq m} (p_j - p_i) \prod_{1 \leq i < j \leq n} (q_j - q_i)}
             {\prod_{i=1}^m \prod_{j=1}^n (p_i + q_j)}
\end{equation}
and
\begin{equation} \label{eqn:c-pq}
    \overline{c}_{\mathbf{p, q}} =
        2^{\tbinom{n-m}{2}-m}
        \prod_{i=1}^m \frac{1}{(2p_i-1)!}
        \prod_{i=1}^n \frac{1}{(2q_i)!}
        \frac{\prod_{1 \leq i < j \leq m} (p_j - p_i) \prod_{1 \leq i < j \leq n} (q_j - q_i)}
             {\prod_{i=1}^m \prod_{j=1}^n (p_i + q_j)},
\end{equation}
respectively. (We note that in \cite{Ci-2005}, $\mathbf{l}$ is used in place of $\mathbf{p}$. We changed notation for
ease of reading.)

Following still~\cite[Part~B, Section~5]{Ci-2005}, for given parameters $\mathbf{p}$ and $\mathbf{q}$ as above, define
polynomials $P_{\mathbf{p, q}}$ and $\overline{P}_{\mathbf{p, q}}$ by
\begin{equation} \label{eqn:Ppq}
    \begin{split}
        P_{\mathbf{p,q}}(x) =& c_{\mathbf{p, q}} B_{m,n}(x+ p_m - m) \\
                             & \times \prod_{i=1}^m \prod_{j=i}^{p_i-1} (x + p_m - j) (x + p_m - m + n + j + 2) \\
                             & \times \prod_{i=1}^n \prod_{j=1}^{q_i-1} (x + p_m - m + n - j + 1) (x + p_m + j + 1)
    \end{split}
\end{equation}
and
\begin{equation} \label{eqn:P-pq}
    \begin{split}
        \overline{P}_{\mathbf{p,q}}(x) =& \overline{c}_{\mathbf{p, q}} \overline{B}_{m,n}(x+ p_m - m) \\
                                        & \times \prod_{i=1}^m \prod_{j=i}^{p_i-1} (x + p_m - j) (x + p_m - m + n + j + 1) \\
                                        & \times \prod_{i=1}^n \prod_{j=1}^{q_i-1} (x + p_m - m + n - j) (x + p_m + j + 1),
    \end{split}
\end{equation}
respectively.  Furthermore, define the following modifications of $\mathbf{p} = (p_1, \ldots, p_m)$:
\begin{equation*}
    \mathbf{p} - 1 =
    \begin{cases}
        (p_1 - 1, \ldots, p_m - 1) & \text{if } p_1 > 1 \\
        (p_2 - 1, \ldots, p_m - 1) & \text{if } p_1 = 1
    \end{cases}.
\end{equation*}
and
\[
    \mathbf{p}^{(m)} = (p_1, \ldots, p_{m-1}).
\]

Let $T = T_d(I)$ be a mirror symmetric region as in Assumption~\ref{assump:mirror} with parameters $b$ and $(h_1,
d_1), \ldots, (h_s, d_s)$, where $s \geq 1$ and $d_2, \ldots, d_{s-1}$ are all even. In order to use the notation
introduced above, define $a := d_1$ and $k := d_2 + d_3 + \cdots + d_s$. Moreover, if $d_s$ is even (i.e., $k$ is even),
then set
\begin{equation*}
    \begin{split}
        \mathbf{p} :=  &  \left (   1, 2,\ldots, \flfr{h_s}{2},  \flfr{d_s + h_s}{2} + 1, \flfr{d_s + h_s}{2} + 2,\ldots, \flfr{h_{s-1}}{2}, \ldots, \right.\\[1ex]
                       & \hspace*{0.2cm} \left. \flfr{d_3 + h_3}{2} +1, \flfr{d_3 + h_3}{2} + 2,\ldots, \flfr{h_2}{2}, \flfr{d_2 + h_2}{2} + 1, \flfr{d_2 + h_2}{2} +2,\ldots, \clfr{h_1}{2} \right )
    \end{split}
\end{equation*}
and $\mathbf{q} = \emptyset$. More precisely, $\mathbf{p}$ is a concatenation of $s$ lists, where the $i^{\rm th}$ list,
counted from the end, consists of consecutive integers $\flfr{d_i + h_i}{2} +1, \flfr{d_i + h_i}{2} + 2,\ldots,
\flfr{h_{i-1}}{2}$ if $2 \leq i \leq s-1$. Thus, if $s=1$, then $\mathbf{p} = (1, 2,\ldots,\clfr{h_1}{2})$. If $s \geq 2$,
then $\mathbf{p}$ has
\[
    m = \flfr{h_s}{2} + \clfr{h_1}{2} - \flfr{d_2 + h_2}{2} + \sum_{i=2}^{s-1} \left ( \flfr{h_{i-1}}{2} - \flfr{d_i + h_i}{2} \right )
\]
entries.

If $d_s$ is odd (i.e., $k$ is odd), then set  $\mathbf{p} := (1, 2, \ldots, \flfr{h_s}{2})$ and
\begin{equation*}
    \begin{split}
        \mathbf{q} := & \left ( \flfr{d_s}{2} + 1,  \flfr{d_s}{2} + 2,\ldots, \flfr{h_{s-1} - h_{s}}{2},  \right. \\[1ex]
                      & \hspace*{0.2cm} \flfr{d_{s-1} + h_{s-1} - h_{s}}{2} +1, \flfr{d_{s-1} + h_{s-1} - h_{s}}{2} + 2, \ldots,  \flfr{h_{s-2} - h_{s}}{2}, \ldots, \\
                      & \hspace*{0.2cm} \left. \flfr{d_{2} + h_{2} - h_{s}}{2}  + 1, \flfr{d_{2} + h_{2} - h_{s}}{2}  +  2,\ldots, \flfr{h_{1} - h_{s}}{2} \right ).
    \end{split}
\end{equation*}
This time the $(i-1)^{\rm st}$ sublist of $\mathbf{q}$, counted from the end, consists of consecutive integers
$\flfr{d_{i} + h_{i} - h_{s}}{2} + 1, \flfr{d_{i} + h_{i} - h_{s}}{2} + 2,\ldots, \flfr{h_{i-1} - h_{s}}{2}$ if
$2 \leq i \leq s$. Thus, if $s=1$, then $\mathbf{q} = \emptyset$. If $s \geq 2$, then $\mathbf{q}$ has
\[
    n = \flfr{h_{s-1} - h_s}{2} - \flfr{d_s}{2} +
    \sum_{i=2}^s \left ( \flfr{h_{i-1} - h_{s}}{2} - \flfr{d_{i} + h_{i} - h_{s}}{2} \right )
\]
entries.

Using this notation, we can now apply a result in \cite{Ci-2005} that enumerates the number of unsigned lozenge tilings
of a mirror-symmetric region in various cases.

\begin{theorem} \label{thm:ciucu-11}
    Let $T = T_d(I)$ be a mirror symmetric satisfying Assumption~\ref{assump:mirror} with parameters
    $b$ and $(h_1, d_1), \ldots, (h_s, d_s)$, where $s \geq 1$ and $d_2, \ldots, d_{s-1}$ are all even.
    Define $a$, $k$, $\mathbf{p}$, and $\mathbf{q}$ be as above.  Then:
    \begin{enumerate}
        \item If $d_1$ is even, $d_s$ is even, and $h_s \geq 2$, then    
            $\per{Z(T)} = 2^m P_{\emptyset, \mathbf{p}}\left(\frac{a+k-2}{2}\right) \overline{P}_{\mathbf{p}-1, \emptyset} \left(\frac{a}{2} \right).$
        \item If $d_1$ is even, $d_s$ is even, and $0 \leq h_s < 2$, then    
            $\per{Z(T)} = 2^m P_{\emptyset, \mathbf{p}}\left(\frac{a+k-2}{2}\right) P_{\mathbf{p}-1, \emptyset} \left(\frac{a}{2} \right).$
        \item If $d_1$ is odd and $d_s$ is even, then      
            $\per{Z(T)} = 2^m P_{\emptyset, \mathbf{p}^{(m)}}\left(\frac{a+k-1}{2}\right) \overline{P}_{\mathbf{p}, \emptyset} \left(\frac{a-1}{2} \right).$
        \item If $d_1$ is even, $d_s$ is odd, and $h_s \geq 2$, then    
            $\per{Z(T)} = 2^{m+n} \overline{P}_{\mathbf{p,q}} \left(\frac{a+k-1}{2}\right) P_{\mathbf{q, p}^{(m)}} \left(\frac{a}{2} \right).$
        \item If $d_1$ and $d_s$ are both odd, then    
            $\per{Z(T)} = 2^{m+n} \overline{P}_{\mathbf{p,q}^{(n)}} \left(\frac{a+k}{2}\right) P_{\mathbf{q, p}} \left(\frac{a-1}{2} \right).$
    \end{enumerate}
    Moreover, in cases {\rm (i)--(iii)}, $|\det{Z(T)}| = \per{Z(T)}$.
\end{theorem}
\begin{proof}
    By Proposition~\ref{pro:per-enum}, we know that $\per{Z(T)}$ enumerates the \emph{unsigned} lozenge tilings of $T$.
    Furthermore, observe that the five conditions in the statement are equivalent to the following conditions on $a$,
    $k$, and $\mathbf{p}$ in the corresponding order:
    \begin{enumerate}
        \item $k$ is even, $a$ is even, and $p_1 = 1$.
        \item $k$ is even, $a$ is even, and $p_1 > 1$.
        \item $k$ is even and $a$ is odd.
        \item $k$ is odd, $a$ is even, and $\mathbf{p} \neq \emptyset$.
        \item $k$ and $a$ are both odd.
    \end{enumerate}
    Now the stated formul\ae\ for $\per Z(T)$ follow from  \cite[Part~B, Theorem~1.1]{Ci-2005}.

    Finally, note that $k$ is even if and and only if all but the top axial punctures have even side length. Hence
    Proposition~\ref{pro:same-sign} gives $|\det{Z}| = \per{Z}$ in the cases (i)--(iii).
\end{proof}

In case (v) of Theorem \ref{thm:ciucu-11}, the determinant is not equal to the permanent. This is also true in case (iv) in general.

\begin{remark}
    In case (v) of Theorem \ref{thm:ciucu-11}, as $d_1$ and $d_s$ are odd, $T_d(I)$ has exactly two axial punctures
    with odd side lengths. Thus, Theorem~\ref{thm:mirror-odd23} gives $\det{Z(T_d(I))} = 0 < \per Z(T)$.

    Moreover, in case (iv), in general, the permanent does not enumerate the \emph{signed} lozenge tilings of
    $T_d(I)$. Consider, for example, the ideal $I = (x^5, y^5, z^5, x^2y^2z^2)$. Then $|\det{Z(T_7(I))}| = 50$ and
    $\per{Z(T_7(I))} = 54$.
\end{remark}

Observe that in Theorem~\ref{thm:ciucu-11}(iv), the conditions that $d_1$ be even and $d_s$ be odd force that $h_s$ is even.
Thus, we left out precisely the case, where $h_s = 0$ although this case is included in \cite{Ci-2005}. It seems (to us)
that in this specific situation the formula given in \cite[Part~B, Theorem~1.1]{Ci-2005} needs an adjustment.

\begin{example}
    Let $T = T_5(I)$, where $I = (x^3, y^4, z^4, y^2 z^2)$; see Figure~\ref{fig:ciucu-12-cn-8}.
    \begin{figure}[!ht]
        \includegraphics[scale=1]{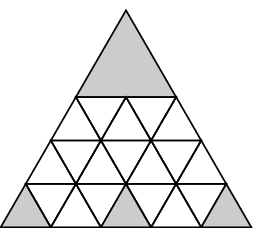}
        \caption{The triangular region $T_5(x^3, y^4, z^4, y^2 z^2)$.}
        \label{fig:ciucu-12-cn-8}
    \end{figure}
    Then $b = 1$, $(h_1, d_1) = (3, 2)$, and $(h_2, d_2) = (0,1)$.  That is, $a = 2$, $b = 1$, $k = 1$, $\mathbf{p} = \emptyset$,
    and $\mathbf{q} = \{1\}$.  If we apply the formula in Theorem~\ref{thm:ciucu-11}(iv) with these parameters, then we get
    \[
        2^1 \cdot \overline{P}_{\emptyset, \{1\}}(1) \cdot P_{\{1\},\emptyset}(1) = 2 \cdot 2 \cdot 3 = 12.
    \]
    However, in this case we can use Proposition~\ref{pro:two-mahonian} with parameters $c = 3$, $a = b = 4$, and
    $\alpha = \beta = 2$ to see that
    \[
    \per Z(T) =     \Mac(1,1,1) \cdot \Mac(1,1,1) \cdot \frac{\HF(2) \HF(2) \HF(3) \HF(5)}{\HF(4) \HF(4) \HF(3) \HF(1)}
        = 2 \cdot 2 \cdot \frac{1 \cdot 1 \cdot 2 \cdot 288}{12 \cdot 12 \cdot 2 \cdot 1} = 8.
    \]
    Similarly, we can compute the determinant of the lattice path matrix of $T$:
    \[
        \det{N(T)} = \det{\left( \begin{array}{cc} 4 & 2 \\ 6 & 1 \end{array} \right)} = -8.
    \]

    We note, however, that if modify the formula in Theorem~\ref{thm:ciucu-11}(iv) to read
    (notice the function $P$ is now the function $\overline{P}$)
    \[
        2^{m+n} \overline{P}_{\mathbf{\emptyset,q}} \left(\frac{a+k-1}{2}\right) \overline{P}_{\mathbf{q, \emptyset}} \left(\frac{a}{2} \right),
    \]
    then we \emph{do} get the correct value:
    \[
        2^1 \cdot \overline{P}_{\emptyset, \{1\}}(1) \cdot \overline{P}_{\{1\},\emptyset}(1) = 2 \cdot 2 \cdot 2 = 8.
    \]
\end{example}

The above modification of the formula in Theorem~\ref{thm:ciucu-11}(iv) gives the correct result in general. It follows
that we get an explicit formula for the permanent, and hence for the determinant of the bi-adjacency matrix of the
associated bipartite graph.

\begin{theorem} \label{thm:ciucu-corrected}
    Let $T = T_d(I)$ be a triangular region as in Assumption~\ref{assump:mirror} with parameters $b$ and $(h_1, d_1),
    \ldots, (h_s, d_s)$, where $s \geq 1$ and $d_2, \ldots, d_{s-1}$ are all even. Define $a$, $k$, $\mathbf{p}$, and
    $\mathbf{q}$ as introduced above Theorem~\ref{thm:ciucu-11}.

    If $d_1$ is even, $d_s$ is odd, and $h_s = 0$, then
    \[
        |\det{Z(T)}| = \per{Z(T)} =
        2^{m+n} \overline{P}_{\mathbf{\emptyset,q}} \left(\frac{a+k-1}{2}\right) \overline{P}_{\mathbf{q, \emptyset}} \left(\frac{a}{2} \right).
    \]
\end{theorem}
\begin{proof}
    We first note that $d_1$ being even implies $a$ is even, $d_s$ being odd implies $k$ is odd, and $h_s = 0$ implies
    $\mathbf{p} = \emptyset$. Since $d_2, \ldots, d_{s-1}$ are all even, all floating punctures of $Z(T)$ have even side
    length. Hence, Proposition~\ref{pro:same-sign} gives $|\det{Z}(T)| = \per{Z(T)}$.

    Proposition~\ref{pro:per-enum} shows that $\per{Z(T)}$ enumerates the \emph{unsigned} lozenge tilings of $T$. Thus,
    by making a single adjustment, the claim follows as in~\cite[Part~B, Section~3, Proof of Theorem~1.1]{Ci-2005}. We
    defer the details to the Appendix (see Proposition~\ref{prop:ciucu-corrected}) as this requires different arguments
    that will not be used again in the body of this work.
\end{proof}

The explicit formul\ae\ above allow us to extract upper bounds on the prime divisors of the signed enumerations of
lozenge tilings. In the next section we will show how this is related to the presence of the weak Lefschetz property in
positive characteristic.

\begin{proposition} \label{pro:ciucu-prime-bounds}
    Let $T = T_d(I)$ be a triangular region as in Assumption~\ref{assump:mirror} with parameters
    $b$ and $(h_1, d_1), \ldots, (h_s, d_s)$, where $s \geq 1$ and $d_2, \ldots, d_{s-1}$ are all even.

    If either (i) $d_s$ is even, or (ii) $d_1$ is even, $d_s$ is odd, and $h_s = 0$, then the prime divisors
    of $\det{Z(T)}$ are less than $d$.
\end{proposition}
\begin{proof}
    Each of the equations~\eqref{eqn:bmn}--\eqref{eqn:P-pq} is given as a product of factors, so the prime divisors are
    bounded by the largest factors. In particular, $B_{m,n}(x)$ and $\overline{B}_{m,n}(x)$, given in \eqref{eqn:bmn}
    and \eqref{eqn:b-mn}, respectively, have largest factors bounded above by $2x + 2n + 2m$. Similarly,
    $c_{\mathbf{p,q}}$ and $\overline{c}_{\mathbf{p,q}}$, given in \eqref{eqn:cpq} and \eqref{eqn:c-pq}, respectively,
    have largest factors bounded by $\max\{p_m - p_1, q_n - q_1\}$. Bringing this together, along with the factors in
    their defining equations, we see that $P_{\mathbf{p,q}}(x)$ and $\overline{P}_{\mathbf{p,q}}(x)$, given in
    \eqref{eqn:Ppq} and \eqref{eqn:P-pq}, respectively, have largest factors bounded above by $\max\{2x + 2p_m + 2n,
    x+p_m+q_n\}$.

    By definition, $q_n \leq p_m = \clfr{h_1}{2}$, and since $h_1 \leq 2b+k-2$, we have that $q_n \leq p_m < b +
    \frac{1}{2}k$. Further, $m$ and $n$ are bounded above by the number of vertebra in $T$ less the number of vertebra
    covered by axial punctures, that is $m$ and $n$ are bounded above by $(b + \flfr{k}{2}) - \flfr{k}{2} = b$.

    In each of the four relevant formul\ae, see Theorems~\ref{thm:ciucu-11}(i)--(iii) and~\ref{thm:ciucu-corrected},
    one of $\mathbf{p}$ and $\mathbf{q}$ is empty. In any case, the largest factor of the first term of all four
    formulae is bounded by $2x + 2\max\{m, n\}$, where $x$ is bounded by $\frac{1}{2}(a+k-1)$, hence it is bounded by
    $a+k-1 + 2b = d -1$. Similarly, in any case, the largest factor of the second term of all four formul\ae\ is bounded
    by $2x + 2\max\{p_m, q_n\}$, where $x$ is bounded by $\frac{1}{2}a$, hence it is bounded by $a + 2b + k - 2 < d-1$.
\end{proof}

The explicit signed enumerations in Theorems~\ref{thm:ciucu-11} and~\ref{thm:ciucu-corrected} were found by a
decomposition of $T$ into two regions which could be enumerated separately using techniques first described
in~\cite{Ci-1997}. Can this approach be used to handle the, conjecturally, one remaining case with non-zero determinant?

\begin{question}
    Let $T = T_d(I)$ be a triangular region as in Assumption~\ref{assump:mirror} with parameters $b$ and
    $(h_1, d_1), \ldots, (h_s, d_s)$. Suppose $d_1$ is even and exactly one of $d_2, \ldots, d_s$ is odd. Is then the
    signed enumeration of lozenge tilings of $T$ not zero, i.e., is $\det{Z(T)} \neq 0$? If so, what is the explicit
    enumeration thereof?
\end{question}

Note that a positive answer to the first of the preceding questions would prove one direction of
Conjecture~\ref{con:zero-mirror}.

\section{The weak Lefschetz property} \label{sec:wlp}

Let $R = K[x_1, \ldots, x_n]$ be the standard graded $n$-variate polynomial ring over an infinite field $K$, and let $A$
be a standard graded quotient of $R$. Suppose $A$ is Artinian, i.e., $A$ is finite dimensional as a vector space over
$K$. Then $A$ is said to have the \emdx{weak Lefschetz property} if there exists a linear form $\ell \in [A]_1$ such
that, for all integers $d$, the multiplication map $\times \ell: [A]_d \rightarrow [A]_{d+1}$ has maximal rank, that is,
the map is injective or surjective. Such a linear form is called a \emdx{Lefschetz element} of $A$.

In what follows, we first derive a few general tools for determining the presence or absence of the weak Lefschetz
property. Using these results we find a more specific criterion for the weak Lefschetz property for monomial ideals in
$K[x,y,z]$. In particular, we relate the bi-adjacency and lattice path matrices to the maps that decide the weak
Lefschetz properties. We show that the prime divisors of the signed enumerations of lozenge tilings of the triangular
region to an ideal govern the presence or absence of the weak Lefschetz property of the ideal. We close this section
with a reinterpretation of a few of the results in the previous sections.

~\subsection{Tools}~\par\label{subsec:tools}

There are some general results that are helpful in order to determine the presence or absence of the weak Lefschetz
property. We recall or derive these tools here.

First, we review the appropriate generalisations of some concepts that were first discussed in Subsection~\ref{sub:socle}.
All $R$-modules in this paper are assumed to be finitely generated and graded. Let $M$ be an Artinian $R$-module. The
\emph{socle} of $M$, denoted $\soc{M}$, is the annihilator of $\mathfrak{m} = (x_1, \ldots, x_n)$, the homogeneous
maximal ideal of $R$, that is, $\soc M = \{y \in M \st y \cdot \mathfrak{m} = 0\}$. The \emph{socle degree} or%
\index{socle!degree}
\emdx{Castelnuovo-Mumford regularity} of $M$ is the maximum degree of a non-zero element in $\soc{M}$. The module $M$ is
said to be \emdx{level} if all socle generators have the same degree, i.e., its socle is concentrated in one degree.

Alternatively, assume that the minimal free resolution of $M$ over $R$ ends with a free module $\bigoplus_{i=1}^{m}
R(-t_i)^{r_i},$ where $0 < t_1 < \cdots < t_m$ and $0 < r_i$ for all $i$. In this case, the socle generators of $M$ have
degrees $t_1 - n, \ldots, t_m - n$. Thus, $M$ is level if and only if $m = 1$.
\smallskip

We now begin with deriving some rather general facts. Once multiplication by a general linear form on an algebra is
surjective, then it remains surjective.

\begin{proposition}{\cite[Proposition~2.1(a)]{MMN-2011}} \label{pro:surj}
    Let $A = R/I$ be an Artinian standard graded $K$-algebra, and let $\ell$ be a general linear form.  If the map
    $\times \ell: [A]_{d} \rightarrow [A]_{d+1}$ is surjective for some $d \geq 0$, then
    $\times \ell: [A]_{d+1} \rightarrow [A]_{d+2}$ is surjective.
\end{proposition}

This can be extended to modules generated in degrees that are sufficiently small.

\begin{lemma} \label{lem:mod-surj}
    Let $M$ be a graded $R$-module such that the degrees of its minimal generators are at most $d$. Let $\ell \in R$ be
    a general linear form. If the map $\times\ell: [M]_{d-1} \rightarrow [M]_{d}$ is surjective, then the map
    $\times\ell: [M]_{j-1} \rightarrow [M]_{j}$ is surjective for all $j \geq d$.
\end{lemma}
\begin{proof}
    Consider the exact sequence $[M]_{d-1} \stackrel{\times\ell}{\longrightarrow} [M]_{d} \rightarrow [M/\ell M]_{d} \rightarrow 0.$
    Notice the first map is surjective if and only if $[M/\ell M]_{d} = 0.$ Thus, the assumption gives $[M/\ell M]_{d} = 0$.
    Hence $[M/\ell M]_{j+1}$ is zero for all $j \geq d$ because $M$ does not have minimal generators having a degree
    greater than $d$, by assumption.
\end{proof}

As a consequence, we get a generalisation of \cite[Proposition~2.1(b)]{MMN-2011}, which considers the case of level algebras.

\begin{corollary}\label{cor:inj}
    Let $M$ be an Artinian graded $R$-module such that the degrees of its non-trivial socle elements are at least
    $\geq d-1$. Let $\ell \in R$ be a general linear form. If the map $\times \ell: [A]_{d-1} \rightarrow [A]_{d}$ is
    injective, then the map $\times \ell: [A]_{j-1} \rightarrow [A]_{j}$ is injective for all $j \leq d$.
\end{corollary}
\begin{proof}
    The $K$-dual of $M$ is $M^{\vee} = \Hom_K (M, K)$. Then $\times \ell: [M]_{j-1} \rightarrow [M]_{j}$ is injective if
    and only if the map $\times \ell: [M^{\vee}]_{-j} \rightarrow [M^{\vee}]_{-j+1}$ is surjective. The assumption on
    the socle of $M$ means that the degrees of the minimal generators of $M^{\vee}$ are at most $-d+1$. Thus, we
    conclude by Lemma~\ref{lem:mod-surj}.
\end{proof}

The above observations imply that to decide the presence of the weak Lefschetz property we need only check near a
``peak'' of the Hilbert function.

\begin{proposition} \label{pro:wlp}
    Let $A \neq 0$ be an Artinian standard graded $K$-algebra. Let $\ell$ be a general linear form. Then:
    \begin{enumerate}
        \item Let $d$ be the smallest integer such that $h_{A}(d-1) > h_{A}(d)$.
        If $A$ has a non-zero socle element of degree less than $d-1$, then $A$ does not have  the weak Lefschetz property.

        \item Let $d$ be the largest integer such that $h_{A}(d-2) < h_{A}(d-1)$. If $A$ has the weak Lefschetz property, then
        \begin{enumerate}
            \item $\times \ell: [A]_{d-2} \rightarrow [A]_{d-1}$ is injective,
            \item $\times \ell: [A]_{d-1} \rightarrow [A]_{d}$ is surjective, and
            \item $A$ has no socle generators of degree less than $d-1$.
        \end{enumerate}

        \item Let $d \geq 0$ be an integer such that $A$ has the following three properties:
        \begin{enumerate}
            \item $\times \ell: [A]_{d-2} \rightarrow [A]_{d-1}$ is injective,
            \item $\times \ell: [A]_{d-1} \rightarrow [A]_{d}$ is surjective, and
            \item $A$ has no socle generators of degree less than $d-2$.
        \end{enumerate}
        Then $A$ has the weak Lefschetz property.
         \end{enumerate}
\end{proposition}
\begin{proof}
    Suppose in case (i) $A$ has a socle element $y \neq 0$ of degree $e < d-1$. Then $\ell y = 0$, and so the map
    $\times \ell: [A]_{e} \rightarrow [A]_{e+1}$ is not injective. Moreover, since $e < d-1$ we have $h_{R/I}(e) \le
    h_{R/I}(e+1)$. Hence, the map $\times \ell: [A]_{e} \rightarrow [A]_{e+1}$ does not have maximal rank. This proves
    claim (i).

    For showing (ii), suppose $A$ has the weak Lefschetz property. Then, by its definition, $A$ satisfies (ii)(a) and
    (ii)(b) because $h_A (d-1) \geq h_A (d)$. Assume (ii)(c) is not true, that is, $A$ has a socle element $y \neq 0$ of
    degree $e < d-1$. Then the map $\times \ell: [A]_{e} \rightarrow [A]_{e+1}$ is not injective. Since $A$ has the weak
    Lefschetz property, this implies $h_A (e) > h_A (e+1)$. Hence the assumption on $d$ gives $e \leq d-3$. However, this
    means that the Hilbert function of $A$ is not unimodal. This is impossible if $A$ has the weak Lefschetz property
    (see \cite{HMNW}).

    Finally, we prove (iii). Corollary~\ref{cor:inj} and Assumptions (iii)(a), and (iii)(c) imply that the map
    $\times \ell: [A]_{i-2} \rightarrow [A]_{i-1}$ is injective if $i \leq d$. Furthermore, using (iii)(b) and
    Proposition~\ref{pro:surj}, we see that $\times \ell: [A]_{i-1} \rightarrow [A]_{i}$ is surjective if $i \geq d$.
    Thus, $A$ has the weak Lefschetz property.
\end{proof}

The same arguments also give the following result.

\begin{corollary}\label{cor:twin-peaks-wlp}
    Let $A$ be an Artinian standard graded $K$-algebra, and let $\ell$ be a general linear form. Suppose there is an
    integer $d$ such that $0 \neq h_{A}(d-1) = h_{A}(d)$ and $A$ has no socle elements of degree less than $d-1$. Then
    $A$ has the weak Lefschetz property if and only if $\times \ell: [A]_{d-1} \rightarrow [A]_{d}$ is bijective.
\end{corollary}

Sometimes we will rephrase the above assumption on the Hilbert function of $A$ by saying that it has ``twin peaks.''%
\index{Hilbert function!twin peaks}

The following easy, but useful observation is essentially the content of \cite[Proposition~2.2]{MMN-2011}.

\begin{proposition} \label{pro:mono}
    Let $A = R/I$ be an Artinian  $K$-algebra, where $I$ is generated by monomials and $K$ is an infinite field.  Let $d$ and $e>0$ be integers. Then the following conditions are equivalent:
    \begin{enumerate}
        \item The multiplication map $\times L^e: [A]_{d-e} \to [A]_d$ has maximal rank, where $L \in R$ is a general linear form.
        \item The multiplication map $\times (x_1 + \cdots + x_n)^e: [A]_{d-e} \to [A]_d$ has maximal rank.
    \end{enumerate}
\end{proposition}
\begin{proof}
    For the convenience of the reader we recall the argument. Let $L = a_1 x_1 + á á á + a_r x_r \in R$ be a general
    linear form. Thus, we may assume that each coefficient $a_i$ is not zero. Rescaling the variables $x_i$ such that
    $L$ becomes $x_1 + \cdots + x_n$ provides an automorphism of $R$ that maps $I$ onto $I$.
\end{proof}

Hence, for monomial algebras, it is enough to decide whether the sum of the variables is a Lefschetz element. As a
consequence, we show that, for a monomial algebra, the presence of the weak Lefschetz property in characteristic zero is
equivalent to the presence of the weak Lefschetz property in some (actually, almost every) positive characteristic. Here
we use that the minimal generators of a monomial ideal are not affected by the characteristic of the ground field $K$.

Recall that a \emdx{maximal minor} of a matrix $B$ is the determinant of a maximal square sub-matrix of $B$. Let us also
mention again that throughout this section we assume that $K$ is an infinite field.

\begin{corollary} \label{lem:wlp-0-p}
    Let $A$ be an  Artinian monomial $K$-algebra.  Then the following conditions are equivalent:
    \begin{enumerate}
        \item $A$ has the weak Lefschetz property in characteristic zero.
        \item $A$ has the weak Lefschetz property in some positive characteristic.
        \item $A$ has the weak Lefschetz property in every sufficiently large positive characteristic.
    \end{enumerate}
\end{corollary}
\begin{proof}
    Let $\ell = x_0 + \cdots + x_n$. By Proposition~\ref{pro:mono}, $A$ has the weak Lefschetz property if, for each
    integer $d$, the map $\times \ell: [A]_{d-1} \to [A]_d$ has maximal rank. As $A$ is Artinian, there are only
    finitely many non-zero maps to be checked. Fixing monomial bases for all non-trivial components $[A]_j$, the
    mentioned multiplication maps are described by zero-one matrices.

    Suppose $A$ has the weak Lefschetz property in some characteristic $q \geq 0$. Then for each of the finitely many
    matrices above, there exists a maximal minor that is non-zero in $K$, hence non-zero as an integer. The finitely
    many non-zero maximal minors, considered as integers, have finitely many prime divisors. Hence, there are only
    finitely many prime numbers, which divide one of these minors. If the characteristic of $K$ does not belong to this
    set of prime numbers, then $A$ has the weak Lefschetz property.
\end{proof}

The following result is a generalisation of~\cite[Proposition~3.7]{LZ}.

\begin{lemma} \label{lem:wlp-p}
    Let $A$ be an Artinian monomial $K$-algebra. Suppose that $a$ is the least positive integer such that $x_i^a \in I$
    whenever $1 \leq i \leq n$, and suppose that the Hilbert function of $A$ weakly increases to degree $s$. Then, for
    any positive prime $p$ such that $a \leq p^m \leq s$ for some positive integer $m$, $A$ fails to have the weak
    Lefschetz property in characteristic $p$.
\end{lemma}
\begin{proof}
    Suppose the characteristic of $K$ is $p$, and let $\ell = x_1 + \cdots + x_n$. Then, by the Frobenius endomorphism,
    $\ell \cdot \ell^{p^m-1} = \ell^{p^m} = x_1^{p^m} + \cdots + x_n^{p^m}$. Moreover, $\ell^{p^m} = 0$ in $A$ as $a
    \leq p^m$. Since $\ell \neq 0$ in $A$, the map $\times \ell^{p^m-1}: [A]_{1} \rightarrow [A]_{p^m}$ is not
    injective. Thus, $A$ does not have the weak Lefschetz property by Proposition~\ref{pro:mono}.
\end{proof}

We conclude this subsection by noting that any Artinian ideal in two variables has the weak Lefschetz property. This was
first proven for characteristic zero in~\cite[Proposition~4.4]{HMNW} and then for arbitrary characteristic
in~\cite[Corollary~7]{MZ}, though it was not specifically stated therein (see~\cite[Remark~2.6]{LZ}). We provide a
brief, direct proof of this fact to illustrate the weak Lefschetz property. Unfortunately, the argument cannot be
extended to the case of three variables, not even for monomial ideals.

\begin{proposition} \label{pro:2-wlp}
    Let $R = K[x,y]$, where $K$ is an infinite field of \emph{arbitrary} characteristic. Then every Artinian graded
    algebra $R/I$ has the weak Lefschetz property.
\end{proposition}
\begin{proof}
    Let $\ell \in R$ be a general linear form, and put $s = \min\{j \in \ZZ \st [I]_j \neq 0 \}$. As $[R]_i = [R/I]_i$
    for $i < s$ and multiplication by $\ell$ on $R$ is injective, we see that $[R/I]_{i-1} \rightarrow [R/I]_{i}$ is
    injective if $i < s$. Moreover, since $R/(I,\ell) \cong K[x]/(x^s)$ and $[K[x]/(x^s)]_i = 0$ for $i \geq s$, the map
    $[R/I]_{i-1} \rightarrow [R/I]_{i}$ has a trivial cokernel if $i \geq s$, that is, the map is surjective if $i \geq
    s$. Hence $R/I$ has the weak Lefschetz property.
\end{proof}~

\subsection{The weak Lefschetz property and perfect matchings}\label{sub:wlp-pm}~\par

If $T \subset {\mathcal T}_d$ is a triangular region, then its associated bi-adjacency matrix $Z(T)$ (see Section \ref{sub:pm})
admits an alternative description using multiplication by $\ell = x+y+z$.

\begin{proposition}\label{prop:interp-Z}
    Let $I$ be a monomial ideal in $R = K[x,y,z]$, and let $\ell = x+y+z$. Fix an integer $d$ and consider the
    multiplication map $\times(x+y+z): [R/I]_{d-2} \rightarrow [R/I]_{d-1}$. Let $M(d)$ be the matrix to this linear map
    with respect to the monomial bases of $[R/I]_{d-2}$ and $[R/I]_{d-1}$ in reverse-lexicographic order. Then the
    transpose of $M(d)$ is the bi-adjacency matrix $Z(T_d (I))$.
\end{proposition}
\begin{proof}
    Set $s = h_{R/I}(d-2)$ and $t = h_{R/I}(d-1)$, and let $\{m_1, \ldots, m_s\}$ and $\{n_1, \ldots, n_t\}$ be the
    distinct monomials in $[R]_{d-2} \setminus I$ and $[R]_{d-1} \setminus I$, respectively, listed in
    reverse-lexicographic order. Then the matrix $M(d)$ is a $t \times s$ matrix. Its column $j$ is the coordinate
    vector of $\ell m_j = x m_i + y m_i + z m_i$ modulo $I$ with respect to the chosen basis of $[R/I]_{d-1}$. In
    particular, the entry in column $j$ and row $i$ is $1$ if and only if $n_i$ is a multiple of $m_j$.

    Recall from Subsection~\ref{sub:pm} that the rows and of $Z(T_d(I))$ are indexed by the downward- and upward-pointing
    unit triangles, respectively. These triangles are labeled by the monomials in $[R]_{d-2} \setminus I$ and
    $[R]_{d-1} \setminus I$, respectively. Since the label of an upward-pointing triangle is a multiple of the label of a
    downward-pointing triangle if and only if the triangles are adjacent, it follows that $Z(T_d(I)) = M(d)^{\rm T}$.
\end{proof}

For ease of reference, we record the following consequence.

\begin{corollary}\label{cor:max-Z}
    Let $I$ be a monomial ideal in $R = K[x,y,z]$. Then the multiplication map $\times(x+y+z): [R/I]_{d-2} \rightarrow [R/I]_{d-1}$
    has maximal rank if and only if the matrix $Z(T_d(I))$ has maximal rank.
\end{corollary}

Combined with Proposition \ref{pro:mono}, we get a criterion for the presence of the weak Lefschetz property.

\begin{corollary}\label{cor:wlp-biadj}
    Let $I$ be an Artinian monomial ideal in $R = K[x,y,z]$. Then $R/I$ has the weak Lefschetz property if and only if,
    for each positive integer $d$, the matrix $Z(T_d(I))$ has maximal rank.
\end{corollary}

Assuming large enough socle degrees, it is enough to consider at most two explicit matrices to check for the weak Lefschetz property.

\begin{corollary}\label{cor:wlp-Z}
    Let $I$ be an Artinian monomial ideal in $R = K[x,y,z]$, and suppose the degrees of the  socle generators of  $R/I$ are at least $d-2$. Then:
    \begin{enumerate}
        \item If $0 \neq h_{R/I}(d-1) = h_{R/I}(d)$, then $R/I$ has the weak Lefschetz property if and only if
            $\det{Z(T_d(I))}$ is not zero in $K$.
        \item If $h_{R/I}(d-2) < h_{R/I}(d-1)$ and $h_{R/I}(d-1) > h_{R/I}(d)$, then $R/I$ has the weak Lefschetz
            property if and only if $Z(T_d(I))$ and $Z(T_{d+1}(I))$ both have a maximal minor that is not zero in $K$.
    \end{enumerate}
\end{corollary}
\begin{proof}
    By Proposition~\ref{pro:mono}, it is enough to check whether $\ell = x+y+z$ is a Lefschetz element of $R/I$. Hence,
    the result follows by combining Corollary~\ref{cor:max-Z} and Proposition~\ref{pro:wlp} and
    Corollary~\ref{cor:twin-peaks-wlp}, respectively.
\end{proof}

In the case, where the region $T_d(I)$ is balanced, we interpreted the determinant of $Z(T_d(I))$ as the signed
enumeration of perfect matchings on the bipartite graph $G(T)$ (see Subsection~\ref{sub:pm}). In general, we can similarly
interpret the maximal minors of $Z(T_d(I))$ by removing unit triangles from $T_d(I)$, since the rows and columns of
$Z(T_d(I))$ are indexed by the triangles of $T_d(I)$. More precisely, let $T = T_d(I)$ be a $\dntri$-heavy triangular
region with $k$ more downward-pointing triangles than upward-pointing triangles. Abusing notation slightly, we define a
\emph{maximal minor}%
\index{triangular region!maximal minor}
of $T$ to be a balanced subregion $U$ of $T$ that is obtained by removing $k$ downward-pointing triangles
from $T$. Similarly, if $T$ is $\uptri$-heavy, then we remove only upward-pointing triangles to get a maximal minor.

Clearly, if $U$ is a maximal minor of $T$, then $\det{Z(U)}$ is indeed a maximal minor of $Z(T)$. Thus, $Z(T)$ has
maximal rank if and only if there is a maximal minor $U$ of $T$ such that $Z(U)$ has maximal rank.

\begin{example} \label{exa:minors-Z}
    Let $I = (x^4, y^4, z^4, x^2 z^2)$. Then the Hilbert function of $R/I$, evaluated between degrees $0$ and $7$, is
    $(1,3,6,10,11,9,6,2)$, and $R/I$ is level with socle degree $7$. Hence, by Corollary~\ref{cor:wlp-Z}, $R/I$ has the
    weak Lefschetz property if and only if $Z(T_5 (I))$ and $Z(T_6 (I))$ both have a maximal minor of maximal rank.

    \begin{figure}[!ht]
        \begin{minipage}[b]{0.54\linewidth}
            \centering
            \includegraphics[scale=1]{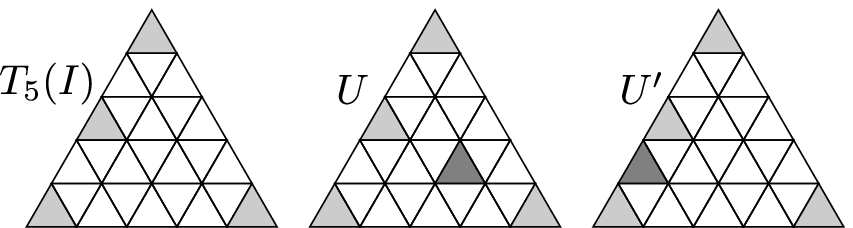}\\
            \emph{(i) $\det{Z(U)} = 0$ and $|\det{Z(U')}| = 4$}
        \end{minipage}
        \begin{minipage}[b]{0.44\linewidth}
            \centering
            \includegraphics[scale=1]{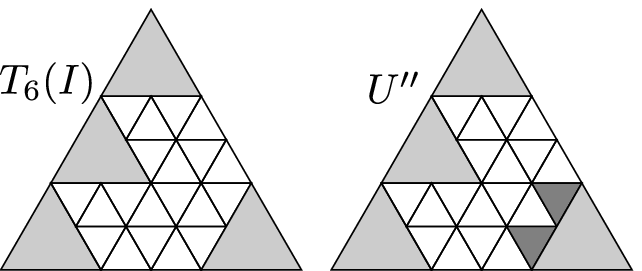}\\
            \emph{(ii) $|\det{Z(U'')}| = 1$}
        \end{minipage}
        \caption{Examples of maximal minors of $T_d(I)$, where $I = (x^4, y^4, z^4, x^2 z^2)$.}
        \label{fig:maximal-minors}
    \end{figure}

    Since $h_{R/I}(3) = 10 < h_{R/I}(4) = 11$, we need to remove $1$ upward-pointing triangle from $T_5(I)$ to get a
    maximal minor of $T_5(I)$; see Figure~\ref{fig:maximal-minors}(i) for a pair of examples. There are $\binom{11}{10} = 11$
    maximal minors, and these have signed enumerations with magnitudes $0, 4,$ and $8$. Thus multiplication from
    degree $3$ to degree $4$ fails injectivity exactly if the characteristic of $K$ is $2$.

    Furthermore, since $h_{R/I}(4) = 11 > h_{R/I}(5) = 9$, we need to remove $2$ downward-pointing triangles from
    $T_6(I)$ to get a maximal minor of $T_6(I)$; see Figure~\ref{fig:maximal-minors}(ii) for an example. There are
    $\binom{11}{9} = 55$ maximal minors, and these have signed enumerations with magnitudes $0, 1,2,3,5,$ and $8$. Thus
    multiplication from degree $4$ to degree $5$ is always surjective (choose the maximal minor whose signed enumeration
    is $1$).

    Hence, we conclude that $R/I$ has the weak Lefschetz property if and only if the characteristic of the base field is
    \emph{not} $2$.
\end{example}

\subsection{The weak Lefschetz property and lattice paths}\label{sub:wlp-nilp}~\par

The key observation in this section is that the lattice path matrix $N(T_d(I))$ (see Section \ref{sub:nilp}) can be used
to study the cokernel of multiplication by $\ell = x+y+z$ on $R/I$.

Let $I$ be a monomial ideal of $R = K[x,y,z]$. Then the cokernel of the multiplication map
$\times(x+y+z): [R/I]_{d-2} \rightarrow [R/I]_{d-1}$ is $[R/(I, x+y+z)]_{d-1}$. This is isomorphic to $[S/J]_{d-1}$,
where $S = K[x,y]$ and $J$ is the ideal generated by the generators of $I$ with $x+y$ substituted for $z$.

\begin{proposition} \label{prop:interp-N}
    Let $I$ be a monomial ideal of $R = K[x,y,z]$, and let $J$ be the ideal of $S = K[x,y]$ generated by the generators
    of $I$ with $x+y$ substituted for $z$. Fix an integer $d$ and set $N = N(T_d(I))$. Then $\dim_K{[S/J]_{d-1}} = \dim_K {\ker{N^{\rm T}}}$.
\end{proposition}
\begin{proof}
    First, we describe a matrix whose rank equals $\dim_K [J]_{d-1}$. Define an integer $a$ as the least power of $x$ in
    $I$ that is less than $d$, and set $a := d$ if no such power exist. Similarly, define an integer $b \leq d$ using
    powers of $y$ in $I$. Let $G_1$ and $G_3$ be the sets of monomials in $x^a [S]_{d-1-a}$ and $y^b [S]_{d-1-b}$,
    respectively. Furthermore, let $G_2$ be the set consisting of the polynomials $x^p y^{d-1-p-e} (x+y)^e \in
    [J]_{d-1}$ such that $x^i y^j z^e$ is a minimal generator of $I$, where $e > 0,\; i \leq p$, and $j \leq d-1-p-e$.
    Replacing $x+y$ by $z$, each element of $G_2$ corresponds to a monomial $x^p y^{d-1-p-e} z^e \in [I]_{d-1}$. Order
    the elements of $G_2$ by using the reverse-lexicographic order of the corresponding monomials in $[I]_{d-1}$, from
    smallest to largest. Similarly, order the monomials in $G_1$ and $G_3$ reverse-lexicographically, from smallest to
    largest. Note that $G_1 \cup G_2 \cup G_3$ is a generating set for the vector space $[J]_{d-1}$. The coordinate
    vector of a polynomial in $[S]_{d-1}$ with respect to the monomial basis of $[S]_{d-1}$ has as entries the
    coefficients of the monomials in $[S]_{d-1}$. Order this basis again reverse-lexicographically from smallest to
    largest. Now let $M$ be the matrix whose column vectors are the coordinate vectors of the polynomials in $G_1, G_2$,
    and $G_3$, listed in this order. Then $\dim_K [J]_{d-1} = \rank M$ because $G$ generates $[J]_{d-1}$.

    Second, consider the lattice path matrix $N = N(T_d(I))$. Its rows and columns are indexed by the starting and
    end points of lattice paths, respectively. Fix a starting point $A_i$ and an end point $E_j$. The monomial label of
    $A_i$ is of the form $x^s y^{d-1-s}$, where $x^s y^{d-1-s} \notin I$. Thus, the orthogonalised coordinates of $A_i
    \in \ZZ^2$ are $(s, s)$. The monomial label of the end point $E_j$ is of the form $x^p y^{d-1-p-e} z^e$, where $x^p
    y^{d-1-p-e} z^e$ is a multiple of a minimal generator of $I$ of the form $x^i y^j z^e$ with the same exponent of
    $z$. The orthogonalised coordinates of $E_j$ are $(p+e, p)$. Hence there are
    \[
        \binom{p+e-s + s - p}{s-p} = \binom{e}{s-p}
    \]
    lattice paths in $\ZZ^2$ from $A_i$ to $E_j$. By definition, this is the $(i, j)$-entry of  the lattice path matrix $N$.

    The monomial label of the end point $E_j$ corresponds to the polynomial
    \[
        x^p y^{d-1-p-e} (x+y)^e = \sum_{k=0}^e \binom{e}{k} x^{p+k} y^{d-1-p-k}
    \]
    in $G_2$. Thus, its coefficient of the monomial label $x^s y^{d-1-s}$ is $N_{(i, j)}$. It follows that the matrix
    $M$ has the form
    \begin{equation*}
        M =
        \begin{pmatrix}
            {\mathcal I}_{d - b} & * & 0 \\
            0 & N & 0 \\
            0 & * & {\mathcal I}_{d - a}
        \end{pmatrix},
    \end{equation*}
    where we used ${\mathcal I}_{k}$ to denote the $k \times k$ identity matrix.

    Notice that the matrices $M$ and $N$ have $d = \dim_K [S]_{d-1}$ and $a+b - d$ rows, respectively. We conclude that
    \begin{equation*}
        \begin{split}
            \dim_K [S/J]_{d-1} &= d - \dim_K [J]_{d-1}\\
                               &= d - \rank M \\
                               &= a+b-d - \rank N \\
                               &= \dim_K \ker N^{\rm T},
        \end{split}
    \end{equation*}
    as claimed.
\end{proof}

The last result provides another way for checking whether the multiplication by $x+y+z$ has maximal rank.

\begin{corollary} \label{cor:max-N}
    Let $I$ be a monomial ideal in $R = K[x,y,z]$. Then the multiplication map
    $\varphi_d = \times(x+y+z): [R/I]_{d-2} \rightarrow [R/I]_{d-1}$ has maximal rank if and only if $N = N(T_d(I))$
has maximal rank.
\end{corollary}
\begin{proof}
    Consider the exact sequence
    \[
        [R/I]_{d-2} \stackrel{\varphi_d}{\longrightarrow}
        [R/I]_{d-1} \longrightarrow [S/J]_{d-1} \longrightarrow 0.
    \]
    It gives that $\varphi_d$ has maximal rank if and only if $\dim_K [S/J]_{d-1} = \max\{0, \dim_K{[R/I]_{d-1}} - \dim_K{[R/I]_{d-2}}\}$.
    By Proposition~\ref{prop:interp-N}, this is equivalent to
    \[
        \dim_K {\ker{N^{\rm T}}} = \max\{0, \dim_K{[R/I]_{d-1}} - \dim_K{[R/I]_{d-2}}\}.
    \]

    Recall that, by construction, the vertices of the lattice $L(T_d(I))$ are on edges of the triangles that are
    parallel to the upper-left edge of ${\mathcal T}_d$, where this edge belongs to just an upward-pointing triangle
    ($A$-vertices), just a downward-pointing triangle ($E$-vertices), or an upward- and a downward-pointing unit
    triangle (all other vertices). Suppose there are $m$ $A$-vertices, $n$ $E$-vertices, and $t$ other vertices. Then
    there are $m + t$ upward-pointing triangles and $n+t$ downward-pointing triangles, that is, $\dim_K{[R/I]_{d-1}} =
    m+t$ and $\dim_K{[R/I]_{d-2}} = n+t$. Hence
    \[
        \dim_K{[R/I]_{d-1}} - \dim_K{[R/I]_{d-2}} = (m+t) - (n+t) = m-n.
    \]
    Since the rows and columns of $N$ are indexed by $A$- and $E$-vertices, respectively, $N$ is an $m \times n$ matrix.
    Hence, $N$ has maximal rank if and only if
    \[
        \dim_K {\ker{N^{\rm T}}} = \max\{0, m-n\} = \max\{0, \dim_K{[R/I]_{d-1}} - \dim_K{[R/I]_{d-2}}\}.
    \]
\end{proof}

The last argument shows in particular that, for any region $T \subset {\mathcal T}_d$, the bi-adjacency matrix $Z(T)$ is
a square matrix if and only if the lattice path matrix $N(T)$ is a square matrix. Hence, using Corollary~\ref{cor:max-N}
instead of Corollary~\ref{cor:max-Z}, we obtain a result that is analogous to Corollary~\ref{cor:wlp-Z}.

\begin{corollary}\label{cor:wlp-N}
    Let $I$ be an Artinian monomial ideal in $R = K[x,y,z]$, and suppose the degrees of the socle generators of $R/I$
    are at least $d-2$. Then:
    \begin{enumerate}
        \item If $0 \neq h_{R/I}(d-1) = h_{R/I}(d)$, then $R/I$ has the weak Lefschetz property if and only if
            $\det{N(T_d(I))}$ is not zero in $K$.
        \item If $h_{R/I}(d-2) < h_{R/I}(d-1)$ and $h_{R/I}(d-1) > h_{R/I}(d)$, then $R/I$ has the weak Lefschetz
            property if and only if $N(T_d(I))$ and $N(T_{d+1}(I))$ both have a maximal minor that is not zero in $K$.
    \end{enumerate}
\end{corollary}

In the case where $T = T_d(I)$ is balanced we interpreted the determinant of $N(T)$ as the signed enumeration of
families of non-intersecting lattice paths in the lattice $L(T)$ (see Subsection~\ref{sub:nilp}). In general, we can
similarly interpret the maximal minors of $N(T)$ by removing $A$-vertices or $E$-vertices from $L(T)$, since the rows
and columns of $N(T)$ are indexed by these vertices. Note that removing the $A$- and $E$-vertices is the same as
removing the associated unit triangles in $T$. For example, $U'$ in Figure~\ref{fig:maximal-minors}(i) corresponds to
removing the starting point $A_1$ from $U$. It follows that the maximal minors of $N(T)$ are exactly the determinants of
maximal minors of $T$ that are obtained from $T$ by removing only unit triangles corresponding to $A$- and $E$-vertices.
We call such a maximal minor a \emph{restricted maximal minor} of $T$.%
\index{triangular region!maximal minor!restricted}

Clearly, $N(T)$ has maximal rank if and only if there is a restricted maximal minor $U$ of $T$ such that $N(U)$ has
maximal rank. As a consequence, for a $\uptri$-heavy region $T$, it is enough to check the restricted maximal minors in
order to determine whether $Z(T)$ has maximal rank.

\begin{proposition} \label{pro:restricted-only}
    Let $T = T_d(I)$ be an $\uptri$-heavy triangular region. Then $Z(T)$ has maximal rank if and only if there is a
    restricted maximal minor $U$ of $T$ such that $Z(U)$ has maximal rank.
\end{proposition}
\begin{proof}
    By Corollaries~\ref{cor:max-Z} and~\ref{cor:max-N}, we have that $Z(T)$ has maximal rank if and only if $N(T)$ has
    maximal rank. Since each restricted maximal minor $U$ of $T$ is obtained by removing upward-pointing triangles, it
    is the triangular region of some monomial ideal. Thus, Theorem~\ref{thm:detZN} gives $|\det{Z(U)}| = |\det{N(U)}|$.
\end{proof}

\begin{remark} \label{rem:restricted}
    The preceding proposition allows us to reduce the number of minors of $Z(T)$ that need to be considered.
    In Example~\ref{exa:minors-Z}(i), there are $11$ maximal minors of $T_5(I)$, but only $2$ \emph{restricted}
    maximal minors.

    In the special case of triangular regions as in Proposition~\ref{pro:ci-enum}, Proposition~\ref{pro:restricted-only}
    was observed by Li and Zanello in \cite[Theorem~3.2]{LZ}.
\end{remark}

We continue to consider Example~\ref{exa:minors-Z}, using lattice path matrices now.

\begin{example} \label{exa:minors-N}
    Recall the ideal $I = (x^4, y^4, z^4, x^2 z^2)$ from Example~\ref{exa:minors-Z}. By Corollary~\ref{cor:wlp-N}, $R/I$
    has the weak Lefschetz property if and only if $N(T_5(I))$ and $N(T_6(I))$ have maximal rank. Since $N(T_5(I))$ is a
    $2 \times 1$ matrix, we need to remove $1$ $A$-vertex to get a maximal minor (see $U'$ in
    Figure~\ref{fig:maximal-minors}(i) for one of the two choices). Both choices have signed enumeration $4$. Since
    $N(T_6(I))$ is a $0 \times 2$ matrix we need to remove $2$ $E$-vertices to get a restricted maximal minor. The
    region $U''$ in Figure~\ref{fig:maximal-minors}(ii) is the only choice, and the signed enumeration is $1$. Thus, we
    see again that $R/I$ has the weak Lefschetz property if and only if the base field $K$ has not characteristic $2$.
\end{example}~

\subsection{Complete Intersections}~\par

We now begin to reinterpret the results in Section~\ref{sec:det} about signed enumerations of triangular regions as
results about the weak Lefschetz property for the associated Artinian ideals. In this subsection we restrict ourselves
to the ideals with the fewest number of generators, namely the ideals of the form $I = (a^a, y^b, z^c)$. These are
monomial complete intersections, and the question whether they have the weak Lefschetz property has motivated a great deal of
research (see \cite{MN-survey} and Remark~\ref{rem:ci-history} below).

We recall a well-known result of Reid, Roberts, and Roitman about the shape of Hilbert functions of monomial complete
intersections.

\begin{lemma}{\cite[Theorem~1]{RRR}} \label{lem:h-ci}
    Let $I = (a^a, y^b, z^c)$, where $a, b,$ and $c$ are positive integers. Then the Hilbert function $h = h_{R/I}$ of
    $R/I$ has the following properties:
    \begin{enumerate}
        \item $h(j-2) < h(j-1)$ if and only if $1 \leq j < \min \{a+b, a+c, b+c, \frac{1}{2}(a+b+c)\}$;
        \item $h(j-2) = h(j-1)$ if and only if $\min \{a+b, a+c, b+c, \frac{1}{2}(a+b+c)\} \leq j \leq \max \{a, b, c, \frac{1}{2}(a+b+c)\}$; and
        \item $h(j-2) > h(j-1)$ if and only if $\max \{a, b, c, \frac{1}{2}(a+b+c)\} < j \leq a+b+c-1$.
    \end{enumerate}
\end{lemma}

Depending on the characteristic of the base field we get the following sufficient conditions that guarantee the weak
Lefschetz property.

\begin{proposition} \label{pro:ci-wlp}
    Let $I = (x^a, y^b, z^c)$, where $a, b,$ and $c$ are positive integers.  Set $d = \flfr{a+b+c}{2}$. Then:
    \begin{enumerate}
        \item If $d < \max\{a,b,c\}$, then $R/I$ has the weak Lefschetz property, regardless of the characteristic of $K$.
        \item If $a+b+c$ is even, then $R/I$ has the weak Lefschetz property in characteristic $p$ if and only if $p$
            does not divide $\Mac(d-a, d-b, d-c)$.
        \item If $a+b+c$ is odd, then $R/I$ has the weak Lefschetz property in characteristic $p$ if and only if $p$
            does not divide at least one of the integers
            \[
                \frac{\binom{d-1}{a-1}}{\binom{d-1}{i}} \binom{d-c}{a-i-1} \Mac(d-a-1, d-b, d-c),
            \]
            where $d-1-b < i < a$.
    \end{enumerate}

    In any case, $R/I$ has the weak Lefschetz property in characteristic $p$ if $p = 0$ or $p \geq \flfr{a+b+c}{2}$.
\end{proposition}
\begin{proof}
    The algebra $R/I$ has exactly one socle generator. It  has degree $a+b+c-3 \geq d-2$.

    If $d < \max\{a,b,c\}$, then without loss of generality we may assume $a > d$, that is, $a > b+c$. In this case,
    $T_d(I)$ has two punctures, one of length $d-b$ and one of length $d-c$. Moreover, $d-b + d-c = a > d$ so the two
    punctures overlap. Hence $T_d(I)$ is balanced and has a unique tiling. That is, $|\det{Z(T)}| = 1$ and so $R/I$ has
    the weak Lefschetz property, regardless of the characteristic of $K$ (see Corollary~\ref{cor:wlp-Z}).

    Suppose $d \geq \max\{a,b,c\}$.  By Lemma~\ref{lem:h-ci}, we have $h_{R/I}(d-2) \leq h_{R/I}(d-1) > h_{R/I}(d)$.

    Assume $a+b+c$ is even. Then Proposition~\ref{pro:ci-enum} gives that $|\det{Z(T_d(I))}| = \Mac(d-a, d-b, d-c)$, and
    so claim (ii) follows by Corollary~\ref{cor:wlp-Z}.

    Assume $a+b+c$ is odd, and so $d = \frac{1}{2}(a+b+c-1)$. In this case it is enough to find non-trivial maximal
    minors of $T_{d}(I)$ and $T_{d+1}(I)$ by Corollary~\ref{cor:wlp-Z}. Consider the hexagonal regions formed by the
    present unit triangles of each $T_{d+1}(I)$ and $T_{d}(I)$. The former hexagon is obtained from the latter by a
    rotation about $180^{\circ}$. Thus, we need only consider the maximal minors of $T_{d}(I)$. This region has exactly
    one more upward-pointing triangle than downward-pointing triangle. Hence, by Proposition~\ref{pro:restricted-only},
    it suffices to check whether the restricted maximal minors of $T_{d}(I)$ have maximal rank. These minors are exactly
    $T_{i} := T_d(x^a, y^b, z^c, x^{i} y^{d-1-i})$, where $d-1-b < i < a$. Using Proposition~\ref{pro:two-mahonian}, we
    get that $|\det{Z(T_{i})}|$ is
    \[
       \Mac(a -1-i, d - a, 1) \Mac(i + b -d, d-b, 1) \frac{\HF(d-a+1) \HF(d-b+1) \HF(d-c+1) \HF(d)}{\HF(a)\HF(b)\HF(c)\HF(1)},
    \]
    where we notice $d - (i + (d-1-i)) = 1$. Since $\Mac(n, k, 1) = \binom{n+k}{k}$ and $\HF(n) = (n-1)! \HF(n-1)$, for
    positive integers $n$ and $k$, we can rewrite $|\det{Z(T_{i})}|$ as
    \[
        \binom{d-1-i}{d-a} \binom{i}{d-b} \frac{(d-b)! (d-c)!}{(a-1)!} \Mac(d-a-1, d-b, d-c).
    \]
    Simplifying this expression, we get part (iii).

    Finally, using both Propositions~\ref{pro:ci-enum} and~\ref{pro:two-mahonian} we see that the prime divisors of
    $|\det{Z(T_{i})}|$ are bounded above by $d-1$ in each case.
\end{proof}

As announced, we briefly comment on the history of the last result and the research it motivated.

\begin{remark} \label{rem:ci-history}
    The presence of the weak Lefschetz property for monomial complete intersections has been studied by many authors.
    The fact that \emph{all} monomial complete intersections, in any number of variables, have the strong Lefschetz
    property in characteristic zero was proven first by Stanley~\cite{Stanley-1980} using the Hard Lefschetz Theorem.
    (See \cite{Co}, and the references contained therein, for more on the history of this theorem.) However, the weak
    Lefschetz property can fail in positive characteristic.

    The weak Lefschetz property in arbitrary characteristic in the case where one generator has much larger degree than
    the others (case (i) in the preceding proposition) was first established by Watanabe~\cite[Corollary~2]{Wa} for
    arbitrary complete intersections in three variables, not just monomial ones. Migliore and
    Mir\'o-Roig~\cite[Proposition~5.2]{MM} generalised this to complete intersections in $n$ variables.

    Part (ii) of the above result was first established by the authors~\cite[Theorem~4.3]{CN-IJM} (with an extra
    generator of sufficiently large degree), and independently by Li and Zanello~\cite[Theorem~3.2]{LZ}. The latter also
    proved part (iii) above (use $i = a - k$). However, while both papers mentioned the connection to lozenge tilings of
    hexagons, it was Chen, Guo, Jin, and Li~\cite{CGJL} who provided the first combinatorial explanation. In particular,
    the case (ii) was studied in~\cite[Theorem~1.2]{CGJL}. We also note that~\cite[Theorem~4.3]{LZ} can be recovered
    from Proposition~\ref{pro:ci-wlp} if we set $a = \beta + \gamma$, $b = \alpha + \gamma$, and $c = \alpha + \beta$.

    More explicit results have been found in the special case where all generators have the same degree, i.e., $I_a =
    (x^a, y^a, z^a)$. Brenner and Kaid used the idea of a syzygy gap to explicitly classify the prime characteristics in
    which $I_a$ has the weak Lefschetz property~\cite[Theorem~2.6]{BK-p}. Kustin, Rahmati, and Vraciu used this result
    in \cite{KRV}, in which they related the presence of the weak Lefschetz property of $R/I_a$ to the finiteness of the
    projective dimension of $I_a : (x^n + y^n + z^n)$. Moreover, Kustin and Vraciu later gave an alternate explicit
    classification of the prime characteristics in which $I_a$ has the weak Lefschetz property~\cite[Theorem~4.3]{KV}.

    As a final note, Kustin and Vraciu~\cite{KV} also gave an explicit classification of the prime characteristics in
    which monomial complete intersections in arbitrarily many variables with all generators of the same degree have the
    weak Lefschetz property. This was expanded by the first author~\cite[Theorem~7.2]{Co} to an explicit classification
    of the prime characteristics, in which the algebra has the \emph{strong Lefschetz property}. In this work another
    combinatorial connection was used to study the presence of the weak Lefschetz property for monomial complete
    intersections in arbitrarily many variables.
\end{remark}~



\subsection{Reinterpretation}~\par\label{subsec:reinterpret}

We provide some rather direct interpretations of earlier results to more general ideals than complete intersections.
More involved uses of our methods will be described in the following sections.

Here we will restrict ourselves to considering balanced regions. In this case we observe the following necessary
condition for the presence of the weak Lefschetz property.

\begin{proposition} \label{pro:non-tileable-non-wlp}
    Let $I$ be a monomial ideal such that $T_d(I)$ is a balanced region that is not tileable. Put $J = I + (x^d, y^d, z^d)$.
    Then $R/J$ never has the weak Lefschetz property, regardless of the characteristic of $K$.
\end{proposition}
\begin{proof}
    Since $T_d(I) = T_d(J)$ is not tileable, Theorem~\ref{thm:pm-matrix} gives $\det Z(T_d(J)) = 0$. Thus, $Z(T_d(J))$
    does not have maximal rank. Now we conclude by Corollary~\ref{cor:wlp-biadj}.
\end{proof}

We illustrate the preceding proposition with an example.

\begin{example}
    Consider the regions depicted in Figure~\ref{fig:nontileable}.
    \begin{figure}[!ht]
        \begin{minipage}[b]{0.48\linewidth}
            \centering
            \includegraphics[scale=1]{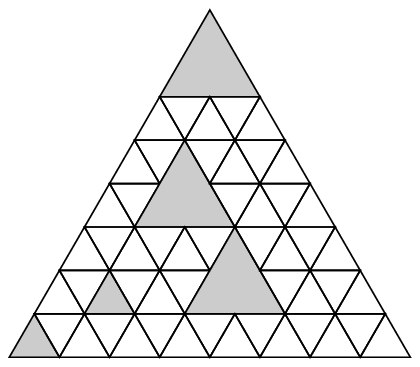}\\
            \emph{(i) $T = T_8(x^6, y^7, z^8, xy^5z, xy^2z^3, x^3y^2z)$}
        \end{minipage}
        \begin{minipage}[b]{0.48\linewidth}
            \centering
            \includegraphics[scale=1]{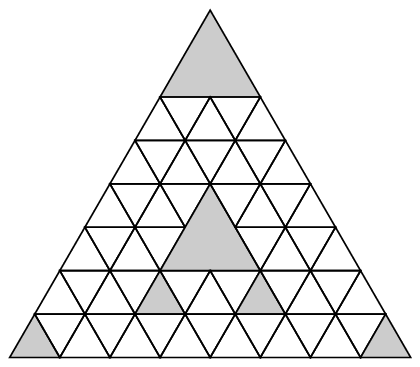}\\
            \emph{(ii) $T' = T_8(x^6, y^7, z^7, xy^4z^2, xy^2z^4, x^2y^2z^2)$}
        \end{minipage}
        \caption{Two balanced non-tileable triangular regions.}
        \label{fig:nontileable}
    \end{figure}
    These regions are both balanced, but non-tileable as they contain $\dntri$-heavy monomial subregions (see
    Theorem~\ref{thm:tileable}). In particular, the monomial subregion assoicated to $xy^2z$ in $T$ and the monomial
    subregion associated to $xy^2z^2$ in $T'$ are both $\dntri$-heavy. Thus, $R/ (x^6, y^7, z^8, xy^5z, xy^2z^3, x^3y^2z)$
    and $R/(x^6, y^7, z^7, xy^4z^2, xy^2z^4, x^2y^2z^2)$ both fail to have the weak Lefschetz property,
    regardless of the characteristic of the base field.
\end{example}

Now we use Propostion~\ref{pro:non-tileable-non-wlp} in order to relate the weak Lefschetz property and semistability of
syzygy bundles (see Section~\ref{sec:syz}). In preparation, we record the following observation. Recall that the
monomial ideal of a triangular region $T \subset {\mathcal T}_d$ is the largest ideal $J$ whose minimal generators have
degrees less than $d$ such that $T = T_d (J)$ (see Subsection~\ref{subsec:gcd}).

\begin{lemma}\label{lem:lcm-overlap}
    Let $J \subset R$ be the monomial ideal of a triangular region $T \subset {\mathcal T}_d$. Then:
    \begin{enumerate}
        \item The region $T$ has no overlapping punctures if and only if each degree of a least common multiple of two
            distinct minimal generators of $J$ is at least $d$.
        \item The punctures of $T$ are not overlapping nor touching if and only if each degree of a least common
            multiple of two distinct minimal generators of $J$ is at least $d+1$.
    \end{enumerate}
\end{lemma}
\begin{proof}
    Let $m_1$ and $m_2$ be two distinct minimal generators of $J$. Then their associated punctures overlap if and only
    if there is a monomial of degree $d-1$ that is a multiple of $m_1$ and $m_2$. The existence of such a monomial means
    precisely that the degree of the least common multiple of $m_1$ and $m_2$ is at most $d-1$. Now claim (i) follows.

    Assertion (ii) is shown similarly by observing that the punctures to $m_1$ and $m_2$ touch if and only if there is a monomial of degree $d$ that is a multiple of $m_1$ and $m_2$.
\end{proof}

The following consequence is useful later on.

\begin{corollary}\label{cor:socle-degree-bound}
    Assume $T \subset {\mathcal T}_d$ is a triangular region whose punctures are not overlapping nor touching, and let $J$
    be the monomial ideal of $T$. Then $R/J$ does not have non-zero socle elements of degree less than $d-1$.
\end{corollary}
\begin{proof}
    Since $J$ is a monomial ideal, every first syzygy of $J$ corresponds to a relation $m_i n_i - m_j n_j = 0$ for
    suitable monomials $n_i$ and $n_j$, where $m_i$ and $m_j$ are distinct monomial minimal generators of $J$. Applying
    Lemma~\ref{lem:lcm-overlap} to the equality $m_i n_i = m_j n_j$, we conclude that the degree of each first syzygy of
    $J$ is at least $d+1$. It follows that the degree of every second syzygy of $J$ is at least $d+2$. Each minimal
    second syzygy of $J$ corresponds to a socle generator of $R/J$ (see the beginning of Subsection~\ref{subsec:tools}).
    Hence, the degrees of the socle generators of $R/J$ are at least $d-1$ .
\end{proof}

The converse of Corollary~\ref{cor:socle-degree-bound} is not true in general. For example, the socle generators of
$R/(x^6, y^7, z^8, xy^5z, xy^2z^3, x^3y^2z)$ have degrees greater than 7, but two punctures of
$T_8 (x^6, y^7, z^8, xy^5z, xy^2z^3, x^3y^2z)$ touch each other (see Figure~\ref{fig:nontileable}).

Recall that perfectly-punctured regions were defined above Corollary~\ref{cor:pp-tileable}. This concept is used in
the proof of the following result.

\begin{theorem}\label{thm:wlp-to-semistab}
    Let $I \subset R$ be an Artinian ideal whose minimal monomial generators have degrees  $d_1,\ldots,d_t$. Set
    \[
        d := \frac{d_1 + \cdots + d_t}{t-1}.
    \]
    Assume $\charf K = 0$ and that the following conditions are satisfied:
    \begin{enumerate}
        \item The number $d$ is an integer.
        \item For all $i = 1,\ldots,t$, one has $d > d_i$.
        \item Each degree of a least common multiple of two distinct minimal generators of $I$ is at least $d$.
    \end{enumerate}
    Then the syzygy bundle of $I$ is semistable if  $R/I$ has the weak Lefschetz property.
\end{theorem}
\begin{proof}
    Consider the triangular region $T = T_d (I)$. By assumption (iii) and Lemma~\ref{lem:lcm-overlap}, we obtain that
    the punctures of $T$ do not overlap. Recall that the side length of the puncture to a minimal generator of degree
    $d_i$ is $d - d_i$. The definition of $d$ is equivalent to
    \[
        d = \sum_{i = 1}^t (d - d_i).
    \]
    We conclude that the region $T$ is balanced and perfectly-punctured. Combined with the weak Lefschetz property of
    $R/I$, the first property implies that $T$ is tileable by Proposition~\ref{pro:non-tileable-non-wlp}. Now
    Theorem~\ref{thm:tileable-semistable} gives the semistability of the syzygy bundle of $I$.
\end{proof}

The converse of the above result is not true, in general.

\begin{remark}
    The mirror symmetric regions considered in Section~\ref{sec:mirror} are all balanced and tileable. Thus,
    Theorem~\ref{thm:tileable-semistable} gives that each ideal of such a region (see Remark~\ref{rem:mirror-prop})
    \[
        J = (x^{h_1}, y^{d-b}, z^{d-b},
             x^{h_2} (y z)^{\frac{1}{2}(d - d_2 - h_2)},
             \ldots,
             x^{h_m} (y z)^{\frac{1}{2}(d - d_m - h_m)}).
    \]
    has a semi-stable syzygy bundle. However, Theorem~\ref{thm:mirror-odd23} shows that $R/J$ does not have the weak
    Lefschetz property if the number of axial punctures of $T_d(J)$ with odd side length is 2 or 3 modulo 4. If,
    instead, all axial punctures, except possibly the top one, do have an even side length, then $R/J$ has the weak
    Lefschetz property (see Proposition~\ref{pro:mirror-even}).
\end{remark}

However, under stronger assumptions the converse to Theorem~\ref{thm:wlp-to-semistab} is indeed true.

\begin{theorem}\label{thm:wlp-iff-semistab}
    Let $I \subset R$ be an Artinian ideal with minimal monomial generators $m_1,\ldots,m_t$. Set
    \[
        d := \frac{d_1 + \cdots + d_t}{t-1},
    \]
    where $d_i = \deg m_i$. Assume $\charf K = 0$ and that the following conditions are satisfied:
    \begin{enumerate}
        \item The number $d$ is an integer.
        \item For all $i = 1,\ldots,t$, one has $d > d_i$.
        \item If $i \neq j$, then the degree of the least common multiple of $m_i$ and $m_j$ is at least $d+1$.
        \item If $m_i$ is not a power of $x, y$, or $z$, then $d - d_i$ is even.
    \end{enumerate}
    Then the syzygy bundle of $I$ is semistable if and only if  $R/I$ has the weak Lefschetz property.
\end{theorem}
\begin{proof}
    By Theorem~\ref{thm:wlp-to-semistab}, it is enough to show that $R/I$ has the weak Lefschetz property if the syzygy
    bundle of $I$ is semistable.

    Consider the region $T = T_d(I)$. In the proof of Theorem~\ref{thm:wlp-to-semistab} we showed that $T$ is balanced
    and perfectly-punctured. Hence $T$ is tileable by Theorem~\ref{thm:tileable-semistable}. Since all floating
    punctures of $T$ have an even side length by assumption (iv), Theorem~\ref{thm:pm-matrix} and
    Proposition~\ref{prop:same-sign} give that $Z(T)$ has maximal rank.

    Assumption (iii) means that the punctures of $T$ are not overlapping nor touching (see Lemma~\ref{lem:lcm-overlap}).
    Hence, Corollary~\ref{cor:socle-degree-bound} yields that the degrees of the socle generators of $R/I$ are at least
    $d-1$. Therefore, Corollary~\ref{cor:twin-peaks-wlp} proves that $R/I$ has the weak Lefschetz property.
\end{proof}

We now show that, for all positive integers $d_1,\ldots,d_t$ with $t \geq 3$ that satisfy the numerical assumptions (i),
(ii), and (iv) of Theorem~\ref{thm:wlp-iff-semistab}, there is a monomial ideal $I$ whose minimal generators have
degrees $d_1,\ldots,d_t$ to which Theorem~\ref{thm:wlp-iff-semistab} applies and guarantees the weak Lefschetz property
of $R/I$.

\begin{example}\label{exa:ideals-with-wlp}
    Let $d_1,\ldots,d_t$ be $t \geq 3$ positive integers satisfying the following numerical conditions:
    \begin{enumerate}
        \item The number $d := \frac{d_1 + \cdots + d_t}{t-1}$ is an integer.
        \item For all $i = 1,\ldots,t$, one has $d > d_i$.
        \item At most three of the integers $d - d_i$ are not even.
    \end{enumerate}
    Re-indexing if needed, we may assume that $d_3 \leq \min \{d_1, d_2\}$ and that $d - d_i$ is even whenever
    $4 \leq i \leq t$. Consider the following ideal
    \begin{equation*}
        I = (x^{d_1}, y^{d_2}, z^{d_3}, m_4,\ldots,m_t),
    \end{equation*}
    where $m_4 = x^{d-d_3} y z^{-d-1 + d_3 +d_4}$ if $t \geq 4$, $m_5 = x^{2d - d_3 - d_4} y^2 z^{2d - 2 +d_3+d_4 + d_5}$
    if $t \geq 5$, and
    \begin{equation*}
        m_i =
        \begin{cases}
            x^{d-d_3} y^{1 + \sum_{k=4}^{i-1} (d- d_k)} z^{-d (i-3) - 1 + \sum_{k=3}^i d_k } & \text{if } 6 \leq i \leq t \text{ and $i$ is even}\\
            x^{-1 + \sum_{k=3}^{i-1} (d- d_k)} y^2 z^{-d (i-3) -1 + \sum_{k=3}^i d_k } & \text{if } 7 \leq i \leq t \text{ and $i$ is odd.}
        \end{cases}
    \end{equation*}
    Note that $\deg m_i = d_i$ for all $i$. One easily checks that the degree of the least common multiple of any two
    distinct minimal generators of $I$ is at least $d+1$, that is, the punctures of $T_d(I)$ do not overlap nor touch
    each other.
    \begin{figure}[!ht]
        \includegraphics[scale=1]{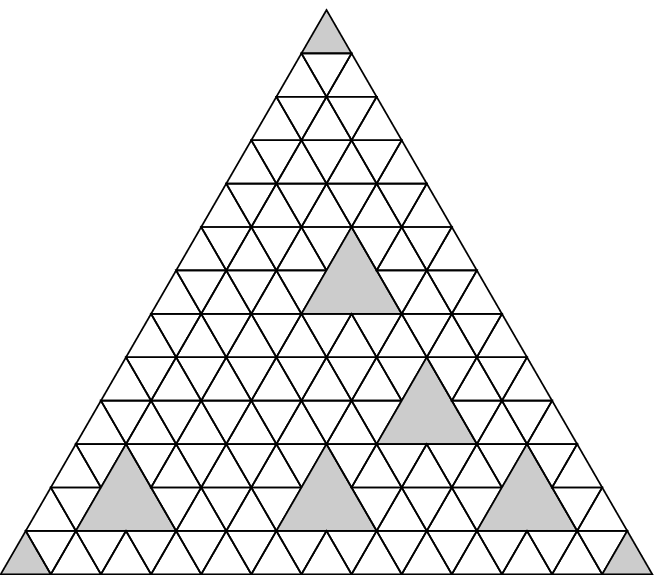}
        \caption{The region corresponding to $d_1 = d_2 = d_3 = 12$ and $d_4 = \cdots = d_8 = 11$ in Example~\ref{exa:ideals-with-wlp}.}
        \label{fig:example-d-13-t-8}
    \end{figure}
\end{example}

\begin{corollary}
    Let $I$ be any ideal as defined in Example~\ref{exa:ideals-with-wlp}. Assume that the base field $K$ has
    characteristic zero. Then $R/I$ has the weak Lefschetz property and the syzygy bundle of $I$ is semistable.
\end{corollary}
\begin{proof}
    By construction, the considered ideals satisfy assumptions (i)--(iv) of Theorem~\ref{thm:wlp-iff-semistab}.
    Furthermore, the region $T_d(I)$ has no over-punctured monomial subregions. Hence, it is tileable by
    Corollary~\ref{cor:pp-tileable}. (Alternatively, one can exhibit a family of non-intersecting lattice paths to check
    tileability.) By Theorem~\ref{thm:tileable-semistable}, it follows that the syzygy bundle of $I$ is semistable, and
    hence $R/I$ has the weak Lefschetz property by Theorem~\ref{thm:wlp-iff-semistab}.
\end{proof}

\begin{remark}
    Given an integer $t \geq 3$, there are many choices for the integers $d_1,\ldots,d_t$, and thus for the ideals
    exhibited in Example~\ref{exa:ideals-with-wlp}. A convenient choice, for which the description of the ideal becomes
    simpler, is $d_1 = 2t -4$, $d_2 = d_3 = d-1$, and $d_4 = \cdots = d_t = d-2$, where $d$ is any integer satisfying
    $d \geq 2t -3$. Then the corresponding ideal is
    \begin{equation*}
        I = (x^{2t - 4}, y^{d - 1}, z^{d-1}, xyz^{d-4}, x^3 y^2 z^{d-7}, m_6,\ldots,m_t ),
    \end{equation*}
    where
    \begin{equation*}
        m_i =
        \begin{cases}
            xy^{2i -7} z^{d+4 - 2i} & \text{if } 6 \leq i \leq t \text{ and $i$ is even}\\
            x^{2i - 8} y^2 z^{d+4 - 2i} & \text{if } 7 \leq i \leq t \text{ and $i$ is odd.}
        \end{cases}
    \end{equation*}
    The ideal in Figure~\ref{fig:example-d-13-t-8} is generated as above with $d = 13$ and $t = 8$.
\end{remark}

\section{Artinian monomial algebras of type two in three variables} \label{sec:type-two}

Boij, Migliore, Mir\'o-Roig, Zanello, and the second author proved in \cite[Theorem~6.2]{BMMNZ} that the Artinian
monomial algebras of type two in three variables that are \emph{level} have the weak Lefschetz property in
characteristic zero. The proof given there is surprisingly intricate and lengthy. In this section, we establish a more
general result using techniques derived in the previous sections.

To begin, we classify the Artinian monomial ideals $I$ in $R = K[x,y,z]$ such that $R/I$ has type two, that is, its socle is
of the form $\soc (R/I) \cong K(-s) \oplus K(-t)$. The algebra $R/I$ is level if the socle degrees $s$ and $t$ are
equal. The classification in the level case has been established in \cite[Proposition~6.1]{BMMNZ}. The following more
general result is obtained similarly.

\begin{lemma} \label{lem:classify-type-two}
    Let $I$ be an Artinian monomial ideal in $R = K[x,y,z]$ such that $R/I$ is of type $2$. Then, up to a change of
    variables, $I$ has one of the following two forms:
    \begin{enumerate}
        \item $I = (x^a, y^b, z^c, x^{\alpha} y^{\beta})$, where $0 < \alpha < a$ and $0 < \beta < b$. In this case, the
            socle degrees of $R/I$ are $a + \beta + c-3$ and $\alpha + b + c-3$. Thus, $I$ is level if and only if
            $a - \alpha = b - \beta$.
        \item $I = (x^a, y^b, z^c, x^{\alpha} y^{\beta}, x^{\alpha} z^{\gamma})$, where $0 < \alpha < a$, $0 < \beta < b$,
            and $0 < \gamma < c$. In this case, the socle degrees of $R/I$ are $a + \beta + \gamma-3$ and $\alpha + b + c-3$.
            Thus, $I$ is level if and only if $a - \alpha = b - \beta + c - \gamma$.
    \end{enumerate}
\end{lemma}
\begin{proof}
    We use Macaulay-Matlis duality. An Artinian monomial algebra of type two over $R$ arises as the inverse system of
    two monomials, say $x^{a_1} y^{b_1} z^{c_1}$ and $x^{a_2} y^{b_2} z^{c_2}$, such that one does not divide the other.
    Thus we may assume without loss of generality that $a_1 > a_2$ and $b_1 < b_2$. We consider two cases: $c_1 = c_2$
    and $c_1 \neq c_2$.

    Suppose first that $c_1 = c_2$.  Then the annihilator of the monomials is the ideal
    \[
        (x^{a_1+1}, y^{b_1+1}, z^{c_1+1}) \cap (x^{a_2+1}, y^{b_2+1}, z^{c_1+1}) =   (x^{a_1+1}, y^{b_2+1}, z^{c_1+1}, x^{a_2+1} y^{b_1+1}),
    \]
    which is the form in (i).  By construction, the socle elements are $x^{a_1} y^{b_1} z^{c_1}$ and $x^{a_2} y^{b_2} z^{c_1}$.

    Now suppose $c_1 \neq c_2$; without loss of generality we may assume $c_1 < c_2$. Then the annihilator of the
    monomials is the ideal
    \[
        (x^{a_1+1}, y^{b_1+1}, z^{c_1+1}) \cap (x^{a_2+1}, y^{b_2+1}, z^{c_2+1}) =
        (x^{a_1 + 1}, y^{b_2+1}, z^{c_2+1}, x^{a_2 + 1} y^{b_1+1}, x^{a_2+1} z^{c_1+1}),
    \]
    which is the form in (ii).  By construction, the socle elements are $x^{a_1} y^{b_1} z^{c_1}$ and $x^{a_2} y^{b_2} z^{c_2}$.
\end{proof}

We give a complete classification of the type two algebras that have the weak Lefschetz property in characteristic zero.

\begin{theorem} \label{thm:type-two}
    Let $I$ be an Artinian monomial ideal in $R = K[x,y,z]$, where $K$ is a field of characteristic zero, such that
    $R/I$ is of type $2$. Then $R/I$ fails to have the weak Lefschetz property in characteristic zero if and only if
    $I = (x^a, y^b, z^c, x^{\alpha} y^{\beta}, x^{\alpha} z^{\gamma})$, up to a change of variables, where
    $0 < \alpha < a$, $0 < \beta < b$, and $0 < \gamma < c$, and there exists an integer $d$ with
    \begin{equation} \label{eqn:type-two}
        \begin{split}
                \max \left\{a, \alpha + \beta, \alpha + \gamma, \frac{a+\alpha+\beta+\gamma}{2} \right\} < d \hspace*{7cm}\\
                < \min \left\{a+\beta + \gamma, \frac{\alpha+b+c}{2}, b+c, \alpha + c, \alpha + b \right\}.
        \end{split}
    \end{equation}
\end{theorem}
\begin{proof}
    According to Corollary \ref{cor:wlp-biadj}, for each integer $d > 0$, we have to decide whether the bi-adjacency
    matrix $Z(T_d(I))$ has maximal rank. This is always true if $d = 1$. Let $d \geq 2$.

    By Lemma~\ref{lem:classify-type-two}, we may assume that $I$ has one of two forms given there. The difference
    between the two forms is an extra generator, $x^{\alpha} z^{\gamma}$. In order to determine the rank of $Z( T_d(I))$
    we split $T = T_d(I)$ across the horizontal line $\alpha$ units from the bottom edge. We call the monomial
    subregion above the line, which is the subregion associated to $x^\alpha$, the \emph{upper portion} of $T$, denoted
    by $T^u$, and we call the isosceles trapezoid below the line the \emph{lower portion} of $T$, denoted by $T^l$. Note
    that $T^u$ is empty if $d \leq \alpha$. Both portions, $T^u$ and $T^l$, are hexagons, i.e., triangular regions
    associated to complete intersections. In particular, if $I$ has four generators, then
    $T^u = T_{d-\alpha}(x^{a-\alpha}, y^{\beta}, z^c)$. Similarly, if $I$ has five generators, then
    $T^u = T_{d-\alpha}(x^{a-\alpha}, y^{\beta}, z^{\gamma})$. In both cases $T^l$ is $T_d(x^{\alpha}, y^b, z^c)$. See
    Figure~\ref{fig:decompose} for an illustration of this decomposition.
    \begin{figure}[!ht]
        \includegraphics[scale=2]{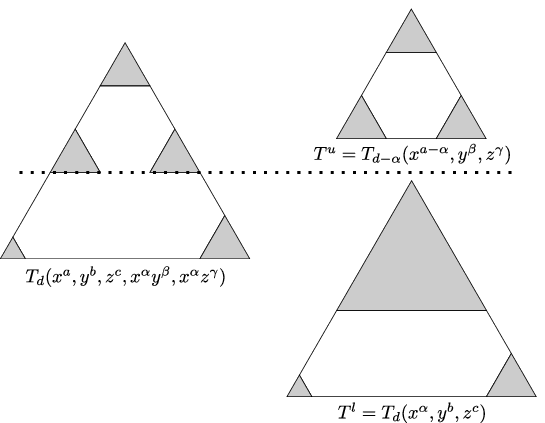}
        \caption{The decomposition of $T_d(I)$ into $T^u$ and $T^l$.}
        \label{fig:decompose}
    \end{figure}

    After reordering rows and columns of the bi-adjacency matrix  $Z(T)$, it becomes a  block matrix of the form
    \begin{equation}\label{eq:block-Z}
        Z =
        \begin{pmatrix}
            Z(T^u) & 0 \\
            Y & Z(T^l)
        \end{pmatrix}
    \end{equation}
    because the downward-pointing triangles in $T^u$ are not adjacent to any upward-pointing triangles in $T^l$. For
    determining when $Z$ has maximal rank, we study several cases, depending on whether $T^u$ and $T^l$ are $\uptri$-heavy,
    balanced, or $\dntri$-heavy.
    \smallskip

    First, suppose one of the following conditions is satisfied: (i) $T^u$ or $T^l$ is balanced, (ii) $T^u$ and $T^l$
    are both $\uptri$-heavy, or (iii) $T^u$ and $T^l$ are both $\dntri$-heavy. In other words, $T^u$ and $T^l$ do not
    ``favor'' triangles of opposite orientations. Since $T^u$ and $T^l$ are triangular regions associated to complete
    intersections, both $Z(T^u)$ and $Z(T^l)$ have maximal rank by Proposition~\ref{pro:ci-wlp}. Combining non-vanishing
    maximal minors of $Z(T^u)$ and $Z(T^l)$, if follows that the matrix $Z$ has maximal rank as well.
    \smallskip

    Second, suppose $T^u$ is $\uptri$-heavy and $T^l$ is $\dntri$-heavy. We will show that $Z$ has maximal rank in this case.

    Let $t_u = \#\uptri(T^u) - \#\dntri(T^u)$ and $t_l = \#\dntri(T^l) - \#\uptri(T^l)$ be the number of excess
    triangles of each region. In a first step, we show that we may assume $t_u = t_l$. To this end we remove enough of
    the appropriately oriented triangles from the more unbalanced of $T^u$ and $T^l$ until both regions are equally
    unbalanced. Set $t = \min\{t_u, t_l\}$.

    Assume $T^u$ is more unbalanced, i.e., $t_u > t$. Since $T^u$ is $\uptri$-heavy, the top $t_u$ rows of ${\mathcal
    T}_d$ below the puncture associated to $x^a$ do not have a puncture. Thus, we can remove the top $t_u - t$
    upward-pointing triangles in $T^u$ along the upper-left edge of ${\mathcal T}_d$, starting at the puncture
    associated to $x^a$, if present, or in the top corner otherwise. Denote the resulting subregion of $T$ by $T'$.
    Notice that $Z$ has maximal rank if $Z(T')$ has maximal rank. Furthermore, the $t_u - t$ rows in which $T$ and $T'$
    differ are uniquely tileable. Denote this subregion of $T'$ by $U$ (see Figure~\ref{fig:type-2-case-8}(i) for an
    illustration). By construction, the upper and the lower portion ${T^u}'$ and ${T^l}' = T^l$, respectively, of $T'
    \setminus U$ are equally unbalanced. Moreover, $Z(T')$ has maximal rank if and only if $Z(T' \setminus U)$ has
    maximal rank by Proposition~\ref{prop:remove-unique-tileable}. As desired, $T$ and $T' \setminus U$ have the same shape.

    Assume now that $T^l$ is more unbalanced, i.e., $t_l > t$. Since $T^l$ is $\dntri$-heavy, the two punctures
    associated to $x^b$ and $x^c$, respectively, cover part of the bottom $t_l$ rows of ${\mathcal T}_d$. Thus, we can
    remove the bottom $t_l - t$ downward-pointing triangles of $T^l$ along the puncture associated to $x^c$. Denote the
    resulting subregion of $T$ by $T'$. Notice that $Z$ has maximal rank if $Z(T')$ has maximal rank. Again, the $t_l -
    t$ rows in which $T$ and $T'$ differ form a uniquely tileable subregion. Denote it by $U$. By construction, the
    upper and the lower portion ${T^u}' = T^u$ and ${T^l}' $, respectively, of $T' \setminus U$ are equally
    unbalanced. Moreover, $Z(T')$ has maximal rank if and only if $Z(T' \setminus U)$ has maximal rank by
    Proposition~\ref{prop:remove-unique-tileable}. As before, $T$ and $T' \setminus U$ have the same shape.

    The above discussion shows it is enough to prove that the matrix $Z$ has maximal rank if $t_u = t_l = t$, i.e., $T$
    is balanced. Since $T$ has no floating punctures, Proposition~\ref{prop:same-sign} gives the desired maximal rank of
    $Z$ once we know that $T$ has a tiling. To see that $T'$ is tileable, we first place $t$ lozenges across the line
    separating $T^u$ from $T^l$, starting with the left-most such lozenge. Indeed, this is possible since $T^u$ has $t$
    more upwards-pointing than downwards-pointing triangles. Next, place all fixed lozenges. The portion of ${T^u}$ that
    remains untiled after placing these lozenges is a hexagon. Hence it is tileable. Consider now the portion of ${T^l}$
    that remains untiled after placing these lozenges. Since $t$ is at most the number of horizontal rows of $T^l$ this
    portion is, after a $60^{\circ}$ rotation, a region as described in Proposition~\ref{pro:two-mahonian}. Thus, it is
    tileable. Figure~\ref{fig:type-2-case-8}(ii) illustrates this procedure with an example.

    It follows that $T$ is tileable. Therefore $Z$ has maximal rank, as desired.
    \smallskip

    \begin{figure}[!ht]
        \begin{minipage}[b]{0.48\linewidth}
            \centering
            \includegraphics[scale=1]{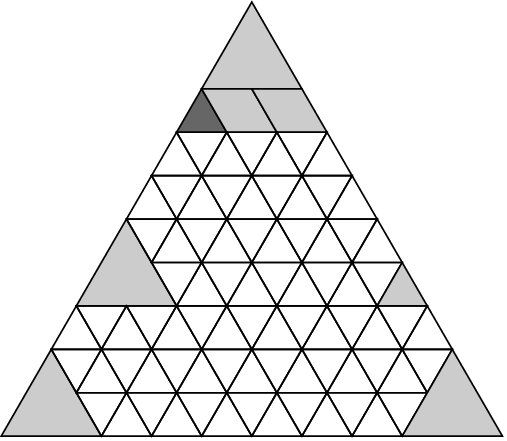}\\
            \emph{(i) A maximal minor of $T$; the removed triangle is darkly-shaded.}
        \end{minipage}
        \begin{minipage}[b]{0.48\linewidth}
            \centering
            \includegraphics[scale=1]{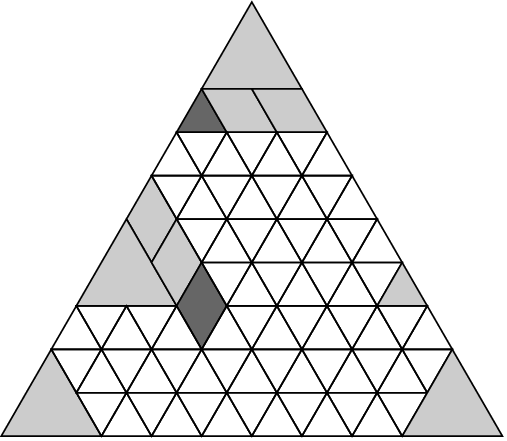}\\
            \emph{(ii) Placing a lozenge on the maximal minor to produce a tiling.}
        \end{minipage}
        \caption{Let $T = T_{10}(x^8, y^8, z^8, x^3 y^5, x^3 z^6)$.
            The lightly shaded lozenges are fixed lozenges.}
        \label{fig:type-2-case-8}
    \end{figure}

    Finally, suppose $T^u$ is $\dntri$-heavy and $T^l$ is $\uptri$-heavy. Consider any maximal minor of $Z(T)$. It
    corresponds to a balanced subregion $T'$ of $T$. Then its upper portion ${T^u}'$ is still $\dntri$-heavy, and its
    lower portion ${T^l}'$ is $\uptri$-heavy. Hence, any covering of ${T^u}'$ by lozenges must also cover some
    upward-pointing triangles of ${T^l}'$. The remaining part of ${T^l}'$ is even more unbalanced than ${T^l}'$. This
    shows that $T'$ is not tileable. Thus, $\det Z(T') = 0$ by Theorem \ref{thm:pm-matrix}. It follows that $Z$ does not
    have maximal rank in this case.
    \smallskip

    The above case analysis proves that $R/I$ fails the weak Lefschetz property if and only if there is an integer $d$
    so that the associated regions $T^u$ and $T^l$ are $\dntri$-heavy and $\uptri$-heavy, respectively. It remains to
    determine when this happens.

    If $I$ has only four generators, then no row of $T^u$ has more downward-pointing than upward-pointing triangles.
    Hence, $T^u$ is not $\dntri$-heavy. It follows that $I$ must have five generators if $R/I$ fails to have the weak
    Lefschetz property. For such an ideal $I$, the region $T^u = T_{d-\alpha}(x^{a-\alpha}, y^{\beta}, z^{\gamma})$ is
    $\dntri$-heavy if and only if
    \[
        \dim_K [R/(x^{a-\alpha}, y^{\beta}, z^{\gamma})]_{d-\alpha-2} > \dim_K [R/(x^{a-\alpha}, y^{\beta}, z^{\gamma})]_{d-\alpha-1},
    \]
    and $T^l = T_d(x^{\alpha}, y^b, z^c)$ is $\uptri$-heavy if and only if
    \[
        \dim_K [R/(x^{\alpha}, y^b, z^c)]_{d-2} < \dim_K [R/(x^{\alpha}, y^b, z^c)]_{d-1}.
    \]
    Using Lemma~\ref{lem:h-ci}, a straight-forward computation shows that these two inequalities are both true if and
    only of $d$ satisfies Condition~\ref{eqn:type-two}.
\end{proof}

\begin{remark}\label{rem:type2}
    The above argument establishes the following more precise version of Theorem~\ref{thm:type-two}:

    Let $R/I$ be a Artinian monomial algebra of type 2, where $K$ is a field of characteristic zero, and let $\ell \in
    R$ be a general linear form. Then the multiplication map $\times \ell: [R/I]_{d-2} \to [R/I]_{d-1}$ does not have
    maximal rank if and only if $I = (x^a, y^b, z^c, x^{\alpha} y^{\beta}, x^{\alpha} z^{\gamma})$, up to a change of
    variables, and $d$ satisfies Condition~\ref{eqn:type-two}.
\end{remark}

Condition~\ref{eqn:type-two} in Theorem~\ref{thm:type-two} is indeed non-vacuous.

\begin{example}
    We provide three examples, the latter two come from \cite[Example~6.10]{BMMNZ}, with various shapes of Hilbert functions.
    \begin{enumerate}
        \item Let $I = (x^4, y^4, z^4, x^3y, x^3z)$. Then $d = 5$ satisfies Condition~\ref{eqn:type-two}.
            Moreover, $T_d(I)$ is a balanced region, and $R/I$ has a strictly unimodal Hilbert function, \\
            $(1,3,6,10,10,9,6,3,1)$.
        \item Let $J = (x^3, y^7, z^7, xy^2, xz^2)$. Then Condition~\ref{eqn:type-two} is satisfied if and only if $d=5$ or  $d = 6$.
            Note that $R/J$ has a non-unimodal Hilbert function, \\
            $(1,3,6,7,6,6,7,6,5,4,3,2,1)$.
        \item Let $J' = (x^2, y^4, z^4, xy, xz)$. Then $d = 3$ satisfies Condition~\ref{eqn:type-two}.
            Moreover, $J'$ has a \emph{non-strict} unimodal Hilbert function $(1,3,3,4,3,2,1)$.
    \end{enumerate}
\end{example}

Using Theorem \ref{thm:type-two}, we easily recover \cite[Theorem~6.2]{BMMNZ}, one of the main results in the recent
memoir~\cite{BMMNZ}.

\begin{corollary} \label{cor:type-two-level}
    Let $R/I$ be a Artinian monomial algebra of type 2 over a field of characteristic zero. Then $R/I$ has the weak
    Lefschetz property.
\end{corollary}
\begin{proof}
    By Theorem~\ref{thm:type-two}, we know that if $I$ has four generators, then $R/I$ has the weak Lefschetz property.
    If $I$ has five generators, then it suffices to show that Condition~\ref{eqn:type-two} is vacuous in this case.
    Indeed, since $R/I$ is level, we have that $a - \alpha = b - \beta + c - \gamma$ by
    Lemma~\ref{lem:classify-type-two}. This implies
    \[
        \frac{a+\alpha+\beta+\gamma}{2} = \frac{2 \alpha+b+c}{2} \geq \alpha + \min\{b, c\}.
    \]
    Hence, no integer  $d$ satisfies Condition~\ref{eqn:type-two}.
\end{proof}

Moreover, in most of the cases when the weak Lefschetz property holds in characteristic zero, we can give a linear lower
bound on the characteristics for which the weak Lefschetz property must hold. Note that while the proof of
Theorem~\ref{thm:type-two} relies on properties of the bi-adjacency matrix $Z(T)$, the following argument also uses
Proposition~\ref{pro:ci-wlp}, which is based on the evaluation of a certain lattice path matrix.

\begin{corollary}\label{cor:type-two-pos-char}
    Let $R/I$ be a Artinian monomial algebra of type 2. Suppose that $R/I$ has the weak Lefschetz property in
    characteristic zero and that there is no integer $d$ such that
    \begin{equation}\label{eq:Cond-on-d}
        \max\left\{ \alpha, b, c, \frac{\alpha + b + c}{2} \right\}
        < d <
        \min\left\{ a + \beta, a + \gamma, \alpha + \beta + c, \frac{a + \alpha + \beta + c}{2} \right\}.
    \end{equation}
    Then $R/I$ has the weak Lefschetz property, provided $K$ has  characteristic $p \geq \frac{1}{2}(a+b+c)$.
\end{corollary}
\begin{proof}
    We use the notation introduced in the proof of Theorem~\ref{thm:type-two}. Fix any integer $d \ge2$. Recall that,
    possibly after reordering rows and columns, the bi-adjacency matrix of $T = T_d (I)$ has the form (see
    Equation~\eqref{eq:block-Z})
    \[
        Z =
        \begin{pmatrix}
            Z(T^u) & 0 \\
            Y & Z(T^l)
        \end{pmatrix}.
    \]
    By assumption, $d$ does not satisfy Condition \eqref{eqn:type-two} nor \eqref{eq:Cond-on-d}. This implies that $T$
    has one of the following properties: (i) $T^u$ or $T^l$ is balanced, (ii) $T^u$ and $T^l$ are both $\uptri$-heavy,
    or (iii) $T^u$ and $T^l$ are both $\dntri$-heavy.

    The matrices $Z(T^u)$ and $Z(T^l)$ have maximal rank by Proposition~\ref{pro:ci-wlp} if the characteristic of $K$ is
    at least $\flfr{a - \alpha + \beta + c}{2}$ and $\flfr{\alpha + b + c}{2}$, respectively. Combining non-vanishing
    maximal minors of $Z(T^u)$ and $Z(T^l)$, if follows that the matrix $Z$ has maximal rank as well if $\charf K \geq \frac{1}{2}(a+b+c)$.
\end{proof}

In order to fully extend Theorem~\ref{thm:type-two} to sufficiently large positive characteristics, it remains to consider the
case where $T^u$ is $\uptri$-heavy and $T^l$ is $\dntri$-heavy. This is more delicate.

\begin{example}
    Let $T = T_{10}(x^8, y^8, z^8, x^3 y^5, x^3 z^6)$ as in Figure~\ref{fig:type-2-case-8}, and let $T'$ be the maximal
    minor given in Figure~\ref{fig:type-2-case-8}(i). In each lozenge tiling of $T'$, there is exactly one lozenge that
    crosses the splitting line. There are four possible locations for this lozenge; one of these is illustrated in
    Figure~\ref{fig:type-2-case-8}(ii). The enumeration of lozenge tilings of $T'$ is thus the sum of the lozenge
    tilings with the lozenge in each of the four places along the splitting line. Each of the summands is the product of
    the enumerations of the resulting upper and lower regions. In particular, we have that
    \begin{equation*}
        \begin{split}
            |\det{N(T')}| &= 20 \cdot 60 + 45 \cdot 64 + 60 \cdot 60 + 50 \cdot 48 \\
                          &= 2^4 \cdot 3 \cdot 5^2 + 2^6 \cdot 3^2 \cdot 5 + 2^4 \cdot 3^2 \cdot 5^2 + 2^5 \cdot 3 \cdot 5^2 \\
                          &= 2^5 \cdot 3^2 \cdot 5 \cdot 7 \\
                          &= 10080.
        \end{split}
    \end{equation*}
    Notice that while the four summands only have prime factors of $2$, $3$, and $5$, the final enumeration also has a
    prime factor of $7$.
\end{example}

Still, we can give a bound in this case, though we expect that it is very conservative. It provides the following
extension of Theorem~\ref{thm:type-two}.

\begin{proposition}\label{prop:type2-pos-char}
    Let $R/I$ be a Artinian monomial algebra of type 2 such that $R/I$ has the weak Lefschetz property in characteristic
    zero. Then $R/I$ has the weak Lefschetz property in positive characteristic, provided
    $\charf K \geq 3^e$, where $e = \frac{1}{2}\binom{\frac{1}{2} (a+b+c) + 2}{2}$.
\end{proposition}

This follows from Lemma~\ref{lem:classify-type-two} and the following general result, which provides an effective bound
for Corollary~\ref{lem:wlp-0-p} in the case of three variables.

\begin{proposition}\label{prop:char-0-to-p}
    Let $R/I$ be any Artinian monomial algebra such that $R/I$ has the weak Lefschetz property in characteristic zero.
    If $I$ contains the powers $x^a, y^b, z^c$, then $R/I$ has the weak Lefschetz property in positive characteristic
    whenever $\charf K > 3^{\frac{1}{2}\binom{\frac{1}{2} (a+b+c) + 2}{2}}$.
\end{proposition}
\begin{proof}
    Define $I' = (x^a, y^b, z^c)$, and let $d'$ be the smallest integer such that $0 \neq h_{R/I'}(d'-1) \geq h_{R/I'}(d')$.
    Thus, $d' - 1 \leq \frac{1}{2}(a+b+c)$ by Lemma~\ref{lem:h-ci}.

    Let $d$ be the smallest integer such that $0 \neq h_{R/I}(d-1) \geq h_{R/I}(d)$. Then $d \leq d'$, as $I' \subset I$
    and adding or enlarging punctures only exacerbates the difference in the number of upward- and downward-pointing
    triangles. Since $R/I$ has the weak Lefschetz property in characteristic zero, the Hilbert function of $R/I$ is
    strictly increasing up to degree $d-1$. Hence, Proposition~\ref{pro:wlp} implies that the degrees of non-trivial
    socle elements of $R/I$ are at least $d-1$. The socle of $R/I$ is independent of the characteristic of $K$.
    Therefore Proposition~\ref{pro:wlp} shows that, in any characteristic, $R/I$ has the weak Lefschetz property if and
    only if the bi-adjacency matrices of $T_d (I)$ and $T_{d+1} (I)$ have maximal rank. Each row and column of a
    bi-adjacency matrix has at most three entries that equal one. All other entries are zero. Moreover the maximal
    square submatrices of $Z(T_d (I))$ and $Z(T_{d+1} (I))$ have at most $h_{R/I} (d-1)$ rows. Since
    $h_{R/I} (d-1) < h_R (d-1) = \binom{d+1}{2} \leq 3e$, Hadamard's inequality shows
    that the absolute values of the maximal minors of $Z(T_d (I))$ and $Z(T_{d+1} (I))$, considered as integers, are
    less than $3^{2e}$. Hence, any prime number $p \geq 3^e$ does not divide any of these non-trivial maximal minors.
\end{proof}

As indicated above, we believe that the bound in Proposition~\ref{prop:type2-pos-char} is far from being optimal.
Through a great deal of computer experimentation, we offer the following conjecture.

\begin{conjecture} \label{con:type-two-pos-char}
    Let $I$ be an Artinian monomial ideal in $R = K[x,y,z]$ such that $R/I$ is of type two. If $R/I$ has the weak
    Lefschetz property in characteristic zero, then $R/I$ also has the weak Lefschetz property in characteristics
    $p > \frac{1}{2}(a+b+c)$.
\end{conjecture}

\section{Artinian monomial almost complete intersections} \label{sec:amaci}

We discuss another generalisation of monomial complete intersections. The latter have three minimal generators. This section
presents an in-depth discussion of the Artinian monomial ideals with exactly four minimal generators. They are called
Artinian monomial almost complete intersections. These ideals have been discussed, for example, in \cite{BK}
and~\cite[Section~6]{MMN-2011}. In particular, we will answer some of the questions posed in \cite{MMN-2011}.
Some of our results are used in \cite{BMMMNW} for studying ideals with the Rees property.

Each Artinian ideal of $K[x,y,z]$ with exactly four monomial minimal generators is of the form
\[
    I_{a,b,c,\alpha,\beta,\gamma} = (x^a, y^b, z^c, x^\alpha y^\beta z^\gamma),
\]
where $0 \leq \alpha < a$, $0 \leq \beta < b,$ and $0 \leq \gamma < c$, such that at most one of $\alpha$, $\beta$, and
$\gamma$ is zero. If one of $\alpha$, $\beta$, and $\gamma$ is zero, then $R/I_{a,b,c,\alpha,\beta,\gamma}$ has type
two. In this case, the presence of the weak Lefschetz property has already been described in Section~\ref{sec:type-two},
see in particular, Theorem~\ref{thm:type-two} and Proposition~\ref{prop:char-0-to-p}. Thus, throughout this section we
assume that the integers $\alpha, \beta$, and $\gamma$ are all positive.

~\subsection{Presence of the weak Lefschetz property}~\par\label{subsec:aci-wlp}

We begin by recalling a few results. The first one shows that $R/I_{a,b,c,\alpha,\beta,\gamma}$ has type three. More precisely:

\begin{proposition}{\cite[Proposition 6.1]{MMN-2011}} \label{pro:amaci-props}
    Let $I = I_{a,b,c,\alpha,\beta,\gamma}$ be defined as above. Then $R/I$ has three minimal socle generators. They
    have degrees $\alpha + b + c - 3$, $a + \beta + c - 3$, and $a + b + \gamma - 3$.

    In particular, $R/I$ is level if and only if $a - \alpha = b - \beta = c - \gamma$.
\end{proposition}

Brenner classified when the syzygy bundle of $I_{a,b,c,\alpha,\beta,\gamma}$ is semistable.

\begin{proposition}{\cite[Corollary~7.3]{Br}} \label{pro:amaci-semistable}
    Let $I = I_{a,b,c,\alpha,\beta,\gamma}$ be defined as above, and suppose $K$ is a field of characteristic zero. Set
    $d = \frac{1}{3}(a+b+c+\alpha + \beta + \gamma)$. Then $I$ has a semistable syzygy bundle if and only if the
    following three conditions are satisfied:
    \begin{enumerate}
        \item $\max\{a, b, c, \alpha + \beta + \gamma\} \leq d$;
        \item $\min\{\alpha + \beta + c, \alpha + b + \gamma, a + \beta + \gamma\} \geq d$; and
        \item $\min\{a+b, a+c, b+c\} \geq d$.
    \end{enumerate}
\end{proposition}

Furthermore, Brenner and Kaid showed that, for almost complete intersections, nonsemistability implies the weak
Lefschetz property in characteristic zero.

\begin{proposition}{\cite[Corollary~3.3]{BK}} \label{pro:amaci-nss-wlp}
    Let $K$ be a field of characteristic zero. Then $I_{a,b,c,\alpha,\beta,\gamma}$ has the weak Lefschetz property if
    its syzygy bundle is not semistable.
\end{proposition}

The conclusion of this result is not necessarily true in positive characteristic.

\begin{example} \label{exa:amaci-nss}
    Let $I = I_{5,5,3,1,1,2}$, and thus $d = 6$. Then the syzygy bundle of $I$ is not semistable as
    $\alpha + \beta + c = 5 < d = 6$. However, the triangular region $T_6(I)$ is balanced and $\det{Z(T_6(I))} = 5$.
    Hence, $I$ does not have the weak Lefschetz property if and only if the characteristic of $K$ is $5$.
\end{example}

The following example illustrates that the assumption on the number of minimal generators cannot be dropped in
Proposition~\ref{pro:amaci-semistable}.

\begin{example} \label{exa:amaci-nss-2}
    Consider the ideal $J = (x^5, y^5, z^5, xy^2z, xyz^2)$ with five minimal generators. Then
    Corollary~\ref{cor:T-stable} gives that the syzygy bundle of $J$ is not semistable. Notice that $T_6(J)$ is
    balanced. However, $\det{Z(T_6(J))} = 0$, and so $R/J$ never has the weak Lefschetz property, regardless of the
    characteristic of $K$.
\end{example}

The number $d$ in Proposition~\ref{pro:amaci-semistable} is not assumed to be an integer. In fact, if it is not, then
the algebra has the weak Lefschetz property.

\begin{proposition}{\cite[Theorem~6.2]{MMN-2011}} \label{pro:amaci-not-3}
    Let $K$ be a field of characteristic zero. Then $I_{a,b,c,\alpha,\beta,\gamma}$ has the weak Lefschetz property if
    $a+b+c+\alpha + \beta + \gamma \not\equiv 0 \pmod{3}$.
\end{proposition}

Again, the conclusion of this result may fail in positive characteristic. Indeed, for the ideal $I_{5,5,3,1,1,2}$ in
Example~\ref{exa:amaci-nss} we get $d = \frac{17}{3}$, but it does not have the weak Lefschetz property in
characteristic $5$.

The following result addresses the weak Lefschetz property in the cases that are left out by
Propositions~\ref{pro:amaci-nss-wlp} and~\ref{pro:amaci-not-3}. Its first part extends \cite[Lemma~7.1]{MMN-2011} from
level to arbitrary monomial almost complete intersections. Observe that balanced triangular regions correspond to an
equality of the Hilbert function in two consecutive degrees, dubbed ``twin-peaks'' in \cite{MMN-2011}.

\begin{proposition} \label{pro:amaci-balanced}
    Let $I = I_{a,b,c,\alpha,\beta,\gamma}$, and assume $d = \frac{1}{3}(a+b+c+\alpha+\beta+\gamma)$ is an integer. If
    the syzygy bundle of $I$ is semistable and $d$ is integer, then $T_d(I)$ is perfectly-punctured and balanced.

    Moreover, in this case $R/I$ has the weak Lefschetz property if and only if $\det{Z(T_d(I))}$ is not zero in $K$.
\end{proposition}
\begin{proof}
    Note that condition (i) in Proposition~\ref{pro:amaci-semistable} says that $T_d(I)$ has punctures of nonnegative
    side lengths $d-a, d-b, d-c$, and $d-(\alpha + \beta + \gamma)$. Furthermore, conditions (ii) and (iii) therein are
    equivalent to the fact that the degree of the least common multiple of any two of the minimal generators of $I$ is
    at least $d$. Hence, Lemma~\ref{lem:lcm-overlap} gives that the punctures of $T_d (I)$ do not overlap. Using the
    assumption that $d$ is an integer, it follows that $T_d(I)$ is perfectly-punctured, and thus balanced.

    Since the punctures of $T_d(I)$ do not overlap, the punctures of $T_{d-1}(I)$ are not overlapping nor touching.
    Using Corollary~\ref{cor:socle-degree-bound}, we conclude that the degrees of the socle generators of $R/I$ are at
    least $d-2$. Hence, Corollary~\ref{cor:wlp-Z} gives that $R/I$ has the weak Lefschetz property if and only if
    $\det{Z(T_d(I))}$ is not zero in $K$.
\end{proof}


In the situation of Proposition~\ref{pro:amaci-balanced}, the fact that $R/I$ has the weak Lefschetz property implies
that $T_d(I)$ is tileable by Theorem~\ref{thm:pm-matrix}. This combinatorial property remains true even if $R/I$ fails
to have the weak Lefschetz property.
\begin{proposition} \label{pro:amaci-ss-tileable}
    Let $I = I_{a,b,c,\alpha,\beta,\gamma}$. If $R/I$ fails to have the weak Lefschetz property in characteristic zero,
    then $d = \frac{1}{3}(a+b+c+\alpha+\beta+\gamma)$ is an integer and $T_d(I)$ is tileable.
\end{proposition}
\begin{proof}
    By Propositions~\ref{pro:amaci-nss-wlp} and~\ref{pro:amaci-not-3}, we know that the syzygy bundle of $I$ is
    semistable and $d = \frac{1}{3}(a+b+c+\alpha+\beta+\gamma)$ is an integer. Hence by Proposition~\ref{pro:amaci-balanced},
    $T_d(I)$ is perfectly-punctured. Now we conclude by Theorem~\ref{thm:tileable-semistable}.
\end{proof}

Specialising results in Section~\ref{sec:det} and \ref{sec:mirror}, we can decide the presence of the weak Lefschetz property in almost all cases.

\begin{theorem} \label{thm:amaci-wlp}
    Let $I = I_{a,b,c,\alpha,\beta,\gamma} = (x^a, y^b, z^c, x^\alpha y^\beta z^\gamma)$ be an Artinian ideal with four
    minimal generators such that $\alpha$, $\beta$, and $\gamma$ are all positive. Assume the base field $K$ has
    characteristic zero, and consider the following conditions:
    \begin{enumerate}
        \item $\max\{a, b, c, \alpha + \beta + \gamma\} \leq d$;
        \item $\min\{\alpha + \beta + c, \alpha + b + \gamma, a + \beta + \gamma\} \geq d$;
        \item $\min\{a+b, a+c, b+c\} \geq d$; and
        \item $d = \frac{1}{3}(a+b+c+\alpha + \beta + \gamma)$ is an integer.
    \end{enumerate}
    Then the following statements hold:
        \begin{itemize}
            \item[(a)] If one of the conditions (i) - (iv) is not satisfied, then $R/I$ has the weak Lefschetz property.
            \item[(b)] Assume all the conditions (i) - (iv) are satisfied. Then:
            \begin{itemize}
                \item[(1)] The multiplication map $\times (x+y+z): [R/I]_{j-2} \to [R/I]_{j-1}$ has maximal rank whenever  $j \neq d$.
                \item[(2)] The algebra $R/I$ has the weak Lefschetz property if one of the following conditions is satisfied:
                \begin{itemize}
                    \item[(I)] Condition (ii) is an equality.
                    \item[(II)] $a+b+c+\alpha+\beta + \gamma$ is divisible by 6.
                    \item[(III)] $c = \frac{1}{2}(a+b+\alpha+\beta+\gamma)$.
                    \item[(IV)] The region $T_d(I)$ has an axes-central puncture (see Subsection~\ref{sub:axes-central})
                        and one of $d-a, d-b, d-c$, and $d-(\alpha+\beta+\gamma)$ is not odd.
                    \item[(V)] $a = b$, $\alpha = \beta$, and $c$ or $\gamma$ is even.
                \end{itemize}
                \item[(3)] The algebra $R/I$ fails to have the weak Lefschetz property if one of the following
                conditions is satisfied:
                \begin{itemize}
                    \item[(IV')] The region $T_d(I)$ has an axes-central puncture (see Subsection~\ref{sub:axes-central})
                        and all of $d-a, d-b, d-c$, and $d-(\alpha+\beta+\gamma)$ are odd; or
                    \item[(V')] $a = b$, $\alpha = \beta$, and both $c$ and $\gamma$ are odd.
                \end{itemize}
            \end{itemize}
    \end{itemize}
\end{theorem}
\begin{proof}
    Assertion (a) follows from Propositions~\ref{pro:amaci-semistable}, \ref{pro:amaci-nss-wlp}, and~\ref{pro:amaci-not-3}.

    Consider now the claims in part (b). Then Proposition~\ref{pro:amaci-balanced} gives that $R/I$ has the weak
    Lefschetz property if and only if $\det Z(T_d(I))$ is not zero.

    The assumptions in (b) guarantee that the punctures of $T = T_d (I)$ do not overlap and the degrees of the socle
    generators of $R/I$ are at least $d-2$. Then condition (I) implies that the puncture to the generator $x^\alpha y^\beta z^\gamma$
    touches another puncture, whereas condition (II) says that this puncture has an even side length. In either case,
    $R/I$ has the weak Lefschetz property by Proposition~\ref{prop:same-sign}.

    The proof of (b)(1) uses the Grauert-M\"ulich splitting theorem. We complete this part below
    Proposition~\ref{prop:splitt-type-semist}.

    The remaining assertions all follow from a result in Section~\ref{sec:det} or \ref{sec:mirror}, when combined with
    Proposition~\ref{pro:amaci-balanced}:

    (III). The condition $c = \frac{1}{2}(a+b+\alpha+\beta+\gamma)$ is equivalent to $d-c = 0$. Thus
    Proposition~\ref{pro:C-is-zero} gives the claim.

    (IV) and (IV').   Use Corollary~\ref{cor:axes-central-Z}.

    (V) and (V'). Use Proposition~\ref{pro:ciucu-prime-bounds} and Theorem~\ref{thm:mirror-odd23}.
\end{proof}

Notice that Theorem~\ref{thm:amaci-wlp}(b)(1) says that, for almost monomial complete intersections, the multiplication map can fail to
have maximal rank in at most one degree.


\begin{remark}\label{rem:q-and-a}
    \begin{enumerate}
        \item Theorem~\ref{thm:amaci-wlp} can be extended to fields of sufficiently positive characteristic by using
            Proposition~\ref{prop:char-0-to-p}. This lower bound on the characteristic can be improved whenever we know the
            determinant of $Z(T_d(I))$ from a result in Section~\ref{sec:det} or \ref{sec:mirror}. We leave the details to the
            reader.
        \item Question~8.2(2c) in \cite{MMN-2011} asked if there exist non-level almost complete intersections which never
            have the weak Lefschetz property. The almost complete intersection $I = I_{3,5,5,1,2,2} = (x^3, y^5, z^5, xy^2z^2)$
            is not level and never has the weak Lefschetz property, regardless of field characteristic, as
            $\det{Z(T_6(I))} = 0$ by Theorem~\ref{thm:mirror-odd23}.
    \end{enumerate}
\end{remark}~

\subsection{Level almost complete intersections}\label{sub:level}~\par

In Subsection~\ref{sub:axes-central}, we considered one way of centralising the inner puncture of a triangular region
associated to a monomial almost complete intersection. We called such punctures ``axes-central.'' In this section, we
consider another method of centralising the inner puncture of such a triangular region. It turns out this method of
centralisation is equivalent to the algebra being level.

Consider the ideal $I = I_{a,b,c,\alpha,\beta,\gamma}$ as above. Let $d$ be an integer and assume that $T = T_d(I)$ has
one floating puncture. We say the inner puncture of $T$ is a \emph{gravity-central puncture}%
\index{puncture!gravity-central}
if the vertices of the puncture are each the same distance from the puncture opposite to it (see Figure~\ref{fig:gravity-central}).

\begin{figure}[!ht]
    \includegraphics[scale=1]{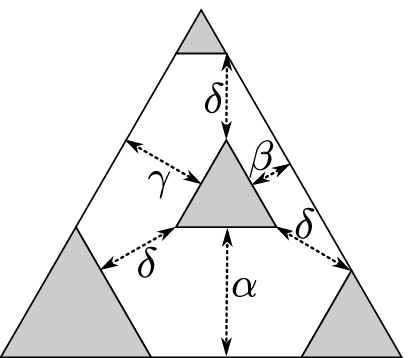}
    \caption{A prototypical figure with a gravity-central puncture.}
    \label{fig:gravity-central}
\end{figure}

\begin{lemma}
    Let $I = I_{a,b,c,\alpha,\beta,\gamma}$ such that $T_d(I)$ has a gravity-central puncture. Then $R/I$ is a level
    type $3$ algebra.
\end{lemma}
\begin{proof}
    The defining property for the distances is $(d-b) + (d-c) - \alpha = (d-a) + (d-c) - \beta = (d-a) + (d-b) - \gamma$.
    This is equivalent to the condition in Proposition~\ref{pro:amaci-props} that $R/I$ is level, i.e.,
    $a - \alpha = b - \beta = c - \gamma$.
\end{proof}

Since having a gravity-central puncture has an algebraic interpretation, it is natural to wonder if this is also true
for the existence of an axes-central puncture.

\begin{question}
    Let $I = I_{a,b,c,\alpha,\beta,\gamma}$. Does the existence of an axes-central puncture in $T_d(I)$ admit an
    algebraic characterisation?
\end{question}

Level almost complete intersections were studied extensively in~\cite[Sections~6 and~7]{MMN-2011}. In particular,
Migliore, Mir\'o-Roig, and the second author proposed a conjectured characterisation for the presence of the weak
Lefschetz property for such algebras. We recall it here, though we present it in a different, but equivalent, form to
better elucidate the reasoning behind it.

\begin{conjecture}{\cite[Conjecture~6.8]{MMN-2011}} \label{conj:level-wlp}
    Let $I = I_{\alpha+t,\beta+t, \gamma+t, \alpha,\beta,\gamma}$ be an ideal of $R = K[x,y,z]$, where $K$ has
    characteristic zero, $0 < \alpha \leq \beta \leq \gamma \leq 2(\alpha+\beta)$, $t \geq \frac{1}{3}(\alpha+\beta+\gamma)$,
    and $\alpha + \beta + \gamma$ is divisible by three. If $(\alpha,\beta,\gamma,t)$ is not $(2,9,13,9)$ or
    $(3,7,14,9)$, then $R/I$ fails to have the weak Lefschetz property if and only if $t$ is even, $\alpha + \beta + \gamma$
    is odd, and $\alpha = \beta$ or $\beta = \gamma$. Furthermore, $R/I$ fails to have the weak Lefschetz property in
    the two exceptional cases.
\end{conjecture}

The necessity part of this conjecture was proven in \cite[Corollary~7.4]{MMN-2011}) by showing that $R/I$ does not have
the weak Lefschetz property if $t$ is even, $\alpha + \beta + \gamma$ is odd, and $\alpha = \beta$ or $\beta = \gamma$.
This result is covered by Theorem~\ref{thm:amaci-wlp}(b)(3)(V') because the region is mirror symmetric. It remained open
to establish the presence of the weak Lefschetz property. Theorem~\ref{thm:amaci-wlp} does this in many new cases.

\begin{proposition} \label{pro:level-wlp}
    Consider the ideal $I = I_{\alpha+t,\beta+t, \gamma+t, \alpha,\beta,\gamma}$ as given in
    Conjecture~\ref{conj:level-wlp}. Then $R/I$ has the weak Lefschetz property if one of the following
    conditions is satisfied:
    \begin{enumerate}
        \item $t$ and $\alpha + \beta + \gamma$ have the same parity; or
        \item $t$ is odd and $\alpha = \beta = \gamma$ is even.
    \end{enumerate}
\end{proposition}
\begin{proof}
    We apply Theorem~\ref{thm:amaci-wlp} with $d = t + \frac{2}{3} (\alpha + \beta + \gamma)$. Then the side length of
    the inner puncture of $T_d(I)$ is $t - \frac{1}{3} (\alpha + \beta + \gamma)$. Hence (i) follows from
    Theorem~\ref{thm:amaci-wlp}(b)(II). Claim (ii) is a consequence of Theorem~\ref{thm:amaci-wlp}(b)(IV) as the given
    condition implies the inner puncture is axes-central.
\end{proof}

\begin{remark}
    Conjecture~\ref{conj:level-wlp} remains open in two cases, both of which are conjectured to have the weak Lefschetz property:
    \begin{enumerate}
        \item $t$ even, $\alpha + \beta + \gamma$ is odd, and $\alpha < \beta < \gamma$; and
        \item $t$ odd, $\alpha + \beta + \gamma$ is even, and $\alpha < \beta$ or $\beta < \gamma$.
    \end{enumerate}
    Note that if true, then Conjecture~\ref{con:zero-mirror} implies part (ii) in the case, where $\alpha = \beta$ or
    $\beta = \gamma$.
\end{remark}

Notice that $T = T_d(I_{a,b,c,\alpha,\beta,\gamma})$ is simultaneously axis- and gravity-central precisely if either
$a = b = c$ and $\alpha = \beta = \gamma$, or $a = b+2 = c+1$ and $\alpha=\beta+2=\gamma+1$. In the former case, the weak
Lefschetz property in characteristic zero is completely characterised below, strengthening
\cite[Corollary~7.6]{MMN-2011}.

\begin{corollary}
    Let $I = I_{a, a, a, \alpha, \alpha, \alpha} = (x^a, y^a, z^a, x^{\alpha}, y^{\alpha}, z^{\alpha})$, where $a > \alpha$.
    Then $R/I$ fails to have the weak Lefschetz property in characteristic zero if and only if $\alpha$ and $a$ are odd
    and $a \geq 2 \alpha + 1$.
\end{corollary}
\begin{proof}
    If $a < 2 \alpha$, then $R/I$ has the weak Lefschetz property by Theorem~\ref{thm:amaci-wlp}(a).

    Assume now $a \geq 2 \alpha$. Then $R/I$ fails the weak Lefschetz property if $\alpha$ and $a$ are odd by
    \cite[Corollary~7.6]{MMN-2011} (or Theorem~\ref{thm:amaci-wlp}(b)(3)(V')). Otherwise, $R/I$ has this property by
    Proposition~\ref{pro:level-wlp}.
\end{proof}

For $a \geq 2 \alpha$, the triangular region $T_{a+\alpha} (I)$ was considered by Krattenthaler in \cite{Kr-06}. He
described a bijection between cyclically symmetric lozenge tilings of the region and descending plane partitions with
specific conditions.

~\subsection{Splitting type and regularity}~\par

The generic splitting type of a vector bundle on projective space is an important invariant. However, its computation is
often challenging. In this section we consider the splitting type of the syzygy bundles of monomial almost complete
intersections in $R$. These are rank three bundles on the projective plane. For the remainder of this section we assume
$K$ is an infinite field.

Let $I = I_{a,b,c,\alpha,\beta,\gamma}$ as above. Recall from Section~\ref{sec:syz} that the syzygy module $\syz{I}$ of
$I$ is defined by the exact sequence
\begin{equation*}
        0
    \longrightarrow
        \syz{I}
    \longrightarrow
        R(-\alpha-\beta-\gamma) \oplus R(-a) \oplus R(-b) \oplus R(-c)
    \longrightarrow
        I
    \longrightarrow
        0,
\end{equation*}
and the syzygy bundle $\widetilde\syz{I}$ on $\PP^2$ of $I$ is the sheafification of $\syz{I}$. Its restriction to any
line $H$ of $\PP^2$ splits as $\SO_H(p) \oplus \SO_H(q) \oplus \SO_H(r)$. The triple $(p, q, r)$ depends on the choice
of the line $H$, but is the same for all general lines. This latter triple is called the \emph{generic splitting type}%
\index{syzygy bundle!generic splitting type}
of $\widetilde\syz{I}$. Since $I$ is a monomial ideal, the arguments in Proposition~\ref{pro:mono} imply that the
generic splitting type $(p, q, r)$ can be determined if we restrict to the line defined by $\ell = x+ y + z$.

For computing the generic splitting type of $\widetilde\syz{I}$, we use the observation that $R/(I, \ell) \cong S/J$,
where $S = K[x,y]$, and $J = (x^a, y^b, (x+y)^c, x^\alpha y^\beta (x+y)^\gamma)$. Define an $S$-module $\syz{J}$ by the
exact sequence
\begin{equation} \label{eqn:syz-J}
        0
    \longrightarrow
        \syz{J}
    \longrightarrow
        S(-\alpha-\beta-\gamma) \oplus S(-a) \oplus S(-b) \oplus S(-c)
    \longrightarrow
        J
    \longrightarrow
        0
\end{equation}
using the, possibly non-minimal, set of generators $\{x^a, y^b, (x+y)^c, x^\alpha y^\beta (x+y)^\gamma\}$ of $J$. Then
$\syz{J} \cong S(p) \oplus S(q) \oplus S(r)$, where $(p, q, r)$ is the generic splitting type of the vector bundle
$\widetilde\syz{I}$. The Castelnuovo-Mumford regularity of the ideal $J$ is $\reg{J}= 1 + \reg S/J$.

For later use we record the following facts.

\begin{remark} \label{rem:splitting-type}
    Adopt the above notation. Then the following statements hold:
    \begin{enumerate}
        \item Using, for example,  the Sequence~(\ref{eqn:syz-J}), one gets  $-(p + q + r) = a+b+c+\alpha+\beta+\gamma$.
        \item If any of the generators of $J$ is extraneous, then the degree of that generator is one of $-p$, $-q$, or $-r$.
        \item As the regularity of $J$ is determined by the Betti numbers of $S/J$, we obtain that
            $\reg{J} + 1 = \max\{-p,-q,-r\}$ if the Sequence~(\ref{eqn:syz-J}) is a minimal free resolution of $J$.
    \end{enumerate}
\end{remark}

Before moving on, we prove a technical but useful lemma.

\begin{lemma} \label{lem:reg-2AMACI}
    Let $S = K[x,y]$, where $K$ is a field of characteristic zero. Consider the ideal $\fa = (x^a, y^b, x^\alpha y^\beta (x+y)^\gamma)$
    of $S$, and assume that the given generating set is minimal. Then $\reg{\fa}$ is
    \[
            -1 + \max \left \{a+ \beta, b+\alpha, \min \left \{a+b, a+ \beta + \gamma, b+ \alpha + \gamma,
                \left\lceil \frac{1}{2}(a+b+ \alpha + \beta + \gamma)\right\rceil  \right \} \right\}.
    \]
\end{lemma}
\begin{proof}
    We proceed in three steps.

    First, considering the minimal free resolution of the ideal $(x^a, y^b, x^\alpha y^\beta)$, we conclude
    \[
        \reg (x^a, y^b, x^\alpha y^\beta) = -1 + \max \{a+ \beta, b+\alpha\}.
    \]

    Second, the algebra $S/(x^a, y^b)$ has the strong Lefschetz property in characteristic zero (see, e.g.,
    \cite[Proposition~4.4]{HMNW}). Thus, the Hilbert function of $S/(x^a, y^b, (x+y)^\gamma)$ is
    \[
        \dim_K{[S/(x^a, y^b, (x+y)^\gamma)]_j} = \max\{0, \dim_K{[S/(x^a, y^b)]_j} - \dim_K{[S/(x^a,y^b)]_{j-\gamma}}\}.
    \]
    By analyzing when the difference becomes non-positive, we get that
    \begin{equation}\label{eq:reg-restr-ci}
        \reg (x^a, y^b, (x+y)^\gamma) = -1 + \min \left \{a+b, a+ \gamma, b+\gamma, \left\lceil \frac{1}{2}(a+b+\gamma)\right\rceil \right \}.
    \end{equation}

    Third, notice that
    \[
        (x^a, y^b, x^\alpha y^\beta (x+y)^\gamma):x^\alpha y^\beta = (x^{a-\alpha}, y^{b-\beta}, (x+y)^\gamma).
    \]
    Hence, multiplication by $x^\alpha y^\beta$ induces the short exact sequence
    \[
        0 \rightarrow [S/(x^{a-\alpha}, y^{b-\beta}, (x+y)^\gamma)](-\alpha-\beta) \stackrel{\times x^\alpha y^\beta}{\longrightarrow}
        S/\fa \rightarrow S/(x^a, y^b, x^\alpha y^\beta) \rightarrow 0.
    \]
    It implies
    \[
        \reg{\fa} = \max\{\alpha + \beta + \reg{(x^{a-\alpha}, y^{b-\beta}, (x+y)^\gamma)},
        \reg{(x^a, y^b, x^\alpha y^\beta)}\}.
    \]

    Using the first two steps, the claim follows.
\end{proof}

Recall that Proposition~\ref{pro:amaci-semistable} gives a characterisation of the semistability of the syzygy bundle
$\widetilde\syz{I_{a,b,c,\alpha,\beta,\gamma}}$, using only the parameters $a$, $b$, $c$, $\alpha$, $\beta$, and
$\gamma$. We determine the splitting type of $\widetilde\syz{I_{a,b,c,\alpha,\beta,\gamma}}$ for the nonsemistable and
the semistable cases separately.

~\subsubsection{Nonsemistable syzygy bundle}~

We first consider the case when the syzygy bundle is not semistable, and therein we distinguish four cases. It turns out
that in three cases, at least one of the generators of the ideal $J$ is extraneous.

\begin{proposition} \label{pro:st-nss}
    Consider the ideal $I = I_{a,b,c,\alpha,\beta,\gamma} = (x^a, y^b, z^c, x^\alpha y^\beta z^\gamma)$ with four
    minimal generators. Assume that the base field $K$ has characteristic zero and, without loss of generality, that
    $a \leq b \leq c$. Set $d := \frac{1}{3}(a+b+c+\alpha+\beta+\gamma)$, and denote by $(p, q, r)$ the generic splitting
    type of $\widetilde\syz{I}$. Assume that $\widetilde\syz{I}$ is not semistable. Then:
    \begin{enumerate}
        \item If $\min \{\alpha + \beta + \gamma, c\} \geq a+b -1$, then
            \[
                (p, q, r) = (-c, -\alpha - \beta - \gamma, -a-b).
            \]
        \item Assume $\min \{\alpha + \beta + \gamma, c\} \leq a+b -2$ and
            \[
                \frac{1}{2}(a+b+ c) \leq \min \left \{a+ \beta + \gamma, b+ \alpha + \gamma, c + \beta + \gamma,
                    \frac{1}{2}(a+b+ \alpha + \beta + \gamma) \right \}.
            \]
            Then
            \[
                (p, q, r) = (-\alpha-\beta-\gamma, - \left\lceil \frac{1}{2}(a+b+c) \right\rceil,
                    - \left\lfloor \frac{1}{2}(a+b+c) \right\rfloor).
            \]
        \item Assume $\min \{\alpha + \beta + \gamma, c\} \leq a+b -2$ and
            \[
                \frac{1}{2}(a+b+ \alpha + \beta + \gamma)  \leq \min \left \{a+ \beta + \gamma, b+ \alpha + \gamma,
                    c + \beta + \gamma, \frac{1}{2}(a+b+ c)  \right \}.
            \]
            Then
            \[
                (p, q, r) =  (-c, q, -a-b-\alpha-\beta-\gamma+q),
            \]
            where $- q = \min \left \{a+ \beta + \gamma, b+ \alpha + \gamma,
            \left\lceil \frac{1}{2}(a+b+ \alpha + \beta + \gamma)\right\rceil  \right \}$.
        \item Assume $\min \{\alpha + \beta + \gamma, c\} \leq a+b -2$ and
            \begin{equation*}
                \begin{split}
                    -r = \min \left \{a+ \beta + \gamma, b+ \alpha + \gamma, c + \beta + \gamma  \right \} < \hspace*{5cm} \\
                    \min \left \{ \frac{1}{2}(a+b+ \alpha + \beta + \gamma), \frac{1}{2}(a+b+ c) \right \}.
                \end{split}
            \end{equation*}
            Then
            \begin{equation*}
                (p, q, r) = \left (  \left\lfloor \frac{1}{2}(-3d-r) \right\rfloor,  \left\lceil \frac{1}{2}(-3 d - r) \right\rceil, r \right).
            \end{equation*}
    \end{enumerate}
\end{proposition}
\begin{proof}
    Set
    \begin{equation*}
        \mu = \min \left \{a+b, a+ \beta + \gamma, b+ \alpha + \gamma, c + \beta + \gamma,   \frac{1}{2}(a+b+ \alpha + \beta + \gamma), \frac{1}{2}(a+b+ c) \right \}.
    \end{equation*}

    Using $a \leq b \leq c$, \cite[Theorem 6.3]{Br} implies that the maximal slope of a subsheaf of $\widetilde\syz{I}$
    is $-\mu$. Since $\widetilde\syz{I}$ is not semistable, we have $\mu < d$ (see
    Proposition~\ref{pro:amaci-semistable}). Moreover, the generic splitting type of $\widetilde\syz{I}$ is determined
    by the minimal free resolution of $J = (x^a, y^b, (x+y)^c, x^\alpha y^\beta (x+y)^\gamma)$ as a module over $S = K[x,
    y]$. We combine both approaches to determine the generic splitting type.

    Since $\reg (x^a, y^b) = a+b-1$, all polynomials in $S$ whose degree is at least $a+b-1$ are contained in $(x^a,
    y^b)$. Hence, $J = (x^a, y^b)$ if $\min \{\alpha + \beta + \gamma, c\} \geq a+b -1$, and the claim in case (i)
    follows by Remark~\ref{rem:splitting-type}.

    For the remainder of the proof, assume $\min \{\alpha + \beta + \gamma, c\} \leq a+b -2$. Then $a+b > \frac{1}{2}(a+b+ c)$,
    and thus $\mu \neq a+b$.

    In case (ii), it follows that $\mu = \frac{1}{2}(a+b+ c)$ and $c \leq \alpha + \beta + \gamma$, and thus $c \leq a+b-2$.
    Using Equation \eqref{eq:reg-restr-ci}, we conclude that
    \[
        \reg (x^a, y^b, (x+y)^c) = -1 + \min \left \{a+b,  \left\lceil \frac{1}{2}(a+b+c)\right\rceil \right \} = -1 + \left\lceil \frac{1}{2}(a+b+c)\right\rceil.
    \]
    Observe now that $d > \mu = \frac{1}{2}(a+b+ c)$ is equivalent to $\alpha + \beta + \gamma > \frac{1}{2}(a+b+ c)$.
    This implies $\alpha + \beta + \gamma > \reg (x^a, y^b, (x+y)^c)$, and thus $J = (x^a, y^b, (x+y)^c)$. Using
    Remark~\ref{rem:splitting-type} again, we get the generic splitting type of $\widetilde\syz{I}$ as claimed in (ii).

    Consider now case (iii). Then $d > \mu = \frac{1}{2}(a+b+ \alpha + \beta + \gamma)$, which gives $c >
    \frac{1}{2}(a+b+ \alpha + \beta + \gamma)$. The second assumption in this case also implies $\frac{1}{2}(a+b+ \alpha
    + \beta + \gamma) \leq a + \beta + \gamma$, which is equivalent to $b+\alpha \leq a + \beta + \gamma$ and also to $b +
    \alpha \leq \frac{1}{2}(a+b+ \alpha + \beta + \gamma)$. Similarly, we have that $\frac{1}{2}(a+b+ \alpha + \beta +
    \gamma) \leq b+\alpha + \gamma$, which is equivalent to $a + \beta \leq b+ \alpha + \gamma$ and also to $a + \beta \le
    \frac{1}{2}(a+b+ \alpha + \beta + \gamma)$. It follows that
    \[
        \max \{a+\beta, b+ \alpha \} \leq  \min \left \{a+ \beta + \gamma, b+ \alpha + \gamma, \frac{1}{2}(a+b+ \alpha + \beta + \gamma)  \right \}.
    \]
    Hence Lemma~\ref{lem:reg-2AMACI} yields
    \begin{equation*}
        \begin{split}
            \reg (x^a, y^b,  x^\alpha y^\beta (x+y)^\gamma) =   \hspace*{9.7cm} \\
            -1 +   \min \left \{a+ \beta + \gamma, b+ \alpha + \gamma, \left\lceil \frac{1}{2}(a+b+ \alpha + \beta + \gamma)\right\rceil  \right \} < c.
        \end{split}
    \end{equation*}
    This shows that $(x+y)^c \in (x^a, y^b, x^\alpha y^\beta (x+y)^\gamma) = J$. Setting $- q = 1 + \reg J$,
    Remark~\ref{rem:splitting-type} provides the generic splitting type in case (iii).

    Finally consider case (iv). Then $\mu = -r$, and $\mu$ is equal to the degree of the least common multiple of two of
    the minimal generators of $I$. In fact, $-\mu = r$ is the slope of the syzygy bundle ${\mathcal O}_{\PP^2}(r)$ of
    the ideal generated by these two generators. Thus, the Harder-Narasimhan filtration (see
    \cite[Definition~1.3.2]{HM}) gives an exact sequence
    \[
        0 \to {\mathcal O}_{\PP^2}(r) \to \widetilde\syz{I} \to {\mathcal E} \to 0,
    \]
    where ${\mathcal E}$ is a semistable torsion-free sheaf on $\PP^2$ of rank two and first Chern class $-a-b-c -
    \alpha - \beta - \gamma -r = -3d -r$. Its bidual ${\mathcal E}^{**}$ is a stable vector bundle. Thus, by the theorem
    of Grauert and M\"ulich (see \cite{GM} or \cite[Corollary 1 of Theorem 2.1.4]{OSS}), its generic splitting type is
    $( \left\lfloor \frac{1}{2}(-3d-r) \right\rfloor, \left\lceil \frac{1}{2}(-3 d - r) \right\rceil)$. Now the claim
    follows by restricting the above sequence to a general line of $\PP^2$.
\end{proof}

We have seen that the ideal $J = (x^a, y^b, (x+y)^c, x^\alpha y^\beta (x+y)^\gamma)$ has at most three minimal
generators in the cases (i) - (iii) of the above proposition. In the fourth case, the associated ideal $J \subset S$ may
be minimally generated by four polynomials.

\begin{example} \label{exa:st-nss-4mingen}
    Consider the ideal
    \[
        I = I_{4,5,5,3,1,1} = (x^4, y^5, z^5, x^3yz).
    \]
    Then the corresponding ideal $J$ is minimally generated by $x^4, y^5, (x+y)^5$, and $x^3y(x+y)$. The syzygy bundle
    of $\widetilde\syz{I}$ is not semistable, and its generic splitting type is $(-7, -6, -6)$ by
    Proposition~\ref{pro:st-nss}(iv).
\end{example}~

\subsubsection{Semistable syzygy bundle}~

Order the entries of the generic splitting type $(p,q,r)$ of the semistable syzygy bundle $\widetilde\syz{I}$ such
that $p \leq q \leq r$. In this case, the splitting type determines the presence of the weak Lefschetz property (see
\cite[Theorem 2.2]{BK}). The following result is slightly more precise.

\begin{proposition}\label{prop:splitt-type-semist}
    Let $K$ be a field of characteristic zero, and assume the ideal $I = I_{a,b,c,\alpha,\beta,\gamma}$ has a semistable
    syzygy bundle. Set $k = \left\lfloor \frac{1}{3}(a+b+c+\alpha+\beta+\gamma) \right\rfloor$. Then the generic
    splitting type of $\widetilde\syz{I}$ is
    \begin{equation*}
        (p, q, r) =
        \begin{cases}
            (-k-1,-k,-k) & \text{if } a+b+c+\alpha+\beta+\gamma = 3k+1;\\
            (-k-1,-k-1,-k) & \text{if } a+b+c+\alpha+\beta+\gamma = 3k+2; \\
            (-k,-k,-k) & \text{if } a+b+c+\alpha+\beta+\gamma = 3k \text{ and} \\
                       & \text{$R/I$ has the weak Lefschetz property}; \\
            (-k-1,-k,-k+1) & \text{if } a+b+c+\alpha+\beta+\gamma = 3k \text{ and} \\
                       & \text{$R/I$ fails to have the weak Lefschetz property}.
        \end{cases}
    \end{equation*}
\end{proposition}
\begin{proof}
    The Grauert-M\"ulich theorem \cite{GM} gives that $r - q$ and $q - p$ are both nonnegative and at most 1. Moreover,
    $p, q$, and $r$ satisfy $a+b+c+\alpha+\beta+\gamma = -(p+q+r)$ (see Remark~\ref{rem:splitting-type}(i)). This gives
    the result if $k \neq d = \frac{1}{3}(a+b+c+\alpha+\beta+\gamma)$.

    It remains to consider the case when $k = d$. Then $(-k,-k,-k) $ and $(-k-1,-k,-k+1)$ are the only possible generic
    splitting types. By Proposition~\ref{pro:amaci-semistable}(i), the minimal generators of the ideal
    $J = (x^a, y^b, (x+y)^c, x^\alpha y^\beta (x+y)^\gamma)$ have degrees that are less than $d$. Hence $\reg J = d$ if and
    only if the splitting type of $\widetilde\syz{I}$ is $(-d-1, -d , -d+1)$. Since $\dim_K [R/I]_{d-2} = \dim_K [R/I]_{d-1}$,
    using Proposition~\ref{pro:amaci-balanced}, we conclude that $\reg J \geq d$ if and only if $R/I$ does not have the
    weak Lefschetz property.
\end{proof}

We are ready to add the missing piece in the proof of Theorem~\ref{thm:amaci-wlp}.

\begin{proof}[Completion of the proof of Theorem~\ref{thm:amaci-wlp}(b)(1)]
    \mbox{ }

    We have just seen that the ideal $J = (x^a, y^b, (x+y)^c, x^\alpha y^\beta (x+y)^\gamma)$ has regularity $d$ if
    $R/I$ fails the weak Lefschetz property. This implies that the multiplication map $\times (x+y+x): [R/I]_{j-2} \to [R/I]_{j-1}$
    is surjective whenever $j > d$. Moroever, since the minimal generators of $J$ have degrees that are less than $d$,
    the exact sequence
    \begin{equation*}
            0
        \longrightarrow
            S(-d+1) \oplus S(-d) \oplus S(-d-1)
        \longrightarrow
            S(-\alpha-\beta-\gamma) \oplus S(-a) \oplus S(-b) \oplus S(-c)
        \longrightarrow
            J
        \longrightarrow
            0
    \end{equation*}
    shows that $\dim_K [S/J]_{d-2} = 3$.

    In the above proof of Theorem~\ref{thm:amaci-wlp} we saw that the four punctures of $T_d (I)$ do not overlap and
    that $T_d(I)$ is balanced. Hence $T_{d-1} (I)$ has 3 more downward-pointing than upward-pointing triangles, that is,
    \[
        \dim_K [R/I]_{d-2} = \dim_K [R/I]_{d-3} + 3.
    \]
    It follows that the multiplication map in the exact sequence
    \[
        [R/I]_{d-3} \longrightarrow [R/I]_{d-2} \longrightarrow S/J \longrightarrow 0
    \]
    is injective. Since the degrees of the socle generators of $R/I$ are at least $d-2$, Corollary~\ref{cor:inj} gives
    that $\times (x+y+x): [R/I]_{j-2} \to [R/I]_{j-1}$ is injective whenever $j \leq d-1$.
\end{proof}

The second author would like to thank the authors of \cite{BMMMNW}; it was during a conversation in the preparation of that paper
that he learned about the use of the Grauert-M\"ulich theorem for deducing the injectivity of the map $[R/I]_{d-3} \longrightarrow [R/I]_{d-2}$
in the above argument.

\begin{example} \label{exa:syzygy}
    Consider the ideal $I_{7,7,7,3,3,3} = (x^7, y^7, z^7, x^3 y^3 z^3)$ which never has the weak Lefschetz property, by
    Theorem~\ref{thm:amaci-wlp}(vii). The generic splitting type of $\widetilde\syz{I_{7,7,7,3,3,3}}$ is $(-11, -10, -9)$.
    Notice that the similar ideal $I_{6,7,8,3,3,3} = (x^6, y^7, z^8, x^3 y^3 z^3)$ has the weak Lefschetz property
    in characteristic zero as $\det{N_{6,7,8,3,3,3}} = -1764$. The generic splitting type of
    $\widetilde\syz{I_{6,7,8,3,3,3}}$ is $(-10,-10, -10)$.
\end{example}

We summarise part of our results for the case where $I$ is associated to a tileable triangular region. Then the weak
Lefschetz property is equivalent to several other conditions.

\begin{theorem} \label{thm:equiv}
    Let $I = I_{a,b,c,\alpha,\beta,\gamma} \subset R = K[x,y,z]$, where $K$ is an infinite field of arbitrary
    characteristic. Assume $I$ satisfies conditions (i)--(iv) in Theorem~\ref{thm:amaci-wlp} and
    $d := \frac{1}{3}(a+b+c+\alpha+\beta+\gamma)$ is an integer.
    Then the following conditions are equivalent:
    \begin{enumerate}
        \item The algebra $R/I$ has the weak Lefschetz property.
        \item The determinant of $N(T_d(I))$ (i.e., the enumeration of signed lozenge tilings of $T_d(I)$)
            is not zero in $K$.
        \item The determinant of $Z(T_d(I))$ (i.e., the enumeration of signed perfect matchings of $G(T_d(I)$) is not
            zero in $K$.
        \item The generic splitting type of $\widetilde\syz{I}$ is $(-d,-d,-d)$.
    \end{enumerate}
\end{theorem}
\begin{proof}
    Regardless of the characteristic of $K$, the arguments for Proposition~\ref{pro:amaci-balanced} show that $T_d(I)$
    is balanced. Moreover, the degrees of the socle generators of $R/I$ are at least $d-2$ as shown in
    Theorem~\ref{thm:amaci-wlp}(b)(1). Hence, Corollary~\ref{cor:twin-peaks-wlp} gives that $R/I$ has the weak
    Lefschetz property if and only if the multiplication map
    \[
        \times (x+y+z): [R/I]_{d-2} \to [R/I]_{d-1}
    \]
    is bijective. Now, Corollary~\ref{cor:wlp-biadj} yields the equivalence of Conditions (i) and (ii). The latter is
    equivalent to condition (iii) by Theorem~\ref{thm:detZN}.

    As above, let $(p, q, r)$ be the generic splitting type of $\widetilde\syz{I}$, where $p \leq q \leq r$, and let
    $J \subset S$ be the ideal such that $R/(I, x+y+z) \cong S/J$. The above multiplication map is bijective if and only if
    $\reg J = d-1$. Since $\reg J + 1 = -r$ and $p+q+r = -3d$, it follows that $\reg J = d-1$ if and only if $(p, q, r)
    = (-d, -d, -d)$. Hence, conditions (i) and (iv) are equivalent.
\end{proof}

\section{Failure of the weak Lefschetz property} \label{sec:failure}

In this section we provide examples of Artinian monomial ideals that fail to have the weak Lefschetz property in
various ways. In particular, in Subsection~\ref{sub:Togliatti} we construct families of triangular regions (hence ideals)
where the rank of the bi-adjacency matrix $Z(T)$ can be made as far from maximal as desired. In
Subsection~\ref{sub:large-p} we give examples of triangular regions such that the determinant of $Z(T)$ has large prime
divisors, relative to the side length of $T$. That is, we offer examples of ideals that fail to have the weak Lefschetz
property in large prime characteristics relative to the degrees of the generators.

As preparation, in Subsection~\ref{sub:red-unit-punct} we prove that for checking maximal rank it is enough to consider
triangular regions with only unit punctures.

~\subsection{Reduction to unit punctures}~\par\label{sub:red-unit-punct}

\index{triangular region!reduction to unit punctures}
We show that each triangular region $T$ is contained in a triangular region $\hat{T}$ such that $\hat{T}$ only has only
unit punctures and $Z(T)$ has maximal rank if and only if $Z(\hat{T})$ has maximal rank. To see this we first replace a
large puncture (side length at least two) by two non-overlapping subpunctures, one of which is a unit puncture. We need
a partial extension of Proposition~\ref{pro:rep-enum}.

\begin{proposition}\label{prop:Z-subregion}
    Let $U$ be a balanced subregion of a triangular region $T$ such that no downward-pointing unit triangle in $U$ is
    adjacent to an upward-pointing unit triangle of $T \setminus U$. Then the following statements are true:
    \begin{itemize}
        \item[(i)] Possibly after reordering rows and columns of the bi-adjacency matrix of $T$, $Z(T)$ becomes a block matrix of the form
            \[
                \begin{pmatrix}
                    Z(T \setminus U) & X \\
                    0 & Z(U)
                \end{pmatrix}.
            \]
        \item[(ii)] If  $\det Z(U) \neq 0$ in $K$, then the bi-adjacency matrix $Z(T \setminus U)$ has maximal rank if and only if $Z(T)$ has maximal rank.
    \end{itemize}
\end{proposition}
\begin{proof}
    The second assertion follows from the first one. The first assertion is immediate from the definition of the
    bi-adjacency matrix $Z(T)$.
\end{proof}

\begin{remark}\label{rem:max-rank-uniques-tile}
    If $U$ is uniquely tileable, then Proposition \ref{pro:per-enum} and Theorem \ref{thm:pm-matrix} show that
    $\per Z(U) = | \det Z(U) | = 1$. Thus, $\det Z(U) \neq 0$, regardless of the characteristic of the base field $K$.
\end{remark}

We are ready to describe the inductive step of our reduction to the case of unit punctures.

\begin{lemma}\label{lem:unit-base}
    Let $T \subsetneq {\mathcal T}_d$ be a triangular region such that the sum of the side lengths of the non-unit
    punctures of $T$ is $m > 0$. Then there exists a triangular region $\hat{T} \subset {\mathcal T}_d$ containing $T$
    such that the following statements are true.
    \begin{enumerate}
        \item The sum of the side lengths of the non-unit punctures of $\hat{T}$ is at most $m-1$.
        \item $\#\uptri(\hat{T}) - \#\dntri(\hat{T}) = \#\uptri(T) - \#\dntri(T)$.
        \item The bi-adjacency matrix $Z(T)$ has maximal rank if and only if $Z(\hat{T})$ has maximal rank.
    \end{enumerate}
\end{lemma}
\begin{proof}
    Among the bottom rows of a puncture of $T$ whose side length is at least two, consider the row that is closest to
    the bottom of ${\mathcal T}_d$. In this row, pick a maximal \emph{strip $S$ of unit punctures}, that is, a sequence
    of adjacent upward- and downward-pointing triangles that all belong to punctures of $T$ such that the
    downward-pointing triangles that are possibly adjacent to the left and right of $S$ do not belong to a puncture of
    $T$. By the choice of the row, the strip $S$ contains at least one downward-pointing triangle. Let $P$ be any of the
    upward-pointing unit triangles in $S$. Denote by $U_1, U_2 \subset S$ the regions that are formed by the triangles
    in $S$ to the left and to the right of $P$, respectively. Set $\hat{T} = T \cup U_1 \cup U_2$.

    By construction, the downward-pointing triangles in $U_1 \cup U_2$ are not adjacent to any upward-pointing triangle
    in $T$. Furthermore, $U_1$ and $U_2$ are uniquely tileable. Thus, $1 = |\det Z(U_1)| = |\det Z(U_2)| $. Hence
    applying Proposition~\ref{prop:Z-subregion} twice, first to $U_1 \subset \hat{T}$ and then to
    $U_2 \subset \hat{T} \setminus U_1$, our assertions follow.
\end{proof}

Notice that instead of using a row in the above argument, one can also use a maximal strip that is parallel and closest
to the North-East or North-West boundary of ${\mathcal T}_d$. This follows from either the above arguments or a suitable
rotation of the region $T$ (compare Remark~\ref{rem:rotations}).

Repeating the procedure described in the preceding proof until the sum of the side lengths of the non-unit punctures is
zero yields a triangular region $\hat{T}$ containing $T$ with the following properties:
\begin{enumerate}
    \item $\hat{T}$ has only unit punctures;
    \item $\#\uptri(\hat{T}) - \#\dntri(\hat{T}) = \#\uptri(T) - \#\dntri(T)$; and
    \item $Z(T)$ has maximal rank if and only if $Z(\hat{T})$ has maximal rank.
\end{enumerate}
We call $\hat{T}$ a \emph{reduction of $T$ to unit punctures}.

\begin{example}\label{exa:pictures-reduction-unit}
    The triangular region $T = T_{8}(x^7, y^7, z^6, x y^4 z^2, x^3 y z^2, x^4 y z)$ and a reduction $\hat{T}$ of $T$ to
    unit punctures are depicted below. Notice that care must be taken when punctures overlap.

    \begin{figure}[!ht]
        \begin{minipage}[b]{0.48\linewidth}
            \centering
            \includegraphics[scale=1]{figs/triregion-gcd-1}\\
            \emph{(i) $T$}
        \end{minipage}
        \begin{minipage}[b]{0.48\linewidth}
            \centering
            \includegraphics[scale=1]{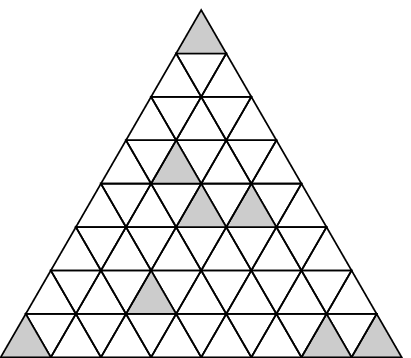}\\
            \emph{(ii) $\hat{T}$}
        \end{minipage}
        \caption{The region $T$ as in Figure~\ref{fig:triregion-gcd}(i) and a
             reduction $\hat{T}$   to unit punctures. }
          \label{fig:unit-reduc}
    \end{figure}
\end{example}

As pointed out above, Lemma~\ref{lem:unit-base} provides the following result.

\begin{proposition}\label{cor:unit-reduc}
    Let $T \subsetneq {\mathcal T}_d$ be a  triangular region.  Then $T$ has a reduction to unit punctures.

    Furthermore, $Z(T)$ has maximal rank if and only if $Z(\hat{T})$ has maximal rank for some (hence any) reduction
    $\hat{T}$ of $T$ to unit punctures.
\end{proposition}

In order to state this result algebraically and more precisely, we refine the definition of the weak Lefschetz property.

\begin{definition}\label{def:wlp-in-degree}
    A graded $K$-algebra $A$ is said to \emph{fail the weak Lefschetz property in degree $d-1$} by $\delta$ if, for a
    general linear element $\ell \in [A]_1$, the rank of the the multiplication map
    $\times \ell: [A]_{d-1} \rightarrow [A]_d$ is $r - \delta$, where $r = \min\{ h_A (d-1), h_A (d)\}$ is the expected rank.
\end{definition}

Thus, ``failing'' the weak Lefschetz property in degree $d-1$ by $0$ means that $A$ has the weak Lefschetz property in degree $d-1$.

Suppose $I$ is a monomial ideal of $R = K[x,y,z]$. Then $R/I$ (or $I$) fails the weak Lefschetz property in degree $d-1$
if and only if $Z(T_{d+1}(I))$ fails to have maximal rank. This follows by Proposition~\ref{pro:mono} and
Corollary~\ref{cor:wlp-biadj}. Hence Proposition~\ref{cor:unit-reduc} implies:

\begin{corollary}\label{cor:reduction-unit-punct-alg}
    Let $I \subset R$ be a monomial ideal such that $[I]_d \neq 0$, where $d \geq 1$. Then there is an Artinian ideal
    $J \subset I$ such that
    \begin{enumerate}
        \item The ideal $J$ has no generators of degree less than $d$, that is, $[J]_{d-1} = 0$.
        \item The minimal generators of $I$ and $J$ whose degrees are at least $d+1$ agree.
        \item If $R/I$ is Artinian, then so is $R/J$.
        \item $h_{R/I} (d) - h_{R/I} (d-1) = h_{R/J} (d) - h_{R/J} (d-1)$.
        \item $R/I$ fails the the weak Lefschetz property in degree $d-1$ by $\delta$  if and only if $R/J$ does.
    \end{enumerate}
\end{corollary}
\begin{proof}
    Let $\hat{T}$ be a reduction of $T = T_{d+1} (I)$ to unit punctures. Then there is a unique monomial ideal $J'
    \subset I$ such that $J'$ is generated in degree $d$ and $\hat{T} = T_{d+1} (J')$. Let $J \subset I$ be the monomial
    ideal that is generated by the monomials in $J'$ and the monomial minimal generators of $I$ whose degree is greater
    than $d$. Then $\hat{T} = T_{d+1} (J)$. Since $\hat{T}$ is a reduction of $T$ to unit punctures, possibly after
    reordering rows and columns of its bi-adjacency matrix, $Z(\hat{T})$ becomes a block matrix of the form
    \[
        \begin{pmatrix}
            Z(T) & X \\
            0 & Y
        \end{pmatrix},
    \]
    where $Y$ is a square matrix with $|\det Y| = 1$ (see Propositions~\ref{prop:Z-subregion} and \ref{cor:unit-reduc}).
    This implies all the assertions, except possibly (iv).

    In order to address the latter, we use the flexibility of the procedure for producing a reduction to unit punctures.
    In fact, assume $R/I$ is Artinian. Then $I$ has minimal generators of the form $x^a, y^b$, and $x^c$. If $a < d$,
    $b < d$, or $c < d$, then we can choose a reduction of $T$ to unit punctures such that the triangles labeled by
    $x^d$, $y^d$, and $z^d$, respectively, are punctures of $\hat{T} = T_{d+1} (J)$. Thus, $x^d \in J$, $y^d \in J$, or
    $z^d \in J$, respectively.
\end{proof}

Observe that the assumption $[I]_d \neq 0$ in the above results is harmless. If $[I]_d = 0$, then $R/I$ always has the
weak Lefschetz property in degree $d-1$.

~\subsection{Togliatti systems and Laplace equations}\label{sub:Togliatti}~\par

Mezzetti, Mir\'{o}-Roig, and Ottaviani showed in \cite{MMO} that in some cases the failure of the weak Lefschetz
property can be used to produce a variety satisfying a Laplace equation. Combined with our methods we produce toric
surfaces that satisfy as many Laplace equations as desired.

We begin by reviewing the needed concepts from differential geometry. Throughout this section we assume that $K$ is an
algebraically closed field of characteristic zero. Let $X \subset \PP^N= \PP^N_K$ be an $n$-dimensional projective
variety, and let $P \in X$ be a smooth point. Choose affine coordinates and a local parametrisation $\varphi$ around $P$,
where $\varphi (0,\ldots,0) = P$ and the $N$ components of $\varphi$ are formal power series. Then the \emph{$s$-th osculating space}%
\index{osculating space}
\index{0@\textbf{Symbol list}!TPsX@$T_P^{(s)} X$}
$T_P^{(s)} X$ to $X$ at $P$ is the projectivised span of the partial derivatives of $\varphi$ of order at most $s$.
Its expected dimension is $\binom{n+s}{s} - 1$. The variety $X$ is said to satisfy $\delta$ Laplace equations of order
$s$ if, for a general point $P$ of $X$,
\[
    \dim T_P^{(s)} X = \binom{n+s}{s} - 1 - \delta.
\]

\begin{remark}\label{rem:Laplace-trivially}
    If $N < \binom{n+s}{s} - 1$, then $X$ clearly satisfies at least one Laplace equation. However, this is not interesting.
\end{remark}

There is a rich literature on varieties satisfying a Laplace equation (see, e.g., \cite{To}, \cite{Perkinson},
\cite{MMO}, \cite{GIV} and the references therein). In \cite{MMO}, Mezzetti, Mir\'{o}-Roig, and Ottaviani found a new
approach to produce such varieties.

Let $I$ be an ideal of $S = K[x_0,\ldots,x_n]$ that is generated by forms $f_1,\ldots,f_r$ of degree $d$. Then $I$
induces a rational map $\varphi_I: \PP^n \dashrightarrow \PP^{r-1}$ whose image we denote by $X_{n, [I]_d}$. It is a
projection of the $d$-uple Veronese embedding of $\PP^n$. Assume now that $S/I$ is Artinian. Let $I^{-1} \subset S$ be
the ideal generated by the Macaulay inverse system of $I$. The forms of degree $d$ in $I^{-1}$ induce a rational map
$\varphi_{I^{-1}}: \PP^n \dashrightarrow \PP^{\binom{n+d}{n} - r - 1}$ whose image we denote by $X_{n,[I^{-1}]_d}$. Note
that in the case, where $I \subset S$ is an Artinian monomial ideal, the ideal $I^{-1}$ is generated by the monomials of
$S$ that are not in $I$. We now give a quantitative version of \cite[Theorem~3.2]{MMO}, which directly follows from the
arguments for the original statement.

\begin{theorem} \label{thm:mmo}
    Let $I$ be an Artinian ideal of $S = K[x_0, \ldots, x_n]$ that is minimally generated by $r \leq \binom{n+d-1}{n-1}$
    forms of degree $d$. Then the following conditions are equivalent:
    \begin{enumerate}
        \item The ideal $I$ fails  the weak Lefschetz property in degree $d-1$ by $\delta$.
        \item The variety $X_{n,[I^{-1}]_d}$ satisfies $\delta$ Laplace equations of order $d-1$.
    \end{enumerate}
\end{theorem}

Note that the assumption $r \leq \binom{n+d-1}{n-1}$ simply ensures that the variety $X_{n,[I^{-1}]_d}$ does not trivially
satisfy a Laplace equation in the sense of Remark~\ref{rem:Laplace-trivially}. The assumption also implies that $R/I$
has the weak Lefschetz property in degree $d-1$ if and only if the multiplication map
$\times \ell: [R/I]_{d-1} \to [R/I]_d$ is injective.

Following \cite{MMO}, an ideal $I$ is said to define a \emdx{Togliatti system} if it satisfies the two equivalent
conditions in Theorem~\ref{thm:mmo}. The name is in honor of Togliatti who proved in \cite{To} that the only example for
$n = 2$ and $d = 3$ is $I = (x_0^3, x_1^3, x_2^3, x_0 x_1 x_2)$. \smallskip

Now we restrict ourselves to the case of three variables, i.e., $n = 2$. An Artinian monomial ideal $I$ of $R = K[x, y, z]$
defines a Togliatti system if it is generated by $r \leq d+1$ monomials of degree $d$ and it fails the weak
Lefschetz property in degree $d-1$. This corresponds precisely to the triangular regions $T \subset \mathcal{T}_{d+1}$
with $r \leq d+1$ unit punctures, which include the three upward-pointing unit triangles in each corner of ${\mathcal T}_{d+1}$,
and a bi-adjacency matrix $Z(T)$ that does not have maximal rank.

Hence, often a monomial ideal $I \subset R$ that fails the weak Lefschetz property in degree $d-1$ gives rise to a
Togliatti system by using Proposition~\ref{cor:unit-reduc}.

\begin{theorem}\label{thm:laplace-eq}
    Let $I \subset R$ be an Artinian monomial ideal, generated by monomials of degree at most $d$. Assume $h_{R/I} (d-1) \leq h_{R/I} (d)$
    and $R/I$ fails the weak Lefschetz property in degree $d-1$ by $\delta > 0$. Then, for every
    reduction $\hat{T} = T_{d+1} (\hat{I})$ of $T = T_{d+1} (I)$ to unit punctures such that the ideal $\hat{I}$ is
    Artinian, $\hat{I}$ defines a Togliatti system. Moreover, the variety $X_{n,[\hat{I}^{-1}]_d}$ satisfies $\delta$
    Laplace equations of order $d-1$.
\end{theorem}
\begin{proof}
    Combine Theorem~\ref{thm:mmo} and Corollary~\ref{cor:reduction-unit-punct-alg}. Note that the latter also guarantees
    that there is always at least one reduction such that $\hat{I}$ is Artinian.
\end{proof}

Observe that the above $X_{n,[\hat{I}^{-1}]_d}$ is a toric surface since $\hat{I}^{-1}$ is a monomial ideal. All toric
surfaces are quasi-smooth by \cite[\$5.2]{GKZ}.

In order to illustrate the last result we exhibit a specific example.

\begin{example} \label{exa:unit-tog}
    Let $I = (x^d, y^d, z^d, xyz)$, for some $d \geq 3$. Then $T = T_{d+1}(I)$ is a balanced region. Moreover,
    $|\det{Z(T)}| = 0$ if $d$ is odd (see Proposition~\ref{pro:level-wlp}(i)) and $|\det{Z(T)}| = 2$ if $d$ is even.
    This is also proven in~\cite[Proposition~3.1]{CN-IJM}.

    Suppose $d$ is odd. Then the ideal $\hat{I} = (x^d, y^d, z^d, m_1, \ldots, m_{d-1})$, where $m_i = x^i y^j z^{d-i-j}$
    for some $1 \leq j \leq d-i-1$, defines a Togliatti system as $T_{d+1}(\hat{I})$ is a reduction of $T$ to
    unit punctures. It is obtained by picking an upward pointing triangle in each row of the puncture associated to
    $xyz$. Using instead a diagonal and two rows produces the region depicted in Figure~\ref{fig:togliatti}.
    \begin{figure}[!ht]
        \begin{minipage}[b]{0.48\linewidth}
            \centering
            \includegraphics[scale=1]{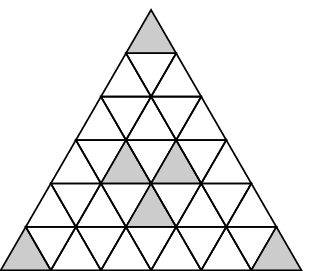}\\
            \emph{(i) $T_{6}(x^5, y^5, z^5, xy^2z^2, x^2yz^2, x^2y^2z)$}
        \end{minipage}
        \begin{minipage}[b]{0.48\linewidth}
            \centering
            \includegraphics[scale=1]{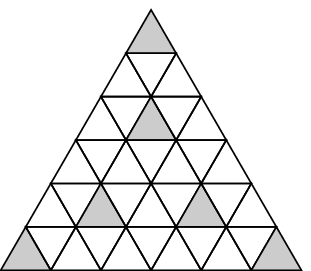}\\
            \emph{(ii) $T_{6}(x^5, y^5, z^5, xyz^3, xy^3z, x^3yz)$}
        \end{minipage}
        \caption{Triangular regions whose associated ideals define Togliatti systems.
            Both regions are formed by reducing $T_{6}(x^5, y^5, z^5, xyz)$.}
        \label{fig:togliatti}
    \end{figure}
    However, Figure~\ref{fig:togliatti}(ii) is not formed by using the procedure described above
    Proposition~\ref{cor:unit-reduc}. Instead, it is obtained by removing a tileable region from the central puncture.
\end{example}

We now describe the Artinian ideals of $R$ with few generators that define Togliatti systems. No such ideal exists with
three generators. Togliatti~\cite{To} proved that the only such ideal with four generators of degree three is
$(x^3, y^3, z^3, xyz)$. Moreover, Mezzetti, Mir\'o-Roig, and Ottaviani showed in \cite[Theorem~5.1]{MMO} that no ideal
of $R$ with four generators whose degree is at least four defines a Togliatti system. We now classify the Artinian
monomial ideal ideals with five generators that define Togliatti systems.

\begin{proposition}\label{prop:tog-5}
    Let $I \subset R$ be an Artinian ideal minimally generated by five monomials all of degree $d$. Then $I$ defines a
    Togliatti system if and only if, up to a change of variables, $I$ is either of the form $(x^4, y^4, z^4, x^2yz, y^2z^2)$
    or $(x^d, y^d, z^d, x^{d-1}y, x^{d-1}z)$ with $d \geq 4$.
\end{proposition}
\begin{proof}
    Since $I$ is Artinian, it must be of the form $I = (x^d, y^d, z^d, m, n)$, where $m$ and $n$ are monomials of degree
    $d$. The assumption forces $d \geq 2$.

    If $m = x^{d-1}y$ and $n = x^{d-1}z$, then the residue class of $x^{d-1}$ is a socle element of $R/I$. Hence, $R/I$
    fails the weak Lefschetz property in degree $d-1$ by Proposition~\ref{pro:wlp}(ii), as claimed. This covers in
    particular the case $d = 2$.

    Suppose $d = 7$. Then there are $\binom{7+2}{2}-3 = 33$ monomials of degree $d$ in $R \setminus \{x^d, y^d, z^d\}$.
    Thus, there are $\binom{33}{2} = 528$ choices for picking the two monomials $m \ne n$. Checking all these ideals
    using a computer yields that the claim is true if $d=7$. Similarly, one checks that no such ideal exists if $3 \leq d \leq 6$.

    Suppose $d \geq 8$.  Write $m = x^a y^b z^c$ and $n = x^{\alpha} y^{\beta} z^{\gamma}$.  We consider two cases.
    \smallskip

    \emph{Case 1:} Assume $2 \leq a \leq \alpha$. By the beginning of the proof we may also assume $(a, \alpha) \neq (d-1, d-1)$.
    We choose in each row of $T = T_{d+1}$ above the row with the puncture associated to $m$ an upward pointing
    triangle as a new puncture, with two exceptions. If $a \neq \alpha$, then we do not pick a puncture in the row that
    contains the puncture to the monomial $n$. If $a = \alpha$, then we do not pick a puncture in the row adjacent to
    the row containing the punctures to $m$ and $n$. In both cases, we get a triangular region $T' \subset T$ that has
    $d-2-a$ more unit punctures than $T$. In fact, $T'$ is a reduction to unit punctures of $T'' = T_{d+1}(x^a, y^d, z^d)$.
    By Proposition~\ref{pro:ci-wlp}, $Z(T'')$ has maximal rank. (Recall that we assume $\charf K = 0$.) Hence,
    $Z(T')$ has maximal rank, and so does $Z(T)$ as $\det Z(T')$ is a maximal minor of $Z(T)$. This concludes the first case.
    \smallskip

    Case 1 takes care of the situations where both numbers in the pairs $a, \alpha)$, $(b, \beta)$, or $(c, \gamma)$
    are at least two. Thus, without loss of generality it remains to consider the following case.
    \smallskip

    \emph{Case 2:} Assume $a, b \leq 1$ and $\gamma \leq 1$. Then we can pick as a new puncture a unit triangle whose
    label is a multiple of $z^{d-2}$ such that the resulting region $T'$ is a reduction of unit punctures of
    $T'' =T_{d+1}(x^d, y^d, z^{d-2}, n)$.

    We may also assume that $\alpha \geq \beta$. Thus $\alpha \geq \frac{d-1}{2}$, so $\alpha \geq 4$. Now we pick a new
    puncture in each row of $T''$ above the row containing the puncture to $n$. Call the result $\tilde{T}$. Then
    $\tilde{T}$ is a reduction to unit punctures of $T_{d+1} (x^{\alpha}, y^d, z^{d-2})$. Using again
    Proposition~\ref{pro:ci-wlp}, it follows that $Z(\tilde{T}), Z(T''), Z(T')$, and thus $Z(T)$ have maximal rank.
\end{proof}

\begin{remark}\label{rem:tog-5}
    \begin{enumerate}
        \item The above result for $d =4$ was shown independently by Di Gennaro, Ilardi, and Vall\`es \cite[Theorem 3.2]{GIV}.
        \item Notice that $T = T_5 (x^4, y^4, z^4, x^2yz, y^2z^2)$ is mirror symmetric with three odd axial punctures. Hence
            $\det Z(T) = 0$ by Theorem~\ref{thm:mirror-odd23}, giving a direct argument that the ideal $(x^4, y^4, z^4, x^2yz, y^2z^2)$
            defines a Togliatti system.
    \end{enumerate}
\end{remark}

In Example~\ref{exa:unit-tog}, we only showed that the bi-adjacency matrices do not have maximal rank. We now consider
how much maximal rank can fail. To this end we construct ideals that have the weak Lefschetz property in all degrees,
except $d-1$, and give rise to varieties satisfying many Laplace equations of order $d-1$.

We begin with a general construction that proves useful for generating such examples. We note that this construction is
based on modifications of triangular regions that are similar to the techniques used in Subsection~\ref{sub:red-unit-punct}.

\begin{proposition}\label{lem:build-I-Jdjk}
    Let $J$ be an Artinian ideal that is minimally generated by $m$ monomials of degree $e \leq d$. Let $j$ and $k$ be integers such that
    \[
        1 \leq j \leq \min \left \{ \frac{d-1}{m}, \frac{d+1}{e+1} \right \} \quad \text{and } \quad  0 \leq k \leq d-mj-1.
    \]
    Define an ideal $I = I_{J,d,j,k}$ by
    \begin{equation*}
         I_{J,d,j,k}   = J \cdot  x^{d+1 - (e+1) j} \cdot (x^{e+1}, y^{e+1})^{j-1} + (y^d) + z^{m j + k+ 1} \cdot (y, z)^{d-1 - mj - k}.
    \end{equation*}
    Then the following statements are true.
    \begin{enumerate}
        \item The ideal $I$ is generated by $d+1-k$ monomials of degree $d$, so $T_{d+1}(I)$ has only unit punctures.
        \item For $i \in \{1,2\}$, the rows and columns of $Z(T_{d+i}(I))$ can be rearranged so that it becomes a block
            matrix of the form
            \[
                \left(
                    \begin{array}{cccccc}
                        X_{0, i}& X_{1, i}      & X_{2, i}      & X_{3, i}& \cdots        & X_{j, i}  \\
                        0       & Z(T_{e+i}(J)) & 0             & 0       & \cdots        & 0 \\
                        0       & 0             & Z(T_{e+i}(J)) & 0       & \cdots        & 0 \\
                        \vdots  & \vdots        & \vdots        & \ddots  & \ddots        & \vdots \\
                        0       & 0             & 0             & \cdots  & Z(T_{e+i}(J)) & 0 \\
                        0       & 0             & 0             & \cdots  & 0             & Z(T_{e+i}(J))
                    \end{array}
                \right),
            \]
            where the matrix $X_{0, i}$ has maximal rank.
    \end{enumerate}
\end{proposition}
\begin{proof}
    The definition of $I$ gives that $I$ is  generated by $ mj + 1 + d - mj - k = d + 1 - k$ monomials of degree $d$.

    By Proposition~\ref{pro:ci-wlp}, the bi-adjacency matrix of the complete intersection region
    $T'' = T_{d+1}(x^{d+1-(e+1) j}, y^d, z^{mj+k+1})$ has maximal rank. We compare this region with the region to the ideal
    \[
        I' = x^{d+1 - (e+1) j} \cdot (x^{e+1}, y^{e+1})^{j-1} + (y^d) + z^{m j + k+ 1} \cdot (y, z)^{d-1 - mj - k}.
    \]
    The region $T' = T_{d+1}(I')$ has $j$ non-overlapping punctures of side length $e+1$ along the top-left edge and
    $d-mj - k$ non-overlapping unit punctures along the bottom edge. It contains $T''$. In fact, $T''$ can be obtained
    from $T'$ by removing uniquely tileable regions, namely rhombi. Applying Proposition~\ref{prop:Z-subregion}
    repeatedly, we get that $Z_{d+1} (T')$ has maximal rank.

    Replacing each puncture of side length $e+1$ of $T'$ by $T_{e+1}(J)$ produces the region $T = T_{d+1} (I)$. Each
    such replacement amounts to removing from $T_{e+1}(J)$ all the present triangles. (See
    Figure~\ref{fig:large-kernel-13-2-0} for an illustration.) Thus, Proposition~\ref{prop:Z-subregion} applies again.
    This proves claim (ii) if $i = 1$. Observe that $X_{0, 1} = T_{d+1} (I')$.

    For $i = 2$, we argue similarly, using the fact that the overlapping subregions $T_{e+2}(J)$ of $T_{d+2} (I)$
    overlap in a unit puncture. Here, the assumption that $J$ is an Artinian ideal is important. It implies that the
    unit triangles in the corners of $T_{e+i} (J)$ are punctures.
\end{proof}

Choosing a suitable ideal $J$ in the above result we can construct the desired Togliatti systems.

\begin{corollary}\label{cor:build-large-kernel}
    Consider the ideal $J = (x^3, y^3, z^3, xyz)$. Let $d \geq 5$, $j$, and $k$ be integers satisfying
    $1 \leq j \leq \frac{d-1}{4}$, and $0 \leq k \leq d-4j-1$. Then $Z(T_{d+1}(I_{J,d,j,k}))$ is
    $\binom{d+1}{2} \times \left(\binom{d+1}{2} + k\right)$ matrix of $\rank Z \leq \binom{d+1}{2} - j$, i.e., it fails
    to have maximal rank by $j$.

    Moreover, for $k = 0$, the bi-adjacancy matrix $Z(T_{d+2}(I_{J,d,j,0}))$ has maximal rank.
\end{corollary}
\begin{proof}
    Set $I = I_{J,d,j,k}$, and $T = T_{d+1}(I)$.

    By Proposition~\ref{lem:build-I-Jdjk}(i), $I$ is generated by $d+1-k$ monomials of degree $d$. Hence
    $h_{R/I}(d-1) = \binom{d+1}{2}$ and $h_{R/I}(d) = \binom{d+2}{2} - (d+1-k) = \binom{d+1}{2} + k$, and so $Z(T)$ is a
    $\binom{d+1}{2} \times \left(\binom{d+1}{2} + k\right)$ matrix. Since $Z(T_{4}(J))$ is a $6 \times 6$ matrix of rank
    $5 < 6$ and there are $j$ copies of $Z(T_4 (J))$ along the block diagonal of $Z(T)$ by
    Proposition~\ref{lem:build-I-Jdjk}(ii), it follows that $\rank Z(T) = \binom{d+1}{2} - j$.

    Suppose now $k = 0$, and consider $T' = T_{d+2}(I)$. Using the notation of Proposition~\ref{lem:build-I-Jdjk}, note
    that $X_{0, 2} = Z( T_{d+2} (I'))$, where
    \[
        I' = x^{d+1 - 4 j} \cdot (x^{4}, y^{4})^{j-1} + (y^d) + z^{4 j + 1} \cdot (y, z)^{d-1 - 4 j}.
    \]

    The bi-adjacency matrix of the complete intersection region $T'' = T_{d+2} ( x^{d+1 - 4 j} ,y^d, z^{4 j + 1})$ has
    two more rows than columns. Since $T''$ is obtained from $T_{d+2} (I')$ by removing balanced regions, also
    $X_{0, 2} = Z( T_{d+2} (I'))$ has two more rows than columns. The matrix $T_5 (J)$ is a $6 \times 3$ matrix of maximal rank.
    Hence Proposition~\ref{lem:build-I-Jdjk}(ii) proves that $Z(T_{d+2} (I))$ has maximal rank as well.
\end{proof}

We illustrate the regions involved in the last statement in a specific case.

\begin{example} \label{ex:laplace-eq}
    Let $I = I_{J,d,j,k}$ be an ideal as in the preceding corollary, where $d = 13,\; j = 2$, and $k = 0$. The related
    triangular regions are depicted below.

    \begin{figure}[!ht]
        \begin{minipage}[b]{0.48\linewidth}
            \centering
            \includegraphics[scale=1]{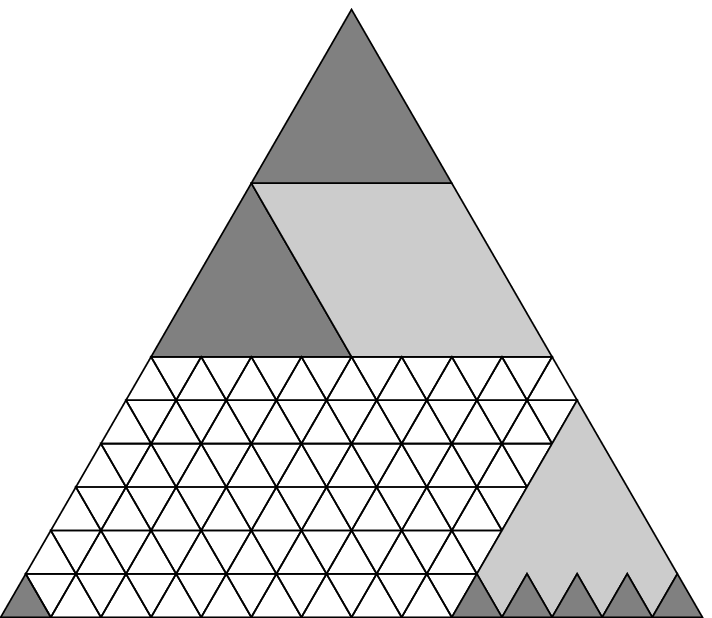}\\
            \emph{(i) $T_{14}(x^6(x^4,y^4) + y^{13} + z^9(y,z)^4)$ with $T_{14}(x^6, y^{13}, z^9)$ highlighted}
        \end{minipage}
        \begin{minipage}[b]{0.48\linewidth}
            \centering
            \includegraphics[scale=1]{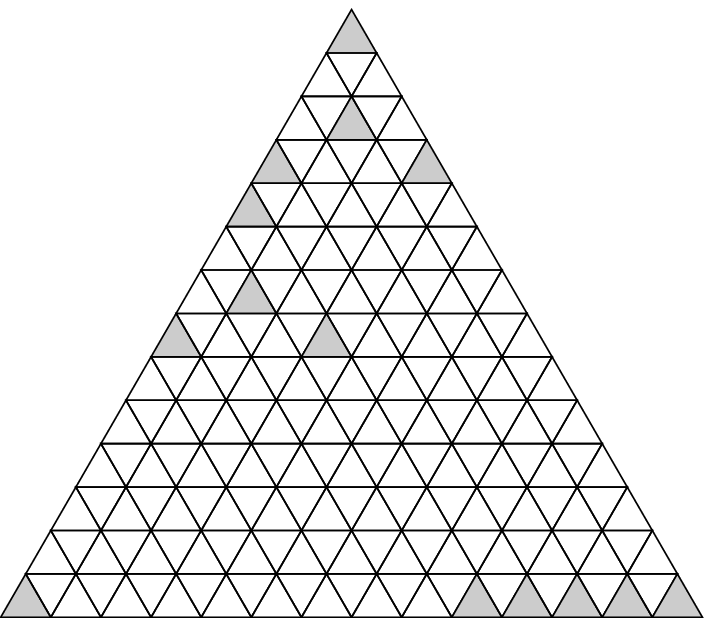}\\
            \emph{(ii) $T = T_{14}(I_{J,13,2,0})$}
        \end{minipage}
        \caption{Construction of the region $T = T_{14}(I_{J,13,2,0})$, where $J = (x^3, y^3, z^3, xyz)$.
            The matrix $Z(T)$ fails to have maximal rank by $2$.}
        \label{fig:large-kernel-13-2-0}
    \end{figure}
\end{example}

Collecting the results about the ideals $I = I_{J,d,j,k}$ in the case $k = 0$, we get the following consequence.

\begin{corollary}\label{cor:many-laplance-eq}
    Let $d$ and $j$ be  integers such that $1 \leq j \leq \frac{d-1}{4}$, and consider the ideal
    \[
        I_j = (x^3, y^3, z^3, xyz) \cdot x^{d+1 - 4j} \cdot (x^4, y^4)^{j-1} + (y^d) + z^{4j+1} (y, z)^{d-1 - 4j} .
    \]
    Then the following statements hold.
    \begin{enumerate}
        \item The algebra $R/I_j$ has the weak Lefschetz property in every degree $i \neq d-1$.

            More precisely, for a general linear form $\ell$, the map $\times \ell: [R/I_j]_{i-1} \to [R/I_j]_i$ is
            injective if $i < d$, and it is surjective if $j > d$.

        \item  In degree $d-1$, $R/I_j$ fails the weak Lefschetz property by $j$.

        \item The variety $X_{n,[(I_j)^{-1}]_d}$ satisfies exactly $j$ Laplace equations of order $d-1$.
    \end{enumerate}
\end{corollary}
\begin{proof}
    Combining Corollary \ref{cor:build-large-kernel} (with $k = 0$) and Proposition~\ref{pro:mono} shows assertions (i)
    and (ii). Notice that $\dim_K [R/I_j]_{d} = \dim_K [R/I_j]_{d+1} + 2$.

    Now Theorem~\ref{thm:laplace-eq} gives (iii) because $\dim_K [R/I_j]_{d-1} = \dim_K [R/I_j]_{d}$.
\end{proof}

It should be noted that one can produce more families with unexpected properties by varying the ideal $J$ used in
Proposition~\ref{lem:build-I-Jdjk}. We illustrate this by constructing a family of algebras that fail the weak Lefschetz
property in two consecutive degrees, where in one degree injectivity was expected and in the other degree surjectivity
was expected.

\begin{corollary}\label{cor:build-small-image}
    Let $d, j$ and $k$ be integers such that $1 \leq j \leq \frac{d-1}{6}$ and $0 \leq k \leq \min\{2j+2, d-6j-1\}$.
    Consider the ideal
    \begin{equation*}
        I_{j,k} = J \cdot x^{d+1 - 4 j} \cdot (x^{4}, y^{4})^{j-1} + (y^d) + z^{6 j + k+ 1} \cdot (y, z)^{d-1 - 6j - k},
    \end{equation*}
    where
    \[
        J = (x^3, y^3, z^3, x^2y, xz^2, y^2z).
    \]
    Then the algebra $R/I_{j, k}$ has the following properties.
    \begin{enumerate}
        \item $\dim_K [R/I_{j, k}]_d = \binom{d+1}{2} + k \geq \binom{d+1}{2} = \dim_K [R/I_{j, k}]_{d-1}$, and
            $R/I_{j,k}$ fails the weak Lefschetz property by at least $2j$.
        \item $\dim_K [R/I_{j, k}]_{d+1} = \dim_K [R/I_{j, k}]_{d} - (2j - k +2) \leq \dim_K [R/I_{j, k}]_{d}$, and
            $R/I_{j, k}$ fails the weak Lefschetz property by $2j + k -2$.
    \end{enumerate}
\end{corollary}
\begin{proof}
    The arguments are similar to the ones used in the proof of Corollary~\ref{cor:build-large-kernel}. Thus, we mainly
    restrict ourselves to mentioning some of the differences.

    We apply Proposition~\ref{lem:build-I-Jdjk} with $e = 3$ and $m = 6$. Note that, for $i \in \{1, 2\}$, the matrix
    $X_{0, i}$ is $X_{0, i} = T_{d+i} (I')$, where
    \[
        I' = x^{d+1 - 4 j} \cdot (x^{4}, y^{4})^{j-1} + (y^d) + z^{6 j + 1 + k} \cdot (y, z)^{d-1 - 6 j - k}.
    \]
    The matrix $X_{0, 1}$ has maximal rank and $2j + k$ more columns than rows. The $6 \times 4$ matrix $Z (T_4(J))$ has
    rank four. Hence, Proposition~\ref{lem:build-I-Jdjk}(ii) shows that $Z(T_{d+1} (I_{j, k}))$ fails to have maximal
    rank by at least $2j$. This proves claim (i).

    Consider now the matrix $X_{0, 2}$. It has maximal rank and $2j + k - 2$ more columns than rows. Since $Z (T_5(J))$
    is a $4 \times 0$ matrix, it follows that the matrix $Z(T_{d+2} (I_{j, k}))$ is obtained from the matrix $X_{0, 2}$
    by appending $4j$ zero rows. Therefore, $Z(T_{d+1} (I_{j, k}))$ fails to have maximal rank by $2j + k -2$.
\end{proof}

\begin{remark}
    The algebras in Corollaries~\ref{cor:build-large-kernel} and \ref{cor:build-small-image} do not have low degree
    socle elements in the sense of Proposition~\ref{pro:wlp}(i). The degrees of their socle elements are at least $d$.
    Thus, the weak Lefschetz property fails for reasons \emph{other} than having socle elements of small degree.
\end{remark}~

\subsection{Large prime divisors}\label{sub:large-p}~\par

Let $T = T_d(I)$ be a balanced triangular region, and consider $\det Z(T)$ as an integer. If it is not zero, then, by
Corollary~\ref{cor:wlp-biadj}, the algebra $R/I$ fails the weak Lefschetz property in degree $d-1$ if and only if the
characteristic $p > 0$ of the base field divides $\det Z(T)$. Throughout Sections~\ref{sec:det}--\ref{sec:amaci}, every
time we are able to provide an upper bound on the prime divisors of $\det{Z(T)}$, it has been at most $d-1$. However,
this is not always true. We have numerous examples where, for small $d$, the prime factors of $\det{Z(T)}$ can be quite
large relative to $d$.

\begin{example}\label{exa:huge-primes}
    We provide Artinian monomial ideals $I$ and integers $d$ such that $0 \neq \det{Z(T_d(I))}$ has a prime divisor that
    is not less than $d$.
    \begin{enumerate}
        \item The smallest $d$ such that $\det{Z(T_d(I))}$ is nonzero and divisible by some prime $p \geq d$ is $d=5$.
            Indeed, if  $I = (x^3, y^4, z^5, xz^3, y^2z^2)$, then $|\det{Z(T_5(I))}| = 5$.
        \item The smallest $d$ such that $\det{Z(T_d(I))}$ is nonzero and divisible by some prime $p > d$ is $d=6$.
            If $I = (x^5, y^5, z^5, x^3y^2, x^2z^3, y^3z^2)$, then $|\det{Z(T_6(I))}| = 35 = 5 \cdot 7$.
        \item A possibly surprising example is $I = (x^{20}, y^{20}, z^{20}, x^3 y^8 z^{13})$ with $d = 28$.  In this case,
            \[
                |\det{Z(T_{28}(I))}| = 2\cdot 3^{2}\cdot 5^{3}\cdot 7\cdot 11\cdot 17^{2}\cdot 19^{6}\cdot 23^{5}\cdot 20554657.
            \]
            Here $\det{Z(T_d(I))}$ is divisible by a prime, $20554657$, which is over $700,000$ times larger than $d$.
        \item In the previous examples, the determinants have only a single prime factor that is larger than $d$. In
            general, there can be more such prime divisors. Indeed, let $I = (x^7, y^{12}, z^{15}, x y^7 z^2)$ and $d = 14$.
            Then $|\det{Z(T_{14}(I))}| = 13 \cdot 17 \cdot 23$.
    \end{enumerate}
    Note that, for a fixed integer $d$, there are finitely many Artinian monomial ideals whose generators have degrees
    at most $d$. We used \emph{Macaulay2}~\cite{M2} to search this finite space to find results such as (i) and (ii) above.
\end{example}

It would be desirable if, at least, there is a uniform upper bound on the prime divisors of $\det{Z(T)}$ that is linear
in $d$. However, this does not appear to be the case in general. The following example suggests that some prime divisors
of $\det{Z(T)}$ can be of the order $d^2$.

\begin{example} \label{exa:non-linear-bound}
    For $t \geq 4$, consider the level and type $3$ algebra  $R/I_t$, where
    \[
        I_t = (x^{1+t}, y^{4+t}, z^{7+t}, x y^4 z^7).
    \]
    In~\cite[Section~6]{CEKZ}, Ciucu, Eisenk\"olbl, Krattenthaler, and Zare argue that the determinant of a punctured
    hexagonal subregion of ${\mathcal T}_d$ is polynomial in $d$, if only the side length of the central puncture
    increases with $d$, whereas the size of the other punctures remains constant. Thus, one can use interpolation to
    determine this polynomial. Applying this procedure to $T_{t+8}(I_t)$, we get that $\det{Z(T_{t+8}(I_t))}$ is
    \[
        \frac{4}{\HF(7)} \cdot \left\{
            \begin{array}{ll}
                (t-3) (t-2) (t-1)^3 t^3 (t+1)^2 (t+2) (t+4) (t+6) (t^2 + 6t - 1) & \mbox{ if $t$ is odd;} \\[0.3em]
                (t-2)^2 (t-1)^2 t^4 (t+1)^2 (t+2) (t+5) (t+7) (t^2 + 2t - 9) & \mbox{ if $t$ is even.}
            \end{array}
        \right.
    \]

    We now recall a number-theoretic conjecture. Let $f \in \ZZ[t]$ be an irreducible polynomial whose degree is
    at least 2, and set $D = \gcd\{f(i) \st i \in \ZZ\}$. In this case, Bouniakowsky conjectured in 1857~\cite{Bou} that
    there are infinitely many integers $t$ such that $\frac{1}{D}f(t)$ is a prime number. We note that the weaker
    Fifth Hardy-Littlewood conjecture is a special case of the Bouniakowsky conjecture. It posits that $t^2 + 1$ is
    prime for infinitely many positive integers $t$.

    Observe that the quadratic factors of the above determinant, $t^2 + 6t - 1$ and $t^2 + 2t -9$, respectively, are
    irreducible polynomials in $\ZZ[t]$. Thus, for infinitely many positive integers $t$, the above determinant has a
    prime divisor of order $t^2$ if Bouniakowsky's conjecture is true.

    It follows that, assuming the Bouniakowsky conjecture, the above ideals provide regions $T_d \subset {\mathcal T}_d$
    such that $\det Z(T_d) \neq 0$ has a prime divisor of order $d^2$ for infinitely many integers $d$.
\end{example}

Given the above examples, it seems unlikely that the prime divisors of $\det{Z(T_d(I))}$ are bounded linearly in $d$.

\section{Further open problems} \label{sec:open-problems}

In this closing section we wish to point out some additional questions and problems that are suggested by this work.
\smallskip

In Section~\ref{sec:signed} we introduced for \emph{every} non-empty subregion $T \subset {\mathcal T}_d$ its bi-adjacency
matrix $Z(T)$ and its perfect matching matrix $N(T)$. These matrices are square matrices if and only if the region $T$ is
balanced. According to Theorem~\ref{thm:detZN}, the determinants of $N(T)$ and $Z(T)$ have the same absolute value if $T$
is a balanced triangular subregion, i.e., its punctures are upward-pointing triangles. However, we are not aware of any
example, where the mentioned equality is not true. This raises the following question:

\begin{question}\label{qe:det}
    Let $T \subset {\mathcal T}_d$ be any non-empty balanced subregion. Is then the equality
    \[
        |\det Z(T)| = |\det N(T)|
    \]
    always true?
\end{question}

An affirmative answer would extend the bijection between signed perfect matchings determined by $T$ and signed
families of non-intersecting lattice path in $T$ from triangular subregions to arbitrary balanced subregions.
\smallskip

In this work we have focussed on studying the weak Lefschetz property of an Artinian monomial ideal
$I \subset R = K[x, y, z]$ by establishing combinatorial interpretations of the multiplications maps
$[A]_{d-1} \to [A]_d$ on $A = R/I$ by $\ell = x+ y + z$.  In order to study the strong Lefschetz property of
$A$ it would be desirable to extend our approach by finding combinatorial interpretations of the multiplication
by powers of $\ell$.  This could also be an approach to determining the Jordan canonical form of the multiplication
map by $\ell$ on $A$.

Quite generally, let $A$ be a graded Artinian $K$-algebra with Hilbert function $h_A (j) = \dim_K [A]_j$. The multiplication
on $A$ by any linear form $0 \neq L \in A$ is a nilpotent map. Hence, the Jordan canonical form $J_L$ of this multiplication map
is given by a partition $\lambda$ of $\dim_K A$. The parts of this partition are determined by the ranks of the multiplication
by powers of $L$.  More precisely, define $m = \max \{h_A (i) \st i \in \ZZ\}$ and, for all $j = 1,\ldots,m$,
\[
    \lambda_j = |\{i \st h_A (i) \ge j\}|.
\]
Then $\lambda = (\lambda_1,\ldots,\lambda_m)$ is called the \emph{expected partition}, determining the expected Jordan canonical
form $J_{\ell}$ of the multiplication by a general linear form $\ell \in A$. This is closely related to the strong Lefschetz property.
In fact, the algebra $A$ has the strong Lefschetz property if and only if the Jordan canonical form of the multiplication on $A$ by
a general linear form is determined by the expected partition.

Since monomial complete intersections have the strong Lefschetz property in characteristic zero this observation allows us to
determine the corresponding Jordan canonical form.

\begin{example}
    Consider the algebra $A = R/I$, where $I = (x^3, y^4, z^5)$ and $K$ is a field of characteristic zero. Its Hilbert function is
    positive from degree zero to degree nine. The corresponding values are
    \[
        1, 3, 6, 9, 11, 11, 9, 6, 3, 1.
    \]
    Hence, the Jordan canonical form of the multiplication on $A$ by $\ell = x+y+z$ (or every general linear form) is given by the
    expected partition of $60$,
    \[
        \lambda = (10, 8, 8, 6, 6, 4, 4, 4, 2, 2).
    \]
\end{example}
\smallskip

The richness of the results on lozenge tilings, perfect matchings, and families of non-intersection lattice paths invites one
to find higher-dimensional generalizations. MacMahon already considered the three-dimensional analogue of plane partitions
in~\cite{MacMahon-60}. However, numerical evidence suggests that there is no simple formula for enumerating such space partitions.
Nevertheless, parts of the theory we developed here do extend to higher dimension. We hope that the connection to the weak Lefschetz
property of quotients of polynomial rings in more than three variables can provide some guidance towards extending some of the
beautiful classical combinatorial results.

\begin{acknowledgement}
    The authors acknowledge the invaluable help of the computer algebra system \emph{Macaulay2}~\cite{M2}. It was used
    extensively throughout both the original research and the writing of this manuscript.
%

    This work supersedes~\cite{CN-2011}. The authors thank a referee of~\cite{CN-2011} for very valuable comments.
\end{acknowledgement}~

\appendix
\section{Completion of the proof of Theorem \ref{thm:ciucu-corrected}.}

Adopt the notation introduced in Chapter \ref{sec:mirror}. In order to prove Theorem~\ref{thm:ciucu-corrected} it
remains to show the following result.

\begin{proposition} \label{prop:ciucu-corrected}
    Let $T = T_d(I)$ be a triangular region as in Assumption~\ref{assump:mirror} with parameters $b$ and
    $(h_1, d_1), \ldots, (h_s, d_s)$, where $s \geq 1$ and $d_2, \ldots, d_{s-1}$ are all even. Define $a$, $k$,
    $\mathbf{p}$, and $\mathbf{q}$ as introduced above Theorem~\ref{thm:ciucu-11}.

    If $d_1$ is even, $d_s$ is odd, and $h_s = 0$, then
    \[
       \per{Z(T)} =
        2^{m+n} \overline{P}_{\mathbf{\emptyset,q}} \left(\frac{a+k-1}{2}\right) \overline{P}_{\mathbf{q, \emptyset}} \left(\frac{a}{2} \right).
    \]
\end{proposition}
\begin{proof}
    This follows as in~\cite[Part~B, Section~3, Proof of Theorem~1.1]{Ci-2005} if we make a single adjustment.

    In order to refrain from making copious new definitions, we assume the reader is familiar with the notation used
    in~\cite[Part~B]{Ci-2005}.

    The penultimate sentence on page~123 of~\cite{Ci-2005} ends with the phrase ``while the one obtained from $R^{-}$
    after the same procedure is congruent to $R_{\mathbf{q,l}^{(m)}}(a/2)$'' (recall that we use $\mathbf{p}$ in place
    of $\mathbf{l}$, as discussed above). This congruence hinges on the assumption given in the second sentence of the
    same paragraph: ``Then the last vertebra below $\mathcal{R}$ is in this case a triangular vertebra$\ldots$''
    However, when $\mathbf{p} = \emptyset$, this assumption fails, and the stated congruence breaks down (see, e.g.,
    Figure~\ref{fig:ciucu-12-cn-8}). Fortunately, we need simply replace
    $R_{\mathbf{q,l}^{(m)}}(a/2) = R_{\mathbf{q},\emptyset}(a/2)$ by $\overline{R}_{\mathbf{q},\emptyset}(a/2)$ for the
    congruence to hold.

    Set $R = R_{\mathbf{q},\emptyset}(a/2)$ and $\overline{R} = \overline{R}_{\mathbf{q},\emptyset}(a/2)$.  Then we will
    show that $R$ is $\overline{R}$ with an extra column of triangles along the Northwestern boundary and an extra row
    of triangles along the Northern boundary, thus expanding the region considered. See Figure~\ref{fig:R-vs-overline-R}
    for an example of $R$ and $\overline{R}$.
    \begin{figure}[!ht]
        \begin{minipage}[b]{0.48\linewidth}
            \centering
            \includegraphics[scale=1]{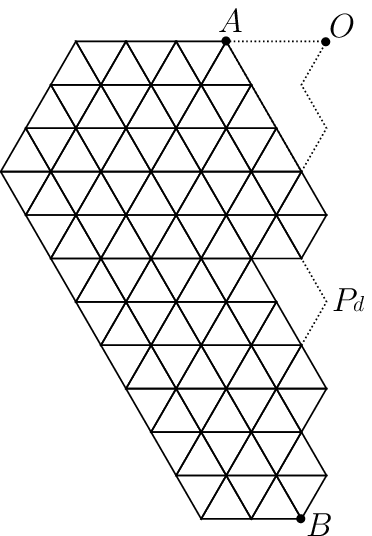}\\
            \emph{(i) $R_{\mathbf{q},\emptyset}(a/2)$}
        \end{minipage}
        \begin{minipage}[b]{0.48\linewidth}
            \centering
            \includegraphics[scale=1]{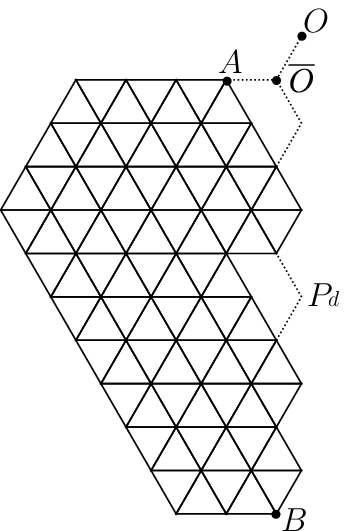}\\
            \emph{(ii) $\overline{R}_{\mathbf{q},\emptyset}(a/2)$}
        \end{minipage}
        \caption{An example of the distinction between $R = R_{\mathbf{q},\emptyset}(a/2)$ and
            $\overline{R} = \overline{R}_{\mathbf{q},\emptyset}(a/2)$, where $\mathbf{q} = \{2,4,5\}$ and $a = 4$.
            The region $R$ is the same as~\cite[Part~B, Figure~2.1(c)]{Ci-2005}, and $\overline{R}$
            is very similar to~\cite[Part~B, Figure~2.2(c)]{Ci-2005}.}
        \label{fig:R-vs-overline-R}
    \end{figure}

    When expanding $R$ and $\overline{R}$ to the symmetric regions of which they are part, we need only consider those
    triangles that are below the horizontal line that goes through the origin $O$ since $\mathbf{p} = \emptyset$.
    Further, the first selected bump is $q_1$, and as $d_s$ is odd, we have $q_1 = 1 + \flfr{d_s}{2} = \frac{1}{2}(d_s+1)$.

    When expanding $R$ to the symmetric region $T$ which it is part of, we notice that the top of the $q_1$ bump is
    $d_s$ triangles below $O$. Thus the Northwest path coming from the top of the bump joins the Northwest edge of the
    bump, and this ray extends to the horizontal ray to the West of $O$. This creates a Northwest edge that is $d_s+1$
    units long. Hence the lowest axial puncture of $T$ has side-length $d_s+1$. In particular, $T = T_{d+1}(\hat{I})$,
    where $\hat{I}$ is generated by the same generators as $I$, but with the following modifications: all generators
    divisible by $x$ are multiplied by $x$ and the pure powers of $y$ and $z$ are multiplied by $y$ and $z$,
    respectively. See Figure~\ref{fig:R-in-T-T-hat}(i).

    \begin{figure}[!ht]
        \begin{minipage}[b]{0.48\linewidth}
            \centering
            \includegraphics[scale=1]{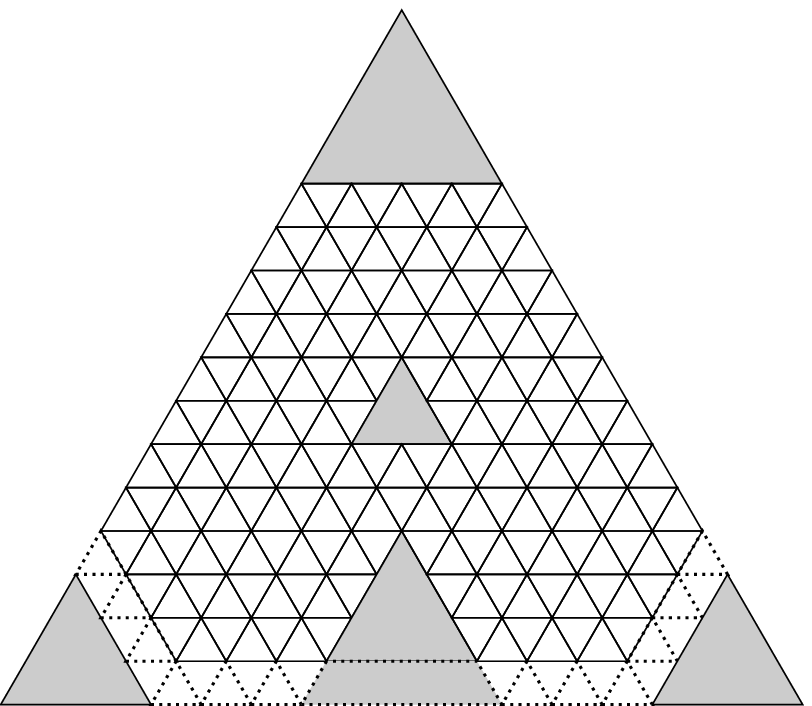}\\
            \emph{(i) $T$ coming from $R$ with the added triangles highlighted by dashed lines}
        \end{minipage}
        \begin{minipage}[b]{0.48\linewidth}
            \centering
            \includegraphics[scale=1]{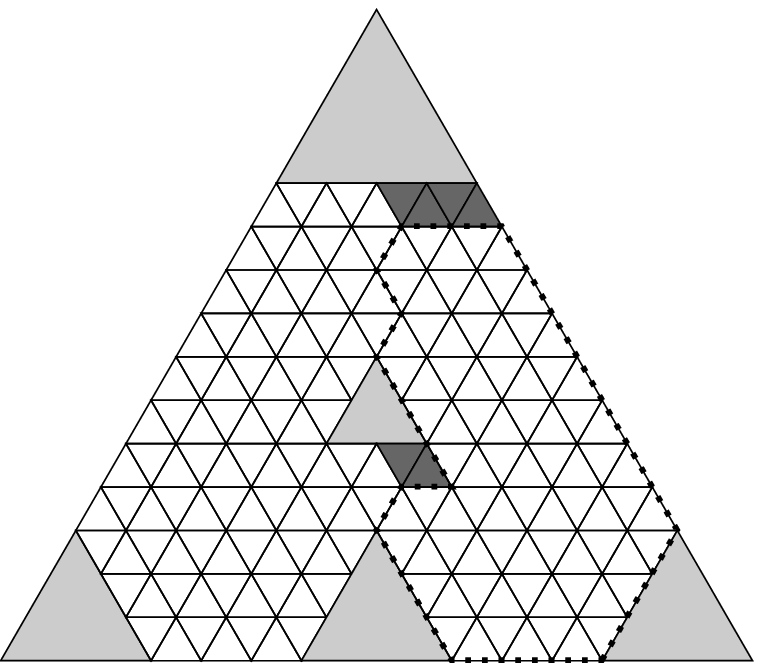}\\
            \emph{(ii) $\overline{T}$ coming from $\overline{R}$; fixed lozenges are darkened}
        \end{minipage}
        \caption{Let $\mathbf{q} = \{2,4,5\}$ and $a = 4$, as in Figure~\ref{fig:R-vs-overline-R}.
            Then $T = T_{16}(x^{12}, y^{13}, z^{13}, x^6 y^4 z^4, y^6 z^6)$ and
                $\overline{T} = T_{15}(x^{11}, y^{12}, z^{12}, x^5 y^4 z^4, y^6 z^6)$.}
        \label{fig:R-in-T-T-hat}
    \end{figure}

    On the other hand, when expanding $\overline{R}$ to the symmetric region $\overline{T}$ which it is part of, we
    notice that the top of the $q_1$ bump is $d_s-1$ triangles below $\overline{O}$. Thus the lowest axial puncture of
    $T$ has side-length $d_s$. Hence $\overline{T} = T_d(I)$, as desired. See Figure~\ref{fig:R-in-T-T-hat}(ii).

    The claim now follows by using the arguments in \cite[Part~B]{Ci-2005} and replacing the region $R$ by the region
    $\overline{R}$.
\end{proof}


\clearpage
\printindex

\end{document}